\documentclass[reqno,11pt]{amsart}

\usepackage{ulem, bbm, braket, stmaryrd, rotating}

\usepackage{amscd,latexsym,amsthm,amsfonts,amssymb,amsmath,amsxtra,dsfont,mathrsfs}

\usepackage{ytableau}
\ytableausetup{smalltableaux}

\usepackage{dynkin-diagrams}

\usepackage{multirow}
\renewcommand{\arraystretch}{1.5}

\usepackage[all]{xypic}

\usepackage[pdftex,bookmarks,colorlinks,pdfmenubar]{hyperref}

\usepackage{verbatim}

\newcommand{\bd}{\boldsymbol{\rm d}}

\newcommand{\BC}{{\mathbb {C}}}

\newcommand{\BN}{{\mathbb {N}}}

\newcommand{\BW}{{\mathbb {W}}}

\renewcommand{\CD}{{\mathcal {D}}}
\newcommand{\CE}{{\mathcal {E}}}
\newcommand{\CF}{{\mathcal {F}}}

\newcommand{\CL}{{\mathcal {L}}}

\newcommand{\CN}{{\mathcal {N}}}
\newcommand{\CO}{{\mathcal {O}}}

\newcommand{\CS}{{\mathcal {S}}}

\newcommand{\CU}{{\mathcal {U}}}

\newcommand{\CY}{{\mathcal {Y}}}

\newcommand{\Fa}{{\mathfrak {a}}}

\newcommand{\Fg}{{\mathfrak {g}}}

\newcommand{\Fl}{{\mathfrak {l}}}

\newcommand{\Fn}{{\mathfrak {n}}}

\newcommand{\RG}{{\mathrm {G}}}

\newcommand{\RO}{{\mathrm {O}}}

\newcommand{\RU}{{\mathrm {U}}}

\newcommand{\ScO}{{\mathscr {O}}}

\newcommand{\Ad}{{\mathrm{Ad}}}

\newcommand{\bb}{\mathbb}

\newcommand{\cal}{\mathcal}

\newcommand{\disc}{{\mathrm{disc}}}

\newcommand{\fg}{\mathfrak{g}}

\newcommand{\GL}{{\mathrm{GL}}}

\newcommand{\Hom}{{\mathrm{Hom}}}

\newcommand{\Ind}{{\mathrm{Ind}}}

\newcommand{\Ker}{{\mathrm{Ker}}}

\newcommand{\rank}{{\mathrm{rank}}}

\newcommand{\Res}{{\mathrm{Res}}}

\newcommand{\SL}{{\mathrm{SL}}}

\newcommand{\SO}{{\mathrm{SO}}}

\newcommand{\sgn}{{\mathrm{sgn}}}
\newcommand{\Sp}{{\mathrm{Sp}}}

\newcommand{\tr}{{\mathrm{tr}}}

\newcommand{\wt}{\widetilde}

\newcommand{\apair}[1]{\left\langle {#1} \right\rangle}
\newcommand{\bpair}[1]{\left[{#1}\right]}

\newcommand{\ppair}[1]{\left( {#1} \right)}

\def\diag{{\rm diag}}

\def\eps{{\epsilon}}

\newtheorem{thm}{Theorem}[section]
\newtheorem{cor}[thm]{Corollary}
\newtheorem{lem}[thm]{Lemma}
\newtheorem{prop}[thm]{Proposition}

\newtheorem {ques/conj}[thm]{Question/Conjecture}

\newtheorem{defn}[thm]{Definition}
\newtheorem{rmk}[thm]{Remark}

\newtheorem{exmp}[thm]{Example}
\newenvironment{sop}{\begin{proof}[\indent  Sketch of proof]}{\end{proof}}

\usepackage[numbers]{natbib}

\newcommand{\Irr}{{\rm Irr}}

\newcommand{\prll}{\pi_{\rho,\Lambda_1,\Lambda_{-1}}}
\newcommand{\prllz}{\pi_{\rho,\Lambda_1,\Lambda_{-1},\eta}}
\newcommand{\prllp}{\pi_{\rho',\Lambda_1',\Lambda'_{-1}}}
\newcommand{\prllps}{\pi_{\rho',\Lambda_1',\Lambda^*_{-1}}}

\newcommand{\prllpz}{\pi_{\rho',\Lambda_1',\Lambda'_{-1},\eta}}

\newcommand{\prllpt}{\pi_{\rho',\Lambda_1',\Lambda^{\prime t}_{-1}}}
\newcommand{\prllpst}{\pi_{\rho',\Lambda_1',\Lambda^{\star t}_{-1}}}

\newcommand{\prllpp}{\pi_{\rho'',\Lambda''_1,\Lambda''_{-1}}}
\newcommand{\prlls}{\pi_{\rho^\star,\Lambda^\star_1,\Lambda^\star_{-1}}}
\newcommand{\prllsp}{\pi_{\rho^{\star\prime},\Lambda^{\star\prime}_1,\Lambda^{\star\prime}_{-1}}}

\newcommand{\ep}{\epsilon_\alpha}
\newcommand{\ev}{\epsilon_\beta}
\newcommand{\epp}{\epsilon_{\alpha}'}
\newcommand{\evp}{\epsilon_{\beta}'}

\newcommand{\er}{\epsilon_\rho}
\newcommand{\erp}{\epsilon_{\rho'}}

\newcommand{\UU}{{\mathrm{U}}}

\newcommand{\llb}{\llbracket}
\newcommand{\rrb}{\rrbracket}

\newcommand{\SCO}{{\cal{O}^{\rm{st}}}}
\newcommand{\SCOP}{{\cal{O'}^{\rm{st}}}}
\newcommand{\SCOPS}{{\cal{O^{\prime }}^{\rm{st}*}}}
\newcommand{\SCOF}{{\cal{O}^{\rm{st}}_{\CF}}}
\newcommand{\mf}{{\mathcal{M}} ( {\mathcal{F}})}

\newcommand{\SOM}{\ScO(\pi)^{\max}}
\newcommand{\SOMS}{\ScO(\pi)^{\max, \rm{st}}}

\title{Rational wavefront sets and the descent method for finite classical groups}

 \author{Zhicheng Wang}
 \address{Department of Mathematics, Jilin University, Changchun 130012, Jilin, P. R. China}
 \email{wangzhicheng@jlu.edu.cn}
\date{\today}

\begin{document}

 \begin{abstract}

Let \(\mathbbm{k}\) be a finite field of sufficiently large odd
characteristic, let \(F\) be the Frobenius map, and let \(G\) be a
symplectic or orthogonal group defined over \(\mathbbm{k}\).  For an
irreducible representation \(\pi\) of \(G^F\), its rational wavefront
set \(\ScO(\pi)\) is the set of \(F\)-rational nilpotent orbits
\(\mathcal O\) such that \(\pi\) occurs in the generalized
Gelfand--Graev representation \(\Gamma_{\mathcal O}\).

We prove that all maximal rational orbits in \(\ScO(\pi)\) lie in a
single \(F\)-stable nilpotent orbit, and that the whole rational
wavefront set is obtained from its maximal rational orbits by
rational closure. We also give an explicit description of these maximal rational
orbits in terms of Lusztig parameters and compute the dimensions of
the corresponding generalized Whittaker models.

For unipotent representations, we prove that the maximal rational
wavefront set is a single fiber of Lusztig's canonical quotient.
For special unipotent representations, this fiber identifies the
kernel of the canonical quotient map, giving a finite-field
representation-theoretic realization of the canonical quotient.  We
also describe the effect of Alvis--Curtis duality on maximal rational
wavefront sets.  

 \end{abstract}

\maketitle
\tableofcontents

\section{Introduction}

Let $\mathbbm{k}$ be a finite field of cardinality $q$ with an algebraic
closure $\overline{\mathbbm{k}}$, and let $p$ be the characteristic of
$\mathbbm{k}$.  Let $F$ be the Frobenius map.
Let $V_\Fn$ be an $\Fn$-dimensional nondegenerate symmetric or
skew-symmetric space over $\mathbbm{k}$.  If $V_\Fn$ is symmetric,
define its discriminant by
\begin{equation}\label{disc1}
        \disc(V_\Fn)
        =
        (-1)^{\frac{\Fn(\Fn-1)}{2}}\det(V_\Fn)
        \in \mathbbm{k}^{\times}/\mathbbm{k}^{\times2}.
\end{equation}
Let $G$ be the isometry group of $V_\Fn$; thus $G$ is $\Sp_\Fn$ or
$\RO_\Fn$.  When $V_\Fn$ is symmetric, we write
\[
        G=\RO(V_\Fn)=\RO_\Fn^{\disc(V_\Fn)}.
\]
Thus the sign in $\RO_\Fn^\pm$ is the discriminant of the
underlying symmetric space.  With our convention for $\disc$,
$\RO_{2m}^+$ is the split even orthogonal group.

Let $G^\circ_{\overline{\mathbbm{k}}}$ denote the identity component
of the base change of $G$ to $\overline{\mathbbm{k}}$, and let
$h_G$ be the Coxeter number of its root system.  Throughout this
paper, we assume that
\[
        p>3(h_G-1).
\]

If $H$ is an $F$-stable subgroup of $G$, and if
$\pi\in\Irr(H^F)$ and $\sigma\in\Irr(G^F)$, define
\[
        \apair{\pi,\sigma}_{H^F}
        :=
        \dim\Hom_{H^F}(\pi,\sigma|_{H^F}).
\]
If $H=G$, we write $\apair{\pi,\sigma}$ instead of
$\apair{\pi,\sigma}_{G^F}$.  As usual, we use the same symbol for
a representation and its character when no confusion can arise.

We write $\Fg$ for the Lie algebra of $G$ and $\CN\subset\Fg$ for
the nilpotent cone.  We denote an $F$-stable nilpotent orbit in
$\CN$ by $\SCO$, and an $F$-rational nilpotent orbit in $\CN^F$ by
$\CO$.  For classical groups, $F$-rational nilpotent orbits are
parametrized by signed Young diagrams; the precise convention and
the rational orbit order are recalled in Section~\ref{lcq}.

We fix once and for all a nontrivial additive character
\[
        \psi:\mathbbm{k}\longrightarrow\mathbb C^\times .
\]
For each $F$-rational nilpotent orbit $\CO\subset\CN^F$, let
$\Gamma_\CO$ be the corresponding generalized Gelfand--Graev representation
of $G^F$ with respect to $\psi$; see Section~\ref{sec:notation} for details.  For $\pi\in\Irr(G^F)$, define its rational wavefront set by
\[
        \ScO(\pi)
        :=
        \bigcup_{\apair{\pi,\Gamma_\CO}\neq0}\CO .
\]
Equivalently,
\[
        \ScO(\pi)
        =
        \set{
        X\in\CN^F
        \mid
        \apair{\pi,\Gamma_{\CO_X}}\neq0
        },
\]
where $\CO_X$ is the $F$-rational nilpotent orbit of $X$. We freely identify a collection of $F$-rational nilpotent orbits
with its union in $\CN^F$.  Under this convention, $\CO\in\ScO(\pi)$ means that $\CO$ is one
of the rational nilpotent orbits appearing in the above union. Whenever the cardinality of such a union is considered, it means
the number of rational nilpotent orbits in the union. We also use the
equivalent formulation in terms of generalized Whittaker models:
\[
        \CO\in\ScO(\pi)
        \quad\Longleftrightarrow\quad
        \operatorname{Wh}_\CO(\pi)\neq0.
\]
Define $\ScO(\pi)^{\max}$ to be the union of all rational
nilpotent orbits $\CO\in\ScO(\pi)$ that are maximal with respect
to the rational orbit order defined in Section~\ref{ssec:PO}.  Define $\ScO(\pi)^{\max,\rm st}$ to be the union of all rational
nilpotent orbits $\CO\in\ScO(\pi)$ such that $\CO^{\rm st}$ is
maximal, with respect to the closure order on $F$-stable
nilpotent orbits, among the orbits $(\CO')^{\rm st}$ with
$\CO'\in\ScO(\pi)$.
By the compatibility of the rational orbit order with the closure
order on $F$-stable nilpotent orbits, we have
\[
        \ScO(\pi)^{\max,\rm st}\subset \ScO(\pi)^{\max}.
\]
The stable orbits $\CO^{\rm st}$ with
$\CO\in\ScO(\pi)^{\max,\rm st}$ form the Kawanaka wavefront set
of $\pi$.  By the results of Lusztig, Geck--Malle,
Achar--Aubert, and Taylor \cite{Lu92,GM00,AA07,Tay16}, this set
consists of a single orbit, called the Kawanaka wavefront
orbit of $\pi$.  Although $\ScO(\pi)^{\max,\rm st}$ is defined as
a union of $F$-rational nilpotent orbits, by abuse of notation,
when no confusion can arise, we also use it for this unique
$F$-stable orbit.  Accordingly, in the sense of this abuse of
notation, the equality
\[
        \CO^{\rm st}=\ScO(\pi)^{\max,\rm st}
\]
means that $\CO^{\rm st}$ is the Kawanaka wavefront orbit of $\pi$.

\subsection{Rational wavefront sets}

Wavefront sets have long been studied for representations of
reductive groups over non-archimedean local fields. In the Harish--Chandra--Howe local character expansion, the
character near the identity is expressed as a linear combination
of Fourier transforms of nilpotent orbital integrals indexed by
rational nilpotent orbits over the local field
\cite{HC99,Howe74,DeB02}. The maximal rational orbits with
nonzero coefficients form the rational wavefront set.
M\oe glin--Waldspurger proved that these are precisely the maximal
orbits for which the associated degenerate Whittaker models are
nonzero \cite{MW87}. Determining this set of rational orbits is difficult, and many
results are formulated instead in terms of the stable wavefront
set obtained after passing to an algebraic closure; see, for
example, \cite{Wal18, CK25}. The relation between rational and stable wavefront sets is
discussed in \cite{Tsai25}.

In the local field case, depth-zero representations are closely related to representations
of the finite reductive quotients of parahoric subgroups.
Barbasch--Moy use generalized Gelfand--Graev characters of these
finite groups to study the coefficients in the local character
expansion \cite{BM97}.  Okada uses their Kawanaka wavefront orbits
to describe the wavefront set of depth-zero representations after extending the local field to
its maximal unramified extension \cite{Okada21}.  The resulting
unramified wavefront set lies between the rational wavefront set
over the original field and the stable wavefront set.  Replacing
the stable finite-group data by rational wavefront sets may therefore
lead to a finer local invariant, closer to the rational wavefront
set.  Related developments for unipotent representations appear in
\cite{CMBO24,CMBO25}.

For finite groups of Lie type, Kawanaka introduced generalized
Gelfand--Graev representations attached to rational nilpotent
orbits \cite{K85}.  The work of Lusztig, Geck--Malle,
Achar--Aubert, and Taylor \cite{Lu92,GM00,AA07,Tay16} determines
for each irreducible representation a unique $F$-stable Kawanaka
wavefront orbit, but does not determine which individual rational
nilpotent orbits occur in $\ScO(\pi)$.  This paper studies this
question for finite symplectic and orthogonal groups.

The guiding philosophy, originating in the descent method of
Ginzburg--Rallis--Soudry \cite{Gin06,GRS03,GRS11} and formulated
in the local conjectures of Jiang--Zhang and
Jiang--Liu--Zhang \cite{JZ18,JLZ22}, is that wavefront sets should
be read by descent.  The finite-field version of this philosophy
was initiated for symplectic groups in \cite{PW23}, where the
method treated only the range accessible through the previously
known composition law.  This range is limited: it does not reach
even the trivial representation of $\Sp_{2n}(\mathbbm{k})$, whose
rational wavefront set consists only of the trivial orbit.

The present paper removes this restriction by establishing new
composition laws for generalized Whittaker models.  Together with
the explicit first-descent theorem, they make the descent method
applicable to all rational nilpotent orbits for finite symplectic
and orthogonal groups.

\subsection{Main results}

\begin{thm}\label{main1}
Let $G$ be a symplectic or orthogonal group defined over $\mathbbm{k}$, and
let $\pi\in\Irr(G^F)$.  Then
\begin{enumerate}
\item[(i)]
\[
        \ScO(\pi)^{\max}
        =
        \ScO(\pi)^{\max,\rm st}.
\]
\item[(ii)]
\[
        \ScO(\pi)=\overline{\ScO(\pi)^{\max}},
\]
where the closure on the right-hand side is taken with respect to the
rational orbit order defined in Section~\ref{ssec:PO}.  
\end{enumerate}
\end{thm}

This theorem is not a formal consequence of the uniqueness of the
Kawanaka wavefront orbit.  A priori, a rational orbit whose
underlying stable orbit is strictly below the Kawanaka wavefront
orbit could nevertheless be maximal in $\ScO(\pi)$ with respect to
the rational orbit order.  Part~(i) rules out this possibility:
every maximal rational orbit in $\ScO(\pi)$ lies over the Kawanaka
wavefront orbit.  Part~(ii) further shows that all the remaining
rational orbits in $\ScO(\pi)$ are obtained from the maximal ones
by rational closure.

The mechanism is descent.  A rational nilpotent orbit for a symplectic
group can be decomposed into basic pieces: even one-block orbits
\[
        \CO_{2\ell,a}
        =
        [(2\ell,a),1^{2n-2\ell}],
        \qquad a\in\{\pm\},
\]
and odd-pair orbits
\[
        \CO_{(2\ell+1)^2}
        =
        [(2\ell+1)^2,1^{2n-4\ell-2}].
\]
Here we use the signed Young diagram notation recalled in
Section~\ref{lcq}.  If the largest part of a rational orbit is even, one
can remove one even block by the known composition law; the resulting
one-step model is the Fourier--Jacobi descent
$\operatorname{Wh}_{\CO_{2\ell,a}}(\pi)$, whose computation is a refined
finite Gan--Gross--Prasad problem solved in \cite{Wang21,Wang23}.

The remaining difficulty is to handle the odd-pair pieces.  The following
composition law reduces $\CO_{(2\ell+1)^2}$ to a composition of
$\CO_{2\ell,a}$ and $\CO_{2\ell+2,\varepsilon(-1)a}$, where
$\varepsilon(-1)$ denotes the square class of $-1$.  The case $\ell=0$ requires a convention: it occurs when one treats the
odd-pair orbit $\CO_{1^2}$.  In this case, the formal first even piece is
$\CO_{0,a}$, which is not a nilpotent orbit. We use the convention (see \eqref{descent0})
\[
        \operatorname{Wh}_{\CO_{0,a}}(\pi)
        :=
        \pi\otimes\omega_{n,\psi_{\varepsilon(-1)a}},
\]
where $\omega_{\psi_{n,\varepsilon(-1)a}}$ is the Weil representation of $\Sp_{2n}(\mathbbm{k})$ attached to
$\psi_{\varepsilon(-1)a}$, with $\psi_+=\psi$ and
$\psi_-(x)=\psi(\eta x)$ for a fixed nonsquare
$\eta\in\mathbbm{k}^{\times}$.  

\begin{thm}[Composition Law]\label{s1}
Let $\pi\in\Irr(\Sp_{2n}(\mathbbm{k}))$.  Suppose that $\ell\leq (n-1)/2$
and
\[
        \operatorname{Wh}_{\CO_{2d,\pm}}(\pi)=0
        \qquad
        \text{for all } d>\ell .
\]
Then the following hold.
\begin{enumerate}
\item[(i)] For every $a\in\{\pm\}$,
\[
\dim
\operatorname{Wh}_{\CO_{2\ell+2,\varepsilon(-1)a}}
\bigl(
\operatorname{Wh}_{\CO_{2\ell,a}}(\pi)
\bigr)
=
2\dim
\operatorname{Wh}_{\CO_{(2\ell+1)^2}}(\pi).
\]
\item[(ii)] Let
\(
        \CO=
        [(\lambda_1^{n_1},a_1),\ldots,(\lambda_r^{n_r},a_r)]
\)
be an $F$-rational nilpotent orbit of
$\Sp_{2n-4\ell-2}(\mathbbm{k})$ such that
$\lambda_1\leq 2\ell+1$.  Then
\[
\dim
\operatorname{Wh}_{\CO_{(2\ell+1)^2}*\CO}(\pi)
=
\dim
\operatorname{Wh}_\CO
\bigl(
\operatorname{Wh}_{\CO_{(2\ell+1)^2}}(\pi)
\bigr).
\]
\end{enumerate}
\end{thm}
\begin{rmk}
Theorem~\ref{s1} is not a direct application of the usual exchange-of-roots
arguments of \cite{GRS11,GGS17}.  The standard method gives the basic
framework, but it does not by itself produce the odd-to-even composition
needed here.  The extra point is the finite-orbit analysis which controls
the relevant subgroup actions and the rational signs.  This is what allows
the odd-pair model to be reduced to two even Fourier--Jacobi descents.  The
same idea may have a local analogue, though one would have to handle the
analytic issues in exchanging unipotent integrals and tracking the
corresponding subgroup actions.
\end{rmk}

Thus, the odd-pair pieces are reduced to the even descents computed by
finite Gan--Gross--Prasad theory.
A descent sequence of
$\pi\in\Irr(\Sp_{2n}(\mathbbm{k}))$ is a sequence
\begin{equation}\label{descent1}
        \pi=\pi_0
        \xrightarrow{(\ell_1,a_1)}
        \pi_1
        \xrightarrow{(\ell_2,a_2)}
        \cdots
        \xrightarrow{(\ell_r,a_r)}
        \pi_r=\mathbbm{1},
\end{equation}
where $\pi_i$ is an irreducible constituent of
$\operatorname{Wh}_{\CO_{2\ell_i,a_i}}(\pi_{i-1})$, and $\mathbbm{1}$ is
the trivial representation of the trivial group $\Sp_0(\mathbbm{k})$.  Its
descent sequence index is denoted by
\[
        \Gamma
        =
        \llb
        (2\ell_1,a_1),(2\ell_2,a_2),\ldots,(2\ell_r,a_r)
        \rrb .
\]
Some descent sequence indexes correspond to signed Young diagrams,
and hence to rational nilpotent orbits, but not every index does.
The mutation order is nevertheless defined on all descent sequence
indexes, whether or not they correspond to rational orbits.
It is proved in Section~\ref{sec4} that this order is compatible
with the rational orbit order in both directions. If
$\Gamma\leqslant\Gamma'$ and they correspond to rational
orbits $\CO$ and $\CO'$, respectively, then
$\CO\leqslant\CO'$.  Conversely, if
$\CO\leqslant\CO'$, then for every index $\Gamma$ corresponding
to $\CO$, there exists an index $\Gamma'$ corresponding to
$\CO'$ such that $\Gamma\leqslant\Gamma'$.  Together with the
relation between descent sequences and generalized Whittaker
models, this allows the study of $\ScO(\pi)^{\max}$ to be reduced
to that of maximal descent sequence indexes.

A first descent is a nonzero descent with the largest possible value of
$\ell$.  The explicit first-descent theorem (see Section~\ref{sec6}),
together with the comparison of descent branches in Section~\ref{sec7},
shows that every maximal descent sequence index is obtained by successive
first descents.  Iterating the explicit description of first descent then
gives an explicit description of all maximal descent sequence indexes.
We further prove that every such index corresponds to a rational nilpotent
orbit.  The comparison between the mutation order and the rational orbit
order therefore identifies these orbits with those in $\SOM$, yielding an
explicit description of $\SOM$ for symplectic groups.  The orthogonal cases
are obtained from the symplectic case by theta correspondence and the
moment-map theorem of Gomez--Zhu (see Section~\ref{sec8}).

\begin{thm}\label{main2}
Let $G$ be a symplectic or orthogonal group defined over $\mathbbm{k}$, and
let $\pi\in\Irr(G^F)$. Then
\[
        \SOM
        =
        \bigcup_{\Gamma}\CO(\Gamma),
\]
where $\Gamma$ runs over all maximal descent sequence indexes of $\pi$ and
$\CO(\Gamma)$ is the rational nilpotent orbit attached to $\Gamma$.
Moreover, after fixing the theta-compatible Lusztig correspondence of
Theorem~\ref{thm:L-can}, these maximal descent sequence indexes are computed
explicitly from the Lusztig parameter of $\pi$.  The explicit formulas of these maximal descent sequence indexes are
given in Theorems~\ref{max}, \ref{maxo1} and \ref{maxo2}.
\end{thm}
\begin{rmk}
By Theorem~\ref{main1}\textup{(ii)}, the explicit description of
$\SOM$ determines the whole rational wavefront set $\ScO(\pi)$ by
rational closure.
\end{rmk}

\begin{exmp}\label{trv}
Let $\mathbbm{1}_n$ be the trivial representation of
$\Sp_{2n}(\mathbbm{k})$.  All positive even descents vanish:
\[
        \operatorname{Wh}_{\CO_{2\ell,a}}(\mathbbm{1}_n)=0
        \qquad(\ell>0).
\]
By the composition law, this also kills the odd-pair descents
$\CO_{(2\ell+1)^2}$ with $\ell>0$.  Thus only the case $\ell=0$, namely
$\CO_{1^2}$, remains.
For this case we use the convention
\[
        \operatorname{Wh}_{\CO_{0,a}}(\mathbbm{1}_n)
        =
        \omega_{n,\psi_{\varepsilon(-1)a}}.
\]
The Weil representation
$\omega_{n,\psi_{\varepsilon(-1)a}}$ has two irreducible constituents, and the  maximal element in rational
wavefront set of each constituent is the single orbit
\(
        \CO_{2,\varepsilon(-1)a}.
\)
Thus
\[
        \dim
        \operatorname{Wh}_{\CO_{2,\varepsilon(-1)a}}
        (\omega_{n,\psi_{\varepsilon(-1)a}})
        =
        2.
\]
By Theorem~\ref{s1},
\[
        \dim\operatorname{Wh}_{\CO_{1^2}}(\mathbbm{1}_n)=1,
\]
and this descent is $\mathbbm{1}_{n-1}$.  Iterating gives
\[
        \ScO(\mathbbm{1}_n)^{\max}
        =
        [1^{2n}].
\]
In summary, we have the following descent diagram:
\[
\xymatrix{
&  \mathbbm{1}_n \ar[dl]_{(0,+)} \ar[dr]^{(0,-)}  &  \\
\omega_{n,\psi_{\varepsilon(-1)}} \ar[dr]_{(2,\varepsilon(-1))}& & \omega_{n,\psi_{-\varepsilon(-1)}} \ar[dl]^{(2,-\varepsilon(-1))} \\
  &\mathbbm{1}_{n-1}. &  }
\]
\end{exmp}

\begin{rmk}
Corollary~\ref{5.6} shows that this kind of symmetry is not special
to the trivial representation.  The paths in its diagrams determine
descent sequences, hence rational nilpotent orbits, while the
different branches record symmetries among rational orbits in the
same wavefront set.
\end{rmk}

We use the
theta-compatible Lusztig correspondence fixed in Theorem~\ref{thm:L-can}.
The following example illustrates why the theta-compatible parametrization
is needed.

\begin{exmp}
Let $G=\Sp_4(\mathbbm{k})$, and let $s\in G^{*F}$ be a semisimple element
such that
\[
        C_{G^{*F}}(s)\simeq \RO_2^\eta(\mathbbm{k}).
\]
The Lusztig series $\mathcal E(G,s)$ contains two irreducible
representations, say $\pi_+$ and $\pi_-$.  The abstract Lusztig
correspondence does not distinguish which one corresponds to the
trivial character of $\RO_2^\eta(\mathbbm{k})$ and which one
corresponds to the sign character.
Nevertheless,
\[
        \ScO(\pi_+)^{\max}
        \neq
        \ScO(\pi_-)^{\max},
\]
although their Kawanaka wavefront orbits coincide.  This distinction is detected by theta correspondence.  Each
$\pi_\epsilon$ is obtained by theta lifting from a representation
of a one-dimensional orthogonal group
$\RO_1^{\epsilon'}(\mathbbm{k})$.  Gomez--Zhu's theorem relates
the rational orbit data in $\ScO(\pi_\epsilon)^{\max}$ to the
isomorphism class of this one-dimensional orthogonal space.
\end{exmp}

We also record an explicit dimension formula for maximal generalized
Whittaker models. The numerical multiplicities are closely related to, and in
principle recoverable from, the formulas of Lusztig, Geck--Malle,
Geck--H\'ezard and Taylor for generalized Gelfand--Graev representations;
see, for example, \cite[Section 4]{GH03} and \cite[Proposition 15.4]{Tay16}.
The point of the present formulation is that it is adapted to the rational
wavefront-set problem. It expresses the multiplicity in terms of the number
of maximal rational nilpotent orbits in $\ScO(\pi)$, the component group of the Kawanaka
wavefront orbit, and Lusztig's canonical quotient on the dual side.

Let $\pi\in\mathcal E(G,s)$ be an irreducible
representation in the Lusztig series associated with a semisimple element
$s\in G^{*F}$.  Let
\[
        \pi^*=\mathcal L^{\rm can}(\pi)
        \in \mathcal E(C_{G^*}(s),1)
\]
be its image under the Lusztig correspondence fixed in
Theorem~\ref{thm:L-can}.  Let $\CO\in\SOM$ and put
\[
        \SCO=\SOMS,
        \qquad
        \SCO^*=\ScO(\pi^*)^{\max,\rm st}.
\]
Write $A(\SCO)$ for the component group of the centralizer of a point in
$\SCO$, and write $\overline A(\SCO^*)$ for Lusztig's canonical quotient
attached to $\SCO^*$ in $C_{G^*}(s)$.

\begin{thm}\label{main3}
With the notation above,
\[
\dim \operatorname{Wh}_\CO(\pi)
=
\begin{cases}
\displaystyle
\frac{|A(\SCO)|}
     {|\SOM|\,|\overline A(\SCO^*)|},
& \text{if $G$ is symplectic or even orthogonal},\\[1.2em]
\displaystyle
\frac{|A(\SCO)|}
     {2|\SOM|\,|\overline A(\SCO^*)|},
& \text{if $G$ is odd orthogonal}.
\end{cases}
\]
Moreover, $\dim\operatorname{Wh}_\CO(\pi)$ is a power of $2$, explicitly
given in Theorems~\ref{max}, \ref{maxo1} and \ref{maxo2}.
\end{thm}

\subsection{Unipotent representations and canonical quotients}

The unipotent case reveals the intrinsic structure behind the formulas.
There are three simplifying features:
\begin{enumerate}
\item The image of an irreducible unipotent representation is
independent of the chosen Lusztig correspondence.
\item For unipotent representations, $\SOM$ is independent of the choice of
the additive character $\psi$; see Theorem~\ref{indep}.
\item The maximal rational wavefront set is governed by Lusztig's canonical
quotient.
\end{enumerate}

Recall the notation from Section~\ref{lcq}.  For each $F$-stable nilpotent
orbit $\SCO$, we define a map
\[
        S:
       \{\text{$F$-rational nilpotent orbits attached to }\SCO\}
        \longrightarrow
        A(\SCO),
\]
and hence a surjection
\[
        \overline S:
        \{\text{$F$-rational nilpotent orbits attached to }\SCO\}
        \longrightarrow
        \overline A(\SCO).
\]
For $x\in\overline A(\SCO)$, let $\ScO(x)$ denote the preimage of $x$
under $\overline S$. In the full orthogonal cases, the domains of $S$ and $\overline S$
are the disjoint union of the rational orbits for the two quadratic
forms. In even dimension, for a rational orbit $\CO$ on either form,
$\SCO$ denotes the $F$-stable nilpotent orbit in the split form with
the same underlying partition as $\CO$; see Section~\ref{lcq}.

\begin{thm}\label{main4}
Let $G$ be a symplectic or orthogonal group defined over $\mathbbm{k}$,
and let $\pi$ be an irreducible unipotent representation of $G^F$ such
that
\[
\SOMS=\SCO.
\]
Then there exists $x\in\overline A(\SCO)$ such that
\[
\SOM=\mathscr O(x).
\]
If $\pi$ is special, regarded in the odd orthogonal case as a
representation of $\RO_{2n+1}^{+}(\mathbbm{k})$, then
\[
S(\SOM)
=
\ker\bigl(A(\SCO)\longrightarrow\overline A(\SCO)\bigr).
\]
Consequently,
\[
A(\SCO)/S(\SOM)\cong\overline A(\SCO).
\]
\end{thm}
Thus the maximal rational wavefront set of a unipotent representation
is a single fiber of Lusztig's canonical quotient. For the special
representation, the image of this fiber under $S$ is the kernel of the
quotient map.

The proof uses the two-step symmetry appearing in
Example~\ref{trv}. There, the trivial representation has two first
descents, one for each sign. The one-step descendants are different,
but after the next descent they return to the same trivial
representation of smaller rank. For a general unipotent representation, extending a maximal descent
sequence of the smaller representation along the analogous two paths
gives either the same rational orbit or two distinct rational orbits. The factor one or two agrees with the ratio between the sizes of the
kernels of the canonical quotient maps attached to the stable
wavefront orbits of the original representation and the smaller representation. See Section~\ref{sec:un} for details.

This canonical quotient description is compatible with Alvis--Curtis
duality.   Recall that Lusztig
partitions the irreducible unipotent representations of $G^F$ into
families.  If $\mathcal F$ is such a family, let
$\CO^{\rm st}_{\mathcal F}$ be the stable nilpotent orbit corresponding to its unipotent support.  If
$\pi\in\mathcal F$ and if $\pi'$ is its Alvis--Curtis dual, lying in a
family $\mathcal F'$, then
\[
        \ScO(\pi)^{\max,\rm st}
        =
        \CO^{\rm st}_{\mathcal F'},
        \qquad
        \ScO(\pi')^{\max,\rm st}
        =
        \CO^{\rm st}_{\mathcal F}.
\]
For symplectic, even orthogonal, and special odd orthogonal groups,
the representations in $\mathcal F$ are parametrized by
\(
\mathcal M(\overline A(\CO^{\rm st}_{\mathcal F})),
\)
where, for a finite group $A$,
\[
\mathcal M(A)
=
\{(x,\tau)\mid x\in A,\ \tau\in\Irr(C_A(x))\}/A.
\]
We write the corresponding parameter as $(x^+,x^-)$ and denote the
representation by
\(
\pi_{\CO^{\rm st}_{\mathcal F},x^+,x^-}.
\)

For the full odd orthogonal group, the parameters $(x^+,x^-)$
parametrize the restriction to the special orthogonal subgroup, while
the additional parameter $\eta\in\{\pm\}$ records the action of the
central element $-I$ and hence distinguishes the two extensions to
the full orthogonal group. We write
\(
\pi_{\CO^{\rm st}_{\mathcal F},x^+,x^-,\eta}.
\)

By comparing the family realization introduced in
Section~\ref{family} with the rational-orbit realization introduced
in Section~\ref{lcq}, Section~\ref{sec:un} associates to each possible
parameter $x^+$ in the family $\mathcal F$ an element
\[
\Phi_{\mathcal F}(x^+)
\in
\overline A(\CO_{\mathcal F}^{\rm st}).
\]
In the odd orthogonal case, $x^+$ is identified with an element of
$\overline A_{\SO}(\CO_{\mathcal F}^{\rm st})$, whereas
\(
\Phi_{\mathcal F}(x^+)
\in
\overline A_{\RO}(\CO_{\mathcal F}^{\rm st}).
\)

\begin{thm}\label{main5}
Let
\(
\pi
=
\pi_{\CO_{\mathcal F}^{\rm st},x^+,x^-,\eta}
\)
be an irreducible unipotent representation of $G^F$, and let
\(
\pi'
=
\pi_{\CO_{\mathcal F'}^{\rm st},
x^{\prime+},x^{\prime-},\eta'}
\)
be its Alvis--Curtis dual. Here the parameters $\eta$ and $\eta'$ are
present only in the odd orthogonal case. Then
\[
\ScO(\pi')^{\max}
=
\mathscr O\bigl(\Phi_{\mathcal F}(x^+)\bigr),
\qquad
\ScO(\pi)^{\max}
=
\mathscr O\bigl(\Phi_{\mathcal F'}(x^{\prime+})\bigr).
\]
\end{thm}

The rigidity of Theorems~\ref{main4} and~\ref{main5} does not persist
for general quadratic unipotent representations. In fact, the
quadratic unipotent case exhibits the opposite extreme.

\begin{thm}\label{main6}
Let $G$ be a symplectic or orthogonal group defined over $\mathbbm{k}$.
For every $F$-rational nilpotent orbit $\CO\subset\Fg^F$, there exists a quadratic
unipotent representation $\pi\in\Irr(G^F)$ such that
\[
        \SOM=\{\CO\}.
\]
\end{thm}

Thus every $F$-rational nilpotent orbit can occur as the unique
maximal orbit in the rational wavefront set of a quadratic
unipotent representation.  This
also gives another proof of \cite[Proposition~4.3]{GH03}, which is used to
establish a conjecture of Kawanaka \cite[(3.3.1)]{K85}.

\subsection{Ideas of the proof}

We briefly explain the main ideas of the proof.  The first step is the
composition law for the odd-pair orbits.  Its proof starts from root
exchange.  In the usual exchange-root argument one chooses unipotent
subgroups $C,X,Y$ of $G$ with
\[
        B=CY,\qquad D=CX,\qquad A=CXY .
\]
The phrase ``exchange of roots'' means that the roots in $B$ are
replaced by those in $D$ without changing the corresponding induced
representation.  More precisely, if $\psi_B$ and $\psi_D$ denote the
extensions of a character $\psi$ of $C^F$ that are trivial on $Y^F$ and $X^F$, respectively,
then
\[
        \operatorname{Ind}_{B^F}^{G^F}\psi_B
        \simeq
        \operatorname{Ind}_{D^F}^{G^F}\psi_D.
\]
This follows from the Stone--von Neumann theorem applied to the finite
symplectic space defined by the pairing
\[
        (x,y)\longmapsto\psi([x,y]),
        \qquad x\in X^F,\quad y\in Y^F.
\]

For Theorem~\ref{s1}, the unipotent subgroups attached to 
\[
\operatorname{Wh}_{\CO_{2\ell+2,\varepsilon(-1)a}}(\operatorname{Wh}_{\CO_{2\ell,a}}(\pi)) \quad\textrm{and}\quad 
\operatorname{Wh}_{\CO_{(2\ell+1)^2}}(\pi)
\] 
do not satisfy the usual exchange-root hypotheses directly.  We first
exchange the roots which can be exchanged.  What remains are some residual
finite character sums.  A suitable Levi subgroup acts on these residual sets.
The two open orbits give the two contributions to the model attached to
$\CO_{(2\ell+1)^2}$, and this is the source of the factor $2$ in
Theorem~\ref{s1} (i).  The other orbits lead to generalized Whittaker models \[
\operatorname{Wh}_{\CO_{2\ell,\varepsilon(-1)b}}(\operatorname{Wh}_{\CO_{2\ell+2,b}}(\pi)) \quad \textrm{with }b\in\{\pm\},
\]
and they vanish by the hypothesis of
Theorem~\ref{s1}.  This proves the odd-pair composition law and expresses the
odd-pair Whittaker models in terms of successive descents.

After the odd-pair composition law is established, the symplectic
case is reduced to descent sequence indexes.  Every step in a
maximal descent sequence must be a first descent, that is, a
descent with the largest possible value of $\ell$.  Indeed, suppose
that two successive steps have the form
\[
\pi\longrightarrow\pi'\longrightarrow\pi''.
\]
If the first step is not a first descent, the explicit
first-descent formulas obtained in our previous work
\cite{Wang21,Wang23} allow us to construct another route
\[
\pi\longrightarrow\pi^\star\longrightarrow\pi''
\]
whose first descent index is strictly larger.  Replacing the
original two steps by this route gives a larger descent sequence
index, contradicting maximality.  Applying this argument
successively proves that every maximal descent sequence is obtained
by first descents at every step.

A first descent may have several irreducible constituents, so
different maximal descent sequences may pass through different
representations.  The underlying phenomenon is already visible in
the example of the trivial representation: the descent branches are
different, but the corresponding descent sequences give the same
rational orbit.  Corollary~\ref{5.6} describes the general symmetry
behind this example.  Together with the explicit first-descent
formulas, this symmetry implies that all maximal descent sequence
indexes have the same numerical sequence.

The explicit first-descent formulas also show that every maximal
descent sequence index is good, meaning that it corresponds to a
rational nilpotent orbit.  By the compatibility between the mutation
order and the rational orbit order, the preceding discussion then
shows that the orbits in $\SOM$ are precisely those corresponding to
maximal descent sequence indexes of $\pi$.  Since every maximal
descent sequence is obtained by successive first descents, iterating
the explicit formulas gives an explicit description of $\SOM$.  This
proves Theorem~\ref{main2} for symplectic groups.

Since all maximal descent sequence indexes of $\pi$ have the same
numerical sequence, the corresponding rational orbits have the same
underlying partition and hence lie in the same $F$-stable nilpotent
orbit, necessarily the stable wavefront orbit of $\pi$.  This proves
Theorem~\ref{main1}\textup{(i)} for symplectic groups.  Keeping track
of the factors of $2$ contributed by the composition laws along a
maximal descent sequence also gives the Whittaker-dimension formula
in Theorem~\ref{main3}.

To recover the remaining rational orbits in $\ScO(\pi)$, we prove
that descent sequence indexes of $\pi$ are downward closed.  Recall
that the mutation order is generated by elementary mutations: if
$\Gamma\leqslant\Gamma'$, then $\Gamma'$ can be obtained from
$\Gamma$ by finitely many replacements of the following form:
\[
(2\ell_i,a_i),(2\ell_{i+1},a_{i+1})
\longmapsto
(2\ell_i+2,b_i),(2\ell_{i+1}-2,b_{i+1}),
\qquad
a_i a_{i+1}=b_i b_{i+1}.
\]
The rational orbit order also admits a description by mutations, but
several different types of mutation occur there.  A main advantage
of passing to descent sequence indexes is that their order is
generated by this single operation, so no case-by-case analysis is
needed.  
To prove downward closure, it is enough to show that the inverse of
one elementary mutation is again the index of a descent sequence of $\pi$.
By applying the same argument after any number of descent sequences,
we may restrict attention to a mutation at the first two positions. Starting from
\(
\pi_0\to\pi_1\to\pi_2,
\)
we choose an auxiliary generic rank-one representation such that the
eigenvalues of the semisimple element defining its Lusztig series
are disjoint from those of $\pi_i$, and combine it with
$\pi_1$ by Deligne--Lusztig induction.  The resulting representation
is a descent of $\pi_0$ with index $\ell_1-1$ and has $\pi_2$ in its
descent with index $\ell_2+1$, with both signs multiplied by the same
$\epsilon\in\{\pm1\}$.  Thus the inverse elementary mutation is
realized.  Iterating this construction proves downward closure.  The
compatibility between the mutation order and the rational orbit
order then gives the rational-closure description of $\ScO(\pi)$,
proving Theorem~\ref{main1}\textup{(ii)} for symplectic groups.

The orthogonal cases are obtained from the symplectic case by theta
correspondence.  The theta-compatible Lusztig correspondence
controls the rational parametrization of the representations, while
the moment-map theorem of Gomez--Zhu transfers the corresponding
rational nilpotent orbits.  Together, these results extend
Theorems~\ref{main1}, \ref{main2}, and \ref{main3} from symplectic
groups to orthogonal groups.

For the unipotent results, the same symmetry among descent sequences,
again described in Corollary~\ref{5.6}, is the key ingredient.  After
deleting the first two steps of a maximal descent sequence, the
different branches either determine the same rational orbit or
determine two distinct rational orbits.  The map obtained by
forgetting these two steps, from $\SOM$ to the maximal rational
wavefront set of the resulting smaller unipotent representation, is
therefore one-to-one or two-to-one.  Its fiber size agrees exactly
with the corresponding change in the kernel defining Lusztig's
canonical quotient.  Induction then shows that the maximal rational
wavefront set is a single fiber of the canonical quotient, proving
Theorem~\ref{main4}.  Comparing the same symmetry on the
rational-orbit and symbol sides gives the compatibility with
Alvis--Curtis duality and proves Theorem~\ref{main5}.  Reversing the
first-descent construction produces a quadratic unipotent
representation having any prescribed rational nilpotent orbit as
its unique maximal rational wavefront orbit, proving
Theorem~\ref{main6}.

\subsection{Remarks on the characteristic}

We include a few remarks on the assumption $p>3(h_G-1)$.  This
assumption ensures that nilpotent orbit theory and generalized
Gelfand--Graev representations behave in the required
large-characteristic setting.  Some parts of the theory, such as
the classification of rational nilpotent orbits and the
construction of generalized Gelfand--Graev representations, can be
developed under weaker hypotheses; see, for example,
\cite{Tay16}.  However, the stated bound is used in several
essential places, especially in the first-descent theorem, in the
construction of auxiliary generic representations, and in the
reduction of the orthogonal case to the symplectic case by theta
correspondence.  If these ingredients were established under weaker
hypotheses, the characteristic assumption could likely be weakened.

\subsection{Organization of this paper}

This paper is organized in four stages.  Sections~\ref{sec2}--\ref{sec4}
set up the representation-theoretic framework and the order structure
of rational nilpotent orbits and descent sequence indexes. We first
review Lusztig's Jordan decomposition and the classification of
irreducible representations of finite classical groups, then recall
generalized Gelfand--Graev representations, generalized Whittaker
models, rational nilpotent orbits, the rational orbit order, and the
relevant component-group and canonical-quotient maps.  We next introduce
the descent operations and the combinatorial structures used to organize
them, including descent sequence indexes, the mutation order, and good
descent sequence indexes.

Sections~\ref{sec5} and \ref{sec6} establish the two main inputs for the
proofs.  Section~\ref{sec5} proves the new composition laws for pairs of
equal odd parts, while Section~\ref{sec6} recalls the explicit
first-descent theorem and derives the structural consequences needed
later.

The main theorems are proved in Sections~\ref{sec7} and \ref{sec8}.
Section~\ref{sec7} treats symplectic groups by combining the composition
laws with first descent, and Section~\ref{sec8} obtains the orthogonal
analogues through theta correspondence and generalized descent.

Finally, Section~\ref{sec:un} studies the unipotent case.  We prove
independence of the additive character, describe the maximal rational
wavefront set through Lusztig's canonical quotient, establish
compatibility with Alvis--Curtis duality, and prove the realization
theorem for quadratic unipotent representations.

\subsection*{Acknowledgement} 
 The author wishes to thank Lei Zhang and Dongwen Liu for their insightful and valuable discussions.
The author would also like to thank the organizers of the Seminar on Harmonic Analysis on Algebraic Groups and the Seminar on the Relative Langlands Program held at the Tianyuan Mathematics Research Center in Kunming in August 2025. This paper was inspired by discussions and feedback received at these two seminars.

 The author gratefully acknowledges financial support from the National Natural Science Foundation of China (Grant No. 12571012) and the Jilin Provincial Postdoctoral Foundation. 

\clearpage

\section{Classification of irreducible representations of $\RO_\Fn$ and $\Sp_{2n}$}\label{sec2}
Let $V_\Fn$ be an $\Fn$-dimensional (skew) symmetric space over $\mathbbm{k}$.
Denote by $G$ the isometry group of $V_\Fn$.
More precisely, $G$ is $\Sp_{\Fn}$ or $\RO_{\Fn}$.

For a symmetric space $V_\Fn$, define its discriminant by 
\begin{equation}\label{eq:disc}
\disc(V_\Fn)=(-1)^{\frac{\Fn(\Fn-1)}{2}}\det(V_\Fn)\in \mathbbm{k}^{\times}/\mathbbm{k}^{\times 2}.
\end{equation}

\subsection{Lusztig's symbols} \label{sec:L-s}
In this section, we follow the notation of \cite{Pan19} to review Lusztig's symbols, which are slightly different from those of \cite{Lu77}.

Let $\lambda=[\lambda_1,\lambda_2,\dots]$ be a partition of $n$ and $|\lambda|:=\sum_i  \lambda_i$.
We often write $\lambda=[\lambda_1^{k_1},\cdots,\lambda_l^{k_l}]$ where the $\lambda_j$ are distinct nonzero parts and the $k_j$ are their
multiplicities.
As is standard, we realize partitions as Young diagrams and denote by $\lambda^t=[\lambda^t_1,\lambda^t_2,\cdots]$ the transpose of $\lambda$. A partition $\lambda$ is called {\it even} if all multiplicities $k_i$ are even.

We say that two partitions $\lambda=[\lambda_1,\lambda_2,\dots]$ and $\lambda'=[\lambda'_1,\lambda'_2,\dots]$ are {\it close} if $|\lambda_i-\lambda_i'|\leqslant 1$ for every $i$.
We define
\[
\lambda\cap\lambda' :=[\lambda_i]_{\set{ i \mid \lambda_i=\lambda_i'}}
\]
to be the partition formed by the common parts of $\lambda$ and $\lambda'$. Following \cite{AMR96}, we say that $\lambda$ and $\lambda'$ are {\it 2-transverse} if they are close and $\lambda\cap \lambda'$ is even.
In particular, if $\lambda$ and $\lambda'$ are close and $\lambda\cap \lambda'=\varnothing$, then $\lambda$ and $\lambda'$ are 2-transverse, and in this case we say that they are {\it transverse}. For example, let
$\lambda=[\lambda_1,\ldots, \lambda_k]$ be a partition of $n$, and let
\[
\lambda_\star=[\lambda_2,\ldots, \lambda_k]
\]
be the partition of $n-\lambda_1$ obtained by removing the first row of $\lambda$. Then $\lambda^t$ and $\lambda_{\star}^t$ are transverse, hence 2-transverse.  Moreover, $\lambda_{\star}^t$ is the unique partition of $n-\lambda_1$ with this property.

A {\it symbol} is an array of the form
\[
\Lambda=
\begin{pmatrix}
A\\
B
\end{pmatrix}
=
\begin{pmatrix}
a_1,a_2,\cdots,a_{m_1}\\
b_1,b_2,\cdots,b_{m_2}
\end{pmatrix}
\]
of two finite sets $A$ and $B$ of nonnegative integers, possibly empty,
with $a_1>\cdots>a_{m_1}\geqslant 0$ and
$b_1>\cdots>b_{m_2}\geqslant 0$.
The {\it rank} and {\it defect} of a symbol $\Lambda$ are defined by
\[
\begin{aligned}
&\rank(\Lambda):=\sum_{a_i\in A}a_i+\sum_{b_i\in B}b_i-\left\lfloor\left(\frac{\#A+\#B-1}{2}\right)^2\right\rfloor, \\
&\textrm{def}(\Lambda):=\#A-\#B.
\end{aligned}
\]
 Note that the definition of ${\mathrm{def}}(\Lambda)$ differs from that of \cite[p.133]{Lu77}.

For a symbol $\Lambda=\binom{A}{B}$, let $\Lambda^*$ (resp. $\Lambda_*$) denote the first row (resp. second row) of $\Lambda$, i.e. $\Lambda^*=A$ and $\Lambda_*=B$, 
and let
$\Lambda^t := \binom{B}{A}$.

Define an equivalence relation generated by the rule
\[
\begin{pmatrix}
a_1,a_2,\cdots,a_{m_1}\\
b_1,b_2,\cdots,b_{m_2}
\end{pmatrix}
\sim
\begin{pmatrix}
a_1+1,a_2+1,\cdots,a_{m_1}+1,0\\
b_1+1,b_2+1,\cdots,b_{m_2}+1,0
\end{pmatrix}.
\]
Note that the defect and rank are functions on the set of equivalence classes of symbols. Except for the terminal convention below, we always consider symbols that
have at most one zero entry, and each equivalence class of symbols has a unique representative with this property. For such symbols $\Lambda$, let $|\Lambda|$ be the number of nonzero entries in $\Lambda$.
Let us explain why we only count nonzero entries instead of all the entries. Lusztig parametrizes unipotent representations by symbols, and we would like to use $|\Lambda|$ to describe the length of the partition of the wavefront set of the representation corresponding to $\Lambda$. Under Lusztig's parametrization, the trivial representation of the trivial symplectic group $\Sp_0(\mathbbm{k})$ corresponds to the symbol
$\Lambda=\binom{0}{-}$.
Thus $|\Lambda |$ should be 0.

Set
\[
\begin{aligned}
& \Lambda^0 :=\Lambda, \quad \textrm{ and }
& \zeta^0(\Lambda)=\zeta(\Lambda):=
\begin{cases}
 +,&\textrm{ if }a_1>b_1;\\
  -,&\textrm{ if }a_1<b_1;\\
  \{\pm\},&\textrm{ if }a_1=b_1.\\
\end{cases}
\end{aligned}
\]
Here when one row of a symbol is empty, we regard its largest entry as
$-\infty$.
Inductively, let
\[
\Lambda^i=\begin{pmatrix}
A^i\\
B^i
\end{pmatrix}
\]
be the symbol obtained by removing the largest number in $\Lambda^{i-1}$. Clearly, we have 
\[
|\Lambda^i|=|\Lambda^{i-1}|-1=\cdots=|\Lambda|-i.
\]
There is a slight ambiguity in the intermediate step.  Suppose that the
largest entries of $A^i$ and $B^i$ are equal.  Then $\Lambda^{i+1}$ may be
chosen to be either $\Lambda^i_A$ or $\Lambda^i_B$, where $\Lambda^i_X$ is
obtained by removing the largest entry from the row $X$.  After the next
step the other equal entry must be removed, so $\Lambda^{i+2}$ is
well-defined.

Set
\[
\zeta^i(\Lambda):=
  \begin{cases}
+,&\textrm{ if the largest number of }\Lambda^i \textrm{ only appears in }A^i;\\
  -,&\textrm{ if the largest number of } \Lambda^i \textrm{ only appears in }B^i.
\end{cases}
\]
In the ambiguous case above, the two possible choices give
\[
        (\zeta^i(\Lambda),\zeta^{i+1}(\Lambda))=(+,-)
        \quad\textrm{or}\quad
        (-,+).
\]
This convention records the rational orbit data: the number removed
determines the size of the corresponding part, while the row from which it
is removed determines the corresponding rational form.  Therefore, the two
possible orders may give two different rational nilpotent orbits, but they
differ only in the two parts corresponding to these two equal entries; see
Corollary~\ref{5.6}.  In the arguments below, whenever such an ambiguity occurs and no choice is specified, we make an arbitrary choice; the argument does not depend on this choice. If the choice affects
the statement or the proof, we will specify and discuss the choice
explicitly.

We also use one terminal convention.  Later, we will often remove entries
simultaneously from the two symbols $\Lambda_1$ and $\Lambda_{-1}$ in the
Lusztig parameter.  If one of the two symbols has no nonzero entry left while
the other still has entries to be removed, we regard the exhausted symbol as
\(
        \binom{0}{0}.
\)
This is the only symbol with two zero entries which occurs in this paper.
The two zeros are treated as two equal largest entries: either zero may be
removed first, and then the remaining zero is removed at the next step.  This
convention is illustrated in Example~\ref{dex2}.

\subsection{Bi-partitions}  \label{sec:bi}
For partitions $\lambda=[\lambda_1,\lambda_2,\cdots,\lambda_k]$ and $\mu=[\mu_1,\mu_2,\cdots,\mu_l]$, we write
\[
\lambda\preccurlyeq\mu
\quad\textrm{if }\quad
\mu_i-1\leqslant \lambda_i \leqslant \mu_i
\textrm{ for each }i,
\]
after appending zeros to the two partitions if necessary.

Let $\mathcal{P}_2(n)$
denote the set of bi-partitions $\bpair{\begin{smallmatrix}
\lambda\\
\mu\end{smallmatrix}}$ of $n$, where $\lambda$, $\mu$ are partitions
such that $|\lambda| + |\mu| = n$. We can associate a bi-partition to each symbol as follows:
\begin{align*}
\Upsilon: &\Lambda
=\begin{pmatrix}
A\\
B
\end{pmatrix}
=\begin{pmatrix}
a_1,a_2,\dots,a_{m_1}\\
b_1,b_2,\dots,b_{m_2}
\end{pmatrix}\\
&\mapsto
\begin{bmatrix}
\lambda\\
\mu
\end{bmatrix} = 
\begin{bmatrix}
a_1-(m_1-1),a_2-(m_1-2),\dots,a_{m_1-1}-1,a_{m_1}\\
b_1-(m_2-1),b_2-(m_2-2),\dots,b_{m_2-1}-1,b_{m_2}
\end{bmatrix}.
\end{align*}
This gives a bijection
\begin{equation}\label{bp}
\Upsilon:\mathcal{S}_{n,\beta}\to 
\begin{cases}
\mathcal{P}_2(n-(\frac{\beta+1}{2})(\frac{\beta-1}{2})), &  \textrm{if }\ \beta\textrm{ is odd};\\
\mathcal{P}_2(n-(\frac{\beta}{2})^2), & \textrm{if }\ \beta\textrm{ is even},
\end{cases}
\end{equation}
where $\mathcal{S}_{n,\beta}$ denotes the set of symbols of rank $n$ and defect $\beta$. Then a symbol $\Lambda$ can be characterized by its defect and the bi-partition $\Upsilon(\Lambda)$.

Set
\[
\Upsilon(\Lambda)^\epsilon:=
  \begin{cases}
\lambda, &\textrm{ if }\epsilon=+;\\
\mu, &\textrm{ if }\epsilon=-.
\end{cases}
\]
For a partition $\lambda$, we set
\[
\lambda_{\epsilon}:=
  \begin{cases}
  \lambda, &\textrm{ if }\epsilon=+;\\
  \textrm{the partition obtained from $\lambda$ by removing its first column},
  &\textrm{ if }\epsilon=-.
  \end{cases}
\]
For example, $(\Upsilon(\Lambda)^-)_-$ is the partition $[\mu_1-1,\mu_2-1,\mu_3-1,\cdots]$, that is, the partition obtained from the bottom partition $\mu$ of
$\Upsilon(\Lambda)$ by removing its first column.

To each symbol $\Lambda$, we associate its Alvis--Curtis dual
${}^{\bd}\Lambda$ \cite{Al79,Cu80} by requiring that
 \[
\textrm{def}({}^{\bd}\Lambda)=\textrm{def}(\Lambda)\quad 
\text{ and }\quad
 \Upsilon({}^{\bd}\Lambda)^\epsilon=\left(\Upsilon(\Lambda)^{-\epsilon}\right)^t, 
 \]
where $\epsilon=\pm$. 
For convenience, throughout this paper we use the convention
\begin{equation}\label{eq:bd-iteration}
{}^{\bd}\Lambda^i:=({}^{\bd}\Lambda)^i,
\end{equation}
where the superscript on the right is defined in Section~\ref{sec:L-s}.
 


\subsection{Lusztig correspondence}\label{mlus} \label{sec3.2}
To parametrize the irreducible representations of $G^F$, we first follow Lusztig's Jordan decomposition (see  \cite[Section 2.3]{Ma17} and \cite[Theorem 5.2]{Gec17} for instance) to
partition $\Irr(G^F)$ according to the semisimple conjugacy classes in the dual group $G^{*F}$ of $G^F$, that is,
\begin{equation}\label{eq:irr-decomp-1}
\Irr(G^F)=\coprod_{(s)\in (G^{*})^{\circ F}_{ss}/ \Ad(G^{*F})}\mathcal{E}(G,s)
\end{equation}
where $(s)$ runs over the $G^{*F}$-conjugacy classes in
$(G^*)^{\circ F}_{ss}$, and
 the Lusztig series $\mathcal{E}(G,s)$ is defined by
\[
\mathcal{E}(G,s) := \set{
\pi \in \Irr(G^F)
\mid
\langle \pi, R_{T^*,s}^G\rangle \ne 0
\textrm{ for some $F$-stable maximal torus $T^*$ containing $s$}
}.
\]
When $G$ is a connected reductive group, there is a natural bijection between
$G^F$-conjugacy classes of pairs $(T,\theta)$, where $T$ is an $F$-stable
maximal torus of $G$ and $\theta$ is a character of $T^F$, and
$G^{*F}$-conjugacy classes of pairs $(T^*,s)$. Here $T^*$ is an
$F$-stable maximal torus of $G^*$ containing $s$.
If $(T, \theta)$ corresponds to $(T^*, s)$, we denote the Deligne--Lusztig character $R_T^G(\theta)$ in \cite{DL76} by
$R_{T^*,s}^G$. 
If $G$ is non-connected, we define
\[
        R^G_{T^*,s}:=
        \Ind^{G^F}_{(G^\circ)^F} R^{G^\circ}_{T^*,s}.
\]

Continuing \eqref{eq:irr-decomp-1},
we recall the correspondence between the Lusztig series $\CE(G,s)$ and the unipotent representations $\CE(C_{G^*}(s),1)$ of $C_{G^{*}}(s)^F$ in \cite[Chap 9]{Lu77}, \cite[Theorem 8.14]{Sr79} and  \cite[Lemma 2.7]{Ma17},
where $C_{G^*}(s)$ is the centralizer  of $s$ in $G^*$. 

\begin{prop}[Lusztig's Jordan decomposition]\label{Lus}
Let $G$ be a connected reductive group defined over $\mathbbm{k}$.
There is a bijection
\[
\CL_s:\mathcal{E}(G,s)\to \mathcal{E}(C_{G^{*}}(s), 1),
\]
which extends by linearity to a map between virtual characters such that
\[
\CL_s(\varepsilon_G R^G_{T^*,s})=\varepsilon_{C_{G^{*}}(s)} R^{C_{G^{*}}(s)}_{T^*,1},
\]
where $\varepsilon_G:=(-1)^{\mathrm{rk}\,G}$ and $\mathrm{rk}\,G$ is the
$\mathbbm{k}$-rank of $G$.  Moreover, $\CL_s$ sends cuspidal representations
to cuspidal representations.
\end{prop}
Lusztig's Jordan decomposition can be extended to even orthogonal groups (see \cite[Proposition 1.7]{AMR96} for instance). 

We remark that the choice of $\CL_s$ is not unique in general. To remove this ambiguity, 
we will introduce a canonical choice in Theorem \ref{thm:L-can}.

Next, we recall the description of $C_{G^{*}}(s)$ (see \cite[Section 1.B]{AMR96}  for instance) and then refine $\CE(C_{G^*}(s),1)$.
The multiset of eigenvalues of $s$ on the standard representation of $G^*$ (cf. \cite[Section 1.B]{AMR96}) is of the form
\[
(s)= \begin{cases}\set{x_1,\ldots,x_n, x_1^{-1}, \ldots, x_n^{-1}, 1},  & \textrm{if }G^*=\SO_{2n+1}; \\
\set{x_1,\ldots,x_n, x_1^{-1}, \ldots, x_n^{-1}}, & \textrm{otherwise}.
\end{cases}
\]
For $a\in \overline{\mathbbm{k}}^\times$, define 
 \[
\nu_a(s):=\#\set{i \mid x_i=a \textrm{ or } x_i^{-1}=a}
\quad \textrm{and} \quad
[a]:=\set{a^{q^k},a^{-q^k} \mid k\in\bb{Z}},
  \]
where $\#X$ denotes the cardinality of a finite set $X$.   
Clearly, $\nu_a(s)$ and $[a]$ depend only on the orbit of $a$ under $\Gamma_{\mathbbm{k}}$, the group of automorphisms of $\overline{\mathbbm{k}}^\times$ generated by the Frobenius map $F$ and the inversion $a\mapsto a^{-1}$.
The centralizer $C_{G^*}(s)$ has a natural decomposition 
\[
C_{G^*}(s)\cong \begin{cases}
\prod_{[a]}G^*_{[a]}(s)\times\{\pm\},&\textrm{if }G=\RO_{2n+1};\\
 \prod_{[a]}G^*_{[a]}(s),&\textrm{otherwise},
\end{cases}
\]
where $[a]$ runs over the orbits $(s)/\Gamma_{\mathbbm{k}}$, and
$G^*_{[a]}(s)$ is obtained by restriction of scalars from a classical
group defined over $\mathbbm{k}(a)$. If $a\ne\pm1$, then
$G^*_{[a]}(s)$ has $\overline{\mathbbm{k}}$-rank
\(
\frac{1}{2}\#[a]\nu_a(s).
\)
For convenience, we set
\[
G^*_{[\ne\pm 1]}(s) :=\prod_{[a], \, a\ne\pm 1 }G^*_{[ a]}(s).
\]
Then we rewrite
\begin{equation}\label{decomp}
 \begin{aligned}
C_{G^*}(s)
\cong &
G^*_{[ \ne \pm 1]}(s)\times G^*_{[ 1]}(s)\times G^*_{[ -1]}(s)\times 
\begin{cases}
\{\pm\},&\textrm{if }G=\RO_{2n+1};\\
1,&\textrm{otherwise}.
\end{cases}
 \end{aligned}
\end{equation}

In particular, we have the following explicit description of $G^*_{[a]}(s)$:
\begin{itemize}
    \item 
If $a\ne \pm 1$, then $G^*_{[a]}(s)$ is either $\Res_{{\mathbbm{k}}(a)/{\mathbbm{k}}}\GL_{\nu_a(s)}$ or the unitary group $\Res_{{\mathbbm{k}}(a)/{\mathbbm{k}}}\RU_{\nu_a(s)}$;
\item
If $G=\RO_{2n+1}$, then we may take $G^*=\Sp_{2n}\times\{\pm 1\} =\Sp(V^*)\times \{\pm1\}$, in which case
$
G^*_{[\pm 1]}(s)=\Sp_{2\nu_{\pm 1}(s)}= \Sp(V^*_{s=\pm 1});
$
\item
If $G=\Sp_{2n}$, then we may take $G^*=\SO_{2n+1}=\SO(V^*)$, in which case  
$G^*_{[ 1]}(s)=\SO_{2\nu_{1}(s)+1}=\SO(V^*_{s=1})$ and $G^*_{[-1]}(s)=\RO_{2\nu_{-1}(s)}=\RO(V^*_{s=-1})$;
\item
If $G=\RO_{2n}(V)$, then $G^*=\RO_{2n}(V^*)$ with $V^*=V$, and  
$
G^*_{[\pm 1]}(s)=\RO_{2\nu_{\pm 1}(s)} = \RO(V^*_{s=\pm 1}).
$
\end{itemize}
Here $V^*$ is the (skew) symmetric space defining the dual group $G^*$, and
$V^*_{s=\pm1}$ is the eigenspace of $s$ on $V^*$ corresponding to the
eigenvalue $\pm1$.
We remark that when $G^*$ is orthogonal,  $\RO(V^*_{s=\pm 1})$ depends on the discriminant of $V^*_{s=\pm 1}$.
 

Substituting \eqref{decomp} into Proposition~\ref{Lus}, we obtain
\begin{align}  \label{lcor}
& \mathcal E(G,s)
\nonumber\\
=\,& \cal{E}(G^*_{[\ne \pm 1]}(s)\times G^*_{[1]}(s)\times G^*_{[-1]}(s),1)
\times \begin{cases}
 \{\pm\},&\textrm{if }G=\RO_{2n+1};\\
  1,&\textrm{otherwise},
\end{cases}
\end{align}
where, by abuse of notation, $\{\pm\}$ also denotes the set of irreducible
representations of $\{\pm\}\subset G^*$ when $G=\RO_{2n+1}$.
Thus, for $\pi\in\mathcal E(G,s)$, we write
\[
        \CL_s(\pi)
        =
        \bigotimes_{[a]\in (s)/\Gamma_{\mathbbm{k}}}\pi[a]\otimes\eta,
\]
where $\eta\in\{\pm\}$ occurs only for $G=\RO_{2n+1}$. Here $\eta$ is the scalar by which the central element $-I$ acts.

\subsection{Theta correspondence} \label{subsec:theta}

We briefly recall the theta correspondence over finite fields, which will be used in the description of  $\Irr(G^F)$ later. 
Let $V$ and $V'$ be an $\epsilon$-symmetric space and a $(-\epsilon)$-symmetric space over $\mathbbm{k}$, respectively, where $\epsilon\in\{\pm\}$.  Let $G$ and $G'$ be their isometry groups. Let $\BW := \textrm{Hom}(V, V')$, and  define a symplectic form $\langle, \rangle_\BW$ on $\BW$ by
\[
        \langle T,S\rangle_\BW:=\operatorname{Tr}(T^*S),
        \quad T,S\in\BW,
\]
where the trace is taken for $T^*S$ as a $\mathbbm{k}$-linear endomorphism
of $V$. Here $T^*\in \Hom(V',V)$ is defined by
\[
(Tv,v')_{V'}=(v,T^*v')_V,\quad v\in V,\ v'\in V'.
\]
Let $\Sp(\BW)$ be the isometry group of $\langle \cdot, \cdot\rangle_\BW$. There is a natural homomorphism $G \times G' \to \Sp(\BW)$ given by
\[
(g, g') \cdot T = g'T g^{-1}\quad
\textrm{for } T \in \textrm{Hom}(V, V'),\  g \in G,\  g' \in G'.
\]
Then $(G, G')$ is a reductive dual pair inside $\Sp(\BW)$.

Fix a nontrivial additive character $\psi$ of $\mathbbm{k}$. 
Let $\omega_{N,\psi}$ be the Weil representation of the finite symplectic
group $\Sp(\BW)^F=\Sp_{2N}(\mathbbm{k})$ associated with the unique
irreducible representation of the Heisenberg group of $\BW$ with central
character $\psi$, where $\dim\BW=2N$.
Write
$\omega_{G,G'}$ for the restriction of $\omega_{N, \psi}$ to $G^F \times G'^F$, which decomposes into a direct sum
\[
\omega_{G, G'}=\bigoplus_{\pi,\pi'} m_{\pi,\pi '} \,\pi\otimes\pi '
\]
where $\pi$ and $\pi'$ run over $\Irr(G^F)$ and $\Irr(G'^F)$ respectively, and $m_{\pi,\pi'}$ are nonnegative integers. We can reassemble this decomposition as
\[
\omega_{G, G'}=\bigoplus_{\pi} \pi\otimes\Theta_{G,G'}(\pi )
\]
 where $\Theta_{G, G'}(\pi ) := \bigoplus_{\pi'} m_{\pi,\pi '}\pi'$ is a (not necessarily irreducible) representation of $G'^F$, called the (big) theta lift of $\pi$ from $G^F$ to $G'^F$. We will write $\pi'\subset \Theta_{G, G'}(\pi)$ if $\pi\otimes\pi'$ occurs in $\omega_{G,G'}$, i.e. $m_{\pi, \pi'}\neq 0$. We remark that even if $\Theta_{G,G'}(\pi)=\pi'$ is irreducible, one only
has in general
\[
        \pi\subset \Theta_{G',G}(\pi'),
\]
and equality need not hold. We will simply write $\Theta_{G, G'}$ as 
 $\Theta$ when the dual pair is clear from the context. 

\subsection{Classification of irreducible representations of classical groups}\label{5.2}

In this section, we first recall Lusztig's classification of unipotent
representations of classical groups \cite{Lu77,Lu81,Lu82,LS77}.  In this
classification, the unipotent representations of the factors
$G^*_{[a]}(s)$ are parametrized by Lusztig symbols as recalled in
Sections~\ref{sec:L-s} and~\ref{sec:bi}. Then we are ready to introduce the unique choice of Lusztig's correspondence $\CL_{\rm can}$ to label irreducible representations $\Irr(G^F)$. 

Continuing the decomposition \eqref{lcor},
if $G^*_{[a]}(s)$ is $\GL_{\nu_a(s)}$ or  $\UU_{\nu_a(s)}$, 
then Lusztig and Srinivasan \cite{LS77} give a canonical bijection between
$\CE(G^*_{[a]}(s),1)$ and the irreducible representations of the Weyl group
$S_{\nu_a(s)}$ of type $A$.
If $G^*_{[\pm1]}(s)$ is symplectic or orthogonal, we follow the notation in Sections~\ref{sec:L-s} and~\ref{sec:bi} and introduce a set $\CS(W)$ of symbols associated with a (skew) symmetric space $W$ of dimension $m$:
\begin{itemize}
\item  $\CS(W)=\set{\Lambda : \textrm{rank}(\Lambda)=\lceil \frac{m-1}{2} \rceil, \, \textrm{def}(\Lambda)\equiv 1\ (\textrm{mod }4)}$ if $W$ is symplectic or odd orthogonal; 
\item $\CS(W)=\set{\Lambda : \textrm{rank}(\Lambda)=\frac{m}{2},\, \textrm{def}(\Lambda)\equiv 0\ (\textrm{mod }4)}$ if $W$ is even orthogonal and $\disc(W)=+$;
\item $\CS(W)=\set{\Lambda : \textrm{rank}(\Lambda)=\frac{m}{2},\, \textrm{def}(\Lambda)\equiv 2\ (\textrm{mod }4)}$ if $W$ is even orthogonal and $\disc(W)=-$.
\end{itemize}
In \cite{Lu81, Lu82}, Lusztig gives bijections between $\CS(V^*)$ and the corresponding unipotent
series
\[
        \CE(\Sp(V^*),1),\quad
        \CE(\SO(V^*),1),\quad
        \CE(\RO(V^*),1),
\]
according to the type of $V^*$.

We now describe the Lusztig correspondence $\CL_s$ in terms of symbols. For fixed $s$, it labels each $\pi\in\mathcal E(G,s)$ by a quadruple
\begin{align*}
\CL_{s}:  &\mathcal{E}(G,s) \longrightarrow 
\mathcal{E}(G^{*}_{[ \ne\pm1]}(s), 1)\times\cal{S}(V^*_{s=1})\times\cal{S}(V^*_{s=-1})\times\{\pm\}^\star\\ 
&\pi\mapsto (\rho, \Lambda_1,\Lambda_{-1},\eta)
\end{align*}
where $\{\pm\}^\star=\{\pm\}$ if $G=\RO_{2n+1}$ and is a singleton otherwise; accordingly, $\eta$ occurs only for $G=\RO_{2n+1}$. 
We denote by $\pi^s_{\rho, \Lambda_1,\Lambda_{-1},\eta}$ or simply $\pi_{\rho, \Lambda_1,\Lambda_{-1},\eta}$  the irreducible representation of $G^F$ for the given quadruple.

For
\(
\pi=\pi^s_{\rho,\Lambda_1,\Lambda_{-1},\eta},
\)
write
\[
\rho=\prod_{[a]\ne[\pm1]}\pi[a],
\]
where $\pi[a]$ is the unipotent representation of
$G^*_{[a]}(s)^F$ parametrized by the partition $\lambda[a]$.
For $i\geq0$, set
\[
\rho_i=\prod_{[a]\ne[\pm1]}\pi[a]_i,
\]
where $\pi[a]_i$ is the unipotent representation of the group of the
same type as $G^*_{[a]}(s)^F$ parametrized by the partition obtained
from $\lambda[a]$ by removing its first $i$ columns.

Finally, we are ready to state our unique canonical choice of Lusztig correspondence $\CL_{\rm can}$. Recall from Section~\ref{subsec:theta} that we have fixed a nontrivial additive character $\psi$ of $\mathbbm{k}$ to define the Weil representation and theta correspondence.
\begin{thm}[\cite{Wang23}]\label{thm:L-can}
 Suppose that the cardinality of $\mathbbm{k}$ is large
enough so that the main result in \cite{Sr79} holds.
There is a unique choice of Lusztig correspondence
\[
\CL_{\rm can} \colon \Irr(G^F) \to 
\coprod_{(s)\in (G^{*})^{\circ F}_{ss}/ \Ad(G^{*F})}
\mathcal{E}(G^{*}_{[ \ne\pm1]}(s), 1)\times\cal{S}(V^*_{s=1})\times\cal{S}(V^*_{s=-1})\times\{\pm\}^\star
\]
satisfying the following conditions.
\begin{enumerate}
\item It is compatible with parabolic induction (cf. \cite[Section 3.4]{Wang23}). 

\item  Consider the dual pair $(G,G')=(\Sp_{2n},\RO^{\epsilon'}_{2n'+1})$.
 Let $\prll^s\in \Irr(\Sp_{2n}(\mathbbm{k}))$, and $\prllpz^{s'}\in \Irr(\RO^{\epsilon'}_{2n'+1}(\mathbbm{k}))$. Then $\prllpz^{s'}$ occurs in $\Theta(\prll^{s})$ if and only if the following conditions hold:
\begin{itemize}
    \item For any $a\in\overline{\mathbbm{k}}^\times$ with $a\ne\pm1$, one has $G^*_{[a]}(s)\cong (G')^*_{[-a]}(s')$ and $\pi[a]\cong \pi'[-a]$, where $\rho=\prod \pi[a]$ and $\rho'=\prod \pi'[a]$;
    \item $\pi_{\Lambda'_{-1}}=\pi_{\Lambda_{1}}$;
    \item $\pi_{\Lambda'_1}$ occurs in $\Theta_{\epsilon'\epsilon_\rho}(\pi_{\Lambda_{-1}})$;
    \item $\eta=\widetilde{\iota_\rho}(-1)^{\left\lfloor\frac{\rm{def}(\Lambda_{-1})}{2}\right\rfloor}\varepsilon(-1)^{\rm{rank}(\Lambda_1)}$.
\end{itemize}
Here $\widetilde{\iota_\rho}$ is defined in
\cite[Section 3.1]{Wang23}.  We set
\begin{equation}
\eps_\rho:=(-1)^{\#\{ [a]\in (s')/\Gamma_{\mathbbm{k}}\mid \  a\ne \pm 1,~G^{\prime *}_{[a]}(s')\text{ is unitary}\}},
\end{equation}
and define
\[
\Theta_{\epsilon'\epsilon_\rho}(\pi):=\begin{cases}
    \Theta(\pi), &  \textrm{if }\epsilon'\epsilon_\rho=+;\\
     \Theta(\rm{sgn}\cdot\pi), & \textrm{otherwise}.
\end{cases}
\]
In other words, we have the following commutative diagram
\[
\xymatrix{
\prll^s\in \Irr(\Sp_{2n}(\mathbbm{k})) \ar[r]^\Theta \ar[d]_{\CL_{\rm can}} & \prllpz^{s'}\in  \Irr(\RO^{\epsilon'}_{2n'+1}(\mathbbm{k})) \ar[d]^{\CL_{\rm can}} \\
\rho\otimes\pi_{\Lambda_{-1}}\otimes\pi_{\Lambda_1} \ar[r]^(0.35){ \mathrm{iso}\otimes\Theta_{\epsilon'\epsilon_\rho}\otimes\mathrm{id} } &  \rho'\otimes\pi_{\Lambda'_1}\otimes\pi_{\Lambda'_{-1}}\otimes\eta.
}
\]

\item Consider the dual pair $(G,G')=(\Sp_{2n},\RO^{\epsilon'}_{2n'})$.
Let $\prll^s\in \Irr(\Sp_{2n}(\mathbbm{k}))$ and
$\prllp^{s'}\in \Irr(\RO^{\epsilon'}_{2n'}(\mathbbm{k}))$.  Then $\prllp^{s'}$
occurs in $\Theta(\prll^s)$ if and only if the following conditions hold:
\begin{itemize}
\item For any $a\in\overline{\mathbbm{k}}^\times$ with $a\ne\pm1$, $G^*_{[a]}(s)\cong (G')^*_{[a]}(s')$ and $\pi'[a]\cong \pi[a]$, where $\rho=\prod\pi[a]$, and $\rho'=\prod\pi'[a]$;
    \item $\pi_{\Lambda_{-1}'}=\pi_{\Lambda_{-1}}$;
    \item $\pi_{\Lambda_1'}$ occurs in $\Theta(\pi_{\Lambda_1})$.
\end{itemize}
In other words, we have the following commutative diagram
\[
\xymatrix{
\prll^s\in \Irr(\Sp_{2n}(\mathbbm{k})) \ar[r]^\Theta  \ar[d]_{\CL_{\rm can}} &  \prllp^{s'}\in  \Irr(\RO_{2n'}^{\eps'}(\mathbbm{k})) \ar[d]^{\CL_{\rm can}} \\ 
\rho\otimes\pi_{\Lambda_1}\otimes\pi_{\Lambda_{-1}}\ar[r]^{\rm{iso}\otimes\Theta\otimes\rm{id}} & \rho'\otimes\pi_{\Lambda_1'}\otimes\pi_{\Lambda_{-1}}.
}
\]

    \end{enumerate}
\end{thm}

\subsection{Families of unipotent representations}\label{family}

Let $W$ be the Weyl group of $G^\circ$. In \cite{Lu84}, Lusztig defined a partition of $\Irr(W)$ into families and, similarly, a
partition of the irreducible unipotent representations of $G^F$ into families
\[
\mathcal{E}(G,1)
=
\bigsqcup_{\mathcal F}\mathcal{E}(G,1)_{\mathcal F},
\]
with $\mathcal{F}$ running over the ($F$-stable) families of ${\rm Irr}(W)$. By abuse of notation, we still denote $\mathcal{E}(G,1)_{\mathcal{F}}$ by $\mathcal{F}$ when there is no confusion.
To each family $\mathcal{F}$, there is a corresponding $F$-stable nilpotent orbit $\SCOF$ of $G$, related via the Springer correspondence. Let $\cal{U}^{\rm{st}}_{\cal{F}}$ be the corresponding $F$-stable unipotent orbit.
It is known \cite{Lu92,AA07,Tay16} that unipotent representations belong to $\mathcal{F}$ if and only if they have unipotent support $\cal{U}^{\rm{st}}_{\cal{F}}$.  For an $F$-stable nilpotent orbit $\SCO$, we denote by
$\mathcal F_{\SCO}$ the family consisting of unipotent representations whose
unipotent support corresponds to $\SCO$.

For an $F$-stable nilpotent orbit $\SCO$ of $G$ and $u\in\SCO$, let
\[
        A(u)=C_G(u)/C_G^\circ(u)
\]
be the component group of the centralizer of $u$ in $G$.  Lusztig defined a
quotient $\overline A(\SCO)$ of $A(u)$, called the canonical quotient.  When \(G\) is an even orthogonal, symplectic, or special odd orthogonal group, Lusztig constructed, for the family \(\mathcal F\) corresponding to \(\SCOF\), a bijection between the unipotent representations in \(\mathcal F\) and
\(
{\cal M}(\overline A(\SCOF))
\),
where, for a finite group ${\cal G}$,
\[
{\cal M}({\cal G})
:=
\set{(x,\tau)\mid x\in{\cal G},\ \tau\in\Irr(C_{\cal G}(x))}/{\cal G}.
\]
For simplicity, we abbreviate ${\cal{M}}(\overline{A}(\SCOF))$ to  $\mf$. If \( \pi \) corresponds to \( (x, \tau) \in \mf\), we denote it by \( \pi_{\SCOF,x,\tau} \).

We note that unipotent representations admit an alternative parametrization via partitions or symbols. In what follows, we briefly clarify the correspondence between these two distinct parametrizations of unipotent representations.

If $G$ is a general linear group or a unitary group, then for each family
$\mathcal F$, the corresponding set $\mf$ is trivial. In this case, an irreducible unipotent representation is assigned the partition that corresponds to its unipotent support.
 
To unify the descriptions of families $\CF$ and sets $\mf$ for symplectic groups, special odd orthogonal groups, and even orthogonal groups, we employ a slightly different notation from \cite{Lu84}. For an even orthogonal symbol, we append $-\infty$ to the top row
when its defect is congruent to $0$ modulo $4$, and to the bottom row
when its defect is congruent to $2$ modulo $4$. Thus the resulting
extended symbol has defect congruent to $1$ modulo $4$. By doing so, we can naturally apply all operators associated with the symbols of symplectic groups to even orthogonal groups as well.
If $\pi_\Lambda\in\CF$, then $\CF$ consists of all unipotent representations $\pi_{\Lambda'}$ of $G^F$ such that the multisets of entries of $\Lambda$ and $\Lambda'$ agree. 
A symbol $\Lambda_M$ with $\pi_{\Lambda_M}\in \CF$ is of the form
\[
\Lambda_M= \begin{pmatrix}
Z_2\sqcup(Z_1-M)\\
Z_2\sqcup M
\end{pmatrix}
\]
where $Z_2$ is the set of elements which appear in both rows of $\Lambda$, $Z_1$
is the set of singles of $\Lambda$ and $M$ is a subset of $Z_1 $ such that
$|M| \equiv d\pmod 2$ with $|Z_1|=2d+1$.

Let $\Lambda_{M_0}$ be the special symbol in $\mathcal F$, so that
$M_0\subset Z_1$ is the set of its singles lying in the second row.
For $\Lambda_M\in\mathcal F$, define
\[
M^\#
=
(M\setminus M_0)\cup(M_0\setminus M),
\]
consisting of the elements that belong to exactly one of $M$ and $M_0$. Thus $M^\#$ is
precisely the set of singles that lie in different rows in
$\Lambda_M$ and in the special symbol. Since $|M|\equiv|M_0|\equiv d\pmod 2$, the set
$M^\#$ has even cardinality.
Let $V_{Z_1}$ be the $\mathbb F_2$-vector space of even-cardinality
subsets of $Z_1$, where the sum of two subsets is the set of elements
belonging to exactly one of them. Then
\[
\Lambda_M\longmapsto M^\#
\]
is a bijection from $\mathcal F$ onto $V_{Z_1}$.
Unlike Lusztig, who orders the elements of $Z_1$ increasingly in
\cite{Lu84}, we write them in descending order as
\(
z_1>z_2>\cdots>z_{2d+1},
\)
and put
\begin{equation}\label{ei}
     e_i=\{z_i,z_{i+1}\},
\qquad 1\leq i\leq 2d.
\end{equation}
Then $e_1,\ldots,e_{2d}$ form a basis of $V_{Z_1}$.  Define
\[
I^{\rm even}_{\mathcal F}
=
\langle e_2,e_4,\ldots,e_{2d}\rangle,
\qquad
I^{\rm odd}_{\mathcal F}
=
\langle e_1,e_3,\ldots,e_{2d-1}\rangle.
\]
Then
\[
V_{Z_1}
=
I^{\rm even}_{\mathcal F}
\oplus
I^{\rm odd}_{\mathcal F}.
\]
By \cite[Sections~4 and~13]{Lu84}, with the conventions adopted above,
the canonical quotient associated with $\mathcal F$ is identified with
\[
\overline A(\SCOF)
\cong
\begin{cases}
I^{\rm even}_{\mathcal F},
   &\text{if $G$ is symplectic or even orthogonal},\\
I^{\rm odd}_{\mathcal F},
   &\text{if $G$ is special odd orthogonal}.
\end{cases}
\]
Accordingly, define
\[
(I^+_{\mathcal F},I^-_{\mathcal F})
=
\begin{cases}
(I^{\rm even}_{\mathcal F},I^{\rm odd}_{\mathcal F}),
   &\text{if $G$ is symplectic or even orthogonal},\\
(I^{\rm odd}_{\mathcal F},I^{\rm even}_{\mathcal F}),
   &\text{if $G$ is special odd orthogonal}.
\end{cases}
\]
Thus $I^+_{\mathcal F}$ is always identified with
$\overline A(\SCOF)$. The natural perfect pairing between
$I^{\rm even}_{\mathcal F}$ and $I^{\rm odd}_{\mathcal F}$ identifies
$I^-_{\mathcal F}$ with the character group of
$I^+_{\mathcal F}$. Consequently, we identify
\[
\mf
=
{\cal M}(\overline A(\SCOF))
\cong
\set{(x^+,x^-)\mid
x^\pm\in I^\pm_{\mathcal F}}.
\]
To describe these components explicitly, for each $j$, let
$\delta_j=+$ if $z_j$ lies in the same row of $\Lambda_M$ as in the
special symbol of $\mathcal F$, and let $\delta_j=-$ otherwise. In the
bases above, the $I^{\rm even}_{\mathcal F}$- and
$I^{\rm odd}_{\mathcal F}$-components of $M^\#$ have coordinates
\[
(\delta_1\delta_2,\,
 \delta_1\delta_2\delta_3\delta_4,\,
 \ldots,\,
 \delta_1\cdots\delta_{2d})
\ \textrm{ and }\
(\delta_1,\,
 \delta_1\delta_2\delta_3,\,
 \ldots,\,
 \delta_1\cdots\delta_{2d-1}),
\]
respectively. Under the preceding convention, these components are
denoted by $x^+$ and $x^-$, with $x^+$ corresponding to the
canonical-quotient parameter. Thus
\[
\pi_{\Lambda_M}
=
\pi_{\SCOF,x^+,x^-}.
\]

For $G=\mathrm O_{2n+1}^{\epsilon}$, irreducible unipotent
representations are parametrized by pairs
$(\Lambda_M,\eta)$ with $\eta\in\{\pm\}$. With the same coordinates
$x^+$ and $x^-$ as above, we write
\[
\pi_{\Lambda_M,\eta}
=
\pi_{\SCOF,x^+,x^-,\eta}.
\]

\section{Wavefront set and nilpotent orbits} \label{sec:notation}

In this section, we recall the definition of the rational wavefront set for representations of $G^F$.  We then parameterize $F$-rational nilpotent orbits and define the rational orbit order used throughout the paper.

Let $\CN$ be the nilpotent cone of the Lie algebra $\Fg$ of $G$, and let
$\CU$ be the unipotent variety of $G$.  There exists a $G$-equivariant
isomorphism
\[
        \CU\longrightarrow \CN
\]
which is compatible with $F$.  For instance, one may take the Cayley map
\[
        \operatorname{Cay}:\CU\longrightarrow \CN
\]
as in \cite[Theorem III.3.14]{SS70}.

For each $F$-rational nilpotent orbit $\CO\subset \CN^{F}$, define
$\Gamma_{\CO}$ to be the corresponding generalized Gelfand-Graev representation of $G^F$.
We recall the definition of $\Gamma_\CO$ following \cite{Lu92} and \cite{GZ14}. Let $\{X, Y, H\}\subset \frak{g}^F$ be an $\frak{sl}_2$-triple with $X\in \CO$, and decompose $\frak{g} = \bigoplus_{i\in \mathbb{Z}}\frak{g}_i$, where 
$\frak{g}_i$ is the $i$-eigenspace of ${\rm ad}\, H$. For $i\in \BN$ put $\Fg_{\geqslant i} := \bigoplus_{j\geqslant i}\Fg_j$. Let $\langle\, ,\,\rangle$ be a fixed $G$-invariant nondegenerate symmetric
bilinear form on $\Fg$, defined over $\mathbbm{k}$. Define a linear form 
\[
\lambda: \Fg_{\geqslant 1} \to \overline{\mathbbm{k}},\quad x\mapsto \langle Y, x\rangle. 
\]
Then $(x,y):=\lambda([x,y])$ defines a non-degenerate symplectic form on $\Fg_1$. Put $\Fg_{\geqslant 1.5}:=\Fl\oplus \Fg_{\geqslant 2}$, where $\Fl$ is a fixed $F$-stable Lagrangian subspace of $\Fg_1$. Then $\Fg_{\geqslant 1.5}$ is a Lie subalgebra corresponding to a closed connected unipotent subgroup $G_{\geqslant 1.5}$ of $G$. 

Recall the fixed  non-trivial additive character
$\psi: \mathbbm{k}\to  \BC^\times$.
Define the character \(\psi_\CO\) of \(G_{\geqslant1.5}^F\) by
\[
G_{\geqslant 1.5}^F \xlongrightarrow{\operatorname{Cay}} \Fg_{\geqslant 1.5}^F \xlongrightarrow{\lambda} \mathbbm{k} \xlongrightarrow{\psi} \BC^\times,
\]
where $\operatorname{Cay}$ denotes the restriction of the fixed
Cayley map to $G_{\geqslant 1.5}$.
The generalized Gelfand-Graev  representation $\Gamma_\CO$ is defined to be the induced representation 
\[
\Gamma_\CO :=\Ind^{G^F}_{G_{\geqslant 1.5}^F}\psi_\CO.
\]
The isomorphism class of $\Gamma_\CO$ is independent of the choices of
the $\frak{sl}_2$-triple and the Lagrangian subspace $\Fl$, and depends
only on $\CO$ and $\psi$; see \cite{Lu92,GZ14}.

\begin{defn}[Rational wavefront set $\ScO(\pi)$]
For a finite-dimensional representation $\pi$ of $G^F$, define its
{\bf rational wavefront set} by
\[
        \ScO(\pi)
        :=
        \bigcup_{\substack{\CO\subset \CN^F\\
        \apair{\pi,\Gamma_\CO}\ne 0}}
        \CO .
\]
Equivalently,
\[
        \ScO(\pi)
        =
        \set{X\in\CN^F\mid \apair{\pi,\Gamma_{\CO_X}}\ne 0},
\]
where $\CO_X$ is the $G^F$-orbit of $X$.
\end{defn}

Recall that the generalized Whittaker model of $\pi$ associated with $\CO$ is the space
\[
\operatorname{Wh}_\CO(\pi) := \Hom_{G_{\geqslant 1.5}^F}(\psi_\CO,\pi).
\]
This definition of $\operatorname{Wh}_\CO(\pi)$ differs from the usual one; we adopt it for convenience in the descent constructions below.
By Frobenius reciprocity,
\[
\langle  \Gamma_\CO,\pi\rangle = \dim \operatorname{Wh}_\CO(\pi).
\]
Alternatively, $\ScO(\pi)$ is the union of all the $\mathbbm{k}$-rational nilpotent orbits $\CO$ such that the corresponding generalized Whittaker models $\operatorname{Wh}_\CO(\pi)$ are nonzero.

Let \(M_\CO:=Z_G(H)\).  It is a Levi subgroup which normalizes
$G_{\geqslant 1}$.  Let 
\(
        M_{\psi_\CO}=\{m\in M_\CO\mid \Ad(m)(Y)=Y\}.
\)
There exists a canonical representation $\omega_{\psi_\CO}$ of
\(
        M_{\psi_\CO}^F\ltimes G_{\geqslant 1}^F
\)
which extends the representation \(     \Ind^{G_{\geqslant 1}^F}_{G_{\geqslant 1.5}^F}\psi_\CO\)
of $G_{\geqslant 1}^F$.  When $\Fg_1\ne 0$, we refer to
$\omega_{\psi_\CO}$ as the Weil representation associated with
$\psi_\CO$.  Then, for a representation $\pi$ of $G^F$, the Hom space
\[
\operatorname{Wh}_\CO(\pi)
\cong
\Hom_{G_{\geqslant 1}^F}
\bigl(\omega_{\psi_\CO},\pi\bigr)
\]
is naturally a representation of $M_{\psi_\CO}^F$. When $\Fg_1=0$, $\omega_{\psi_\CO}$ is the extension of
$\psi_\CO$ to
$M_{\psi_\CO}^F\ltimes G_{\geqslant 1}^F$
which is trivial on $M_{\psi_\CO}^F$.

\subsection{Parameterization of rational orbits}  \label{lcq}

In this subsection, we recall the parametrization of rational nilpotent
orbits by signed Young diagrams, following
\cite[Section~9.2]{CM93} and \cite{GZ14}. We then describe the maps
$S$ and $\overline S$ from rational orbits to the corresponding
component groups and  Lusztig's canonical quotients.

Following the definition of discriminant in \eqref{eq:disc}, we may identify $\mathbbm{k}^{\times}/(\mathbbm{k}^\times)^2$ with $\{\pm\}$ to assign the signs on the Young diagram. 
For a (skew) symmetric space $V$,
let \(\CY(V)\) denote the set of
signed Young diagrams
$[(\lambda_1^{n_1},a_1),\dots,(\lambda_l^{n_l},a_l)]$ satisfying the following conditions:
\begin{itemize}
\item
$\lambda_1>\cdots>\lambda_l>0$ and $n_i\in\mathbb Z_{>0}$;
    \item $\sum_i n_i\lambda_i=\dim V$ and
$a_i\in\mathbbm{k}^\times/(\mathbbm{k}^{\times})^2=\{\pm\}$, 
    \item  if $V$ is skew-symmetric and $\lambda_i$ is odd, then $n_i$ is even and $a_i=+$,
    \item  if $V$ is symmetric and $\lambda_i$ is even, then $n_i$ is even and $a_i=+$,
    \item if $V$ is symmetric, then $\prod_{1\leqslant k<t\leqslant l}\varepsilon(-1)^{n_k n_t}\cdot \prod_{i} a_i=\disc(V)$,
\end{itemize}
where 
\begin{equation} \label{eq:eps-1}
        \varepsilon(-1)\in\mathbbm{k}^{\times}/(\mathbbm{k}^{\times})^2
\end{equation}
denotes the square class of $-1$.

For each part $\lambda_i^{n_i}$, there is an auxiliary space of dimension $n_i$ carrying a form determined by the parity of $\lambda_i$ and by the form on $V$. When this auxiliary form is symmetric, the sign $a_i$ records its discriminant: $a_i=+$ if the discriminant is a square, and $a_i=-$ otherwise. For convenience, we sometimes omit the sign $a_i=+$ when
$V$ is skew-symmetric and $\lambda_i$ is odd, or when $V$ is symmetric and $\lambda_i$ is even.

As in \cite[Proposition~3.5]{GZ14}, there is a  one-to-one
correspondence  
\[
\CN^F/\Ad(G^F) \longleftrightarrow
\CY(V).
\]
By abuse of notation, we also denote by $[(\lambda_1^{n_1},a_1),(\lambda_2^{n_2},a_2),\dots,(\lambda_l^{n_l},a_l)]$ the $F$-rational nilpotent orbit corresponding to this Young diagram.

For symplectic and special odd orthogonal groups, given a unipotent
element $u\in G^F$, let $\SCO$ be the $F$-stable nilpotent orbit
containing $\operatorname{Cay}(u)$. By
\cite[Proposition~4.8]{BDT20}, there is a bijection
\[
S_u:\{\text{rational nilpotent orbits in }\SCO\}
\longrightarrow H^1(F,A(\SCO)).
\]
We may choose $u$ so that $F$ acts trivially on $A(\SCO)$; hence
$H^1(F,A(\SCO))=A(\SCO)$, and $S_u$ induces a surjection
$\overline S_u$ onto $\overline A(\SCO)$.

For full orthogonal groups, let $V^\epsilon$ be the quadratic space of
the relevant dimension with $\disc(V^\epsilon)=\epsilon$; in even
dimension, $V^+$ is split. In odd dimension, rational orbits associated
with $V^+$ and $V^-$ are regarded as distinct orbits. We write $A_{\SO}(\SCO)$ and $A_{\RO}(\SCO)$ for the
component groups defined using $\SO(V^\epsilon)$ and
$\RO(V^\epsilon)$, respectively.
On $V^+$, choose $u\in\SO(V^+)^F$ and let $\SCO$ be the $F$-stable
orbit containing $\operatorname{Cay}(u)$. The cocycle construction
gives an injective map
\[
S_u^+:
\{\text{rational nilpotent orbits in $\SCO$ associated with }V^+\}
\longrightarrow A_{\RO}(\SCO),
\]
and, by composition with the quotient map, a map $\overline S_u^+$.
In odd dimension, these maps are obtained from those for $\SO(V^+)$
through the inclusion
$A_{\SO}(\SCO)\subset A_{\RO}(\SCO)$. At the end of this subsection,
we extend them to the rational orbits associated with $V^-$.
When no confusion can arise, we write $A(\SCO)$ and
$\overline A(\SCO)$ for $A_{\RO}(\SCO)$ and
$\overline A_{\RO}(\SCO)$, respectively.

We first fix a normalization for symplectic, special odd orthogonal
groups and full split even orthogonal groups. Suppose that the orbit $\SCO$
has partition
\(
[\lambda_1^{n_1},\ldots,\lambda_l^{n_l}],
\)
and set
\[
k_i:=n_i\left(\sum_{m=1}^{i-1}n_m\right).
\]
We choose $u$ so that $\operatorname{Cay}(u)$ lies in
\[
[(\lambda_1^{n_1},\varepsilon(-1)^{k_1}),\ldots,
 (\lambda_l^{n_l},\varepsilon(-1)^{k_l})],
\]
which we abbreviate as
\(
[(\lambda_i^{n_i},\varepsilon(-1)^{k_i})].
\)
The values $S_u(\CO)$ or
$S_u^+(\CO)$ are independent of the choice of $u$ in
this orbit, and in both cases we write them as $S(\CO)$.
This normalization declares
\(
[(\lambda_i^{n_i},\varepsilon(-1)^{k_i})]
\)
to be the origin on the rational-orbit side. It is chosen so that, for unipotent representations corresponding to special symbols, the set $S(\SOM)$ maps to the origin of
Lusztig's canonical quotient; see Theorem~\ref{main4}.

To explicitly describe the map $\overline{S}$, we recall Sommers' description \cite{So01} of Lusztig's canonical quotient. He first characterizes the kernel $H(\SCO)$ of $\overline{A}(\SCO)=A(\SCO)/H(\SCO)$, then defines a subgroup $K(\SCO)$ of $A(\SCO)$ which maps bijectively onto $\overline{A}(\SCO)$. The description of $K(\SCO)$ provided by Sommers is as follows:
\begin{itemize}
    \item In type A, Lusztig's canonical quotient is trivial.
    \item In other classical types, $\overline{A}(\SCO)$ is an elementary 2-group, and there exists a subgroup $K(\SCO)$ of $A(\SCO)$ such that $K(\SCO) \cong \overline{A}(\SCO)$. Assume the partition of $\SCO$ has the form $[\lambda_1^{m_1}, \cdots, \lambda_l^{m_l}]$. Then $A(\SCO)$ is a subgroup of an elementary 2-group with basis $\{x_i \mid i = 1, \cdots, l\}$.
\item For odd special orthogonal groups, $A(\SCO)$ consists of sums of
an even number of elements in
$\{x_i\mid\lambda_i\text{ is odd}\}$, and $K(\SCO)$ consists of sums
of an even number of elements in
\[
\left\{
x_i\ \middle|\
\lambda_i\text{ is odd and }\sum_{j=1}^i m_j\text{ is odd}
\right\}.
\]
For full odd orthogonal groups, $A(\SCO)$ is generated by
$\{x_i\mid\lambda_i\text{ is odd}\}$, and $K(\SCO)$ is generated by
the displayed distinguished subset.

\item For full even orthogonal groups, $A(\SCO)$ is generated by
$\{x_i\mid\lambda_i\text{ is odd}\}$, and $K(\SCO)$ is generated by
\[
\left\{
x_i\ \middle|\
\lambda_i\text{ is odd and }\sum_{j=1}^i m_j\text{ is even}
\right\}.
\]
    \item For symplectic groups, $A(\SCO)$ is generated by $\{x_i \mid \lambda_i \text{ is even}\}$, and $K(\SCO)$ is the subgroup generated by 
    \[
\left\{
x_i\ \middle|\
\lambda_i\text{ is even and }\sum_{j=1}^i m_j\text{ is even}
\right\}.
\]
\end{itemize}

 We identify $\overline A(\SCO)$ with $K(\SCO)$. The quotient map
\(
A(\SCO)\longrightarrow\overline A(\SCO)\cong K(\SCO)
\)
is determined by the kernel $H(\SCO)$ described in the proof of
\cite[Theorem~6]{So01}.

Let $\{x_{r_1},x_{r_2},\ldots\}$ be the distinguished subset appearing
in the description of $K(\SCO)$, and set $r_0=0$. For a rational orbit
\(
\CO=[(\lambda_1^{n_1},a_1),\ldots,(\lambda_l^{n_l},a_l)],
\)
if $V_i$ is the quadratic space of dimension $n_i$ with
$\disc(V_i)=a_i$, then we set
\[
y_t:=\disc\left(
V_{r_{t-1}+1}\oplus\cdots\oplus V_{r_t}
\right)
=
\prod_{j=r_{t-1}+1}^{r_t}\delta_j \quad\textrm{with}\quad
\delta_i:=a_i\varepsilon(-1)^{k_i}.
\]
We write elements of $K(\SCO)$ in sign coordinates with respect to
$\{x_{r_t}\}$, so that $(y_1,y_2,\ldots)$ denotes
\(
\sum_{t:\,y_t=-}x_{r_t}.
\)
Then
\begin{equation}\label{S}
\begin{matrix}
\overline S:
&
\set{\CO\mid\CO\subset\SCO}
&
\longrightarrow
&
K(\SCO)\cong\overline A(\SCO)
\\
&
[(\lambda_1^{n_1},a_1),(\lambda_2^{n_2},a_2),\ldots]
&
\longmapsto
&
(y_1,y_2,\ldots).
\end{matrix}
\end{equation}
The chosen normalization gives
\[
\overline S
\bigl([(\lambda_i^{n_i},\varepsilon(-1)^{k_i})]\bigr)
=
(+,\ldots,+).
\]

For full orthogonal groups, \eqref{S} has so far been defined only on
the rational orbits associated with $V^+$. We now extend both $S$ and
$\overline S$ to the rational orbits associated with $V^-$.
Fix a partition
\(
[\lambda_1^{n_1},\ldots,\lambda_l^{n_l}],
\)
and let $\SCO$ denote the $F$-stable orbit for $V^+$ with this
partition. For a rational orbit
\(
\CO=[(\lambda_1^{n_1},a_1),\ldots,(\lambda_l^{n_l},a_l)]
\)
associated with either $V^+$ or $V^-$, define
\[
S(\CO)
=
\sum_{\substack{i\\
\lambda_i\text{ is odd}\\
\delta_i=-}}
x_i
\in A_{\RO}(\SCO).
\]
On $V^+$, this agrees with the cocycle parametrization $S$ under
the normalization chosen above.
Since
\(
\disc(V^\epsilon)
=
\prod_{\lambda_i\text{ is odd}}\delta_i,
\)
this gives a bijection
\[
S:
\coprod_{\epsilon\in\{\pm\}}
\left\{
\begin{array}{c}
\text{rational nilpotent orbits for $V^\epsilon$}\\
\text{having partition }
[\lambda_1^{n_1},\ldots,\lambda_l^{n_l}]
\end{array}
\right\}
\xrightarrow{\;\sim\;}
A_{\RO}(\SCO).
\]

The restriction of $S$ to the non-split even form does not arise as
$S_u$ from any choice of a base rational orbit on that form. Indeed,
$S_u$ must send its base orbit to the identity of $A(\SCO)$, whereas
the orbit sent to the identity by this formal map is the one associated
with the special symbol, and this orbit belongs to the split even
orthogonal group.

In odd dimension $2n+1$, this $S$ agrees with the extension determined by
$A_{\RO}(\SCO)= A_{\SO}(\SCO)\rtimes\langle -I_{2n+1}\rangle$: the orbits for $V^+$ map to
$A_{\SO}(\SCO)$, while those for $V^-$ map to the other coset.

We define $\overline S$ as the composition of $S$ with the canonical
quotient map. Equivalently, it is given on both forms by the same
sign-coordinate formula \eqref{S}.

\subsection{Partial order on rational orbits} \label{ssec:PO} 

To study the rational structure of $\ScO(\pi)$, we introduce a partial order on the $F$-rational nilpotent orbits in $\CN^{F}$ for orthogonal groups and symplectic groups. 

Let $X=[(\lambda_1^{n_1},a_1),(\lambda_2^{n_2},a_2),\dots]$ and $Y=[(\mu_1^{m_1},b_1),(\mu_2^{m_2},b_2),\dots]$ be two rational nilpotent orbits in $\CN^F$. 
Let $c_i(X)$ be the number of boxes in the $i$-th column of the signed Young diagram $X$ and
\[
        C_r(X):=\sum_{i=1}^r c_i(X)
\]
be the number of boxes in the first $r$ columns of the signed Young diagram
of $X$. Set
\[
{\rm sgn}_r(X)=\prod_{\lambda_i\geqslant r} a_i
\prod_{\substack{i<t\\ \lambda_i,\lambda_t\geqslant r}}
        \varepsilon(-1)^{n_i n_t}.
\]

Assume first that $V$ is skew-symmetric. We define
$X\leqslant Y$ by the following two conditions. 
\begin{enumerate}
\item[(i)] For every $r\geqslant 1$,
\(
        C_r(X)\geqslant C_r(Y).
\)
\item[(ii)] For every $s\geqslant 1$ such that
\(
        C_{2s-1}(X)=C_{2s-1}(Y),
\)
one has
\(
        {\rm sgn}_{2s-1}(X)={\rm sgn}_{2s-1}(Y),
\)
\end{enumerate}

Assume next that $V$ is symmetric. We define
$X\leqslant Y$ by the following two conditions.
\begin{enumerate}
\item[(i)] For every $r\geqslant 1$,
\(
        C_r(X)\geqslant C_r(Y).
\)
\item[(ii)] For every $s\geqslant 1$ such that
\(
        C_{2s}(X)=C_{2s}(Y),
\)
one has
\(
        {\rm sgn}_{2s}(X)={\rm sgn}_{2s}(Y),
\)
\end{enumerate}

For each $\CO\subset\CN^F$, define its rational closure by
\[
        \overline{\CO}
        :=
        \bigcup_{\substack{\CO'\subset\CN^F\\ \CO'\leqslant\CO}}
        \CO'.
\]
More generally, for a union $\mathcal S$ of rational nilpotent orbits, set
\[
        \overline{\mathcal S}
        :=
        \bigcup_{\substack{\CO'\leqslant\CO\\ \CO\subset\mathcal S}}
        \CO'.
\]
Define $\ScO(\pi)^{\max}$ to be the union of all rational nilpotent orbits
in $\ScO(\pi)$ which are maximal with respect to this rational orbit order.
Define $\ScO(\pi)^{\max,\rm st}$ to be the union of all rational nilpotent
orbits $\CO\in\ScO(\pi)$ such that $\CO^{\rm st}$ is maximal, with respect to the stable closure order, among stable orbits
which contain rational orbits occurring in $\ScO(\pi)$.

The $F$-stable nilpotent orbits that are maximal, with respect to the
stable closure order, among those containing an $F$-rational orbit
$\CO\in\ScO(\pi)$ form the Kawanaka wavefront set of $\pi$.
When $\pi$ is irreducible, this set consists of a single orbit
\cite{Lu92,GM00,AA07,Tay16}, also called the Kawanaka wavefront orbit
of $\pi$. Although $\ScO(\pi)^{\max,\rm st}$ is defined as a union of
$F$-rational nilpotent orbits, when no confusion can arise, by abuse
of notation we also use it for the unique $F$-stable orbit containing
them, namely the Kawanaka wavefront set of $\pi$.

Motivated by \cite{D82}, we describe the rational orbit order by
mutations of signed Young diagrams.  We use the convention
\[
        [(\lambda^m,a),(\lambda^{m'},a')]
        =
        [(\lambda^{m+m'},\varepsilon(-1)^{mm'}aa')],
\]
and hence
\[
        [(\lambda,a_1),\ldots,(\lambda,a_m)]
        =
        [(\lambda^m,
        \varepsilon(-1)^{\frac{m(m-1)}{2}}a_1\cdots a_m)].
\]
A mutation is an unordered local operation: before applying it, equal parts
may be split using the above convention; after applying it, the result is
reordered and equal parts are recombined.

We write \(X\rightsquigarrow_{\rm orb}Y\) if \(Y\) is obtained from \(X\) by
one of the following local operations.  Here
\(\lambda\geqslant\delta>0\) are even if \(V\) is skew-symmetric and odd if
\(V\) is symmetric, and \(a,b\in\{\pm\}\):
\[
\begin{array}{ll}
{\rm (1)}&
(\lambda-1)^2\to
(\lambda,a),(\lambda-2,\varepsilon(-1)a),\qquad \lambda>1,
\\[2pt]
{\rm (2)}&
(\lambda,a),(\delta,b)\to
(\lambda+2,a),(\delta-2,b),\qquad \delta>1,
\\[2pt]
{\rm (3)}&
(\lambda,a),(\delta,b)\to
(\lambda+2,-a),(\delta-2,-b),\qquad \delta>1,
\\[2pt]
{\rm (4)}&
(\lambda^2,a),(\delta,b)\to
(\lambda+1)^2,(\delta-2,ab),\qquad \delta>1,
\\[2pt]
{\rm (5)}&
(\lambda,a),(\delta^2,b)\to
(\lambda+2,ab),(\delta-1)^2,
\\[2pt]
{\rm (6)}&
(\lambda^2,a),(\delta^2,a)\to
(\lambda+1)^2,(\delta-1)^2 .
\end{array}
\]
A zero part is omitted.

\begin{lem}\label{mutation}
If \(X\leqslant Y\) are rational nilpotent orbits in \(\CN^F\), then there is
a finite chain
\[
        X=X_0\rightsquigarrow_{\rm orb}X_1
        \rightsquigarrow_{\rm orb}\cdots
        \rightsquigarrow_{\rm orb}X_m=Y .
\]
\end{lem}

\begin{proof}[Sketch of proof]
We treat the skew-symmetric case; the symmetric case is obtained by
interchanging the two parities.  The case \(X=Y\) is immediate, so
assume that \(X<Y\).  Let \(r\) be minimal such that
\[
        C_r(X)>C_r(Y),
\]
and let \((\mu_d^{k_d},a_d)\) be the block containing the lowest box
of \(X\) in the \(r\)-th column.  We construct from \(X\), by a mutation, a rational orbit
\(Z\leqslant Y\).  The lemma then follows by induction.
The comparison between \(C_i(Y)\) and \(C_i(Z)\) is routine: a
mutation moves only two boxes, so only the columns between their old
and new positions can change. The main point is to verify all rational sign conditions between
\(Y\) and \(Z\), choosing the signs in the mutation as needed.

Suppose first that \(r\) is odd.  If \(\mu_d\) is odd, mutation~{\rm(1)}, with a suitable
\(x\in\{\pm\}\), replaces two copies of \(\mu_d\) by
\[
        (\mu_d+1,x),\qquad
        (\mu_d-1,\varepsilon(-1)x),
\]
leaving the remaining copies unchanged.  If \(\mu_d\) is even and
either \(k_d\geqslant5\), or \(k_d=4\) and \(a_d=+\), mutation~{\rm(6)}
gives
\[
 [(\mu_d^{k_d},a_d)]
 \rightsquigarrow_{\rm orb}
 [(\mu_d+1)^2,(\mu_d^{k_d-4},a_d),(\mu_d-1)^2].
\]
In these cases, a direct calculation shows that the rational orbit
\(Z\) obtained by the indicated mutation satisfies \(Z\leqslant Y\).  
The cases of small \(k_d\) are handled by a suitable choice among
mutations~{\rm(2)}--{\rm(5)}. However, 
the guiding observation is the same in every case.  

First assume that
\(k_d>1\).  If \(\mu_d>r+1\), then, by parity,
\(\mu_d\geqslant r+3\).  Since \(Y\) has at least one fewer box than \(X\) in every column
from \(r\) through \(\mu_d\), we have
\[
 C_s(X)-C_s(Y)\geqslant s-r+1
 \qquad(r\leqslant s\leqslant\mu_d).
\]
For the mutations used here, any affected column not exceeding
\(\mu_d\) is at least \(\mu_d-1\).  Since
\(\mu_d\geqslant r+3\), the column-sum difference between $X$ and $Y$ there is at least
three and hence absorbs the change caused by moving two boxes in the mutation.  The remaining affected columns and any sign condition arising at an
equality are checked directly from the chosen mutation and the
condition used to select it.  Hence \(Z\leqslant Y\).
Likewise, if \(\mu_d=r+1\) and
\(
        C_r(X)-C_r(Y)\geqslant2,
\)
the same column comparison shows that \(Z\leqslant Y\).

The only delicate situation is therefore
\[
        \mu_d=r+1,\qquad C_r(X)-C_r(Y)=1.
\]
If \(c_{r+1}(Y)<c_r(Y)\), then $C_{r+1}(X)-C_{r+1}(Y)$
is already at least three.  Thus only the sign condition ${\rm sgn}_{r}(Z)={\rm sgn}_{r}(Y)$  at \(r\) can become
relevant, and the available free signs in the mutation are
chosen so that \({\rm sgn}_r(Z)\) agrees with
\({\rm sgn}_r(Y)\).

Assume instead that
\(
        c_{r+1}(Y)=c_r(Y).
\)
Let \(y\) be the length of the row of \(Y\) containing its lowest box
in column \(r\); then \(y\geqslant r+1\).  Different configurations
of \(Y\) require different mutations to obtain \(Z\).  Nevertheless,
regardless of \(Y\) and the corresponding mutation, the argument
establishing \(Z\leqslant Y\) is the same: the comparison between
\(C_i(Z)\) and \(C_i(Y)\) proceeds as in the preceding cases.
For the sign conditions, if only \(\operatorname{sgn}_r\) needs to be
checked, we make the
available sign choices in the mutation so that
\(
        \operatorname{sgn}_r(Z)=\operatorname{sgn}_r(Y).
\)
Suppose that the signs at one or more further odd columns \(r'>r\)
must also be compared.  This can occur only when
\(C_{r'}(Z)=C_{r'}(Y)\). A direct calculation then shows that \(y\geqslant r'\), and that the
same rows of \(Y\), as well as the corresponding rows of \(Z\), meet
every column from \(r\) through \(r'\).  Thus the rows involved in the
definitions of \({\rm sgn}_{i}(Y)\) and \({\rm sgn}_{i}(Z)\) do not
change for \(r\leqslant i\leqslant r'\).  Hence
\[
        \operatorname{sgn}_{r'}(Y)=\operatorname{sgn}_r(Y)
       =
        \operatorname{sgn}_{r}(Z)=\operatorname{sgn}_{r'}(Z)
\]
at every such column \(r'\).  Thus all the required sign comparisons are satisfied by the same
available sign choices.

When \(k_d=1\), one may have to mutate the next block
\((\mu_{d-1}^{k_{d-1}},a_{d-1})\), or the following block as well,
but the argument is identical: either the signs of \(Y\) and \(Z\)
need to be matched at only one column, or all the relevant signs of
\(Y\) coincide and all the corresponding signs of \(Z\) coincide.

Now suppose that \(r\) is even. By the minimality of \(r\), the column sums of \(X\) and \(Y\) agree
through column \(r-1\); in particular,
\(c_{r-1}(X)=c_{r-1}(Y)\).  It follows that
\[
 C_r(X)-C_r(Y)
 =
 m_{r-1}(Y)-m_{r-1}(X)=:m,
\]
where \(m_x(A)\) denotes the multiplicity of \(x\) in the underlying
partition of \(A\).  The part \(r-1\)
is odd, so the multiplicities \(m_{r-1}(Y)\) and \(m_{r-1}(X)\) are
even; consequently, \(m\) is a positive even integer.

Thus \(Y\) has exactly \(m\) more rows of length \(r-1\) than \(X\).
Since \(m\) is even and, as \(r-1\) is odd, every row of length
\(r-1\) has sign \(+\), these additional rows do not change the
discriminant when one passes from the \((r-1)\)-st column to the
\(r\)-th column.
Therefore
\[
        {\rm sgn}_{r-1}(Y)={\rm sgn}_r(Y).
\]
Although the sign at the even index \(r\) is not itself part of the
order condition, it provides a common reference sign for the subsequent comparison between \({\rm sgn}_{t}(Y)\) and
\({\rm sgn}_{t}(Z)\) at odd indices \(t>r\).

If \(m\geqslant4\), the same column comparison as in the odd \(r\)
case shows that no new equality is created, and hence no new sign
condition arises.  The only delicate case is \(m=2\).
The appropriate mutation, possibly involving the next block
\((\mu_{d-1}^{k_{d-1}},a_{d-1})\), or also the following block when
\(k_d=1\), is chosen according to the parity and multiplicity of the
relevant blocks and the last column at which equality already holds.  If an odd column \(t>r\) becomes a new equality, equality in
the column estimates forces no row of \(Y\) to end between columns
\(r\) and \(t\).  Hence
\[
        {\rm sgn}_t(Y)
        ={\rm sgn}_r(Y)
        ={\rm sgn}_{r-1}(Y).
\]
Since \(X\leqslant Y\) and \(C_{r-1}(X)=C_{r-1}(Y)\), we have
\({\rm sgn}_{r-1}(X)={\rm sgn}_{r-1}(Y)\).  Moreover, none of the moves in the mutation can remove a box from the
\((r-1)\)-st column, since otherwise
\(C_{r-1}(Z)<C_{r-1}(Y)\).  As the mutation preserves the
product of the signs of the rows meeting this column, it follows
automatically that
\(
        {\rm sgn}_{r-1}(Z)={\rm sgn}_{r-1}(X).
\)
We then make the available sign choices, when necessary, so that
\({\rm sgn}_t(Z)={\rm sgn}_{r-1}(Z)\).  Together with
\({\rm sgn}_t(Y)={\rm sgn}_{r-1}(Y)\), this gives
\[
        {\rm sgn}_t(Z)={\rm sgn}_t(Y),
\]
as required.

We have therefore produced, whenever \(X<Y\), a nonempty sequence of
mutations ending at an orbit \(Z\leqslant Y\) whose underlying
partition is strictly larger than that of \(X\).  Iterating this
construction terminates, and the terminal orbit must be \(Y\).
\end{proof}

\begin{exmp}
Let $G=\Sp_6$.
Then $[(4,+),(2,\varepsilon(-1))]\geqslant[(3,3)]$ and $[(4,-),(2,-\varepsilon(-1))]\geqslant[(3,3)]$.
However, the three orbits $[(4,-),(2,\varepsilon(-1))]$, $[(4,+),(2,-\varepsilon(-1))]$, and $[(3,3)]$ are pairwise incomparable.
\end{exmp}

\section{Descents of symplectic groups} \label{sec4}

We introduce the two types of elementary descents used to decompose generalized
Whittaker models for symplectic groups.  They are attached to the
rational nilpotent orbits
\[
[(2\ell,a),1^{2n-2\ell}]
\quad\text{and}\quad
[(2\ell+1)^2,1^{2n-2(2\ell+1)}],
\]
and are called the Fourier--Jacobi descent and the
\((2\ell+1)^2\)-descent, respectively.

\subsection{Fourier--Jacobi descents} \label{sec4.1}
In this section, we recall the Fourier--Jacobi descents of $\Sp_{2n}$ associated with the partition $[2\ell, 1^{2n-2\ell}]$.
There are two rational nilpotent orbits 
\[
\CO=\CO_{2\ell, a}:=[(2\ell,a), 1^{2n-2\ell}], \quad a\in \{\pm\}
\]
corresponding to the partition $[2\ell, 1^{2n-2\ell}]$, where $0<\ell\leqslant n$. For convenience, we sometimes omit $1^{2n-2\ell}$, and abbreviate the corresponding signed Young diagrams to $[(2\ell,a)]$.

Throughout this paper, we define $\Sp_{2n}$ using the skew-symmetric matrix
 \begin{equation}\label{eq:J2n}
 J_{2n} := \begin{pmatrix}
 0 & w_n \\ -w_n & 0
 \end{pmatrix},
 \quad \textrm{where}\quad 
 w_n:=\ppair{\begin{smallmatrix} 0 & \cdots & 1 \\
 & \begin{sideways} $\ddots$ \end{sideways} & \\
 1 & \cdots & 0\end{smallmatrix}}_{n\times n},
 \end{equation}
and we  choose the semisimple element $H$ in the $\frak{sl}_2$-triple corresponding to $\CO$ to be the diagonal matrix with decreasing diagonal entries. Set $N_\CO := G_{\geqslant 1.5}$ to be the unipotent group of $G$ consisting of the upper triangular matrices
\begin{equation}
u = \ppair{\begin{smallmatrix}
    x & v & y  \\
      & I_{2n-2\ell} & v' \\
      & & x^*
\end{smallmatrix}}
\end{equation} 
where $x \in Z_\ell$, $ x^*:= w_{\ell} (x^{-1})^t w_{\ell}$,  $v\in M_{\ell\times (2n-2\ell)}^0$ and $y\in M'_{\ell\times \ell}$.
Here and throughout the paper, we denote by $Z_\ell$ the subgroup of
unipotent upper triangular matrices of $\GL_\ell$. For \(m'\geqslant1\) and \(n'\geqslant0\), set
\begin{equation}\label{M0}
M_{m'\times 2n'}^0
:=
\set{v\in M_{m'\times 2n'} \mid
v_{m',1}=\cdots=v_{m',n'}=0}.
\end{equation}
For $\ell\geqslant 1$, set
\[
 M'_{\ell\times \ell}
 :=
 \set{y\in M_{\ell\times \ell}\mid w_{\ell}y=y^t w_{\ell}}.
\]
The character $\psi_{\CO}$ of $N_\CO^F$ is given by
\begin{equation}\label{eq:psi-F-J}
\psi_{\CO}(u):= \psi(x_{1,2}+\cdots+x_{\ell-1,\ell}+\Fa \cdot y_{\ell,1}),
\end{equation}
where $\Fa\in \mathbbm{k}^\times$ is a non-square element if $a=-$, and is a square if $a=+$, i.e., $\disc(\Fa)=a$. 

Then
\(
        M_{\CO}\simeq \GL_1^\ell\times \Sp_{2n-2\ell},
\)
where the $\GL_1$-factor is embedded diagonally in the product of the
$\GL_1$-factors arising from the grading.  The stabilizer of
$\psi_\CO$ in $M_\CO$ is
\[
        M_{\psi_\CO}\simeq \RO_1\times \Sp_{2n-2\ell}.
\]
For a representation $\pi$ of $G^F$, we define the descent
$\CD^{\rm FJ}_{2\ell,a}(\pi)$ of $\pi$ relative to the orbit
$[(2\ell,a)]$ by
\[
        \CD^{\rm FJ}_{2\ell,a}(\pi)
        :=
        \operatorname{Wh}_{\CO_{2\ell,a}}(\pi)
        \qquad(\ell>0),
\]
and we regard $\CD^{\rm FJ}_{2\ell,a}(\pi)$ as a representation of $\rm{Sp}_{2(n-\ell)}(\mathbbm{k})$.
For \(\ell=0\), the symbol \(\CO_{0,a}\) is formal.  We define its descent by
extending the multiplicity formula for positive Fourier--Jacobi descents.  For
\(\ell>0\), if \(\CO=\CO_{2\ell,a}\) and
\(\sigma\in\Irr(M_{\psi_\CO}^F)\), then
\[
        \dim\Hom_{M_{\psi_\CO}^F}
        \left(\sigma,\CD^{\rm FJ}_{2\ell,a}(\pi)\right)
        =
        \dim
        \Hom_{M_{\psi_\CO}^F\ltimes G_{\geqslant 1}^F}
        \left(\sigma\otimes\omega_{\psi_\CO},\pi\right).
\]
For \(\ell=0\), we keep the right-hand side of this formula in the formal
case, and obtain
\[
        \Hom_{G^F}
        \left(\sigma\otimes\omega_{n,\psi_a},\pi\right)
        \simeq
        \Hom_{G^F}
        \left(\sigma,\pi\otimes\omega_{n,\psi_a}^{\vee}\right).
\]
Since
\(
        \omega_{n,\psi_a}^{\vee}
        \simeq
        \omega_{n,\psi_{\varepsilon(-1)a}},
\)
we define
\begin{equation}\label{descent0}
        \CD^{\rm FJ}_{0,a}(\pi)
        :=
        \operatorname{Wh}_{\CO_{0,a}}(\pi)
        :=
        \pi\otimes\omega_{n,\psi_{\varepsilon(-1)a}}.
\end{equation}

\begin{defn}
Let $\pi$ be a finite-dimensional representation of $\Sp_{2n}(\mathbbm{k})$.
The largest $\ell\ge 0$ such that $\CD^{\rm FJ}_{2\ell, a}(\pi)\neq 0$ for at least one \(a\in\{\pm\}\) is called the first descent index of
\(\pi\).  At this maximal value of $\ell$,
$\CD^{\rm FJ}_{2\ell,a}(\pi)$ is called a first descent of $\pi$
for every $a\in\{\pm\}$ for which it is nonzero.
\end{defn}

We next turn to compositions of descents.  Since a descent
$\CD^{\rm FJ}_{2\ell,a}(\pi)$ need not be irreducible, we keep track of
irreducible constituents at each step.

A {\it descent sequence} of an irreducible representation
$\pi\in\Irr(\Sp_{2n}(\mathbbm{k}))$ is a sequence
\[
\pi=\pi_0
\xrightarrow{(\ell_1,a_1)}
\pi_1
\xrightarrow{(\ell_2,a_2)}
\cdots
\xrightarrow{(\ell_r,a_r)}
\pi_r=\mathbbm{1},
\]
where, for each $i$ with $\ell_i>0$, there exists
$\chi_i\in\Irr(\RO_1^F)$ such that
$\chi_i\boxtimes\pi_i$ is an irreducible constituent of
$\CD^{\rm FJ}_{2\ell_i,a_i}(\pi_{i-1})$.  If $\ell_i=0$, then $\pi_i$ is an
irreducible constituent of $\CD^{\rm FJ}_{0,a_i}(\pi_{i-1})$.  In both cases,
$\pi_i$ is regarded as a representation of the symplectic factor, and
$\mathbbm{1}$ is the trivial representation of the trivial group
$\Sp_0(\mathbbm{k})$. We require the sequence to terminate when it first reaches
$\Sp_0(\mathbbm{k})$.
 The corresponding {\it descent sequence index} is
\[
        \Gamma
        =
        \llb
        (2\ell_1,a_1),(2\ell_2,a_2),\ldots,(2\ell_r,a_r)
        \rrb .
\]
Note that $[2\ell_1,\ldots,2\ell_r]$ need not be a partition.

For two descent sequence indexes
\[
\Gamma=\llb(2\ell_1,a_1),(2\ell_2,a_2),\ldots\rrb,\qquad
\Gamma'=\llb(2\ell'_1,a'_1),(2\ell'_2,a'_2),\ldots\rrb,
\]
we write \(\Gamma\leqslant\Gamma'\) if \(\Gamma'\) can be obtained from
\(\Gamma\) by a finite sequence of elementary sequence mutations of the
following form:
\[
 (2\ell_i,a_i),(2\ell_{i+1},a_{i+1})
 \longmapsto
 (2\ell_i+2,b_i),(2\ell_{i+1}-2,b_{i+1}),
\]
where \(a_i a_{i+1}=b_i b_{i+1}\) and \(\ell_{i+1}>0\).  A descent sequence index of \(\pi\) is called maximal if it is maximal with respect to this order among the descent sequence indexes of \(\pi\).  A descent sequence is called maximal if its index is maximal.

\begin{defn}\label{defn:good}
Let
\(
\Gamma=
\llb(2\ell_1,a_1),(2\ell_2,a_2),\ldots,(2\ell_r,a_r)\rrb
\)
be a descent sequence index.  We say that $\Gamma$ is {\it good} if the
following conditions hold.  We use the convention
\(
        \ell_0=+\infty, \ell_{r+1}=0.
\)
\begin{enumerate}
\item[(a)] For each $i$, one of the following three possibilities holds:
 \begin{itemize}
 \item $\ell_{i-1}\geqslant \ell_i\geqslant \ell_{i+1}$;
 \item $\ell_i=\ell_{i+1}-1$ and $\ell_{i-1}>\ell_i$;
 \item $\ell_i=\ell_{i-1}+1$ and $\ell_{i+1}<\ell_i$.
 \end{itemize}
\item[(b)] If $\ell_i=\ell_{i+1}-1$, then
\[
        a_i=\varepsilon(-1)a_{i+1}.
\]
\end{enumerate}
\end{defn}

Roughly speaking, a descent sequence index is good if it is sufficiently
close to an $F$-rational nilpotent orbit.  When
$\ell_1>\ell_2>\cdots>\ell_r$, the descent sequence index directly
corresponds to a rational nilpotent orbit with even parts.  Since descent
sequence indexes record only even numbers, an odd pair
\[
        (2k+1)^2
\]
in a nilpotent orbit is represented by the adjacent pair
\[
        (2k,a),\quad (2k+2,\varepsilon(-1)a).
\]

The order \(\leqslant\) defined above is imposed on all descent sequence
indexes, not only on good ones.  Thus, the intermediate indexes in a mutation
chain may fail to be good.

We now attach a rational nilpotent orbit \(\CO(\Gamma)\) to a good descent
sequence index \(\Gamma\). For each \(i\), define a part \(p_i\) as follows:
\begin{itemize}
\item If \(\ell_{i-1}\geqslant\ell_i\geqslant\ell_{i+1}\), then set
\(
p_i=(2\ell_i,a_i).
\)
\item If \(\ell_i=\ell_{i+1}-1\) and \(\ell_{i-1}>\ell_i\), then set
\(
p_i=2\ell_i+1.
\)
\item If \(\ell_i=\ell_{i-1}+1\) and \(\ell_{i+1}<\ell_i\), then set
\(
p_i=2\ell_i-1.
\)
\end{itemize}
The rational nilpotent orbit \(\CO(\Gamma)\) is obtained from the list
\[
p_1,p_2,\ldots,p_r
\]
by omitting all zero parts and combining equal even parts according to the
convention
\[
[(\lambda,a_1),\ldots,(\lambda,a_m)]
=
[(\lambda^m,
\varepsilon(-1)^{\frac{m(m-1)}2}a_1\cdots a_m)].
\]
The odd parts occur in equal consecutive pairs by the definition of
goodness, so this gives a well-defined rational nilpotent orbit.
Conversely, a rational nilpotent orbit may give rise to more than one good
sequence index.  We denote by \(\mathfrak G(\CO)\) the set of all good
sequence indexes attached to \(\CO\). 

\begin{lem}\label{lem:orbit-sequence-order}
Let \(\CO\) and \(\CO'\) be rational nilpotent orbits.
\begin{enumerate}
\item[(i)]
If \(\CO\leqslant\CO'\), then for every
\(\Gamma\in\mathfrak G(\CO)\), there exists
\(\Gamma'\in\mathfrak G(\CO')\) such that
\(
        \Gamma\leqslant\Gamma'.
\)
\item[(ii)]
If \(\Gamma\) and \(\Gamma'\) are good sequence indexes and
\(
        \Gamma\leqslant\Gamma',
\)
then
\(
        \CO(\Gamma)\leqslant\CO(\Gamma').
\)
\end{enumerate}
\end{lem}

\begin{proof}
We first prove~(i). By Lemma~\ref{mutation}, it is enough to lift a single
orbit mutation. Fix an orbit mutation
\(
\CO_1\rightsquigarrow_{\rm orb}\CO_2
\)
and a good sequence index
\(\Gamma_1\in\mathfrak G(\CO_1)\). An orbit mutation is an unordered local
operation on the signed Young diagram; the parts involved in the mutation
need not be adjacent and may consist of more than two parts. By contrast, a
sequence index is ordered, and equal parts may be split among several
parts. Nevertheless, after choosing \(\Gamma_1\), the orbit mutation can
be realized by changing two parts of the sequence index, and these two
parts need not be adjacent.

When the two parts of \(\Gamma_1\) whose change realizes the orbit mutation are adjacent, the required operation is a single elementary sequence mutation. The following table lists all possible cases. In each row, only these two parts are changed; the other displayed parts record possible splittings of equal parts and remain unchanged. Here \(a,b,c\in  \{\pm\}\).

\begin{center}
\footnotesize
\setlength{\tabcolsep}{3pt}
\renewcommand{\arraystretch}{1.25}
\begin{tabular}{@{}p{0.34\linewidth}@{\hspace{0.02\linewidth}}p{0.60\linewidth}@{}}
\hline
\makebox[\linewidth][c]{\textbf{Orbit mutation}}
&
\makebox[\linewidth][c]{\textbf{Local lift on sequence indexes}}
\\
\hline

\makebox[\linewidth][c]{$
(\lambda-1)^2
\longrightarrow
(\lambda,a),(\lambda-2,\varepsilon(-1)a)
$}
&
\makebox[\linewidth][c]{$
(\lambda-2,b),(\lambda,\varepsilon(-1)b)
\longrightarrow
(\lambda,a),(\lambda-2,\varepsilon(-1)a)
$}
\\
\hline

\makebox[\linewidth][c]{$
(\lambda,a),(\delta,b)
\longrightarrow
(\lambda+2,a),(\delta-2,b)
$}
&
\makebox[\linewidth][c]{$
(\lambda,a),(\delta,b)
\longrightarrow
(\lambda+2,a),(\delta-2,b)
$}
\\
\hline

\makebox[\linewidth][c]{$
(\lambda,a),(\delta,b)
\longrightarrow
(\lambda+2,-a),(\delta-2,-b)
$}
&
\makebox[\linewidth][c]{$
(\lambda,a),(\delta,b)
\longrightarrow
(\lambda+2,-a),(\delta-2,-b)
$}
\\
\hline

\makebox[\linewidth][c]{$
\begin{array}{c}
(\lambda^2,a),(\delta,b)\\[-2pt]
{}\longrightarrow(\lambda+1)^2,(\delta-2,ab)
\end{array}
$}
&
\makebox[\linewidth][c]{$
\begin{array}{c}
(\lambda,ac),(\lambda,\varepsilon(-1)c),(\delta,b)\\[-2pt]
{}\longrightarrow
(\lambda,ac),(\lambda+2,\varepsilon(-1)ac),(\delta-2,ab)
\end{array}
$}
\\
\hline

\makebox[\linewidth][c]{$
\begin{array}{c}
(\lambda,a),(\delta^2,b)\\[-2pt]
{}\longrightarrow(\lambda+2,ab),(\delta-1)^2
\end{array}
$}
&
\makebox[\linewidth][c]{$
\begin{array}{c}
(\lambda,a),(\delta,bc),(\delta,\varepsilon(-1)c)\\[-2pt]
{}\longrightarrow
(\lambda+2,ab),(\delta-2,c),(\delta,\varepsilon(-1)c)
\end{array}
$}
\\
\hline

\makebox[\linewidth][c]{$
\begin{array}{c}
(\lambda^2,a),(\delta^2,a)\\[-2pt]
{}\longrightarrow(\lambda+1)^2,(\delta-1)^2
\end{array}
$}
&
\makebox[\linewidth][c]{$
\begin{array}{c}
(\lambda,b),(\lambda,\varepsilon(-1)ab),
(\delta,c),(\delta,\varepsilon(-1)ac)\\[-2pt]
{}\longrightarrow
(\lambda,b),(\lambda+2,\varepsilon(-1)b),
(\delta-2,ac),(\delta,\varepsilon(-1)ac)
\end{array}
$}
\\
\hline
\end{tabular}
\end{center}

We now consider the non-adjacent case. Let the \(i\)-th and \(j\)-th
parts
\[
(2\ell_i,a_i)
\qquad\text{and}\qquad
(2\ell_j,a_j),
\qquad i<j-1,
\]
be the two parts whose change realizes the given orbit mutation. We apply
elementary sequence mutations successively to the adjacent pairs
\[
(j-1,j),\ (j-2,j-1),\ \ldots,\ (i,i+1).
\]
At each step, we choose the new signs so that every intervening part
recovers its original sign, while the two originally selected parts acquire
the signs prescribed by the corresponding row of the table. The resulting
sequence index therefore belongs to \(\mathfrak G(\CO_2)\). Applying this
construction along an orbit-mutation chain proves~(i).

We prove~(ii) by induction on the length of a chain
\[
\Gamma=\Gamma_0\leqslant\Gamma_1\leqslant\cdots
\leqslant\Gamma_N=\Gamma'
\]
of elementary sequence mutations, where \(\Gamma_0\) and \(\Gamma_N\)
are good. The assertion is clear when \(N=0\).

Suppose that the first mutation acts on the \(i\)-th and
\((i+1)\)-st parts:
\[
(2\ell_i,a_i),(2\ell_{i+1},a_{i+1})
\longmapsto
(2\ell_i+2,b_i),(2\ell_{i+1}-2,b_{i+1}),
\qquad
b_ib_{i+1}=a_ia_{i+1}.
\]
If \(\Gamma_1\) is good, comparison with the table shows that this
mutation realizes one orbit mutation, and the result follows by induction.

Assume that \(\Gamma_1\) is not good. Since \(\Gamma_0\) is good, the \(i\)-th and \((i+1)\)-st parts of
\(\Gamma_1\) still satisfy the decreasing case in the definition of
goodness. Thus \(\Gamma_1\) can fail to be good only through the conditions involving
the neighboring pairs \(i-1,i\) or \(i+1,i+2\). Suppose first that the
conditions fail at \(i-1,i\). A direct check of the three cases in
Definition~\ref{defn:good} shows that the first subsequent mutation which
can correct these conditions must act on \(i-1,i\). If the resulting
sequence is still not good, the same argument applies to the next adjacent
pair. The \(i+1,i+2\) case is treated in the same way; if both occur, the
argument is applied to both successively.

These successive neighboring mutations can be placed consecutively in the
original chain. Mutations on disjoint pairs commute, while mutations on
overlapping pairs may be interchanged whenever both orders are legal:
both orders change
\[
(x,y,z)\quad\text{to}\quad(x+1,y,z-1),
\]
and the intermediate signs may be chosen to give the same final signs
because each mutation preserves the product of the affected signs. If a
zero prevents one of the two orders, goodness places it in a terminal
string
\(
0,1,0,1,\ldots;
\)
the mutations on its pairs \((0,1)\) are included among these consecutive
mutations, so no illegal interchange is needed.

Let \(\Gamma_M\) be the first good sequence index reached after these
mutations have been placed consecutively. Replacing each adjacent rise
\(k,k+1\) occurring in \(\Gamma_0\) and \(\Gamma_M\) by the corresponding
odd pair \((2k+1)^2\), a direct comparison with the table, together with
condition~(b) in Definition~\ref{defn:good} and the preservation of sign
products, shows that these mutations decompose into finitely many
non-adjacent versions of the local lifts listed there. Hence
\(
\CO(\Gamma_0)\leqslant\CO(\Gamma_M).
\)
Applying the induction hypothesis to the remaining chain proves~(ii).
\end{proof}

\subsection{$(2\ell+1)^2$-descent}

There is a unique rational nilpotent orbit corresponding to the partition
\[
        [(2\ell+1)^2,1^{2n-4\ell-2}],
\]
and we denote it by
\[
        \CO=\CO_{(2\ell+1)^2}.
\]
Here $0\leqslant\ell\leqslant (n-1)/2$. 
 For convenience, we sometimes omit $1^{2n-4\ell-2}$ and denote signed Young diagrams of the above form by $[(2\ell+1)^2]$. In particular, if $\ell=0$, then $[(2\ell+1)^2]=[1^{2n}]$.

The case \(\ell=0\) will be treated separately at the end of the section. we assume that \(\ell>0\).
Note that $\Fg_1= 0$. Set $N_\CO :=G_{\geqslant 1.5} = G_{\geqslant 2}$, which is the unipotent subgroup of $G$ consisting of the upper triangular matrices
\begin{equation}
u =\ppair{\begin{smallmatrix}
I_2& x_1&\cdots&\cdots&\vdots&&&\\
&I_2&\cdots&\cdots&\vdots&&&\\
&&\ddots& x_{\ell-1}&\vdots&&&\\ 
&&&I_2& y &&&\\
&&& &I_{2n-4\ell}&y'&&\\
&&& &&I_2& x'_{\ell-1} &  \\
&&& &&&\ddots& \\
&&& &&&&I_2  	
\end{smallmatrix}}.
\end{equation}
Then the associated Levi factor \(M_\CO\) is isomorphic to
\(
        \GL_2^\ell\times \Sp_{2n-4\ell-2},
\)
where the \(\GL_2\)-factor is embedded diagonally in the product of the
\(\GL_2\)-factors arising from the grading.
The non-degenerate character $\psi_\CO$ of $N_\CO^F$ is given by
\begin{equation}\label{eq:psi-2l+1}
\psi_\CO(u)=\psi(\tr(x_1+x_2+\cdots+x_{\ell-1})+y_{1,n-2\ell}+y_{2,n-2\ell+1}).
\end{equation}
Note that all non-degenerate characters of $N_\CO^F$ are conjugate to $\psi_\CO$ under the action of $M_\CO^F$, and the stabilizer of $\psi_\CO$ in $M_\CO$ is
\[
M_{\psi_\CO}=\SL_2^\triangle\times \Sp_{2n-4\ell-2}.
\]
For a representation $\pi$ of $G^F$ and $\ell>0$, we define its
{\bf $(2\ell+1)^2$-descent} by
\[
\CD_{(2\ell+1)^2}(\pi)
:=
\operatorname{Wh}_{\CO}(\pi),
\]
regarded as an $M_{\psi_\CO}^F$-module. If $\ell=0$, we define the $1^2$-descent by the formal convention
\[
\CD_{1^2}(\pi)
:=\CD^{\rm FJ}_{2,\varepsilon(-1)}(\CD^{\rm FJ}_{0,+}(\pi)).
\]

 \section{Composition Law}\label{sec5}

The purpose of this section is to prove the composition law announced in the
introduction. The known composition laws handle even one-block orbits
\(\CO_{2\ell,a}\), while the remaining case is an odd pair
\(\CO_{(2\ell+1)^2}\). The calculation gives an identity relating
the dimensions of the generalized Whittaker models attached to
\(\CO_{(2\ell+1)^2}\), the composition of the two adjacent
Fourier--Jacobi descents
\[
\CO_{2\ell,a}
\quad\textrm{and}\quad
\CO_{2\ell+2,\varepsilon(-1)a},
\]
and the generalized Whittaker models attached to
\[
\CO_{(2\ell+2,a),(2\ell,\varepsilon(-1)a)}
\quad\textrm{and}\quad
\CO_{(2\ell+2,-a),(2\ell,-\varepsilon(-1)a)}.
\]
Under the first-descent assumption, the latter two models vanish.
Theorem~\ref{2l+1} proves this identity, and
Theorem~\ref{thm5.1} is obtained by composing each term with the
remaining orbit \(\CO\).

 In Section \ref{ssec:PO} we have parametrized rational nilpotent orbits by signed Young diagrams. Let $\CO = [(d_1, a_1), \ldots, (d_k, a_k)]$ be a rational nilpotent orbit of $\Sp_{2n-4\ell-2}$, with
 $\ell \leqslant (n-1)/2$. Assume that $2\ell+1 \geqslant d_1$. Define the composition of signed Young diagrams
 \[
 [(2\ell+2, a')] *  [(d_1, a_1), \ldots, (d_k, a_k)] := [ (2\ell+2, a'), (d_1, a_1), \ldots, (d_k, a_k)],\quad a'\in\{\pm\}, 
 \]
 which corresponds to the  rational nilpotent orbit of 
\[
 G':=\Sp_{2n-2\ell}.
\]
We denote this orbit by $\CO_{2\ell+2,a'} * \CO$.
 Similarly
 \[
 [(2\ell+1)^2]* [(d_1, a_1), \ldots, (d_k, a_k)] := [ (2\ell+1)^2, (d_1, a_1), \ldots, (d_k, a_k)]
 \]
 corresponds to a rational nilpotent orbit \(\CO_{(2\ell+1)^2}*\CO\) of \(G=\Sp_{2n}\). The Fourier--Jacobi descent 
 \[
 \CD^{\rm FJ}_{2\ell, a}(\pi) = {\rm Wh}_{\CO_{2\ell, a}}(\pi)
 \]
 is a representation of $G'^F = \Sp_{2n-2\ell}(\mathbbm{k})$. 
 Consider the composition of generalized Whittaker models
 \[
{\rm Wh}_{\CO_{2\ell+2,\varepsilon(-1)a}*\CO}  \left( {\rm Wh}_{\CO_{2\ell, a}}(\pi)\right), \quad a\in\{\pm\}.
 \]

\begin{thm}\label{thm5.1}
Let \(\pi\) be a finite-dimensional representation of \(G^F\), and let
\(0<\ell\leqslant (n-1)/2\). Let
\(\CO=[(d_1,a_1),\ldots,(d_k,a_k)]\) be a rational nilpotent orbit of
\(\Sp_{2n-4\ell-2}\) with \(d_1\leqslant 2\ell+1\). Then, for
\(a\in \{\pm\}\),
\[
\begin{aligned}
2\dim{\rm Wh}_{\CO_{(2\ell+1)^2}*\CO}(\pi)
={}&
\dim{\rm Wh}_{\CO_{2\ell+2,\varepsilon(-1)a}*\CO}
\left({\rm Wh}_{\CO_{2\ell,a}}(\pi)\right)\\
&+\frac{q+1}{2}
\dim{\rm Wh}_{\CO}\left(
{\rm Wh}_{\CO_{(2\ell+2,a),(2\ell,\varepsilon(-1)a)}}(\pi)\right)\\
&+\frac{q-1}{2}
\dim{\rm Wh}_{\CO}\left(
{\rm Wh}_{\CO_{(2\ell+2,-a),(2\ell,-\varepsilon(-1)a)}}(\pi)\right).
\end{aligned}
\]
In particular, if \(\ell\) is the first descent index of \(\pi\), then
\[
\dim{\rm Wh}_{\CO_{2\ell+2,\varepsilon(-1)a}*\CO}
\left({\rm Wh}_{\CO_{2\ell,a}}(\pi)\right)
=
2\dim{\rm Wh}_{\CO_{(2\ell+1)^2}*\CO}(\pi).
\]
The latter identity also holds when the first descent index is
\(\ell=0\), under the convention \eqref{descent0}.
\end{thm}

\begin{proof}
By Lemma~\ref{erL3},
\[
\begin{aligned}
&\dim{\rm Wh}_{\CO_{2\ell+2,\varepsilon(-1)a}*\CO}
\left({\rm Wh}_{\CO_{2\ell,a}}(\pi)\right)\\
=&
\dim{\rm Wh}_{\CO}\left(
{\rm Wh}_{\CO_{2\ell+2,\varepsilon(-1)a}}
\left({\rm Wh}_{\CO_{2\ell,a}}(\pi)\right)\right).
\end{aligned}
\]
All root exchanges and Fourier expansions in the proof of
Theorem~\ref{2l+1} are equivariant under the residual factor
\(\Sp_{2n-4\ell-2}(\mathbbm{k})\). Performing the same calculation
after applying the \(\CO\)-Whittaker projector therefore gives
\[
\begin{aligned}
&2\dim{\rm Wh}_{\CO}\left(
{\rm Wh}_{\CO_{(2\ell+1)^2}}(\pi)\right)\\
={}&
\dim{\rm Wh}_{\CO}\left(
{\rm Wh}_{\CO_{2\ell+2,\varepsilon(-1)a}}
\left({\rm Wh}_{\CO_{2\ell,a}}(\pi)\right)\right)\\
&+\frac{q+1}{2}
\dim{\rm Wh}_{\CO}\left(
{\rm Wh}_{\CO_{(2\ell+2,a),
(2\ell,\varepsilon(-1)a)}}(\pi)\right)\\
&+\frac{q-1}{2}
\dim{\rm Wh}_{\CO}\left(
{\rm Wh}_{\CO_{(2\ell+2,-a),
(2\ell,-\varepsilon(-1)a)}}(\pi)\right).
\end{aligned}
\]
Finally, Theorem~\ref{erl2} gives
\[
\dim{\rm Wh}_{\CO}\left(
{\rm Wh}_{\CO_{(2\ell+1)^2}}(\pi)\right)
=
\dim{\rm Wh}_{\CO_{(2\ell+1)^2}*\CO}(\pi),
\]
and the asserted identity follows.

If \(\ell\) is the first descent index of \(\pi\), then
\({\rm Wh}_{\CO_{2\ell+2,\pm}}(\pi)=0\). By
Lemma~\ref{erL3}, the two additional terms therefore vanish, proving
the last assertion. The case of $\ell=0$ follows immediately that of Theorem~\ref{2l+1}.
\end{proof}

\begin{cor}\label{od}
    Let \(\pi\in\Irr(G^F)\), and let \(\CO\) be an
\(F\)-rational nilpotent orbit.
 \begin{enumerate} 
    \item[(i)] If 
    $ {\rm Wh}_{\CO}(\pi)\ne 0,$
then every
\(\Gamma\in\mathfrak G(\CO)\) is a descent sequence index of \(\pi\);
\item[(ii)] If $\Gamma$ is a good descent sequence index of $\pi$, then 
$     {\rm Wh}_{\CO(\Gamma)}(\pi)\ne 0.$
 \end{enumerate}
\end{cor}
\begin{proof}
We argue by induction on the length of \(\Gamma\).

For~\textup{(i)}, let
\(
\Gamma=\llb(2\ell_1,a_1),(2\ell_2,a_2),\ldots\rrb
\in\mathfrak G(\CO).
\)
By definition, \(\Gamma\) is good; hence either
\(\ell_1\geqslant\ell_2\) or \(\ell_1=\ell_2-1\).

Suppose that $\ell_1\geqslant\ell_2$. Then the first part of $\CO$ is $(2\ell_1,a_1)$. Write
\(\CO=\CO_{2\ell_1,a_1}*\CO'\). Then this case follows immediately from Lemma~\ref{erL3} and the induction hypothesis.

Suppose now that \(\ell_1=\ell_2-1\). Goodness gives
\(
(2\ell_1,a_1),(2\ell_2,a_2)
=(2\ell,a),(2\ell+2,\varepsilon(-1)a),
\)
which represents the odd pair \((2\ell+1)^2\). Write
\(
\CO=\CO_{(2\ell+1)^2}*\CO'.
\) 
If \(\ell>0\), Theorem~\ref{thm5.1} shows that either
\[
{\rm Wh}_{\CO_{2\ell+2,\varepsilon(-1)a}*\CO'}
\left({\rm Wh}_{\CO_{2\ell,a}}(\pi)\right)\ne0,
\]
in which case the induction hypothesis applied to the sequence obtained
from \(\Gamma\) by deleting its first term immediately proves the
assertion, or
\[
{\rm Wh}_{\CO'}\left(
{\rm Wh}_{\CO_{(2\ell+2,b),(2\ell,\varepsilon(-1)b)}}(\pi)
\right)\ne0
\]
for some \(b\in\{\pm\}\). In the latter case, Lemma~\ref{erL3} and induction give
a descent sequence of \(\pi\) whose first two terms are
\(
(2\ell+2,b),(2\ell,\varepsilon(-1)b)
\)
and remaining terms agree with those of \(\Gamma\). Applying
Lemma~\ref{lem7.7}, we obtain another descent sequence of \(\pi\) with
the same remaining terms and first two terms
\(
(2\ell,a),(2\ell+2,\varepsilon(-1)a).
\)
Its index is precisely \(\Gamma\), proving \textup{(i)} in this case.

It remains to treat \(\ell=0\). In this case, goodness forces all
numerical indices of \(\Gamma\) to be \(0\) or \(1\), and hence
\(\CO=[1^{2n}]\). We claim that every
\(\sigma\in\Irr(\Sp_{2m}(\mathbbm{k}))\) admits a descent sequence
beginning with
\(
(0,a),(2,\varepsilon(-1)a).
\)
If \(\sigma\) is trivial, this follows from
Proposition~\ref{trivial}\textup{(ii)}. Otherwise,
Proposition~\ref{trivial}\textup{(i)}, together with repeated
applications of Lemma~\ref{lem7.7}, gives
\[
{\rm Wh}_{\CO_{2,b}}(\sigma)\ne0
\]
for some \(b\in\{\pm\}\). Choose a nonzero irreducible constituent
\(\sigma'\) of this descent. By \eqref{descent0},
\(
{\rm Wh}_{\CO_{0,\varepsilon(-1)b}}(\sigma')\ne0.
\)
Thus \(\sigma\) has a descent sequence beginning with
\(
(2,b),(0,\varepsilon(-1)b).
\)
Applying Lemma~\ref{lem7.7} once more gives a descent sequence beginning
with
\(
(0,a),(2,\varepsilon(-1)a).
\)
Then the assertion then follows immediately from the induction hypothesis applied after these two descents to the sequence obtained from \(\Gamma\) by deleting its first two terms.

For~\textup{(ii)}, induction applies after an even first contribution
by Lemma~\ref{erL3}, and after an odd pair \((2\ell+1)^2\) with
\(\ell>0\) by Theorem~\ref{thm5.1}. If \(\ell=0\), goodness gives
\(\CO(\Gamma)=[1^{2n}]\), and hence
\[
{\rm Wh}_{\CO(\Gamma)}(\pi)
={\rm Wh}_{[1^{2n}]}(\pi)=\pi\ne0.
\]

Although Proposition~\ref{trivial} and Lemma~\ref{lem7.7} are stated
later, the proof of Proposition~\ref{trivial} uses only the explicit
first-descent results, while the proof of Lemma~\ref{lem7.7} uses only
the results of Section~\ref{sec6}. Neither proof uses any result of
the present section.
\end{proof}

\subsection{Exchanging Roots Lemma} \label{sec:exchange}
Let $C\subset G$ be an $F$-stable subgroup of a maximal unipotent subgroup of $G$, and
let $\psi_C$ be a nontrivial character of $C^F$. Assume that there are two unipotent
$F$-stable subgroups $X$, $Y$, such that the following conditions are satisfied.
 \begin{enumerate}
 \item    $X$ and $Y$ normalize $C$;
\item   $X\cap C$ and $Y \cap C$ are normal in $X$ and $Y$, respectively, and $(X\cap C)\backslash X$ and
$(Y \cap C)\backslash Y $ are abelian;
\item  $X^F$ and $Y^F$ preserve $\psi_C$ (when acting by conjugation);
\item   $\psi_C$ is trivial on $(X \cap C)^F$ and on $(Y \cap C)^F$;
\item   $[X, Y ] \subset C$;
\item  The pairing $(X \cap C)^F\backslash X^F \times(Y \cap C)^F\backslash Y^F \to \bb{C}^\times$, given by
\[
(x, y) \longmapsto \psi_C([x, y]),
\]
is multiplicative in each coordinate, non-degenerate, and identifies $(Y \cap C)^F\backslash Y^F$
with the dual of $(X \cap C)^F\backslash X^F$ and $(X \cap C)^F\backslash X^F$ with the dual of $(Y \cap C)^F\backslash Y^F$.
\end{enumerate}
We represent the above setup by the following diagram
\begin{equation}\label{er}
\begin{smallmatrix}
&&A &&\\
&\nearrow&&\nwarrow&\\
B=CY&&&&D=CX\\
&\nwarrow&&\nearrow&\\
&&C&&
\end{smallmatrix}.
\end{equation}
Here, $A = BX = DY = CXY $. Extend the character $\psi_C$ to a character $\psi_B$ of
$B^F$, and to a character $\psi_D$ of $D^F$, by making it trivial on $Y^F$ and on $X^F$.

\begin{lem}[{Exchanging Roots Lemma \cite[Lemma 7.5]{PW23}}]\label{erl}
For a finite-dimensional representation $\pi$ of $G^F$, 
\begin{align*}
    \langle \pi,\psi_{B}\rangle_{B^F} 
    & =\frac{1}{|B^F|}\sum_{u\in B^F}\pi(u)\psi_B(u^{-1})\\
    & =\frac{1}{|D^F|\,\bigl|(Y \cap C)^F \backslash Y^F\bigr|}
      \sum_{y\in (Y \cap C)^F \backslash Y^F}\sum_{u\in D^F}\pi(uy)\psi_D(u^{-1}).
\end{align*}
\end{lem}
Lemma~\ref{erl} is a finite-field analogue of Lemma~7.1 in \cite{GRS11}.
In particular, Lemma~\ref {erl} allows us to compose the corresponding generalized Whittaker models.

\begin{lem}[\cite{PW23}, Lemma 7.6]\label{erL3} 
Let \(\ell>0\). Let $\pi$ be a finite-dimensional representation of $\Sp_{2n}(\mathbbm{k})$, and let $\CO = [(d_1, a_1),\ldots, (d_k, a_k)]$ be a rational nilpotent orbit of $\Sp_{2n-2\ell}$ with $d_1\leqslant 2\ell$. Then
\[
 \dim {\rm Wh}_{\CO}  \left( {\rm Wh}_{\CO_{2\ell,a}}(\pi) \right) 
=  \dim {\rm Wh}_{\CO_{2\ell,a}*\CO}(\pi).
\]
\end{lem}
\begin{rmk}
Although \cite[Lemma 7.6]{PW23} is stated for irreducible representations, the same proof applies to finite-dimensional representations.
\end{rmk}

\subsection{The two adjacent even descents}

\begin{thm}\label{2l+1}
Let \(\pi\) be a finite-dimensional representation of \(G^F\), and let
\(1\leqslant\ell\leqslant(n-1)/2\). Then, for \(a\in\{\pm\}\),
\[
\begin{aligned}
2\dim {\rm Wh}_{\CO_{(2\ell+1)^2}}(\pi)
={}&
\dim {\rm Wh}_{\CO_{2\ell+2,\varepsilon(-1)a}}
\left({\rm Wh}_{\CO_{2\ell,a}}(\pi)\right)\\
&+\frac{q+1}{2}
\dim {\rm Wh}_{\CO_{(2\ell+2,a),(2\ell,\varepsilon(-1)a)}}(\pi)\\
&+\frac{q-1}{2}
\dim {\rm Wh}_{\CO_{(2\ell+2,-a),(2\ell,-\varepsilon(-1)a)}}(\pi).
\end{aligned}
\]
In particular, if \(\ell\) is the first descent index of \(\pi\), then
\[
\dim {\rm Wh}_{\CO_{2\ell+2,\varepsilon(-1)a}}
\left({\rm Wh}_{\CO_{2\ell,a}}(\pi)\right)
=
2\dim {\rm Wh}_{\CO_{(2\ell+1)^2}}(\pi).
\]
The latter identity also holds when the first descent index is
\(\ell=0\), under the convention \eqref{descent0}.
\end{thm}

 The proof of this theorem invokes the following lemma.
\begin{lem}\label{ex2}
Let $1\leqslant i\leqslant \ell$ and $N^{[i]}$ be the unipotent subgroup of $\Sp_{2n}$ consisting of the elements
\[
u^{[i]} = 
\ppair{\begin{smallmatrix}
     x^{[i]}_1 &  a^{{[i]}} & b^{[i]} & c^{[i]} & y^{[i]}_1 \\
     & x_2^{[i]} & v^{[i]} & y_2^{[i]} & c^{ {[i]\prime}} \\
     & & n^{[i]} & v^{ {[i]}\prime} & b^{ {[i]}\prime} \\
     & & & x_2^{{[i]}*} & a^{ {[i]}\prime} \\
     & & & & x_1^{ {[i]}*}
\end{smallmatrix}}
\]
where  $x^{[i]}_1\in Z_{\ell-i}$, $x_2^{[i]}\in Z_{\ell-i+1}$, $(a^{[i]},b^{[i]},c^{[i]})\in M_{(\ell-i)\times (2n-2\ell+2i)}$,  
\[
v^{[i]}\in \set{v\in M_{\ell-i+1, 2n-4\ell+4i-2} | v_{k,1}=0 \text{ for } 1\leqslant k\leqslant \ell-i+1},
\]
and $n^{[i]}\in N_{\CO_{(2i)^2,a}}\subset \Sp_{2n-4\ell+4i-2}$. Let $\psi_{[i]}$ be the character of $N^{[i] F}$ defined as follows:
\begin{itemize}
    \item On $x_j^{[i]}=((x_j^{[i]})_{i',j'})$, $j=1,2$,  let 
$\psi_{[i]}(x^{[i]}_j):=\psi(\sum_k(x^{[i]}_j)_{k,k+1})$.
\item 
 On $n^{[i]}$, the character $\psi_{[i]}$ coincides with  $\psi_{\CO_{(2i)^2,a}}$. 
 \item On $b^{[i]}$,  let $\psi_{[i]}(b^{[i]}):=\psi((b^{[i]})_{\ell-i,1})$.
 \item On $v^{[i]}$, let $\psi_{[i]}(v^{[i]}):=\psi((v^{[i]})_{\ell-i+1,2})$.
\item On all remaining variables of \(u^{[i]}\), the character $\psi_{[i]}$ is trivial. 
\end{itemize}
Then for any finite-dimensional representation $\pi$ of $\Sp_{2n}(\mathbbm{k})$ and $1\leqslant i\leqslant \ell-1$, we have
\[
\frac{1}{|N^{{[i]} F}|}
\sum_{u^{[i]}\in N^{{[i]} F}}\pi(u^{{[i]} })\bar\psi_{[i]}(u^{[i]})=\frac{1}{|N^{{[i+1]} F}|}
\sum_{u^{[i+1]}\in N^{{[i+1]} F}}\pi(u^{{[i+1]} })\bar\psi_{[i+1]}(u^{[i+1]}).
\]

\end{lem}

\begin{proof}
   To apply Lemma \ref{erl}, we must first specify some unipotent subgroups $A, B, C, X, Y$ as in the Exchanging Roots Lemma.
Let
\[
C:=\set{u^{[i]}\in N^{[i]} | (a^{[i]})_{\ell-i,k}=0 {\text{ for } 1\leqslant k\leqslant \ell-i+1}}.
\]
In other words, for $u^{[i]}\in C$, the last row of $a^{[i]}$ is zero.  Let 
    \[
    X:=\left\{I_{2n}+\sum_{k=\ell-i+1}^{2\ell-2i+1}\alpha_k(e_{\ell-i,k}-e_{2n+1-k,2n+1-\ell+i})\right\}
    \]
    where $e_{i,j}$ is the $2n\times 2n$ matrix with one at the $(i,j)$ position and zero otherwise. The group $X$ is a complement to $C$ in $N^{[i]}$, and
\[
N^{[i]}=XC.
\]
Let 
\[
Y:=\left\{I_{2n}+\sum_{k=\ell-i+1}^{2\ell-2i+1}\beta_k(e_{k,2\ell-2i+2}-e_{2n-2\ell+2i-1,2n+1-k})\right\}.
\]
The group $Y$ corresponds to the zeros of $v^{[i]}$ and $v^{[i]\prime}$.
It is easy to check that the groups $C$, $X$ and $Y$ satisfy the conditions (1)--(4) in the Exchanging Roots Lemma. 
For $x\in X$ and $y\in Y$ with matrix entries 
$\alpha_k$ and $\beta_k$ as above, one has
\[
 [x,y] = I_{2n}+\left(\sum_{k=\ell-i+1}^{2\ell-2i+1}\alpha_k\beta_k\right)(e_{\ell-i,2\ell-2i+2}-e_{2n-2\ell+2i-1,2n+1-\ell+i})
\in C.
\]
This commutator corresponds to an entry in the bottom left corner of $b^{[i]}$ and an entry in the top right corner of $b^{[i]\prime}$. Moreover
\[
\bar\psi_{[i]}([x,y])=\bar\psi\left(\sum_{k=\ell-i+1}^{2\ell-2i+1}\alpha_k\beta_k\right).
\]
Hence conditions~(5) and~(6) are satisfied. Applying Lemma \ref{erl}, we have
\begin{align*}
\frac{1}{|N^{{[i]} F}|}
\sum_{u^{[i]}\in N^{{[i]} F}}\pi(u^{{[i]} })\bar\psi_{[i]}(u^{[i]}) & =
\frac{1}{|X^FC^F|}
\sum_{x\in X^{ F},c\in C^F}\pi(xc)\bar\psi_{[i]}(xc) \\
& =\frac{1}{|Y^FC^F|}
\sum_{y\in Y^{ F},c\in C^F}\pi(yc)\bar\psi_{[i]}(yc).
\end{align*}

Let $w^{[i]}$ be the Weyl group element 
\[
\diag\{I_{\ell-i-1}, 
\ppair{\begin{smallmatrix} 0 & I_{\ell-i} \\ 1 & 0 \end{smallmatrix}},
I_{2n-4\ell+4i},
\ppair{\begin{smallmatrix} 0 & 1 \\ I_{\ell-i} & 0 \end{smallmatrix}}, I_{\ell-i-1}\}.
\]
The lemma follows from the fact that  $w^{[i]}$ conjugates $YC$ (resp. $\psi_{[i]}$) to $N^{[i+1]}$ (resp. $\psi_{[i+1]}$).
\end{proof}

 \begin{proof}[Proof of Theorem \ref{2l+1}]

For the additional case \(\ell=0\) in the last assertion,
Proposition~\ref{trivial} (i) shows that every irreducible constituent
of \(\pi\) is trivial. The assertion then follows from
Proposition~\ref{trivial} (ii) and the convention \eqref{descent0}.
We henceforth assume that \(\ell>0\).

For the case of \(\ell>0\), the proof follows the strategy of
\cite[Lemma~2.6]{GRS03}. Unlike in loc. cit., however, exchanging
roots does not directly relate the Whittaker model
\({\rm Wh}_{\CO_{2\ell+2,\varepsilon(-1)a}}
({\rm Wh}_{\CO_{2\ell,a}}(\pi))\) to
\({\rm Wh}_{\CO_{(2\ell+1)^2}}(\pi)\). Instead, after several root
exchanges, we Fourier-expand the two models. The two expansions have
the same isotropic contribution, while their anisotropic terms
correspond to
\({\rm Wh}_{\CO_{(2\ell+2,a),(2\ell,\varepsilon(-1)a)}}(\pi)\) and
\({\rm Wh}_{\CO_{(2\ell+2,-a),(2\ell,-\varepsilon(-1)a)}}(\pi)\).
Under the first-descent assumption, these anisotropic terms vanish.

 For \(a\in\{\pm\}\), put
\(
R_a:=
\dim {\rm Wh}_{\CO_{(2\ell+2,a),(2\ell,\varepsilon(-1)a)}}(\pi).
\)
We need to show that
\begin{equation}\label{eq1}
\begin{aligned}
& \frac{1}{ |N_{\CO_{2\ell+2,\varepsilon(-1)a}}^F| \cdot | N_{\CO_{2\ell,a}}^F| } \sum_{u_1\in N_{\CO_{2\ell,a}}^F, u_2\in 
N_{\CO_{2\ell+2,\varepsilon(-1)a}}^F} \pi(u_1 u_2) \bar \psi_{\CO_{2\ell, a}}(u_1) \bar \psi_{\CO_{2\ell+2,\varepsilon(-1)a}}(u_2)\nonumber\\
= \ &
\frac{2}{| N_{\CO_{(2\ell+1)^2}}^F|}
\sum_{u\in N_{\CO_{(2\ell+1)^2}}^F}
\pi(u)\bar\psi_{\CO_{(2\ell+1)^2}}(u)\\
&-\frac{q+1}{2}R_a-\frac{q-1}{2}R_{-a}.
\end{aligned}
\end{equation}
The unipotent subgroup $N_{\CO_{2\ell+2,\varepsilon(-1)a}} \cdot  N_{\CO_{2\ell,a}}$ consists of the elements 
\begin{equation} \label{mat:comp}
u = 
\ppair{\begin{smallmatrix}
     x_1 &  a & b & c & y_1 \\
     & x_2 & v & y_2 & c' \\
     & & I_{2n-4\ell-2} & v' & b' \\
     & & & x_2^* & a' \\
     & & & & x_1^*
\end{smallmatrix}}
\end{equation}
where $x_1\in Z_{\ell}$, $x_2\in Z_{\ell+1}$, $(a,b,c)\in M_{\ell\times (2n-2\ell)}^0$, and $v\in 
M_{\ell+1, 2n-4\ell-2}^0$ (see \eqref{M0}).

Take a Weyl group element 
\[
w:=\diag\{I_{\ell-1}, 
\ppair{\begin{smallmatrix} 0 & I_{\ell} \\ 1 & 0 \end{smallmatrix}},
I_{2n-4\ell},
\ppair{\begin{smallmatrix} 0 & 1 \\ I_{\ell} & 0 \end{smallmatrix}}, I_{\ell-1}\}.
\]
Conjugation by $w$ takes \(N_{\CO_{2\ell+2,\varepsilon(-1)a}}\cdot N_{\CO_{2\ell,a}}\) to the group $N^{[1]}$ of elements
\[
u = 
\ppair{\begin{smallmatrix}
     x^{[1]}_1 &  a^{{[1]}} & b^{[1]} & c^{[1]} & y^{[1]}_1 \\
     & x_2^{[1]} & v^{[1]} & y_2^{[1]} & c^{\prime {[1]}} \\
     & & n^{[1]} & v^{\prime {[1]}} & b^{\prime {[1]}} \\
     & & & x_2^{*{[1]}} & a^{\prime {[1]}} \\
     & & & & x_1^{* {[1]}}
\end{smallmatrix}}
\]
where  $x^{[1]}_1\in Z_{\ell-1}$, $x_2^{[1]}\in Z_{\ell}$, $(a^{[1]},b^{[1]},c^{[1]})\in M_{(\ell-1)\times (2n-2\ell+2)}$,  
\[
v^{[1]}\in 
\set{x\in M_{\ell, 2n-4\ell+2} | \textrm{ the first column of $x$ is zero}}
\]
and $n^{[1]}\in N_{\CO_{2^2,+}}\subset \Sp_{2n-4\ell+2}$. 
By the normalization of the rational signs in \eqref{eq:psi-F-J}, after the
conjugation by \(w\), the characters on the left hand side of \eqref{eq1}
agree with \(\psi_{\CO_{2^2,+}}\) on \(N_{\CO_{2^2,+}}\) in the Levi factor
\(\Sp_{2n-4\ell+2}\).

Applying Lemma \ref{ex2} $\ell-1$ times, we obtain the following equation
\begin{equation}\label{eq2}
   \textrm{LHS of (\ref{eq1})}
=\frac{1}{|N^{{[1]} F}|}
\sum_{u\in N^{{[1]} F}}\pi(u)\bar\psi_{[1]}(u)
=\frac{1}{|N^{{[\ell]} F}|}
\sum_{u\in N^{{[\ell]} F}}\pi(u)\bar\psi_{[\ell]}(u),
\end{equation}
where the character $\psi_{[\ell]}$ on the variable $n^{[\ell]}$ in the Levi factor $\Sp_{2n-2}$ equals   $\psi_{\CO_{(2\ell)^2,+}}$.
Let us describe the RHS of (\ref{eq2}) precisely. The unipotent subgroup $N^{[\ell]}$ consists of the elements
\begin{equation}\label{eq:Nell-u}
u=\ppair{\begin{smallmatrix}
    1  & v^{[\ell]} & \cdots& \\
       & I_2 & n_1 &\cdots  \\
       &     &\ddots &\ddots  \\
     &    &     &I_2 &n_{\ell-1}&\cdots  \\
     &  &     &     &I_2  & v_\ell &n_\ell  \\
     &  &     &      &  &I_{2n-4\ell-2}&v_\ell' &\vdots   \\
     & &      &     & &   & I_2 &n_{\ell-1}'&  & \\
     & &      &   &   &  &    & I_2 & \ddots & \vdots&  \\
        & &      &   &   &    &  & & \ddots &  n_1'& \vdots \\
        &    &   &   &      &   &  &  & & I_2 & v^{[\ell]\prime}\\
      & &      &  &  &  &   &  &  & & 1
\end{smallmatrix}}
\end{equation}
where  $v^{[\ell]}=(0,u_{1,3})$ with $u_{1,3}\in\mathbbm{k}$, 
\[
v_\ell\in\set{x\in M_{2\times (2n-4\ell-2)} | x_{ij}=0 {\textrm{ if $j\leqslant n-2\ell-1$ }}},
\]
and $n_i\in M_{2\times 2}$. Write $n_u:=\sum_{i=1}^{\ell-1}{\rm{tr}}(n_i)$ and
 $n_\ell=\ppair{\begin{smallmatrix}
  u_1&u_2\\
   u_3&u_4
\end{smallmatrix}}$.
The character $\psi_{[\ell]}$ in (\ref{eq2}) is  given by
\[
\psi_{[\ell]}(u)=\psi(u_{1,3}+n_u+{\Fa}u_2-{\Fa}u_3),\quad  u\in N^{[\ell]},
\]
where $u_{1,3}$ refers to the $(1,3)$-th entry of $u$, and $\disc(\Fa)=a$ as in \eqref{eq:psi-F-J}.

Set \[
\phi(t):=\frac{1}{|N^{{[\ell]} F}|}
\sum_{u\in N^{{[\ell]} F}}\pi(b(t)u)\bar\psi_{[\ell]}(u)
\]
 where $t\in \mathbbm{k}$ and
 $
 b(t):=I_{2n}+t\bigl(e_{1,2}-e_{2n-1,2n}\bigr).
 $ 
The last term in \eqref{eq2} is \(\phi(0)\), which will be evaluated below.

Since the group \(\{b(t)\mid t\in\mathbbm{k}\}\) normalizes \(N^{[\ell]}\),
write \(        N':=
        N_{\CO_{(2\ell+2,\varepsilon(-1)a),(2\ell,a)}}.\) Then \( N^{\prime F}=\{b(t)\mid t\in\mathbbm{k}\}\ltimes N^{[\ell]F}.\)
By Fourier expansion on the group $\{b(t)\}$, 
we combine the variables $t$ and $u$, and  obtain
\begin{equation}\label{fourier}
\begin{aligned}
\phi(0)=\ &\frac{1}{|\mathbbm{k}|\cdot |N^{{[\ell]} F}|}\sum_{k\in \mathbbm{k}}\left(\sum_{t\in \mathbbm{k}}\sum_{u\in N^{{[\ell]} F}}\pi(b(t)u)\bar\psi_{[\ell]}(u)\bar\psi(kt)\right)\\
=\ &
\sum_{k\in \mathbbm{k}}\left(\frac{1}{|N^{{\prime} F}|}\sum_{u\in   N^{\prime F}}\pi(u)\bar\psi_{(k,1)}(u)\right),
\end{aligned}
\end{equation}
where 
\[
\psi_{(k,t)}(u)
=\psi(ku_{1,2}+tu_{1,3}+n_u+{\Fa}u_2-{\Fa}u_3).
\]

We consider the factor $\GL_1\times\GL_2^{\times \ell }$ in the Levi subgroup $M_{\CO_{(2\ell+2,\varepsilon(-1)a),(2\ell,a)}}$, which normalizes $N'$.
It induces an action of
\(\GL_1^F\times(\GL_2^{\times\ell})^F\) on the Pontryagin dual
\(\widehat{N^{\prime F}}\) by conjugation, and hence on
\(\{\psi_{(k,1)}\mid k\in\mathbbm{k}\}\).
Its action on $\set{\psi_{(k,1)} | k\in \mathbbm{k}}$ meets two {\it stable} orbits, according as the vector
\((k,1)\in W\) is isotropic or anisotropic.  The anisotropic orbit splits into the two rational orbits distinguished below by the square class of \(k^2-1\).
In particular, this action restricts to the action of $\GL_1^F\times \RO_2(W)^F$ on $\set{\psi_{(k,1)} |  k\in \mathbbm{k}}$ via
\[
(z,g)\cdot (k,1):=(zk,z)g^{t}\quad \text{ for }(z,g)\in \GL_1^F\times \RO_2(W)^F,
\]
where
$\GL_1\times \RO_2(W)$ is the subgroup of $\GL_1\times\GL_2^{\times \ell }$ with $\RO_2(W)$ diagonally embedded into $\GL_2^{\times \ell}$ and $W$ is a 2-dimensional quadratic space with the quadratic form defined by $\diag\{\Fa,-\Fa\}$.
If we consider $(k,1)$ as vectors in $W$, then
the two stable orbits are classified by whether $(k,1)$ is isotropic (equivalently, $k=\pm 1$) or anisotropic.

Now suppose that \(k\ne\pm1\). Set
\[
A_k=\begin{pmatrix}k&1\\1&k\end{pmatrix}
\quad\text{and}\quad
g_k=\diag\{1,A_k\otimes I_\ell,I_{2n-4\ell-2},
A_k^*\otimes I_\ell,1\}.
\]
Since \(g_k \in \GL_1^F\times(\GL_2^{\times \ell})^F\) normalizes \(N'\), changing variables by
\(u\mapsto g_kug_k^{-1}\) and computing the action on
\((u_{1,2},u_{1,3})\) and \(n_\ell\) gives
\[
\begin{aligned}
&\frac{1}{|N^{\prime F}|}\sum_{u\in N^{\prime F}}
\pi(u)\bar\psi_{(k,1)}(u)\\
=&
\frac{1}{|N^{\prime F}|}\sum_{u\in N^{\prime F}}
\pi(u)\bar\psi\left(
u_{1,2}+n_u+\Fa(k^2-1)u_2-\Fa(k^2-1)u_3
\right).
\end{aligned}
\]
Thus this is the dimension of the generalized Whittaker model attached to
\[
\CO_{(2\ell+2,b_k),(2\ell,\varepsilon(-1)b_k)},
\qquad
b_k=\disc\bigl(\Fa(k^2-1)\bigr)
=a\,\disc(k^2-1).
\]
Consequently, the inner sum is \(R_a\) if \(k^2-1\) is a square and
\(R_{-a}\) if \(k^2-1\) is a nonsquare. There are \((q-3)/2\)
elements \(k\in\mathbbm{k}\setminus\{\pm1\}\) of the first kind and
\((q-1)/2\) of the second kind. Let \(\phi_{\rm iso}\) denote the contribution of the two terms
\(k=\pm1\). It follows that
\begin{equation}\label{fourier1}
\phi(0)
=
\phi_{\rm iso}
+\frac{q-3}{2}R_a
+\frac{q-1}{2}R_{-a}.
\end{equation}

Now, consider $k=\pm 1$. 
Take the following element in $\GL_1^F\times(\GL_2^{\times \ell})^F$
\[
g=\diag\{1, \ppair{\begin{smallmatrix}
 \frac{1}{2} & -\Fa^{-1} \\
 \frac{1}{2} & \Fa^{-1} \\
\end{smallmatrix}}\otimes I_{\ell},I_{2n-4\ell-2},\ppair{\begin{smallmatrix}
 \frac{\Fa}{2} & 1 \\
 -\frac{\Fa}{2} & 1 \\
\end{smallmatrix}}\otimes I_{\ell},1\}.
\]
Since $g$ normalizes $N'$, we may change variables by $u\mapsto gug^{-1}$  in the sum over $N^{\prime F}$ and obtain 
\[
\begin{aligned}
\phi_{\rm iso}=&\frac{1}{|N^{\prime F}|}\sum_{k=\pm 1}\sum_{u\in   N^{\prime F}}\pi(u)\bar\psi(ku_{1,2}+u_{1,3}+n_u+\Fa u_2- \Fa u_3)\\
 =&\frac{1}{|N^{\prime F}|}\sum_{k=\pm 1}\sum_{u\in   N^{\prime F}}\pi(gug^{-1})\bar\psi((1+k)u_{1,2}+\frac{\Fa(1-k)}{2}u_{1,3}+n_u+u_1 + u_4)\\
=&\frac{1}{|N^{\prime F}|}\sum_{k=\pm 1}\sum_{u\in   N^{\prime F}}\pi(u)\bar\psi((1+k)u_{1,2}+\frac{\Fa(1-k)}{2}u_{1,3}+n_u+u_1 + u_4).
\end{aligned}
\]
By conjugating by a suitable element in $\GL_1^F\times \RO_2(W)^F$ for $k=\pm 1$ respectively, where
$\RO_2(W)\cong \RO_2^+$ is defined using the matrix $w_2$ in \eqref{eq:J2n}, one has
\begin{equation}\label{fourier2}
\phi_{\rm iso}=\frac{2}{|N^{{\prime} F}|}\sum_{u\in   N^{\prime F}}\pi(u)\bar\psi(u_{1,3}+n_{u}+u_1+ u_4).     
\end{equation}

We now evaluate the odd-pair side
\({\rm Wh}_{\CO_{(2\ell+1)^2}}(\pi)\) by a Fourier expansion
and compare it with \eqref{fourier1}. Set
\begin{equation}\label{f3}
\Psi:=
\frac{2}{|N^{[\ell]F}|}
\sum_{u'\in N^{[\ell]F}}
\pi(u')\bar\psi(u'_{1,3}+n_{u'}+u'_1+u'_4).
\end{equation}
By Fourier expansion for the variable $u_{1,2}$ on the right hand side of \eqref{f3},  we have 
\[
\begin{aligned}
 \Psi=\ \frac{2}{|N^{\prime F}|}\sum_{k\in \mathbbm{k}}\sum_{u'\in   N^{\prime F}}\pi(u')\bar\psi(ku'_{1,2}+u'_{1,3}+n_{u'}+u'_1+ u'_4).
\end{aligned}
\]
For \(k\ne0\), set
\[
s_k=\frac{k}{2}-\Fa^{-1},
\qquad
t_k=\frac{k}{2}+\Fa^{-1}
\]
and
\[
B_k=\begin{pmatrix}s_k&t_k\\ t_k&s_k\end{pmatrix}
\quad\text{and}\quad
g'_k=\diag\{1,B_k\otimes I_\ell,I_{2n-4\ell-2},
B_k^*\otimes I_\ell,1\}.
\]
By conjugation with \(g'_k\), the \(k\)-th inner sum is transformed
into the standard Whittaker model with coefficient
\[
\Fa(s_k^2-t_k^2)=-2k.
\]
It is therefore \(R_a\) when \(\disc(-2k)=a\), and \(R_{-a}\) when
\(\disc(-2k)=-a\). As \(k\) ranges over \(\mathbbm{k}^{\times}\), each case occurs
\((q-1)/2\) times. The contribution of \(k=0\), including the outer
factor \(2\), is \(\phi_{\rm iso}\). Hence
\begin{equation}\label{eq4psi}
\Psi=\phi_{\rm iso}+(q-1)R_a+(q-1)R_{-a}.
\end{equation}

It remains to show that
\(\Psi=2\dim{\rm Wh}_{\CO_{(2\ell+1)^2}}(\pi)\). Namely,
\begin{equation}\label{eq4}
\Psi
=\frac{2}{| N_{\CO_{(2\ell+1)^2}}^F|}\sum_{u\in N_{\CO_{(2\ell+1)^2}}^F}\pi(u) \bar\psi_{\CO_{(2\ell+1)^2}}(u).    
\end{equation}
Consider the subgroup $C\subset N^{[\ell]}$ consisting of the elements with  $(n_i)_{2,1}=0$ for $i=1,\cdots,\ell$, and $(v_\ell)_{2,j}=0$ for $j> n-2\ell-1$ in \eqref{eq:Nell-u}. The characters on both sides of \eqref{eq4} have the same restriction
to \(C^F\); denote this common restriction by \(\psi_C\). Let 
\begin{align*}
X:=& \left\{
I_{2n}+\alpha_\ell e_{2\ell+1,2n-2\ell}+\sum_{i=1}^{\ell-1}\alpha_i (e_{2i+1,2i+2}-e_{2n-2i-1,2n-2i})\right. \\
& \left.+\sum_{i=1}^{n-2\ell-1}\alpha'_i (e_{2\ell+1,n+i}+e_{n-i+1,2n-2\ell})
\right\}
\end{align*}
be a complement to $C$ in $N^{[\ell]}$
and
\begin{align*}
Y:=&\left\{
I_{2n}+\sum_{i=1}^{\ell}\beta_i (e_{2i,2i+1}-e_{2n-2i,2n-2i+1})\right. \\
&\left.+\sum_{i=1}^{n-2\ell-1}\beta'_i(e_{2\ell,2
\ell+i+1}+e_{2n-2\ell-i,2n-2\ell+1})
\right\}
\end{align*}
be a complement to $C$ in $N_{\CO_{(2\ell+1)^2}}$.
For any $x\in X$ and $y\in Y$ with matrix entries $\alpha_i$, $\alpha_i'$ and $\beta_i$, 
$\beta_i'$ as above, a direct commutator calculation modulo \(\ker\psi_C\) gives
\[
\bar\psi_C([x,y])=\bar\psi\left(\sum_{i=1}^{\ell-1}\alpha_i(\beta_{i+1}-\beta_i)-2\left(\alpha_\ell\beta_\ell +\sum_{i=1}^{n-2\ell-1}\alpha_i'\beta_{n-2\ell-i}'\right)\right).
\]
Hence the groups \(C,X,Y\) satisfy conditions (1)--(6) in the Exchanging
Roots Lemma.  Moreover,
\[
        XC=N^{[\ell]},\qquad N_{\CO_{(2\ell+1)^2}}=YC.
\]
Thus \eqref{eq4} follows from Lemma~\ref{erl}.  Combining
\eqref{eq2}, \eqref{fourier1}, \eqref{eq4psi}, and \eqref{eq4}, we get
\[
\begin{aligned}
2\dim {\rm Wh}_{\CO_{(2\ell+1)^2}}(\pi)
={}&
\dim {\rm Wh}_{\CO_{2\ell+2,\varepsilon(-1)a}}
\left({\rm Wh}_{\CO_{2\ell,a}}(\pi)\right)\\
&+\frac{q+1}{2}R_a
+\frac{q-1}{2}R_{-a}.
\end{aligned}
\]
This proves the first assertion.

Suppose now that \(\ell\) is the first descent index of \(\pi\).
By Lemma~\ref{erL3},
\[
R_a
=
\dim {\rm Wh}_{\CO_{2\ell,\varepsilon(-1)a}}
\left({\rm Wh}_{\CO_{2\ell+2,a}}(\pi)\right).
\]
By the definition of the first descent index,
\(
{\rm Wh}_{\CO_{2\ell+2,b}}(\pi)=0,
\)
and hence \(R_{a}=R_{-a}=0\).  The general identity therefore reduces to
\[
\dim {\rm Wh}_{\CO_{2\ell+2,\varepsilon(-1)a}}
\left({\rm Wh}_{\CO_{2\ell,a}}(\pi)\right)
=
2\dim {\rm Wh}_{\CO_{(2\ell+1)^2}}(\pi).
\]
This proves the second assertion.
\end{proof}

\subsection{Composing the odd-pair model with the remaining orbit}

\begin{thm}\label{erl2}
    Let \(1\leqslant\ell\leqslant(n-1)/2\).
Let \(\pi\) be a finite-dimensional representation of
\(\Sp_{2n}(\mathbbm{k})\), and let
\(\CO=[(d_1,a_1),\ldots,(d_k,a_k)]\) be a rational nilpotent orbit of
\(\Sp_{2n-4\ell-2}\) with \(d_1\leqslant 2\ell+1\).
Then
\[
\dim {\operatorname{Wh}}_{\CO_{(2\ell+1)^2}*\CO}(\pi)=\dim {\operatorname{Wh}}_\CO({\operatorname{Wh}}_{\CO_{(2\ell+1)^2}}(\pi)).
\]
\end{thm}

\begin{proof}

To prove this theorem, we only need to show that
\begin{equation}\label{ex3}
\begin{aligned}
&\frac{1}{ |N_{\CO}^F| \cdot | N_{\CO_{(2\ell+1)^2}}^F| } \sum_{u\in N_{\CO_{(2\ell+1)^2}}^F, u_1\in
N_{\CO}^F} \pi(u u_1) \bar \psi_{\CO_{(2\ell+1)^2}}(u) \bar \psi_{\CO}(u_1)\\
 =&\frac{1}{| N_{\CO_{(2\ell+1)^2}*\CO}^F|}\sum_{u_2\in N_{\CO_{(2\ell+1)^2}*\CO}^F}\pi(u_2) \bar\psi_{\CO_{(2\ell+1)^2}*\CO}(u_2).
\end{aligned}
\end{equation}

The proof follows the successive exchanging-roots argument in
\cite[Lemma 7.6]{PW23}. We recall the relevant unipotent subgroups and
verify the conditions needed at each exchange.

{\it Step} 1: Define the unipotent subgroups involved in the exchanges.

Choose \(\frak{sl}_2\)-triples corresponding to \(\CO\) and
\(\CO_{(2\ell+1)^2}*\CO\), and denote their semisimple elements by
\(h_\CO\) and \(h_{\CO_{(2\ell+1)^2}*\CO}\), respectively.  Arrange
the diagonal entries of \(h_\CO\) in decreasing order.  Let \(w\) be
the Weyl group element such that
\begin{equation}\label{eq:ht}
\hbar
:=
w h_{\CO_{(2\ell+1)^2}*\CO}w^{-1}
=
\diag\{\mathtt T,h_\CO,\mathtt T'\},
\end{equation}
where
\begin{align*}
\mathtt T
&=
\diag\{2\ell,2\ell,2\ell-2,2\ell-2,\ldots,2,2,0\}
\in\frak{gl}_{2\ell+1},\\
\mathtt T'
&=
\diag\{0,-2,-2,\ldots,-(2\ell-2),-(2\ell-2),-2\ell,-2\ell\}
\in\frak{gl}_{2\ell+1}.
\end{align*}

Conjugation by $w$ takes RHS of (\ref{ex3}) to 
\begin{equation}\label{w}
 \begin{aligned}
\frac{1}{| N^{\prime F}|}\sum_{u_2\in N^{\prime F}}\pi(u_2) \bar{\psi'}(u_2)
\end{aligned}   
\end{equation}
where $N':=wN_{\CO_{(2\ell+1)^2}*\CO}w^{-1}$ and $\psi'(u):=\psi_{\CO_{(2\ell+1)^2}*\CO}(w^{-1}uw)$ for $u\in N^{\prime F}$.
To write down $ N'$ explicitly, we let
\begin{itemize}
 \item $z_i$ be the number of diagonal entries of \(h_\CO\) which are strictly greater than $2\ell -2{\lceil \frac{i}{2}\rceil}$;
\item  $z^1_i$ the number of diagonal entries of \(h_\CO\) which are strictly greater than $2\ell -2{\lceil \frac{i}{2}\rceil}+1$;
\item $z'_j$ the number of diagonal entries of \(h_\CO\) which are greater than or equal to $  2\ell +4-2{\lceil \frac{j}{2}\rceil}$;
\item $z^{\prime 1}_j$ the number of diagonal entries of \(h_\CO\) which are greater than or equal to  $2\ell +3-2{\lceil \frac{j}{2}\rceil}$.
\end{itemize}
Write
\[
 N'=\left\{\ppair{\begin{smallmatrix}
u&h+e&g\\
&v&h^*+e^*\\
&&u^*
\end{smallmatrix}}
\ppair{\begin{smallmatrix}
I_{2\ell+1}&h'+e'& (h'+e')^\star\\
k+f&I_{2(n-2\ell-1)}&h^{\prime *}+e^{\prime *}\\
(k+f)^\star&k^*+f^*&I_{2\ell+1}
\end{smallmatrix}}\right\}
\]
where $v\in N_{\CO}$ (the unipotent subgroup of  $\Sp_{2(n-2\ell-1)}$ associated with $\CO$), 
\begin{itemize}
\item $u$ runs over $\set{z\in Z_{2\ell+1}| z_{i,i+1}=0 \textrm{ for }i=1,3,5,\cdots,2\ell-1}$;
 \item $h$ runs over $\set{(\alpha_{i,j})_{(2\ell +1)\times2(n-2\ell -1)} | \alpha_{i,j}=0\textrm{ if }j\leqslant z_i\textrm{ or }i=2\ell+1}$;
 \item $h'$ runs over $\set{(\alpha_{i,j})_{(2\ell +1)\times2(n-2\ell -1)} | \alpha_{i,j}=0\textrm{ if }j\leqslant z_i\textrm{ or }i<2\ell+1}$;
\item $e $ runs over $\set{ (\alpha_{i,j})_{(2\ell +1)\times2(n-2\ell -1)} \in M_{1.5}| \alpha_{i,j}=0 \textrm{ if } i=2\ell+1  }$;
\item $e' $ runs over $\set{ (\alpha_{i,j})_{(2\ell +1)\times2(n-2\ell -1)} \in M_{1.5}| \alpha_{i,j}=0 \textrm{ if } i<2\ell+1  }$;
\item $ M_{1.5}=\set{(\alpha_{i,j})_{(2\ell +1)\times2(n-2\ell -1)} | \alpha_{i,j}=0\textrm{ if }j\leqslant {z^1_i},\textrm{ or $i$ even, or }j>z_i}$;
\item  no additional support condition is imposed on \(g\); it is still
subject to the condition that the resulting block matrix lies in
\(\Sp_{2n}\);
\item $k$ runs over $ M_{\geqslant 2}':=\set{(\alpha_{i,j})_{ 2(n-2\ell -1)\times (2\ell +1)} | \alpha_{i,j}=0\textrm{ if }i> z'_{j}}$;
\item $f$ runs over  \[ M_{1.5}'=\set{(\alpha_{i,j})_{ 2(n-2\ell -1)\times (2\ell +1)} | \alpha_{i,j}=0\textrm{ if }i> {z^{\prime 1}_{j}},\textrm{ or $j$ is odd, or }i\leqslant z^{\prime }_{j}};\]
\item $(k+f)^\star$ is determined by $k+f$.
\end{itemize}
In fact, the condition $i$ even (resp. $j$ odd) in $M_{1.5}$ (resp. $M'_{1.5}$) picks up a maximal isotropic subspace in the eigenspace of \(\hbar\) acting on \(\mathfrak{sp}_{2n}\)
with the eigenvalue 1.
If $\CO$ has only odd parts, then $M_{1.5}$ and $M_{1.5}'$ are trivial.

On the other hand, we have
\[
  N_{\CO_{(2\ell+1)^2}}
=
\left\{\ppair{\begin{smallmatrix}
u&h+e&g\\
&I&h^{ *}+e^{ *}\\
&&u^*
\end{smallmatrix}}
\ppair{\begin{smallmatrix}
I_{2\ell+1}&d&d^\star\\
&I_{2n-4\ell-2}&d^{*}\\
&&I_{2\ell+1}
\end{smallmatrix}}
\right\}
\]
where
\begin{itemize}
\item $u$, $g$, $h$ and $e$ run over the same sets as above;
\item $d$ runs over $\set{(\alpha_{ij}) | \alpha_{ij}=0\textrm{ if $i$ odd and }j>{z^1_i},\textrm{ or $i$ even and }j>z_i, \textrm{ or }i=2\ell+1 }$.
 \end{itemize}

We set
\[
C:=\left\{\ppair{\begin{smallmatrix}
u&h+e&g\\
&v&h^*+e^*\\
&&u^*
\end{smallmatrix}}\right\},\quad X:=\left\{\ppair{\begin{smallmatrix}
I&d&d^\star\\
&I&d^{ *}\\
&&I
\end{smallmatrix}}
\right\}
\]
and
\[
Y:=\left\{\ppair{\begin{smallmatrix}
I&h'+e'&\\
k+f&I&h^{\prime *}+e^{\prime *}\\
(k+f)^\star&k^*+f^*&I
\end{smallmatrix}}\right\},
\]
where $v$, $k$,  $f$, $h'$ and $e'$ run over the same sets as above.
Denote by $\frak{x}$ and $\frak{y}$  the Lie algebras of $X$ and $Y$, respectively. 
Note that $X$ and $Y$ intersect $C$ trivially. 
Let
\[
B:=CY= N',\ D:=CX=N_{\CO} \cdot   N_{\CO_{(2\ell+1)^2}}\textrm{ and }A:=CXY.
\]
Denote by $\psi_C$ the restriction of $\psi'$ to $C^F$.  We record the
part of this character which will be used below.  The character \(\psi_C\)
factors through the quotient $C^F/(C^F\cap G_{>2}^F)$.
On the middle block \(v\in N_\CO^F\), it is the character \(\psi_\CO(v)\).
On the \((2\ell+1)^2\)-part, with respect to the full \(2n\times 2n\)
matrix coordinates, it is nontrivial precisely on the degree-two
coordinates belonging to the two Jordan strings, namely
\[
e_{1,3},e_{2,4},\ldots,e_{2\ell-1,2\ell+1},
        \quad e_{2\ell,2n-2\ell}.
\]
Equivalently, if \(c=1+Z\in C^F\), then, up to the fixed normalization of the
rational nilpotent orbit,
\[
 \psi_C(c)
 =
 \psi_\CO(v)\,
 \psi\left(
       z_{1,3}+z_{2,4}+\cdots+z_{2\ell-1,2\ell+1}
       +z_{2\ell,2n-2\ell}
      \right).
\]
All other coordinates of the \((2\ell+1)^2\)-part of \(C\) either lie in
\(C\cap G_{>2}\) or are in the kernel of \(\psi_C\).

{\it Step} 2: Apply the Exchanging Roots Lemma successively.

Our goal is to exchange
\[
CY=N'
\qquad\textrm{and}\qquad
CX=N_{\CO}\cdot N_{\CO_{(2\ell+1)^2}}.
\]
The Exchanging Roots Lemma cannot be applied directly to the whole
groups \(X\) and \(Y\); instead, we exchange their graded pieces one
at a time.

Set
\(
\Fg_r:=\{x\in\mathfrak{sp}_{2n}\mid[\hbar,x]=rx\}.
\)
For \(1\leqslant r\leqslant2\ell\), let \(X_r\) and \(Y_r\) be the
subgroups with Lie algebras
\[
\mathfrak x_r=\mathfrak x\cap\Fg_{2-r},
\qquad
\mathfrak y_r=\mathfrak y\cap\Fg_r.
\]
Since \(X\) and \(Y\) are abelian,
\[
X=\prod_{r=1}^{2\ell}X_r,
\qquad
Y=\prod_{r=1}^{2\ell}Y_r.
\]
For \(1\leqslant r\leqslant2\ell\), set
\[
C_r:=C\prod_{s<r}X_s\prod_{s>r}Y_s.
\]
We first apply the Exchanging Roots Lemma to
\(C_1,X_1,Y_1\), since \(C_1Y_1=CY\).  The resulting group is
\(C_1X_1=C_2Y_2\).  Continuing in this way, we apply the lemma
successively to \(C_r,X_r,Y_r\), for
\(r=1,\ldots,2\ell\), and finally obtain
\(C_{2\ell}X_{2\ell}=CX\).

A direct calculation from the definitions shows that \(\psi_C\)
extends trivially to a character \(\psi_r\) of \(C_r^F\), and that
\((C_r,X_r,Y_r,\psi_r)\) satisfies conditions (1)--(5) of the
Exchanging Roots Lemma.  It remains to verify condition (6), namely
the non-degeneracy of
\[
X_r^F\times Y_r^F\longrightarrow\mathbb C^\times,
\qquad
(x_r,y_r)\longmapsto\bar\psi_C([x_r,y_r]).
\]
We now calculate this pairing. Let
\[
\cal{X}_r:=\set{(i,j) | e_{i,j}-e_{2n+1-j,\,2n+1-i}\in \mathfrak{x}_r} 
\]
and
\[
\cal{Y}_r:=\set{(i,j) | e_{i,j}-e_{2n+1-j,\,2n+1-i}\in \frak{y}_r}.
\]
Denote by
$\beta_i$ the $i$-th diagonal entry in the element given by
\eqref{eq:ht}, and  put $\delta(r,i):=\frac{2\ell -\beta_i+r}{2}$. We describe the set $\cal{X}_r$ and $\cal{Y}_r$ explicitly, as follows.
\begin{itemize}
\item If $r>1$, then 
\begin{align*}
\cal{X}_r & =\set{(i,j) |  2\ell +2\leqslant j\leqslant n,\ \lceil \frac{i}{2}\rceil=\delta(r,j)\textrm{ and }i<2\ell +1}, \\
\cal{Y}_r &=\set{(i,j) | \bigl(2\ell +2\leqslant i\leqslant n,\ \lceil \frac{j}{2}\rceil=\delta(r,i)+1,\ j\leqslant 2\ell +1\bigr) \textrm{ or } \bigl(i=2\ell+1,\ \beta_j=-r\bigr)};
\end{align*}

\item If $r=1$, then
\begin{align*}
\cal{X}_r & =\set{(i,j) | 2\ell +2\leqslant j\leqslant n,\ \lceil \frac{i}{2}\rceil=\delta(r,j)\textrm{, $i$ even, and }i<2\ell +1}, \\
\cal{Y}_r
&=\set{(i,j) |
\bigl(2\ell +2\leqslant i\leqslant n,\
\lceil \frac{j}{2}\rceil=\delta(r,i)+1,\
j\textrm{ even},\
j\leqslant 2\ell +1\bigr)
\textrm{ or }
\bigl(i=2\ell+1,\ \beta_j=-1\bigr)}.
\end{align*}
 \end{itemize}
We write a typical $x_r\in X_r $ (resp. $y_r\in Y_r $) as follows:
\[
x_r=I+\sum_{(i,j)\in\cal{X}_r} \alpha_{i,j}(e_{i,j}-e_{2n+1-j,2n+1-i}) \ \textrm{ (resp. } y_r=I+\sum_{(i,j)\in\cal{Y}_r} \alpha_{i,j}'(e_{i,j}-e_{2n+1-j,2n+1-i})\textrm{)}
\]
for some $\{\alpha_{i,j}\}$ (resp. $\{\alpha'_{i,j}\}$).
Let \(B_r:=\set{k\in\{2\ell+2,\ldots,n\}\mid\beta_k=r}\)
and $k_r:=\max(B_r\cup \{2\ell+1\})$. 
Then
\[
\begin{aligned}
&\bar\psi_C([x_r,y_r])\\
=&
\begin{cases}
 \bar\psi\left(
 \begin{matrix}
\sum_{k\in B_r}
 \left(\alpha_{2\delta(r,k)-1,k}\cdot\alpha'_{k,2\delta(r,k)+1}
 - \alpha_{2\delta(r,k),k}\cdot\alpha'_{2\ell+1,2n+1-k}\right)    \\
  +\sum_{k=k_r+1}^{n }\left(\alpha_{2\delta(r,k)-1,k}\cdot\alpha'_{k,2\delta(r,k)+1}+\alpha_{2\delta(r,k),k}\cdot\alpha'_{k,2\delta(r,k)+2}\right)
 \end{matrix}\right),    &  \textrm{if }r >1;\\
 \bar\psi(
 \sum_{k\in B_1}
 \left(- \alpha_{2\delta(r,k),k}\cdot\alpha'_{2\ell+1,2n+1-k}\right)
 +\sum_{k=k_1+1}^{n }\left(
 \alpha_{2\delta(r,k),k}\cdot\alpha'_{k,2\delta(r,k)+2})\right),     &  \textrm{if }r =1.
\end{cases}
\end{aligned}
\]
Here, we use full \(2n\times 2n\) matrix coordinates.  For \(k\in B_r\) one has
\(\delta(r,k)=\ell\).  Thus the variable
\(\alpha_{2\delta(r,k),k}=\alpha_{2\ell,k}\) is paired with the variable
\(\alpha'_{2\ell+1,2n+1-k}\).  Indeed, the corresponding element of
\(\frak y_r\) contains the term
\[
        e_{2\ell+1,2n+1-k}-e_{k,2n-2\ell},
\]
and hence its commutator with \(e_{2\ell,k}\) contributes
\[
        -e_{2\ell,2n-2\ell}.
\]
This is exactly the last coordinate in the \((2\ell+1)^2\)-part on which
\(\psi_C\) is nontrivial.
It follows that the pairing $X_r^F \times  Y_r^F \to \bb{C}^\times$ given by
\[
(x_r, y_r) \mapsto \bar\psi_C([x_r, y_r]),
\]
is multiplicative in each variable, non-degenerate, and identifies $Y_r^F$
with the dual of $X_r^F$ and $X_r^F$ with the dual of $Y_r^F$ respectively.

{\it Step} 3: Complete the proof of \eqref{ex3}.

The successive exchanges in Step 2 show that the LHS of
\eqref{ex3} and \eqref{w} are equal, respectively, to
\[
\frac{1}{|C^F||X^F||Y^F|}
\sum_{u\in C^F}\sum_{x\in X^F}\sum_{y\in Y^F}
\pi(uyx)\bar\psi_C(u)
\]
and to the same expression with \(yx\) replaced by \(xy\).
Since \(\pi\) is a class function and conjugation by \(X^F\) preserves
both \(C^F\) and \(\psi_C\), a change of variables in \(C^F\) shows
that the two expressions are equal.  This proves \eqref{ex3} and completes the proof.
 \end{proof}

\section{Descendants}\label{sec6}

This section records the explicit branching law which computes Fourier--Jacobi
descents of irreducible representations of finite symplectic groups.  If
\(\pi\in\Irr(\Sp_{2n}(\mathbbm{k}))\), then its Fourier--Jacobi descent
\(\CD^{\rm FJ}_{2\ell,a}(\pi)\) is a representation of
\(\Sp_{2n-2\ell}(\mathbbm{k})\).  Thus determining its irreducible constituents
amounts to deciding, for
\(\pi'\in\Irr(\Sp_{2n-2\ell}(\mathbbm{k}))\), when
\[
        \Hom_{\Sp_{2n-2\ell}(\mathbbm{k})}
        \left(\pi',\CD^{\rm FJ}_{2\ell,a}(\pi)\right)
        \neq 0 .
\]
This is precisely the finite Gan--Gross--Prasad branching problem for
Fourier--Jacobi models, and an explicit answer is given in
\cite{Wang21,Wang23}. The full statement is combinatorial, involving Lusztig symbols, signs,
and interlacing conditions. In the later arguments, part~\textup{(i)}
will be used to test general descents, while the explicit calculation
of first descents relies on part~\textup{(ii)}. After applying
Alvis--Curtis duality, the latter becomes the simple recipe given
below, together with the formula for the first descent index in
Corollary~\ref{fdi}.  

We first recall the relevant results from \cite{Wang23} in the
notation of this paper.

\begin{thm}[{\cite[Theorem 7.2 (1), Section 7.3]{Wang23}}]\label{ggp}
Let $\prll\in \Irr(\Sp_{2n}(\mathbbm{k}))$ and $\prllp \in \Irr(\Sp_{2n'}(\mathbbm{k}))$, where $n\geqslant n'$, and let $\ell= n-n'$.  Let
$\epsilon_{-1}\in\{\pm1\}$ be determined by $\pi_{\Lambda_{-1}}\in \Irr(\RO^{\epsilon_{-1}}_{2m}(\mathbbm{k}))$, where $\Lambda_{-1}$ is the (-1)-symbol coming from $\prll$.  When \(m=0\), we use the convention \(\epsilon_{-1}=+\).  Write
$
\rho=\prod_{[b]\ne [\pm 1]}\pi[b]
$
and
$
\rho'=\prod_{[b]\ne [\pm 1]}\pi'[b],
$ 
where $\pi[b]$ (resp. $\pi'[b]$) is the irreducible unipotent representation corresponding to the partition $\lambda[b]$ (resp. $\lambda'[b]$). 
\begin{enumerate}
\item[(i)] 
$
\Hom_{\Sp_{2n'}(\mathbbm{k})}(\prllp,\CD^{\rm FJ}_{2\ell, a}(\prll))\ne 0
$
if and only if $\prll$ and $\prllp$ satisfy the following conditions:
\begin{enumerate}
\item there  are $\ep,\ev\in\{\pm\}$ such that
\begin{align*}
& a=\ep\ev\er\erp,\\
& \Upsilon(\Lambda_{-1}')^{-\ep\ev{\epsilon_{-1}}\varepsilon(-1) }\preccurlyeq \Upsilon(\Lambda_1)^{\ep }, \quad \Upsilon(\Lambda_1)^{-\ep } \preccurlyeq \Upsilon(\Lambda'_{-1})^{\ep\ev{\epsilon_{-1}}\varepsilon(-1)}, \\
& \Upsilon(\Lambda_{1}')^{\epsilon_{-1}} \preccurlyeq \Upsilon(\Lambda_{-1})^{-\ev}, \quad \Upsilon(\Lambda_{-1})^{\ev } \preccurlyeq \Upsilon(\Lambda'_{1})^{-\epsilon_{-1}}, \\
& {\mathrm{def}}(\Lambda_1')=-\epsilon_{-1}\ev{\mathrm{def}}(\Lambda_{-1})+\epsilon_{-1}, \\
& {\mathrm{def}}(\Lambda_{-1}')=-\epsilon_{-1}\ev\varepsilon(-1){\mathrm{def}}(\Lambda_1)-\ep\ev\epsilon_{-1}\varepsilon(-1).
\end{align*}
\item  For every \([b]\ne[\pm1]\),
\[
\begin{cases}
 \textrm{$\lambda[b]$ and $\lambda'[-b]$ are 2-transverse}, &\textrm{ if $\#[b]$ is odd};\\
   \textrm{$\lambda[b]$ and $\lambda'[-b]$ are close}, &\textrm{ if $\#[b]$ is even}.
\end{cases}
\]
 \end{enumerate}

\item[(ii)] Assume that $\CD^{\rm FJ}_{2\ell, a}(\prll)$ is a first descent. Then $\prllp$ appears in $\CD^{\rm FJ}_{2\ell, a}(\prll)$ if and only if its triple $(\rho',\Lambda'_{1},\Lambda'_{-1})$ satisfies the following conditions:
\begin{enumerate}
\item for $\ep=-\zeta({}^{\bd}\Lambda_1)$ and $\ev=\zeta({}^{\bd}\Lambda_{-1})$, we have
\begin{align*}
    & \Upsilon(\Lambda_{-1}')^{-\ep\ev{\epsilon_{-1}\varepsilon(-1)} } = \Upsilon(\Lambda_1)^{\ep }_-, \quad \Upsilon(\Lambda'_{-1})^{\ep\ev{\epsilon_{-1}}\varepsilon(-1)}=\Upsilon(\Lambda_1)^{-\ep }, \\
    & \Upsilon(\Lambda_{1}')^{\epsilon_{-1}}=\Upsilon(\Lambda_{-1})^{-\ev}_-,\quad
    \Upsilon(\Lambda'_{1})^{-\epsilon_{-1}}=\Upsilon(\Lambda_{-1})^{\ev}, \\
    & {\mathrm{def}}(\Lambda_1') = -\epsilon_{-1}\ev{\mathrm{def}}(\Lambda_{-1})+\epsilon_{-1}, \\
    & {\mathrm{def}}(\Lambda_{-1}')= -\epsilon_{-1}\ev\varepsilon(-1){\mathrm{def}}(\Lambda_1) -\ep\ev\epsilon_{-1}\varepsilon(-1).
\end{align*}
\item For every \([b]\ne[\pm1]\),
$
\lambda'[-b]=\lambda[b]_-.
$
\item 
$
a=-\zeta({}^{\bd}\Lambda_1)\zeta({}^{\bd}\Lambda_{-1})\epsilon_{\rho}\epsilon_{\rho^-_1}.
$
\item $\rho'=\rho_1^-$ where $\rho_1^-$ is defined in \cite[(3.4),(3.5)]{Wang23}.
\end{enumerate}
Moreover, we have
\[
\dim \Hom_{\Sp_{2n'}(\mathbbm{k})}(\prllp,\CD^{\rm FJ}_{2\ell, a}(\prll))=1.
\]
\end{enumerate}
\end{thm}

\begin{rmk}
When the largest entries in the two rows of
\({}^{\bd}\Lambda_{\pm1}\) are equal, either value of
\(\zeta({}^{\bd}\Lambda_{\pm1})\) may be chosen. The resulting descendants are illustrated in
Examples~\ref{dex2}, \ref{dex3}, and \ref{dex4}.   
\end{rmk}

\begin{rmk}
When $\ell=0$, the Fourier--Jacobi descent
$\CD^{\rm FJ}_{0,a}(\prll)$ is understood by the convention
\eqref{descent0}. The same convention is used in \cite{Wang23}, and the corresponding
result there also includes the case $\ell=0$.
\end{rmk}

\subsection{A concise description of first descents}\label{sec6.1}

We now extract from Theorem~\ref{ggp} (ii) the form used later.  After applying Alvis--Curtis duality, the first descent of the symbol part is described by the
following recipe.  It produces \(\Lambda'_1\) from \(\Lambda_{-1}\), and
\(\Lambda'_{-1}\) from \(\Lambda_1\).

 \begin{enumerate}
 \item  Take the Alvis--Curtis dual ${}^{\bd}\Lambda_{-1}$ (resp. ${}^{\bd}\Lambda_1$);
\item   Remove a largest entry of \({}^{\bd}\Lambda_{-1}\) (resp.
\({}^{\bd}\Lambda_1\)).  If the largest entry occurs in both rows, either choice
is allowed; this choice records the corresponding sign
\(\zeta({}^{\bd}\Lambda_{\pm1})\);
\item   In the construction of $\Lambda_1'$, swap the two rows if
\(
\epsilon_{-1}\ev=+.
\)
In the construction of $\Lambda_{-1}'$, swap the two rows if
\(
\epsilon_{-1}\ev\varepsilon(-1)=+;
\)
\item  Take the Alvis--Curtis dual again.
\end{enumerate}
Several signs occur in Theorem~\ref{ggp}. The following examples explain the roles of these signs and
illustrate the recipe.

\begin{exmp}\label{dex1}
    Let $\pi_{\Lambda_{-1}}=\pi_{-,-,\Lambda_{-1}}\in \Irr(\Sp_{2k^2}(\mathbbm{k}))$ with $k$ even and
   \[\Lambda_{-1}=\begin{pmatrix}
2k-1,2k-2,\cdots,1,0\\
-
\end{pmatrix}.\]
Then $\pi_{\Lambda_{-1}}$ is a cuspidal representation, and ${}^{\bd}\Lambda_{-1}=\Lambda_{-1}$. Thus $\ev=\zeta({}^{\bd}\Lambda_{-1})=+$. Removing the largest element in $\Lambda_{-1}$, we obtain the symbol
\[
\Lambda_{1}^*=\begin{pmatrix}
2k-2,\cdots,1,0\\
-
\end{pmatrix}
\]
with defect $2k-1$. However, this symbol does not correspond to a representation of a symplectic group, as $\rm{def}(\Lambda_1^*)\equiv3\pmod 4$. Since $k$ is even, we have $\epsilon_{-1}=+$.
Then $\epsilon_{-1}\ev=+$. According to the third step of the above recipe, we need to interchange the two rows of $\Lambda_1^*$ and get
\[
\Lambda_{1}'
=\begin{pmatrix}
-\\
2k-2,\cdots,1,0
\end{pmatrix}.
\]
Note that $\Lambda_1'$ is the symbol corresponding to the unique unipotent cuspidal representation of $\Sp_{2k(k-1)}(\mathbbm{k})$, and the Alvis--Curtis dual sends $\Lambda_1'$ to itself. Then we have the first descent $\CD^{\rm FJ}_{2\ell, a}(\pi_{\Lambda_{-1}})=\pi_{\Lambda_1'}$ with $\ell=k$.

Let us calculate the sign $a$ in the first descent $\CD^{\rm FJ}_{2\ell, a}(\pi_{\Lambda_{-1}})$.  There is no $\rho$-part in $\pi_{\Lambda_{-1}}$, hence $\er$ and $\epsilon_{\rho^-_1}$ are not involved.
Recall that we regard the  trivial representation of $\Sp_{0}(\mathbbm{k})$ (the trivial group) as a unipotent representation corresponding to the symbol
   \[\Lambda=\begin{pmatrix}
0\\
-
\end{pmatrix}
\quad \textrm{instead of }
 \begin{pmatrix}
-\\
-
\end{pmatrix},
\] 
since the defect of a symbol corresponding to a unipotent representation of a symplectic group has to be odd. Then We have to remove the 0 in the top array. 
Thus $\ep=-$ and $a=\ev\ep=-$.

In summary, we have the descent diagram
\[
\xymatrix{
& \pi_{\Lambda_{-1}} \ar[dl]_{(k,+)} \ar[dr]^{(k,-)} & \\
0 & & \pi_{\Lambda_1'}.
}
\]
\end{exmp}

\begin{exmp}\label{dex2}
    Let $\pi_{\Lambda_{1}}=\pi_{-,\Lambda_{1},-}\in \Irr(\Sp_{2k(k+1)}(\mathbbm{k}))$ with $k$ even and
       \[\Lambda_{1}=\begin{pmatrix}
2k,2k-1,2k-2,\cdots,1,0\\
-
\end{pmatrix}.\]
    Then $\pi_{\Lambda_1}$ is the unique unipotent cuspidal representation of $\Sp_{2k(k+1)}(\mathbbm{k})$, and $\ep=-\zeta({}^{\bd}\Lambda_1)=-$. 
    
 Since $\pi_{\Lambda_1}$ has no $\rho$-part, the signs
$\epsilon_\rho$ and $\epsilon_{\rho_1^-}$ are not involved, and hence
\(
a=\ep\ev=-\ev.
\)
Since $\pi_{\Lambda_1}$ has no $(-1)$-part, we take
$\epsilon_{-1}=+$. By the same argument as in
Example~\ref{dex1}, the first descent index is \(k\), and
\[
\CD^{\rm FJ}_{2k,a}(\pi_{\Lambda_1})
=
\begin{cases}
\pi_{\Lambda_{-1}'}, & \textrm{if }a=+,\\
\pi_{\Lambda_{-1}^{\prime t}}, & \textrm{if }a=-,
\end{cases}
\]
where
\[
\Lambda_{-1}'
=
\begin{cases}
\begin{pmatrix}
2k-1,\ldots,1,0\\
-
\end{pmatrix},
& \textrm{if }\varepsilon(-1)=+,\\[4mm]
\begin{pmatrix}
-\\
2k-1,\ldots,1,0
\end{pmatrix},
& \textrm{if }\varepsilon(-1)=-.
\end{cases}
\]
As explained in Section~\ref{sec:L-s}, once the \((-1)\)-part is exhausted, we use the convention for the exhausted
symbol and regard it as the symbol
\[
 \begin{pmatrix}
0\\
0
\end{pmatrix}.
\]
The two zeros are treated as two equal largest entries.  Thus either zero may be
removed first.  After the possible row-switch in Step~(3) of the above recipe, both
choices lead to the same remaining trivial symbol
\[
 \begin{pmatrix}
0\\
-
\end{pmatrix}.
\]
Therefore the $(-1)$-part of the descendant is still trivial, but the choice of the
row from which the first zero is removed records the sign
\(\ev=\zeta({}^{\bd}\Lambda_{-1})\).  Hence \(\ev\) can be chosen to be either
\(+\) or \(-\). Thus we have the descent diagram
\[
\xymatrix{
& \pi_{\Lambda_{1}} \ar[dl]_{(k,+)} \ar[dr]^{(k,-)} & \\
\pi_{\Lambda_{-1}^{\prime }}&& \pi_{\Lambda_{-1}^{\prime t}}=\sgn\cdot\pi_{\Lambda_{-1}^{\prime}}.
}
\]
\end{exmp}

Assume that $k$ is even. Combining Example \ref{dex1} and Example \ref{dex2}, we have the following descent diagram  
\[
\xymatrix{
& & \pi_{k(k+1)} \ar[dl]_{(k,+)} \ar[dr]^{(k,-)} & &  \\
& \pi_{k^2,\alpha} \ar[dl]_{(k,\varepsilon(-1))} \ar[dr]^{(k,-\varepsilon(-1))} & & 
\pi_{k^2,\beta} \ar[dl]_{(k,\varepsilon(-1))} \ar[dr]^{(k,-\varepsilon(-1))} & \\
0 & & \pi_{k(k-1)} \ar[dl]_{(k-1,+)} \ar[dr]^{(k-1,-)} & & 0 \\
 & \pi_{(k-1)^2,\alpha} & & \pi_{(k-1)^2,\beta} & 
}
\]
where $\pi_{k(k+1)}$ is the unique unipotent cuspidal representation of $\Sp_{2k(k+1)}(\mathbbm{k})$, and $\pi_{k^2,\alpha}$ and $\pi_{k^2,\beta}$ are the only two cuspidal representations in $\cal{E}(\Sp_{2k^2},s)$. Here $s$ has the eigenvalue $-1$ with multiplicity $2k^2$ and $C_{G^*}(s)\cong {\rm O}^{(-1)^k}_{2k^2}$.

When the largest entries in the two rows of
\({}^{\bd}\Lambda_1\) are equal, the two possible choices of
\(\ep=-\zeta({}^{\bd}\Lambda_1)\) may lead to first descents
with different signs and different constituents, as the next
example illustrates.

\begin{exmp}\label{dex3}
Consider the trivial representation $\mathbbm{1}$ of $\Sp_2(\mathbbm{k})$, which corresponds to the symbol 
\[
\Lambda_1=\begin{pmatrix}
1\\
-
\end{pmatrix}, \quad\textrm{with }
{}^{\bd}\Lambda_1=\begin{pmatrix}
1,0\\
1
\end{pmatrix}.
\]
In this case,  $\ep=-\zeta({}^{\bd}\Lambda_1)$  can be chosen to be either \(+\) or \(-\). By the same argument as in Example \ref{dex2}, $\ev$ can also be chosen to be either \(+\) or \(-\). Then we have the following descent diagram
\[
\xymatrix{
&  \mathbbm{1} \ar[dl]_{(0,\varepsilon(-1))} \ar[dr]^{(0,-\varepsilon(-1))}  &  \\
\pi_{\Lambda_{-1}^{\prime }} \oplus \pi_{\Lambda_{-1}^{\star }}  \ar[d]_{(1,+)} & & \pi_{\Lambda_{-1}^{\prime t}}\oplus\pi_{\Lambda_{-1}^{\star t }} \ar[d]^{(1,-)} \\
\mathbbm{1} \oplus \mathbbm{1} & & \mathbbm{1} \oplus \mathbbm{1}
}
\]
where
\[
\Lambda_{-1}'=\begin{pmatrix}
1\\
0
\end{pmatrix}\ 
\textrm{ and }\
\Lambda_{-1}^\star=\begin{pmatrix}
-\\
1,0
\end{pmatrix}.
\]
The representations $\pi_{\Lambda_{-1}^{\prime }}$ and $\pi_{\Lambda_{-1}^{\star}}$ in the $(0,\varepsilon(-1))$-descent of $\mathbbm{1}$ correspond to the pairs $(\ep,\ev)=(+,\varepsilon(-1))$ and $(\ep,\ev)=(-,-\varepsilon(-1))$, respectively. Note that $\pi_{\Lambda_{-1}^{\prime }}$ and $\pi_{\Lambda_{-1}^{\star}}$ have different cuspidal supports. Namely,  $\pi_{\Lambda_{-1}^{\star}}$ itself is cuspidal, while the cuspidal support of $\pi_{\Lambda_{-1}^{\prime}}$ is the trivial representation of $\Sp_0(\mathbbm{k})$. 
Thus, even within a fixed first descent, different choices of
\(\ep\) may give rise to representations with different
cuspidal supports.
\end{exmp}

Examples \ref{dex1}, \ref{dex2} and \ref{dex3} tell us that for quadratic unipotent representations, the sign $a$ plays a key role in the descent. On the other hand, if there are no quadratic unipotent parts in $\pi$, then the different choices of sign $a$ will give us the same descent.

\begin{exmp}\label{dex4}
    Let $\pi_{\rho}=\pi_{\rho,-,-}\in \Irr(\Sp_{2n}(\mathbbm{k}))$. Assume that 
    $\rho$ is a generic representation. From Example \ref{dex1} we know that 
    $\zeta({}^{\bd}\Lambda_1)=+$ if there is no unipotent part. From Example \ref{dex2} we know that 
    $\zeta({}^{\bd}\Lambda_{-1})$  can be chosen to be either \(+\) or \(-\) if there is no $(-1)$-part. Therefore the sign $a$ can also be chosen to be either $+$ or $-$, and we have the descent diagram
    \[
    \xymatrix{
    & \pi_{\rho} \ar[dl]_{(n,+)} \ar[dr]^{(n,-)} & \\
    \mathbbm{1} & & \mathbbm{1}.
    }
    \]
\end{exmp}

Based on the above first-descent recipe, we now enumerate all possible first descents in the general case. In the diagrams in Corollary~\ref{5.6} below, every arrow with a nonzero target denotes a first descent. The label of the arrow, namely the corresponding descent index and sign, is omitted, since it is determined by Theorem~\ref{ggp}. An arrow ending at zero means that the corresponding choice gives no first descent. When two vertices are connected by the symbol \(\oplus\), they are the two irreducible constituents arising from the same descent. For simplicity, we do not write explicitly the \(\rho\)-parts and symbols of all the representations appearing in the diagrams, since they are determined by Theorem~\ref{ggp}. Instead, the diagrams record the relative relations among these representations, in particular when the symbol in one branch is obtained from that in another by interchanging its two rows. Thus they display both the possible constituents of the first descent of \(\prll\) and the way in which these constituents descend further.

\begin{cor}\label{5.6}
   Let $\prll\in  \Irr(\Sp_{2n}(\mathbbm{k}))$. Assume that
   \[
   {}^{\bd}\Lambda_1=
 \begin{pmatrix}
a_1,a_2,\cdots,a_{k_1}\\
b_1,b_2,\cdots,b_{h_1}
\end{pmatrix}
   \quad\textrm{and}\quad
   {}^{\bd}\Lambda_{-1}=
 \begin{pmatrix}
a'_1,a'_2,\cdots,a'_{k_{-1}}\\
b'_1,b'_2,\cdots,b'_{h_{-1}}
\end{pmatrix}.
   \]
Following the convention of Section~\ref{sec:L-s}, the largest
entry of an empty row is regarded as \(-\infty\). In the following sets
\(\{a_2,b_2\}\) and \(\{a'_2,b'_2\}\), only entries that actually
exist are included.

\begin{enumerate}
\item[(i)] If $a_1=b_1$ and $a_1'= b_1'$, then we have the first descent diagram
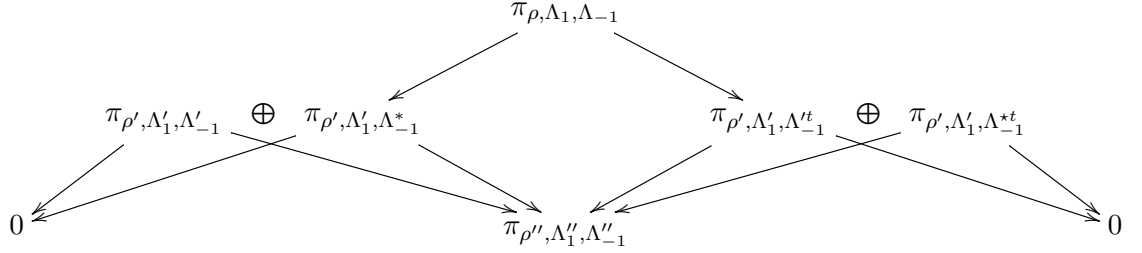
\[
\xymatrix{
& & & \prll \ar[dl] \ar[dr] & & &  \\
& \prllp \ar[dl] \ar[drr] \ar@{}[r]|\bigoplus & \prllps \ar[dll] \ar[dr] & & \prllpt \ar[dl] \ar[drr] \ar@{}[r]|\bigoplus & \prllpst \ar[dll] \ar[dr] & \\
0 & & & \prllpp & & & 0
}
\]

\item [(ii)] If $a_1\ne b_1$, ${\rm min}\{a_1,b_1\}\notin\{a_2,b_2\}$ and $a_1'= b_1'$, then we have the first descent diagram
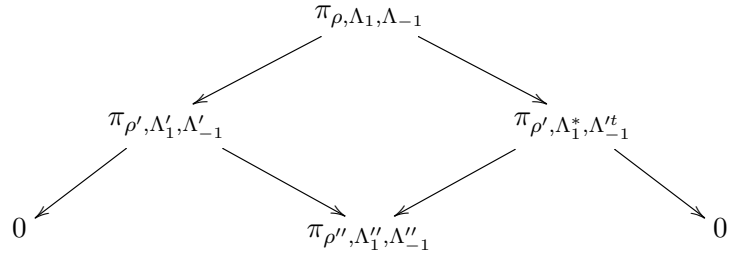
\[
\xymatrix{
& & \prll \ar[dl] \ar[dr] & &  \\
& \prllp \ar[dl] \ar[dr] & & \pi_{\rho',\Lambda_1^*,\Lambda_{-1}^{\prime t}} \ar[dl] \ar[dr] & \\
0 & & \prllpp & & 0
}
\]

 \item [(iii)] If $a_1\ne b_1$, ${\rm min}\{a_1,b_1\}\in\{a_2,b_2\}$ and $a_1'= b_1'$, then we have the first descent diagram
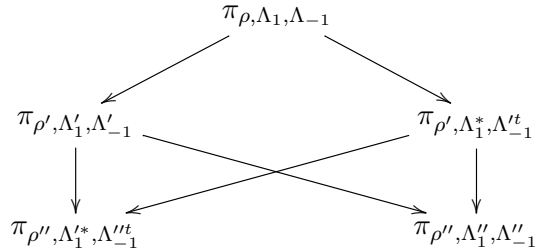
\[
\xymatrix{
& & \prll \ar[dl] \ar[dr] & &  \\
& \prllp \ar[d] \ar[drr] & & \pi_{\rho',\Lambda_1^*,\Lambda_{-1}^{\prime t}} \ar[dll] \ar[d] & \\
&\pi_{\rho'',\Lambda_1^{\prime*},\Lambda_{-1}^{\prime\prime t}}  & & \prllpp  &
}
\]

\item[(iv)] If $a_1=b_1$, $a_1'\ne b_1'$ and ${\rm min}\{a'_1,b'_1\}\notin\{a'_2,b'_2\}$, then we have the first descent diagram
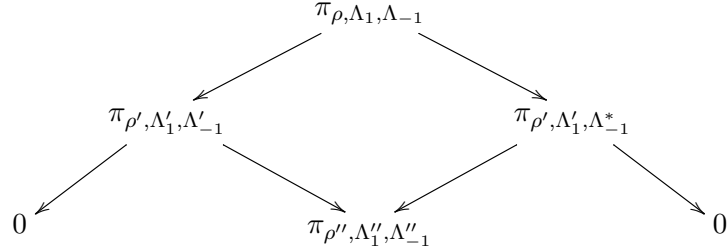
\[
\xymatrix{
& & \prll \ar[dl] \ar[dr] & &  \\
& \prllp \ar[dl] \ar[dr] & & \pi_{\rho',\Lambda_1',\Lambda^*_{-1}} \ar[dl] \ar[dr] & \\
0 & & \prllpp & & 0
}
\]
\item[(v)]   If $a_1\ne b_1$, ${\rm min}\{a_1,b_1\}\notin\{a_2,b_2\}$, $a_1'\ne b_1'$ and ${\rm min}\{a'_1,b'_1\}\notin\{a'_2,b'_2\}$, then we have the first descent diagram
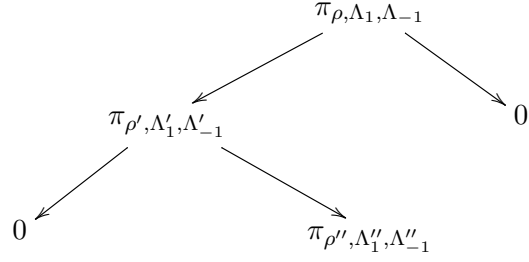
\[
\xymatrix{
& & \prll \ar[dl] \ar[dr] & &  \\
& \prllp \ar[dl] \ar[dr] & & 0  & \\
0 & & \prllpp & & 
}
\]

 \item [(vi)] If $a_1\ne b_1$, ${\rm min}\{a_1,b_1\}\in\{a_2,b_2\}$, $a_1'\ne b_1'$ and ${\rm min}\{a'_1,b'_1\}\notin\{a'_2,b'_2\}$, then we have the first descent diagram
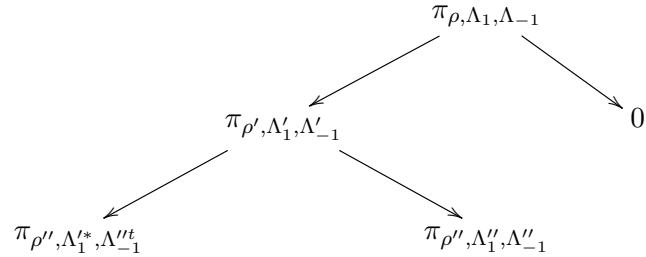
\[
\xymatrix{
& & \prll \ar[dl] \ar[dr] & &  \\
& \prllp \ar[dl] \ar[dr] & & 0  & \\
\pi_{\rho'',\Lambda^{\prime *}_{1},\Lambda^{\prime\prime t}_{-1}} & & \prllpp & & 
}
\]
\item[(vii)] If $a_1=b_1$, $a_1'\ne b_1'$ and ${\rm min}\{a'_1,b'_1\}\in\{a'_2,b'_2\}$, then we have the first descent diagram
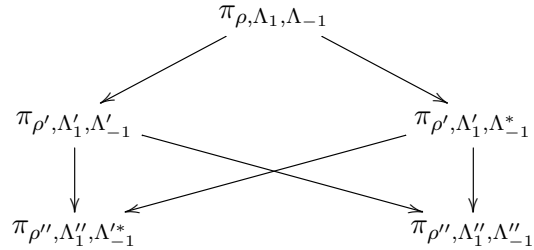
\[
\xymatrix{
& & \prll \ar[dl] \ar[dr] & &  \\
& \prllp \ar[d] \ar[drr] & & \pi_{\rho',\Lambda_1',\Lambda^*_{-1}} \ar[dll] \ar[d] & \\
 &\pi_{\rho'',\Lambda^{\prime \prime}_{1},\Lambda^{\prime* }_{-1}} & &\prllpp &
}
\]

\item[(viii)] If $a_1\ne b_1$, ${\rm min}\{a_1,b_1\}\notin\{a_2,b_2\}$, $a_1'\ne b_1'$ and ${\rm min}\{a'_1,b'_1\}\in\{a'_2,b'_2\}$, then we have the first descent diagram
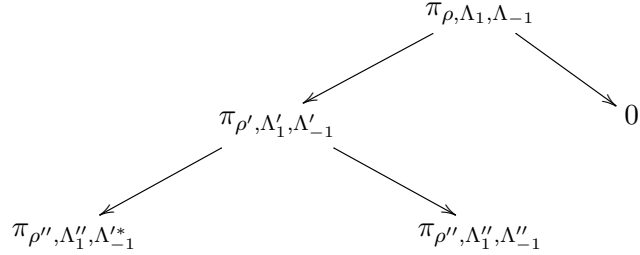
\[
\xymatrix{
& & \prll \ar[dl] \ar[dr] & &  \\
& \prllp \ar[dl] \ar[dr] & & 0  & \\
 \pi_{\rho'',\Lambda^{\prime \prime}_{1},\Lambda^{\prime* }_{-1}}& & \prllpp & &
}
\]

\item[(ix)] If $a_1\ne b_1$, $\min\{a_1,b_1\}\in\{a_2,b_2\}$, $a_1'\ne b_1'$ and $\min\{a'_1,b'_1\}\in\{a'_2,b'_2\}$, then we have the first descent diagram
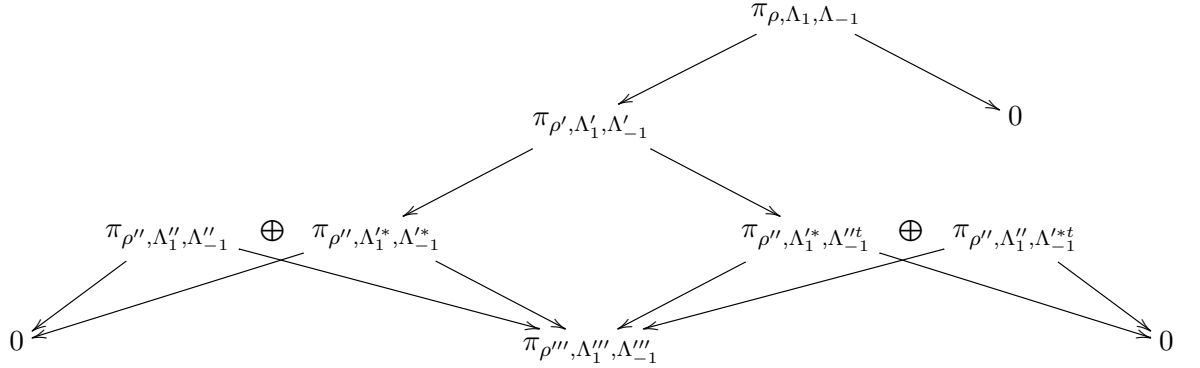
\[
\xymatrix{
&&& & \prll \ar[dl] \ar[dr] & &  \\
&&& \prllp \ar[dl] \ar[dr] & & 0  & \\
&\prllpp \ar[dl] \ar[drr]  \ar@{}[r]|\bigoplus &\pi_{\rho'',\Lambda^{\prime *}_{1},\Lambda^{\prime* }_{-1}} \ar[dll] \ar[dr]& & \pi_{\rho'',\Lambda^{\prime *}_1,\Lambda^{\prime\prime t}_{-1}} \ar[dl] \ar[drr]  \ar@{}[r]|\bigoplus &  \pi_{\rho'',\Lambda''_1,\Lambda^{\prime* t}_{-1}}  \ar[dll] \ar[dr] &\\
0& & & \pi_{\rho''',\Lambda_1''',\Lambda_{-1}'''} &  & & 0 &
}
\]
\end{enumerate}
\end{cor}
\begin{proof}
The assertion follows by applying the recipe above separately to
the \(1\)-part and the \((-1)\)-part, with the corresponding
descent signs determined by Theorem~\ref{ggp}.  For either symbol, there are three possibilities: the two row
maxima are equal; or they are unequal and, after the larger one
is removed, the new maximal entry is unique; or the new maximal
entry occurs in both rows. Combining the three possibilities for
the \(1\)-symbol with the three possibilities for the
\((-1)\)-symbol gives the nine cases above. The arrows in the diagrams are obtained by iterating
the same recipe.
\end{proof}

\subsection{Calculating first descent indices}

Using Theorem~\ref{ggp}\textup{(ii)} and a direct calculation,
we obtain the following formula for the first descent index.

\begin{cor}\label{fdi}
Let
\(
\pi=\pi_{\rho,\Lambda_1,\Lambda_{-1}}
\in\mathcal E(\Sp_{2n}(\mathbbm{k}),s),
\)
and let
\(
\pi'=\pi_{\rho',\Lambda'_1,\Lambda'_{-1}}
\)
be an irreducible constituent of a first descent of \(\pi\), whose
index is denoted by \(\ell\).  We define \(\rank(\rho)\) by
\(
\pi_{\rho}\in
\Irr\bigl(\Sp_{2\rank(\rho)}(\mathbbm{k})\bigr),
\)
and set
\begin{itemize}
    \item $\ell[\ne\pm1]:=\rank(\rho)-\rank(\rho')$;
    \item $\ell[1]:=\rank({}^{\bd}\Lambda_1)-\rank({}^{\bd}\Lambda'_{-1})$;
\item $\ell[-1]:=\rank({}^{\bd}\Lambda_{-1})-\rank({}^{\bd}\Lambda'_1)$.
\end{itemize}
Then
\[
\ell=\ell[\ne\pm1]+\ell[1]+\ell[-1].
\]
Assume moreover that $|{\mathrm{def}}(\Lambda_1)|=2k_1+1$ and $|{\mathrm{def}}(\Lambda_{-1})|=2k_{-1}$. Then
   \begin{itemize}
    \item 
    $
    \ell[\ne \pm 1]=\sum_{[a]\ne[\pm 1]}\widetilde{\#[a]}\cdot\lambda[a]_1^t,
    $
    \item
    $
\widetilde{\#[a]}:=
 \begin{cases}
 \#[a],&\textrm{ if $G_{[a]}^*(s)$ is a general linear group};\\
  \frac{\#[a]}{2},&\textrm{ if $G_{[a]}^*(s)$ is a unitary group},
\end{cases}
$
    \item
 $
\ell[1]=
\begin{cases}
\left(\Upsilon(\Lambda_1)^{-\zeta({}^{\bd}\Lambda_1)}\right)^t_1-(k_1+1), &\textrm{ if }\zeta({}^{\bd}\Lambda_1)\cdot{\mathrm{def}}(\Lambda_1)<0;\\
\left(\Upsilon(\Lambda_1)^{-\zeta({}^{\bd}\Lambda_1)}\right)^t_1+k_1, &\textrm{ if }\zeta({}^{\bd}\Lambda_1)\cdot{\mathrm{def}}(\Lambda_1)>0,
\end{cases}
$
\item
$
\ell[-1]=
\begin{cases}
\left(\Upsilon(\Lambda_{-1})^{\zeta({}^{\bd}\Lambda_{-1})}\right)^t_1-k_{-1}, &\textrm{ if }{\zeta({}^{\bd}\Lambda_{-1})}\cdot{\mathrm{def}}(\Lambda_{-1})\leqslant 0;\\
\left(\Upsilon(\Lambda_{-1})^{\zeta({}^{\bd}\Lambda_{-1})}\right)^t_1+k_{-1}, &\textrm{ if }{\zeta({}^{\bd}\Lambda_{-1})}\cdot{\mathrm{def}}(\Lambda_{-1})> 0.
\end{cases}
$
\end{itemize}
In particular, we have \(\ell[1]\geqslant k_1\) and \(\ell[-1]\geqslant k_{-1}\).
\end{cor}

\begin{proof}
By Theorem~\ref{ggp}\textup{(ii)}, for every
\([a]\ne[\pm1]\) we have
\(
\lambda'[-a]=\lambda[a]_-.
\)
Hence
\[
\rank(\rho)-\rank(\rho')
=
\sum_{[a]\ne[\pm1]}
\widetilde{\#[a]}
\bigl(|\lambda[a]|-|\lambda'[-a]|\bigr)
=
\sum_{[a]\ne[\pm1]}
\widetilde{\#[a]}\lambda[a]^t_1.
\]

Under the canonical Lusztig correspondence, for a representation
\(\pi_{\rho,\Lambda_1,\Lambda_{-1}}\) of
\(\Sp_{2m}(\mathbbm{k})\), the rank \(m\) is the sum of the ranks of
\(\rho\), \({}^{\bd}\Lambda_1\), and
\({}^{\bd}\Lambda_{-1}\). Applying this to \(\pi\) and \(\pi'\), the definitions immediately give
\[
\ell=\ell[\ne\pm1]+\ell[1]+\ell[-1].
\]

To compute the two symbol contributions, we use the first-descent
recipe in Section~\ref{sec6.1}.  For the \(1\)-part, Step~(2) of the first-descent recipe removes a largest entry of
\({}^{\bd}\Lambda_1\), with the choice of row recorded by
\(\zeta({}^{\bd}\Lambda_1)\), while the possible row transposition in
Step~(3) does not change the rank.  The rank formula for symbols,
together with the change of defect, therefore gives the stated formula
for \(\ell[1]\).  The same calculation applied to
\({}^{\bd}\Lambda_{-1}\) gives the formula for \(\ell[-1]\).

It remains to prove the inequalities.  For the \(1\)-part, replace
the sign
\(
\ep=-\zeta({}^{\bd}\Lambda_1)
\)
by the opposite sign in the conditions of
Theorem~\ref{ggp}(i), while keeping the
\((\ne\pm1)\)- and \((-1)\)-parts fixed. Let
\(\ell_{\mathrm{op}}[1]\) be the largest possible contribution of the
\(1\)-part under this choice.  If
\(\ell_{\mathrm{op}}[1]>\ell[1]\), then Theorem~\ref{ggp}(i) gives a
nonzero descent of index
\[
\ell-\ell[1]+\ell_{\mathrm{op}}[1]>\ell,
\]
contradicting the definition of the first descent index.  Hence
\(
\ell_{\mathrm{op}}[1]\leqslant\ell[1].
\)
The formulas corresponding to the two sign choices show that one of
\(\ell[1]\) and \(\ell_{\mathrm{op}}[1]\) is obtained by adding
\(k_1\) to a first-column length.  Therefore
\(
\ell[1]\geqslant k_1.
\)

For the \((-1)\)-part, the assertion is immediate when \(k_{-1}=0\).
When \(k_{-1}>0\), replace
\(\ev=\zeta({}^{\bd}\Lambda_{-1})\) by the opposite sign and make the
corresponding row transposition in the \(1\)-part.  Since this does not
change the rank contribution of the \(1\)-part, the same argument gives
\(
\ell[-1]\geqslant k_{-1}.
\)
\end{proof}

\begin{cor}\label{fdi-comp}
Let
\(
\pi'=\pi_{\rho',\Lambda'_1,\Lambda'_{-1}}
\)
be an irreducible constituent of a first descent of
\(\pi=\pi_{\rho,\Lambda_1,\Lambda_{-1}}\). Let \(\ell[\pm1]\) and \(\ell'[\pm1]\) denote the symbol
contributions to the first descent indices of \(\pi\) and \(\pi'\),
respectively, as in Corollary~\ref{fdi}.  Then
\[
\ell[-1]\geqslant\ell'[1],
\qquad
\ell[1]\geqslant\ell'[-1]-1.
\]
Equality in the first (resp. second) inequality holds if and only if the largest
entries in the two rows of \({}^{\bd}\Lambda_{-1}\) (resp. \({}^{\bd}\Lambda_1\)) are equal.
\end{cor}

\begin{proof}
In the cases that one of the symbols is exhausted,
we use the convention of
Section~\ref{sec:L-s}; the following calculation is then
unchanged.

By the definitions in Corollary~\ref{fdi} and the first-descent
recipe in Section~\ref{sec6.1}, \(\ell[-1]\) and \(\ell'[1]\) are the
successive rank drops obtained by removing the two largest entries of
\({}^{\bd}\Lambda_{-1}\); the possible row transpositions do not
change the ranks.

Let \(N\) be the total number of entries of
\({}^{\bd}\Lambda_{-1}\), and let \(x_1\geqslant x_2\) be its two
largest entries.  Since \({\mathrm{def}}(\Lambda_{-1})\) is even, \(N\) is even, and
the rank formula for symbols gives
\[
\ell[-1]=x_1-\frac N2+1,
\qquad
\ell'[1]=x_2-\frac N2+1.
\]
This proves the first inequality.  Equality holds precisely when
\(x_1=x_2\), that is, when the largest entries in the two rows are
equal.

Similarly, let \(M\) be the total number of entries of
\({}^{\bd}\Lambda_1\), and let \(y_1\geqslant y_2\) be its two
largest entries.  Since \({\mathrm{def}}(\Lambda_1)\) is odd, \(M\) is odd, and hence
\[
\ell[1]=y_1-\frac{M-1}{2},
\qquad
\ell'[-1]=y_2-\frac{M-1}{2}+1.
\]
The second inequality and its equality statement follow.
\end{proof}

\begin{prop}\label{trivial}
\begin{enumerate}
\item[(i)] Let $\pi\in\Irr(\Sp_{2n}(\mathbbm{k}))$. Then $\CD^{\rm FJ}_{2\ell,a}(\pi) \equiv 0$ for every $\ell>0$ and $a\in\{\pm\}$ if and only if $\pi$ is the trivial representation.
\item[(ii)] The first descent index of $\omega_{n,\psi_a}$ is $1$, and 
$\CD^{\rm FJ}_{2,a}(\omega_{n,\psi_a})=\mathbbm{1}_{2(n-1)}\oplus \mathbbm{1}_{2(n-1)}$ and $\CD^{\rm FJ}_{2,-a}(\omega_{n,\psi_a})=0$.

\end{enumerate}
\end{prop}

\begin{proof}
We first prove (i).  If
\(
        \CD^{\rm FJ}_{2\ell,a}(\pi)\equiv 0
\)
for every \(\ell>0\) and \(a\in\{\pm\}\), then the first descent index of
\(\pi\) is \(0\).  By Corollary~\ref{fdi},
\(
\ell =\ell[\ne\pm1]+\ell[1]+\ell[-1]=0.
\)
All three summands are nonnegative. Hence
\(
\ell[\ne\pm1]=\ell[1]=\ell[-1]=0.
\)
The first equality implies that the $\rho$-part is trivial, while
Corollary~\ref{fdi} gives
\(
k_1=k_{-1}=0.
\)
Substituting these values into the formulas for $\ell[1]$ and
$\ell[-1]$ shows directly that $\Lambda_1$ and $\Lambda_{-1}$ are the
symbols occurring in the trivial representation. Therefore
\(
\pi=\mathbbm{1}_{2n}.
\)
The converse follows from the same formulas. This proves (i).

For \textup{(ii)}, by \eqref{descent0},
\(
\CD^{\rm FJ}_{0,a}(\mathbbm{1}_{2n})
=
\omega_{n,\psi_{\varepsilon(-1)a}}.
\)
Consequently, computing the first descent of
$\omega_{n,\psi_{\varepsilon(-1)a}}$ equals computing the first
descent of $\CD^{\rm FJ}_{0,a}(\mathbbm{1}_{2n})$.
Corollary~\ref{5.6}\textup{(i)} together with Example~\ref{dex3} gives the corresponding two-step
descent diagram of $\mathbbm{1}_{2n}$. In this diagram, the first step corresponding to the $(0,a)$-descent gives the two irreducible
constituents of $\CD^{\rm FJ}_{0,a}(\mathbbm{1}_{2n})$.

By the same direct calculation as in Example~\ref{dex3}, both
constituents of $\CD^{\rm FJ}_{0,a}(\mathbbm{1}_{2n})$ have first descent index \(1\), and their descents with
sign $\varepsilon(-1)a$ are the trivial representation. Corollary~\ref{5.6}\textup{(i)}
shows that the branches in the second step with the opposite sign
terminate at zero. Therefore
\[
\CD^{\rm FJ}_{2,\varepsilon(-1)a}
\bigl(\omega_{n,\psi_{\varepsilon(-1)a}}\bigr)
=
\mathbbm{1}_{2(n-1)}
\oplus
\mathbbm{1}_{2(n-1)}
\]
and
\[
\CD^{\rm FJ}_{2,-\varepsilon(-1)a}
\bigl(\omega_{n,\psi_{\varepsilon(-1)a}}\bigr)
=
0.
\]
Replacing $\varepsilon(-1)a$ by \(a\), we obtain that
the first descent index of $\omega_{n,\psi_a}$ is \(1\).
\end{proof}

\section{Rational wavefront sets for symplectic groups}\label{sec7}

\subsection{Descent sequences vs rational orbits}

The purpose of this section is to translate the order structure of descent
sequences into statements about rational wavefront sets.

\begin{thm}\label{ds}
Let \(\pi\in\Irr(\Sp_{2n}(\mathbbm{k}))\), and let
\(\mathcal D(\pi)\) be the set of descent sequence indexes of
descent sequences of \(\pi\).
\begin{enumerate}
\item[(i)] Every maximal descent sequence is obtained by successive first
descents.  Moreover, the numerical sequence
\(
        (2\ell_1,2\ell_2,\ldots)
\)
is the same for every maximal descent sequence index
\(
        \Gamma=\llb(2\ell_1,a_1),(2\ell_2,a_2),\ldots\rrb
\) in \(\mathcal D(\pi)\).

\item[(ii)] Every maximal element of \(\mathcal D(\pi)\) is good.

\item[(iii)] If \(\Gamma\in\mathcal D(\pi)\) and
\(\Gamma'\leqslant\Gamma\), then \(\Gamma'\in\mathcal D(\pi)\).
\end{enumerate}
\end{thm}

\begin{cor}\label{7.2}
   Let $\pi\in \Irr(\Sp_{2n}(\mathbbm{k}))$. Then
   \begin{enumerate}
 \item [(i)] Recall that \(\ScO(\pi)^{\max,\mathrm{st}}\) consists of the rational
orbits in \(\ScO(\pi)\) lying over the Kawanaka wavefront orbit. Then \[
   \ScO(\pi)^{\max, \rm{st}} = \ScO(\pi)^{\max}.
   \]

\item[(ii)] We have
\[
        \ScO(\pi)=\overline{\ScO(\pi)^{\max}},
\]
where the bar denotes rational closure.
\end{enumerate}
\end{cor}
\begin{proof}
Consider the set
\[
\ScO(\pi)^{\rm max\text{-}descent}
:=
\bigcup_{\Gamma}\CO(\Gamma),
\]
where \(\Gamma\) runs over the maximal descent sequence indexes of \(\pi\).
By Theorem~\ref{ds} (ii), these indexes are good, so the orbits \(\CO(\Gamma)\)
are well-defined.  By Lemma~\ref{lem:orbit-sequence-order} and
Corollary~\ref{od}, maximal descent sequence indexes correspond exactly to
maximal rational orbits in \(\ScO(\pi)\).  Hence
\[
        \ScO(\pi)^{\rm max\text{-}descent}
        =
        \ScO(\pi)^{\max}.
\]
On the other hand, Theorem~\ref{ds} (i) implies that all \(F\)-rational orbits
in \(\ScO(\pi)^{\rm max\text{-}descent}\) lie over the same stable orbit.
Therefore
\[
        \ScO(\pi)^{\max,\rm st}
        =
        \ScO(\pi)^{\rm max\text{-}descent}
        =
        \ScO(\pi)^{\max}.
\]
This proves (i).

The second part follows from Theorem~\ref{ds} (iii), together with
Lemma~\ref{lem:orbit-sequence-order} and Corollary~\ref{od}.
\end{proof}

\subsection{Proof of Theorem \ref{ds} (i)}

We prove Theorem~\ref{ds}(i).  This consists of two assertions.  First, every
maximal descent sequence is obtained by successive first descents.  Second,
the numerical sequence
\(
        (2\ell_1,2\ell_2,\ldots)
\)
is independent of the choice of the maximal descent sequence.  

We now prove the first assertion. Let
\[
        \prll\to\prllp\to\prllpp\to\cdots
\]
be a maximal descent sequence.  Its tail after any arrow is again a
maximal descent sequence.  It is therefore enough to prove that the
first arrow of every maximal descent sequence is a first descent.

By Corollary~\ref{fdi} and Theorem~\ref{ggp}(ii), the first arrow is a
first descent once the following three facts are verified: its
\((\pm1)\)-symbol data attain the minimal values allowed by
Theorem~\ref{ggp}(i), the corresponding sign pair is the first-descent sign
pair, and its \((\ne\pm1)\)-part is the minimal one allowed by
Theorem~\ref{ggp}(i).

More precisely, we prove the following three assertions:
\begin{enumerate}
\item[(a)] for some sign pair \((\ep,\ev)\), the \((\pm1)\)-symbol part is
sharp.  Equivalently, for this fixed pair \((\ep,\ev)\), the bipartitions of \(\Lambda'_1\) and \(\Lambda'_{-1}\) are the
minimal ones allowed by the inequalities in Theorem~\ref{ggp}(i);

\item[(b)] this sign pair can be chosen to be the first-descent sign pair
\[
        (\ep,\ev)
        =
        (-\zeta({}^{\bd}\Lambda_1),\zeta({}^{\bd}\Lambda_{-1}));
\]

\item[(c)] the \(\rho\)-part is sharp.  Equivalently, for every class
\([b]\) occurring in the \((\ne\pm1)\)-part, the target partition is the
minimal choice allowed by Theorem~\ref{ggp}(i), namely
\(
        \lambda'[-b]=\lambda[b]_- .
\)
\end{enumerate}
Propositions~\ref{lem73}, \ref{lem75}, and \ref{lem74} prove (a), (b), and
(c), respectively.  Hence they prove that the first arrow is a first descent.
After these propositions, we return to the numerical assertion in
Theorem~\ref{ds}(i).

\begin{prop}\label{lem73}
The first arrow
\(
        \prll\to\prllp
\)
is sharp with respect to one of the sign pairs appearing in the conditions of
Theorem~\ref{ggp}(i) for this descent.  More explicitly, there exist
\(\ep,\ev\in\{\pm\}\) such that
\[
\begin{array}{@{}l@{\qquad}l@{}}
\Upsilon(\Lambda_{-1}')^{-\ep\ev\epsilon_{-1}\varepsilon(-1)}
        =
        \Upsilon(\Lambda_1)^{\ep}_-,
&
\Upsilon(\Lambda_{-1}')^{\ep\ev\epsilon_{-1}\varepsilon(-1)}
        =
        \Upsilon(\Lambda_1)^{-\ep},
\\[2mm]
\Upsilon(\Lambda_1')^{\epsilon_{-1}}
        =
        \Upsilon(\Lambda_{-1})^{-\ev}_-,
&
\Upsilon(\Lambda_1')^{-\epsilon_{-1}}
        =
        \Upsilon(\Lambda_{-1})^{\ev}.
\end{array}
\]
The corresponding defects are given by
\[
\begin{aligned}
&{\mathrm{def}}(\Lambda_1')
        =
        -\epsilon_{-1}\ev{\mathrm{def}}(\Lambda_{-1})
        +\epsilon_{-1},\\
&{\mathrm{def}}(\Lambda_{-1}')
        =
        -\epsilon_{-1}\ev\varepsilon(-1){\mathrm{def}}(\Lambda_1)
        -\ep\ev\epsilon_{-1}\varepsilon(-1).
\end{aligned}
\]

\end{prop}

\begin{proof}
Fix a maximal descent sequence
\[
        \prll\to\prllp\to\prllpp\to\cdots.
\]
If the maximal descent sequence has length one, the assertion is immediate. We may therefore assume that it has length at least two. Suppose, for contradiction, that the first arrow
\(\prll\to\prllp\) is not sharp with respect to any sign pair. We will replace
this first arrow by a sharp one and construct another irreducible
representation \(\prlls\) fitting into a diagram
\begin{equation}\label{dia73}
    \xymatrix{
& \prllp \ar[dr]^{(\ell'',a'')} & &  \\
\prll \ar[ur]^{(\ell',a')} \ar[dr]_{(\ell^\star,a^\star)}
& & \prllpp \ar[r] & \cdots \\
& \prlls \ar[ur]_{(\ell^{\star\prime},a^{\star\prime})} & &
}
\end{equation}
such that the lower path is again a descent sequence with
\[
        \ell^\star>\ell',
        \qquad
        a'a''=a^\star a^{\star\prime}.
\]
Then the lower path gives a descent sequence of \(\prll\) whose first index is
strictly larger than that of the upper maximal descent sequence, while the
square-class contribution of the first two arrows is unchanged.  This
contradicts maximality.

Let \((\ep,\ev)\) be the sign pair appearing in the conditions of
Theorem~\ref{ggp}(i) for the occurrence of \(\prllp\) in the descent
\(
        \prll\to\prllp .
\)

We now give the combinatorial verification.  We first apply the induction
hypothesis to the tail, then replace the first arrow by its sharp version, and
finally verify that \(\prlls\to\prllpp\) is still a descent.

{\it Step 1.} Apply the induction hypothesis to the tail.

By the tail we mean the remaining sequence
\[
        \prllp\to\prllpp\to\cdots
\]
obtained by deleting the first arrow.  This is again a maximal descent
sequence; otherwise, replacing the tail by a larger descent sequence of
\(\prllp\) would give a descent sequence of \(\prll\) larger than the fixed
one.

For the \(-1\)-symbol $\Lambda_{-1}'$ in $\prllp$, the corresponding representation \(\pi_{\Lambda_{-1}'}\) is a representation of
\(\RO_{2m}^{-\ep}(\mathbbm{k})\). Then the sign
\(\epsilon_{-1}\) in Theorem~\ref{ggp}(i) is \(-\ep\) for this tail.  By the
induction hypothesis, the first arrow of this tail,
\(
        \prllp\to\prllpp,
\)
is sharp with respect to a sign pair \((\epp,\evp)\).  Thus
\begin{equation}\label{711}
\begin{array}{@{}l@{\qquad}l@{}}
\Upsilon(\Lambda_{-1}'')^{\ep\epp\evp\varepsilon(-1)}
        =
        \Upsilon(\Lambda_1')^{\epp}_-,
&
\Upsilon(\Lambda_{-1}'')^{-\ep\epp\evp\varepsilon(-1)}
        =
        \Upsilon(\Lambda_1')^{-\epp},
\\[2mm]
\Upsilon(\Lambda_1'')^{-\ep}
        =
        \Upsilon(\Lambda_{-1}')^{-\evp}_-,
&
\Upsilon(\Lambda_1'')^{\ep}
        =
        \Upsilon(\Lambda_{-1}')^{\evp}.
\end{array}
\end{equation}
The corresponding defect relations are
\[
\begin{aligned}
& {\mathrm{def}}(\Lambda_1'')
        =
        \ep\evp{\mathrm{def}}(\Lambda_{-1}')-\ep, \\
& {\mathrm{def}}(\Lambda_{-1}'')
        =
        \ep\evp\varepsilon(-1){\mathrm{def}}(\Lambda_1')
        +\ep\epp\evp\varepsilon(-1).
\end{aligned}
\]

{\it Step 2.} Record the original first arrow.

With respect to the sign pair \((\ep,\ev)\) fixed above, the symbol
inequalities in Theorem~\ref{ggp}(i) for the descent
\(\prll\to\prllp\) are
\begin{equation}\label{712}
\begin{array}{@{}l@{\qquad}l@{}}
\Upsilon(\Lambda_{-1}')^{-\ep\ev\epsilon_{-1}\varepsilon(-1)}
        \preccurlyeq
        \Upsilon(\Lambda_1)^{\ep},
&
\Upsilon(\Lambda_1)^{-\ep}
        \preccurlyeq
        \Upsilon(\Lambda_{-1}')^{\ep\ev\epsilon_{-1}\varepsilon(-1)},
\\[2mm]
\Upsilon(\Lambda_1')^{\epsilon_{-1}}
        \preccurlyeq
        \Upsilon(\Lambda_{-1})^{-\ev},
&
\Upsilon(\Lambda_{-1})^{\ev}
        \preccurlyeq
        \Upsilon(\Lambda_1')^{-\epsilon_{-1}}.
\end{array}
\end{equation}
Here \(\epsilon_{-1}\) is determined by
\(
\pi_{\Lambda_{-1}}
\in
\Irr\bigl(\RO_{2m}^{\epsilon_{-1}}(\mathbbm{k})\bigr).
\)  The corresponding defect equalities are
\[
\begin{aligned}
& {\mathrm{def}}(\Lambda_1')
        =
        -\epsilon_{-1}\ev{\mathrm{def}}(\Lambda_{-1})
        +\epsilon_{-1}, \\
& {\mathrm{def}}(\Lambda_{-1}')
        =
        -\epsilon_{-1}\ev\varepsilon(-1){\mathrm{def}}(\Lambda_1)
        -\ep\ev\epsilon_{-1}\varepsilon(-1).
\end{aligned}
\]

By the contradiction assumption, the first arrow is not sharp with respect to
this pair \((\ep,\ev)\).  We now replace it by the sharp arrow with respect to
the same pair.  Let \(\prlls\) be the resulting representation.  Thus
\(\rho^\star=\rho'\), and the symbols of \(\prlls\) are defined by
\begin{equation}\label{713}
\begin{array}{@{}l@{\qquad}l@{}}
\Upsilon(\Lambda_{-1}^\star)^{-\ep\ev\epsilon_{-1}\varepsilon(-1)}
        =
        \Upsilon(\Lambda_1)^{\ep}_-,
&
\Upsilon(\Lambda_{-1}^\star)^{\ep\ev\epsilon_{-1}\varepsilon(-1)}
        =
        \Upsilon(\Lambda_1)^{-\ep},
\\[2mm]
\Upsilon(\Lambda_1^\star)^{\epsilon_{-1}}
        =
        \Upsilon(\Lambda_{-1})^{-\ev}_-,
&
\Upsilon(\Lambda_1^\star)^{-\epsilon_{-1}}
        =
        \Upsilon(\Lambda_{-1})^{\ev}.
\end{array}
\end{equation}
Their defects are
\[
\begin{aligned}
{\mathrm{def}}(\Lambda_1^\star)
        &=
        -\epsilon_{-1}\ev{\mathrm{def}}(\Lambda_{-1})
        +\epsilon_{-1},\\
{\mathrm{def}}(\Lambda_{-1}^\star)
        &=
        -\epsilon_{-1}\ev\varepsilon(-1){\mathrm{def}}(\Lambda_1)
        -\ep\ev\epsilon_{-1}\varepsilon(-1).
\end{aligned}
\]
Since the first arrow in the upper path of \eqref{dia73} is not sharp
with respect to \((\ep,\ev)\), at least one of the four inequalities
in \eqref{712} is not an equality. Hence the sharp target \(\prlls\)
has strictly smaller rank than \(\prllp\), and therefore
\[
        \ell^\star>\ell'.
\]

{\it Step 3.} Verify that \(\prlls\to\prllpp\) is still a descent.

We do not claim that $\prlls\to\prllpp$ is sharp.  We only verify that it satisfies the
criterion in Theorem~\ref{ggp}(i). Since \(\rho^\star=\rho'\), the conditions on the
\((\ne\pm1)\)-part for the lower arrow are exactly the same as for
\(\prllp\to\prllpp\).  It therefore remains only to check the symbol
inequalities and the defect equalities.

We first combine the two arrows in the upper path.  The equalities in
\eqref{711} identify the relevant bipartitions of \(\Lambda_{\pm1}''\) with
those of \(\Lambda_{\pm1}'\), and the inequalities in \eqref{712} then compare
them with the bipartitions of \(\Lambda_{\pm1}\).  This gives
\begin{equation}\label{714}
\begin{array}{@{}l@{\qquad}l@{}}
\Upsilon(\Lambda_{-1}'')^{\ep\epp\evp\varepsilon(-1)}
        \preccurlyeq
        \Upsilon(\Lambda_{-1})^{-\epp\ev\epsilon_{-1}}_{-\epp\epsilon_{-1}},
&
\Upsilon(\Lambda_{-1})^{\epp\ev\epsilon_{-1}}_{\epp\epsilon_{-1}}
        \preccurlyeq
        \Upsilon(\Lambda_{-1}'')^{-\ep\epp\evp\varepsilon(-1)},
\\[2mm]
\Upsilon(\Lambda_1'')^{-\ep}
        \preccurlyeq
        \Upsilon(\Lambda_1)^{\ev\evp\epsilon_{-1}\varepsilon(-1)}
        _{-\ep\ev\evp\epsilon_{-1}\varepsilon(-1)},
&
\Upsilon(\Lambda_1)^{-\ev\evp\epsilon_{-1}\varepsilon(-1)}
        _{\ep\ev\evp\epsilon_{-1}\varepsilon(-1)}
        \preccurlyeq
        \Upsilon(\Lambda_1'')^{\ep}.
\end{array}
\end{equation}
The passage from \eqref{712} to \eqref{714} amounts to bookkeeping of the row signs.  For example, consider the first inequality.  By the
first equality in \eqref{711}, its left-hand side can be written as
\[
        \Upsilon(\Lambda_{-1}'')^{\ep\epp\evp\varepsilon(-1)}
        =
        \Upsilon(\Lambda_1')^{\epp}_{-},
\]
one needs to compare \(\Upsilon(\Lambda_1')^{\epp}_{-}\) with the
corresponding row of \(\Upsilon(\Lambda_{-1})\). This is obtained from \eqref{712} as
\[
\begin{cases}
\Upsilon(\Lambda_1')^{\epsilon_{-1}}_{-}
        \preccurlyeq
        \Upsilon(\Lambda_{-1})^{-\ev}_{-},
        & \epp=\epsilon_{-1},\\[1mm]
\Upsilon(\Lambda_1')^{-\epsilon_{-1}}_{-}
        \preccurlyeq
        \Upsilon(\Lambda_{-1})^{\ev}_{+},
        & \epp=-\epsilon_{-1}.
\end{cases}
\]
In the case of $\epp=-\epsilon_{-1}$, we use the equivalent deleted-row form of the fourth
inequality in \eqref{712}, according to the definition of
\(\preccurlyeq\).  In both cases this is the uniform inequality
\[
        \Upsilon(\Lambda_1')^{\epp}_{-}
        \preccurlyeq
        \Upsilon(\Lambda_{-1})^{-\epp\ev\epsilon_{-1}}
        _{-\epp\epsilon_{-1}}.
\]
The other three inequalities in \eqref{714} are obtained in the same way.

We now use the definition of the sharp replacement \(\prlls\).  For example,
the first inequality in \eqref{714} becomes
\[
\Upsilon(\Lambda_{-1}'')^{\ep\epp\evp\varepsilon(-1)}
        \preccurlyeq
        \Upsilon(\Lambda_{-1})^{-\epp\ev\epsilon_{-1}}_{-\epp\epsilon_{-1}}
        =
        \Upsilon(\Lambda_1^\star)^{\epp},
\]
by \eqref{713}.  The other three inequalities are obtained in the same way.
Hence \eqref{714} gives
\begin{equation}\label{7145}
\begin{array}{@{}l@{\qquad}l@{}}
\Upsilon(\Lambda_{-1}'')^{\ep\epp\evp\varepsilon(-1)}
        \preccurlyeq
        \Upsilon(\Lambda_1^\star)^{\epp},
&
\Upsilon(\Lambda_1^\star)^{-\epp}
        \preccurlyeq
        \Upsilon(\Lambda_{-1}'')^{-\ep\epp\evp\varepsilon(-1)},
\\[2mm]
\Upsilon(\Lambda_1'')^{-\ep}
        \preccurlyeq
        \Upsilon(\Lambda_{-1}^\star)^{-\evp},
&
\Upsilon(\Lambda_{-1}^\star)^{\evp}
        \preccurlyeq
        \Upsilon(\Lambda_1'')^{\ep}.
\end{array}
\end{equation}
Moreover, \(\pi_{\Lambda_{-1}^{\star}}\in\Irr\bigl(\RO_{2m^\star}^{-\ep}(\mathbbm{k})\bigr),\) so the sign \(\epsilon_{-1}\) attached to the source of the lower arrow in Theorem~\ref{ggp}(i) is \(-\ep\). These are precisely the symbol inequalities in Theorem~\ref{ggp}(i) for 
\(
        \prlls\to\prllpp
\)
with the sign pair \((\epp,\evp)\).

It remains to check the defect equalities for this lower arrow.  By the
construction of \(\prlls\), the symbols \(\Lambda_1^\star\) and
\(\Lambda_1'\) have the same defect, and the symbols
\(\Lambda_{-1}^\star\) and \(\Lambda_{-1}'\) have the same defect.  Therefore
the defect relations following \eqref{711} can be rewritten as
\[
\begin{aligned}
{\mathrm{def}}(\Lambda_1'')
        &=
        \ep\evp{\mathrm{def}}(\Lambda_{-1}^\star)-\ep,\\
{\mathrm{def}}(\Lambda_{-1}'')
        &=
        \ep\evp\varepsilon(-1){\mathrm{def}}(\Lambda_1^\star)
        +\ep\epp\evp\varepsilon(-1).
\end{aligned}
\]
Thus all the conditions in Theorem~\ref{ggp}(i) hold, and hence
\(\prllpp\) occurs in a descent of \(\prlls\).

We have therefore constructed a descent sequence along the lower path
with \(\ell^\star>\ell'\).  Since the two paths in \eqref{dia73} have
the same endpoint, use the same sign pairs in the two arrows, and have
the same initial and final \(\rho\)-parts, Theorem~\ref{ggp}(i) gives
\(a'a''=a^\star a^{\star\prime}\); hence the lower two-step index is
obtained from the upper one by a positive number of elementary sequence
mutations, contradicting maximality.

\end{proof}

Applying Alvis--Curtis duality to the equalities in
Proposition~\ref{lem73} gives the following equivalent description. 

\begin{enumerate}

\item Take the Alvis--Curtis duals
\({}^{\bd}\Lambda_{-1}\) and \({}^{\bd}\Lambda_1\), respectively.

\item Remove the largest entry in the top row or bottom row according to the
signs: for \({}^{\bd}\Lambda_{-1}\), remove from the top row if
\(\ev=+\) and from the bottom row if \(\ev=-\); for
\({}^{\bd}\Lambda_1\), remove from the top row if \(\ep=-\) and from the
bottom row if \(\ep=+\).

\item If
\(
        \epsilon_{-1}\ev=+
\)
in the construction of \(\Lambda_1'\), swap the two rows.  If
\(
        \epsilon_{-1}\ev\varepsilon(-1)=+
\)
in the construction of \(\Lambda_{-1}'\), swap the two rows.

\item Take the Alvis--Curtis dual again.
\end{enumerate}
Recall that \(\zeta({}^{\bd}\Lambda)\) denotes the sign of a row
containing a largest entry of \({}^{\bd}\Lambda\). If the two largest entries are equal, we
choose the value of \(\zeta({}^{\bd}\Lambda)\) corresponding to the row
used by the descent under consideration.

By Proposition~\ref{lem73}, once a sign pair \((\ep,\ev)\) is fixed, the sharp descent removes the largest entry from the row encoded by \(-\ep\) in
\({}^{\bd}\Lambda_1\), and from the row encoded by \(\ev\) in
\({}^{\bd}\Lambda_{-1}\).  Hence the removed entries are largest entries of
the whole symbols precisely when
\[
        (\ep,\ev)
        =
        (-\zeta({}^{\bd}\Lambda_1),\zeta({}^{\bd}\Lambda_{-1})).
\]
The next proposition proves that this is the sign pair forced by maximality.

\begin{prop}\label{lem75}
Let \(\prll\in \Irr(\Sp_{2n}(\mathbbm{k}))\), and let
\[
        \prll\xrightarrow{(\ell,a)}\prllp\to\prllpp\to\cdots
\]
be a maximal descent sequence.  Let \((\ep,\ev)\) be a sign pair for the first
arrow satisfying Proposition~\ref{lem73}.  With the tie convention above, one
has
\[
        (\ep,\ev)
        =
        (-\zeta({}^{\bd}\Lambda_1),\zeta({}^{\bd}\Lambda_{-1})).
\]
\end{prop}

\begin{proof}
We argue by induction on the length of maximal descent sequences.  Suppose,
for contradiction, that \((\ep,\ev)\) does not satisfy the conclusion of the
proposition.  We shall construct two descent sequences
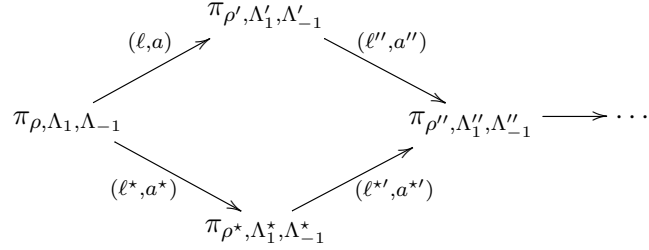
\[
\xymatrix{
& \prllp \ar[dr]^{(\ell'',a'')} & & \\
\prll \ar[ur]^{(\ell,a)} \ar[dr]_{(\ell^\star,a^\star)}
& & \prllpp \ar[r] & \cdots \\
& \prlls \ar[ur]_{(\ell^{\star\prime},a^{\star\prime})} & &
}
\]
where the upper path is the fixed maximal descent sequence.
For the bottom first arrow, take
\[
        (\ep^\star,\ev^\star)
        =
        (-\zeta({}^{\bd}\Lambda_1),\zeta({}^{\bd}\Lambda_{-1})).
\]
Let
\(
        \prll\xrightarrow{(\ell^\star,a^\star)}\prlls
\)
be the sharp descent with respect to \((\ep^\star,\ev^\star)\), in the sense
of Proposition~\ref{lem73}, taking the same \((\ne\pm1)\)-part as in
\(\prllp\).  Since the assumed equality of sign pairs fails, the upper first
arrow does not remove a largest entry of the whole symbol in at least one of
\({}^{\bd}\Lambda_1\) and \({}^{\bd}\Lambda_{-1}\), whereas the bottom first
arrow does.  Hence
\[
        \ell^\star>\ell .
\]
If the fixed maximal sequence has length one, this already contradicts
maximality.  Thus we may assume that the tail is nonempty.  The tail is again
maximal; otherwise replacing it by a larger tail would give a larger descent
sequence of \(\prll\) with the same first arrow.

It remains to prove that \(\prlls\) descends to \(\prllpp\).  The
\(\rho\)-part causes no difficulty: the conditions on the \((\ne\pm1)\)-part
in Theorem~\ref{ggp}(i) are unchanged when only the \((\pm1)\)-symbol sign
pair is changed.  We compare only the \((\pm1)\)-symbol
part.

We use the recipe following Proposition~\ref{lem73} to read off, at each
step, the row from which the largest entry is removed and the corresponding
row-switch sign.  With our sign convention, when a row chosen after the first
step is viewed in the initial symbol, its label is multiplied by the
row-switch sign of the first step.  If the same row occurs twice, the two
entries are removed successively from that row.

Since the tail
\[
        \prllp\xrightarrow{(\ell'',a'')}\prllpp\to\cdots
\]
is maximal, the induction hypothesis gives the sign pair
\[
        (-\zeta({}^{\bd}\Lambda_1'),\zeta({}^{\bd}\Lambda_{-1}'))
\]
for its first arrow.

For the upper path, the rows of the initial symbols from which entries are
removed are
\[
\begin{array}{c|c|c}
\text{initial symbol} & \text{first removed row}
& \text{second removed row, pulled back} \\ \hline
{}^{\bd}\Lambda_1
& -\ep
& \epsilon_{-1}\ev\varepsilon(-1)\zeta({}^{\bd}\Lambda_{-1}') \\[1mm]
{}^{\bd}\Lambda_{-1}
& \ev
& \epsilon_{-1}\ev\zeta({}^{\bd}\Lambda_1') .
\end{array}
\]
Here \(\epsilon_{-1}\ev\varepsilon(-1)\) and \(\epsilon_{-1}\ev\) are the
row-switch signs in the first arrow.

Let \((\alpha,\beta)\) be the sign pair for a possible second arrow from
\(\prlls\).  
For the bottom path, the corresponding rows are
\[
\begin{array}{c|c|c}
\text{initial symbol} & \text{first removed row}
& \text{second removed row, pulled back} \\ \hline
{}^{\bd}\Lambda_1
& \zeta({}^{\bd}\Lambda_1)
& \epsilon_{-1}\zeta({}^{\bd}\Lambda_{-1})\varepsilon(-1)\beta \\[1mm]
{}^{\bd}\Lambda_{-1}
& \zeta({}^{\bd}\Lambda_{-1})
& \epsilon_{-1}\zeta({}^{\bd}\Lambda_{-1})(-\alpha).
\end{array}
\]

For \({}^{\bd}\Lambda_1\), the upper path selects the two rows
\[
        -\ep,\qquad
        \epsilon_{-1}\ev\varepsilon(-1)\zeta({}^{\bd}\Lambda_{-1}'),
\]
whereas the bottom path first selects the row
\(\zeta({}^{\bd}\Lambda_1)\).  Thus, if the bottom endpoint is to be
\(\prllpp\), the second removed row, pulled back to the initial symbol, must be
the row which makes the two-row multiset agree with the two rows in the upper
table.  Hence
\[
        \epsilon_{-1}\zeta({}^{\bd}\Lambda_{-1})
        \varepsilon(-1)\beta
        =
        -\ep\zeta({}^{\bd}\Lambda_1)\,
        \epsilon_{-1}\ev\varepsilon(-1)
        \zeta({}^{\bd}\Lambda_{-1}'),
\]
and therefore
\[
        \beta
        =
        -\ep\ev\,
        \zeta({}^{\bd}\Lambda_1)
        \zeta({}^{\bd}\Lambda_{-1})
        \zeta({}^{\bd}\Lambda_{-1}').
\]

Similarly, for \({}^{\bd}\Lambda_{-1}\), the upper path selects the two rows
\[
        \ev,\qquad
        \epsilon_{-1}\ev\zeta({}^{\bd}\Lambda_1'),
\]
whereas the bottom path first selects the row
\(\zeta({}^{\bd}\Lambda_{-1})\).  Therefore the second removed row, pulled back
to the initial symbol, must be the row which makes the two-row multiset agree
with the two rows in the upper table, and we get
\[
        \epsilon_{-1}\zeta({}^{\bd}\Lambda_{-1})(-\alpha)
        =
        \ev\zeta({}^{\bd}\Lambda_{-1})\,
        \epsilon_{-1}\ev\zeta({}^{\bd}\Lambda_1').
\]
Thus
\[
        \alpha=-\zeta({}^{\bd}\Lambda_1').
\]
We therefore choose the bottom second arrow to have sign pair
\[
        (\ep^{\star\prime},\ev^{\star\prime})
        =
        \left(
        -\zeta({}^{\bd}\Lambda_1'),
        -\ep\ev\,
        \zeta({}^{\bd}\Lambda_1)
        \zeta({}^{\bd}\Lambda_{-1})
        \zeta({}^{\bd}\Lambda_{-1}')
        \right).
\]
Define a representation \(\pi^{\star\prime}\) as follows.  Its symbols
are obtained by taking the sharp choice associated with sign pair $(\ep^{\star\prime},\ev^{\star\prime})$,
and its \(\rho\)-part is the same as that of \(\prllpp\). The criterion in
Theorem~\ref{ggp}(i) then gives a descent from \(\prlls\) to  \(\pi^{\star\prime}\).  Thus \(\pi^{\star\prime}\) is the endpoint of
the lower path, and we now prove that
\(\pi^{\star\prime}=\prllpp\). By construction, the upper and lower
paths remove the same entries from both initial symbols.  Their row-switch
signs also agree.  Indeed, for the symbols coming from
\({}^{\bd}\Lambda_1\) and \({}^{\bd}\Lambda_{-1}\), the total row-switch
signs along the upper path are, respectively,
\[
        \epsilon_{-1}\ev\varepsilon(-1)\cdot
        (-\ep)\zeta({}^{\bd}\Lambda_{-1}')
\quad\text{and}\quad
        \epsilon_{-1}\ev\cdot
        (-\ep)\zeta({}^{\bd}\Lambda_{-1}')\varepsilon(-1).
\]
Both are equal to
\[
        \epsilon_{-1}\ev\cdot
        (-\ep)\zeta({}^{\bd}\Lambda_{-1}')\varepsilon(-1).
\]
The corresponding total row-switch signs along the bottom path are,
respectively,
\[
        \epsilon_{-1}\zeta({}^{\bd}\Lambda_{-1})\varepsilon(-1)
        \cdot \zeta({}^{\bd}\Lambda_1)\ev^{\star\prime}
\quad\text{and}\quad
        \epsilon_{-1}\zeta({}^{\bd}\Lambda_{-1})
        \cdot \zeta({}^{\bd}\Lambda_1)\ev^{\star\prime}\varepsilon(-1).
\]
Both are equal to
\[
        \epsilon_{-1}\zeta({}^{\bd}\Lambda_{-1})\cdot
        \zeta({}^{\bd}\Lambda_1)\ev^{\star\prime}\varepsilon(-1),
\]
which, by the definition of \(\ev^{\star\prime}\), is the same as the upper
total row-switch sign.

The two defect equalities in Theorem~\ref{ggp}(i) depend only on the
initial defects and the corresponding row-switch signs.  Since the total
row-switch signs along the two paths have just been shown to agree, the
defects at the lower endpoint are equal to those of
\(\Lambda_1''\) and \(\Lambda_{-1}''\). Hence the
bottom endpoint $\pi^{\star\prime}$ has the same \(\rho\)-part, the same removed entries, the
same row-switch signs, and the same defects as \(\prllpp\).  Therefore \(\pi^{\star\prime}=\prllpp\). By the preceding sign comparison, the products of
the two descent signs along the two paths are equal. Together with
\(\ell^\star>\ell\), this shows that the lower two-step index is obtained
from the upper one by a positive number of elementary sequence
mutations, contradicting maximality. The proposition follows.
\end{proof}

We next treat the \((\ne\pm1)\)-part.  The following proposition is the
corresponding sharpness statement for the \(\rho\)-part.

\begin{prop}\label{lem74}
Let \(\prll\in \Irr(\Sp_{2n}(\mathbbm{k}))\), and let
\[
        \prll\xrightarrow{(\ell,a)}\prllp\to\prllpp\to\cdots
\]
be a maximal descent sequence.  Write
\[
        \rho=\prod_{[b]}\pi[b],
        \qquad
        \rho'=\prod_{[b]}\pi'[b],
\]
where \([b]\) runs over the classes occurring in the \((\ne\pm1)\)-part, and
\(\pi[b]\) and \(\pi'[b]\) correspond to the partitions
\(\lambda[b]\) and \(\lambda'[b]\), respectively.  Then, for every such class
\([b]\),
\[
        \lambda'[-b]=\lambda[b]_- .
\]
\end{prop}

\begin{proof}
The proof is the same maximality argument as in Proposition~\ref{lem73}.
Suppose that
\[
        \lambda'[-b]\ne \lambda[b]_-
\]
for some class \([b]\) in the \((\ne\pm1)\)-part.  By the
\((\ne\pm1)\)-part of Theorem~\ref{ggp}(i), \(\lambda'[-b]\) satisfies the
required condition relative to \(\lambda[b]\), but it is not the sharp choice
\(\lambda[b]_-\).  Thus the target \(\rho\)-part of
\(\prll\to\prllp\) is not minimal.

Replace only the \([-b]\)-part of \(\rho'\) by \(\lambda[b]_-\), and keep
\(\Lambda_1'\), \(\Lambda_{-1}'\), and all other \((\ne\pm1)\)-parts
unchanged.  Denote the new \(\rho\)-part by \(\rho^\star\), and let the
resulting representation be \(\prlls\).  Theorem~\ref{ggp}(i) gives a
descent from \(\prll\) to \(\prlls\).  Since its target \(\rho\)-part is strictly smaller than that of
\(\prllp\), the index of this descent is strictly larger than \(\ell\).

It remains to check that \(\prllpp\) is still obtained from \(\prlls\) by a descent.  Let \(\rho''\) be the \(\rho\)-part of \(\prllpp\), and let \(\lambda''[b]\) be its \([b]\)-part.  By the induction hypothesis,
\[
\lambda''[b]=\lambda'[-b]_-,
\]
while, by construction,
\[
\lambda^\star[-b]=\lambda[b]_-.
\]
Since \(\prllp\) occurs in a descent of \(\prll\), the partitions
\(\lambda'[-b]\) and \(\lambda[b]\) are close or \(2\)-transverse.
Deleting the first column from both preserves this condition, so
\(\lambda''[b]\) and \(\lambda^\star[-b]\) satisfy the required condition
in Theorem~\ref{ggp}(i).  All other partition and symbol data are unchanged;
hence \(\prllpp\) occurs in a descent of \(\prlls\).

The replacement changes only the intermediate \(\rho\)-signs.  Since their
total contribution along the two successive descents depends only on the
endpoints \(\rho\) and \(\rho''\), the product of the two descent signs is
unchanged.  The modified path therefore has the same endpoint and sign
product, but a strictly larger first index, contradicting maximality.
\end{proof}

Combining Propositions~\ref{lem73}, \ref{lem75}, and \ref{lem74}, the first
arrow in any maximal descent sequence is a first descent.  Since the tail of a
maximal descent sequence is again maximal, induction shows that every maximal
descent sequence is obtained by successive first descents.

It remains to prove the independence of the numerical sequence.  By
Corollary~\ref{fdi}, the first number \(2\ell_1\) is fixed.  We next
show that \(2\ell_2\) is fixed.   By the
first-descent recipe in Section~\ref{sec6.1}, all irreducible
constituents of the first descent have the same \(\rho'\)-part.  The
only possible ambiguity in their symbol parts occurs when the largest
entries in the two rows of \({}^{\bd}\Lambda_1\), or those in the two
rows of \({}^{\bd}\Lambda_{-1}\), are equal: in that case either of the
two equal entries may be removed.  Whichever one is removed in the first
descent, the other remains and is removed in the next first descent.
The symbol contributions to $\ell_2$ depend
only on the values of the removed entries, and not on the rows containing
them.  Hence the symbol contributions to the second first-descent index
are the same for all constituents.  Corollary~\ref{fdi} therefore shows
that \(2\ell_2\) is fixed.  Repeating this argument after each successive
first descent proves that
\(
(2\ell_1,2\ell_2,\ldots)
\)
is independent of the maximal descent sequence.  This proves
Theorem~\ref{ds}(i).

\subsection{Proof of Theorem \ref{ds} (ii)} \label{sec7.3}

By Theorem~\ref{ds}(i), every maximal descent sequence is obtained by
successive first descents.  We prove that its index is good.  Recall from
Definition~\ref{defn:good} that this has two parts: a numerical condition on
the sequence \(\ell_i\), and a sign condition in the case
\(\ell_i=\ell_{i+1}-1\).  We first check the numerical condition.

By Corollary~\ref{fdi},
\(
\ell_i=\ell[\ne\pm1]_i+\ell[1]_i+\ell[-1]_i.
\)
Moreover, the explicit formula for the \((\ne\pm1)\)-part and
Corollary~\ref{fdi-comp} give
\[
\ell[\ne\pm1]_i\geqslant\ell[\ne\pm1]_{i+1},\qquad
\ell[-1]_i\geqslant\ell[1]_{i+1},\qquad
\ell[1]_i\geqslant\ell[-1]_{i+1}-1.
\]
Hence \(\ell_i\geqslant\ell_{i+1}-1\).  If equality holds, then all
three inequalities above are equalities.  By Corollary~\ref{fdi-comp},
the largest entries in the two rows of both dual symbols before the
\(i\)-th descent are equal.  After one entry in each pair is removed,
both dual symbols have a unique largest entry, so
Corollary~\ref{fdi-comp} applied to the next pair gives
\[
\ell_{i+1}>\ell_{i+2}.
\]
If \(i>1\), for both dual symbols before the \(i\)-th descent to have
tied largest entries, both dual symbols before the preceding descent
must have unique largest entries.  Applying
Corollary~\ref{fdi-comp} to the preceding pair gives
\[
\ell_{i-1}>\ell_i.
\]
For \(i=1\), this follows from the convention
\(\ell_0=+\infty\).  This proves the numerical condition in
Definition~\ref{defn:good}.

It remains to verify the sign condition.  Since every tail of a maximal
descent sequence is again maximal, it is enough to consider \(i=1\).
Let
\[
\prll\xrightarrow{(\ell_1,a_1)}\prllp
\xrightarrow{(\ell_2,a_2)}\prllpp
\]
and assume \(\ell_1=\ell_2-1\).  Then
\[
\ell[\ne\pm1]_1=\ell[\ne\pm1]_2,\qquad
\ell[-1]_1=\ell[1]_2,\qquad
\ell[1]_1=\ell[-1]_2-1,
\]
and the largest entries in the two rows of both
\({}^{\bd}\Lambda_1\) and \({}^{\bd}\Lambda_{-1}\) are equal.
We now compare the signs \(a_1\) and \(a_2\).  Recalling the definition of
\(\zeta\), in this tied case \(\zeta({}^{\bd}\Lambda)\) may be chosen to
be either row sign.  We fix choices
\[
        s_1=\zeta({}^{\bd}\Lambda_1),
        \qquad
        s_{-1}=\zeta({}^{\bd}\Lambda_{-1}).
\]
Thus the sign pair of the first descent is
\(
(\ep,\ev)=(-s_1,s_{-1}).
\)
 Once the first step is fixed, the next
first descent removes the largest remaining entries, namely those in the
opposite rows \(-s_1\) and \(-s_{-1}\).

By the sign formula in Theorem~\ref{ggp}(ii), the first two descent signs are
\[
        a_1
        =
        -s_1s_{-1}\epsilon_\rho\epsilon_{\rho^-_1},
        \qquad
        a_2
        =
        -\zeta({}^{\bd}\Lambda'_1)\zeta({}^{\bd}\Lambda'_{-1})
        \epsilon_{\rho'}\epsilon_{(\rho')^-_1}.
\]
The \((\ne\pm1)\)-part contributes equally to \(a_1\) and \(a_2\).
Indeed, by the explicit formula in Corollary~\ref{fdi}, after reindexing
the classes one has
\[
\ell[\ne\pm1]_1
=
\sum_{[a]\ne[\pm1]}
\widetilde{\#[a]}\lambda[a]^t_1,
\qquad
\ell[\ne\pm1]_2
=
\sum_{[a]\ne[\pm1]}
\widetilde{\#[a]}\lambda[a]^t_2.
\]
Since
\(\lambda[a]^t_1\geqslant\lambda[a]^t_2\) for every \([a]\), the equality
\(
\ell[\ne\pm1]_1=\ell[\ne\pm1]_2
\)
forces equality factor by factor.  Hence the corresponding
\(\rho\)-sign factors in the formulas for \(a_1\) and \(a_2\) are equal.

It remains to compare the \((\pm1)\)-symbol contributions.  By the definition of $s_{\pm 1}$, we have
\[
\begin{aligned}
        \zeta({}^{\bd}\Lambda'_1)
        &=
        (-s_{-1})\,\ev\epsilon_{-1},\\
        \zeta({}^{\bd}\Lambda'_{-1})
        &=
        (-s_1)\,\varepsilon(-1)\ev\epsilon_{-1}.
\end{aligned}
\]
Here the extra factors are the row-switch signs: \(\ev\epsilon_{-1}\) for the
\(-1\)-symbol becoming the \(1\)-symbol, and
\(\varepsilon(-1)\ev\epsilon_{-1}\) for the \(1\)-symbol becoming the
\(-1\)-symbol.
Therefore
\[
        \zeta({}^{\bd}\Lambda'_1)
        \zeta({}^{\bd}\Lambda'_{-1})
        =
        \varepsilon(-1)s_1s_{-1}.
\]
The sign formula in
Theorem~\ref{ggp}(ii) gives
\[
        a_1=\varepsilon(-1)a_2 .
\]
Changing the choices of \(s_1\) and \(s_{-1}\) may
change the individual signs \(a_1\) and \(a_2\), but the above ratio is
unchanged. Thus the adjacent pair satisfies the sign condition in
Definition~\ref{defn:good}, independently of these choices. Together with the numerical condition proved
above, this shows that every maximal descent sequence index
is good.  This proves Theorem~\ref{ds}(ii).

\subsection{Proof of Theorem \ref{ds} (iii)} \label{sec7.4}

We prove Theorem~\ref{ds}(iii), namely that if
\[
        \Gamma=\llb(2\ell_1,a_1),(2\ell_2,a_2),\cdots\rrb
        \in\mathcal D(\prll)
\]
and
\[
        \Gamma'=\llb(2\ell_1',a_1'),(2\ell_2',a_2'),\cdots\rrb
        \leqslant \Gamma ,
\]
then \(\Gamma'\in\mathcal D(\prll)\).
By the definition of the order on descent sequence indexes in
Section~\ref{sec4.1}, it suffices to treat one elementary mutation.  Thus we
may assume that \(\Gamma\) and \(\Gamma'\) differ only at one adjacent pair
\(i,i+1\), with
\begin{itemize}
    \item $\ell_{i}=\ell'_{i}+1$;
    \item $\ell_{i+1}=\ell'_{i+1}-1$;
    \item $a_i'a_{i+1}'=a_ia_{i+1}$.
\end{itemize}
It is enough to treat the case \(i=1\).  The general case follows by applying
the same argument to the tail of the descent sequence starting at the
\((i-1)\)-st representation and then splicing it with the preceding arrows.

\begin{lem}\label{lem7.7}
    Let $\prll\in \Irr(\Sp_{2n}(\mathbbm{k}))$. Assume that there is a descent sequence 
    \[
    \prll\to\prllp\to\prllpp\to\cdots
    \]
    with descent sequence index $\Gamma=\llb(2\ell_1,a_1),(2\ell_2,a_2),\cdots\rrb$. Consider a descent sequence index $\Gamma'=\llb(2\ell_1',a_1'),(2\ell_2',a_2'),\cdots\rrb$ with
    \begin{itemize}
        \item $\ell_1'=\ell_1-1$,
        \item $\ell_2'=\ell_2+1$,
        \item $a_1'a_2'=a_1a_2$,
        \item \((\ell_i',a_i')=(\ell_i,a_i)\) for \(i>2\).
    \end{itemize}
    Then there is a descent sequence of $\prll$ with index $\Gamma'$.
\end{lem}
 \begin{proof}
 If \(n=1\), the assertion follows by direct verification.  We henceforth
assume that \(n\geqslant2\).

Choose semisimple elements \(s\), \(s'\), and \(s''\) such that
\[
\prll\in\mathcal E(\Sp_{2n}(\mathbbm{k}),s),\qquad
\prllp\in\mathcal E(\Sp_{2n'}(\mathbbm{k}),s'),\qquad
\prllpp\in\mathcal E(\Sp_{2n''}(\mathbbm{k}),s'').
\]
Since \(n\geqslant2\) and our assumption of characteristic \(p>3(h_G-1)\), we can choose
\[
s^+\in\GL_1(\mathbbm{k})=\mathbbm{k}^{\times},
\qquad
s^-\in\mathrm U_1(\mathbbm{k})
\subset\overline{\mathbbm{k}}^{\times}
\]
such that
\begin{itemize}
    \item \(s^\pm\) is not an eigenvalue of any of
    \(-s\), \(s'\), and \(-s''\);
    \item \(s^\pm\ne\pm1\).
\end{itemize}
Let \(\chi^+\) and \(\chi^-\) be the corresponding characters of
\(\GL_1(\mathbbm{k})\) and \(\mathrm U_1(\mathbbm{k})\), respectively.
Let
\begin{align*}
& \pi_{\rho^+,\Lambda'_1,\Lambda'_{-1}}=R^{\Sp_{2(n-\ell_1+1)}}_{\GL_1\times\Sp_{2(n- \ell_1)}}(\chi^+\otimes\prllp)\\
& (\textrm{resp. }\pi_{\rho^-,\Lambda'_1,\Lambda'_{-1}}=R^{\Sp_{2(n-\ell_1+1)}}_{\RU_1\times\Sp_{2(n-\ell_1)}}(\chi^-\otimes\prllp)).
\end{align*}
Since the new semisimple eigenvalue \(s^\pm\) is disjoint from the eigenvalues of \(s'\), the corresponding Deligne--Lusztig induction is irreducible. Here \(\rho^+\) (resp. \(\rho^-\)) is the
\((\ne\pm1)\)-parameter whose \(s^+\)-component
(resp. \(s^-\)-component) is \(\chi^+\) (resp. \(\chi^-\))
and whose remaining components agree with those of \(\rho'\). For \(\epsilon=\pm1\), write \(\rho^\epsilon\) for \(\rho^+\) or
\(\rho^-\), respectively. Then
$
 \epsilon_{\rho^\epsilon}=\epsilon \epsilon_{\rho'}.
 $

For the arrow $\prll\to\pi_{\rho^\pm,\Lambda'_1,\Lambda'_{-1}}$, the \([-s^\pm]\)-part of \(\prll\) is
\(\varnothing\), while the \([s^\pm]\)-part of
\(\pi_{\rho^\pm,\Lambda'_1,\Lambda'_{-1}}\) is \([1]\).  For the
arrow $\pi_{\rho^\pm,\Lambda'_1,\Lambda'_{-1}} \to \prllpp$, these parts are \([1]\) and \(\varnothing\),
respectively.  They satisfy both the close and the required
\(2\)-transversality conditions.  All the remaining partition and
symbol data are unchanged.  Hence
Theorem~\ref{ggp}(i) gives
 \[
\Hom_{\Sp_{2(n-\ell_1+1)}}(\pi_{\rho^\epsilon,\Lambda'_1,\Lambda'_{-1}},\CD^{\rm FJ}_{2(\ell_1-1), \epsilon a_1}(\prll))\ne 0
 \]
 and
  \[
\Hom_{\Sp_{2(n-\ell_1-\ell_2)}}(\prllpp,\CD^{\rm FJ}_{2(\ell_2+1), \epsilon a_2}( \pi_{\rho^\epsilon,\Lambda'_1,\Lambda'_{-1}}))\ne 0.
 \]
It follows that, for each \(\epsilon\in\{\pm1\}\), there is a descent sequence
\[
        \prll
        \xrightarrow{(\ell_1-1,\epsilon a_1)}
        \pi_{\rho^\epsilon,\Lambda'_1,\Lambda'_{-1}}
        \xrightarrow{(\ell_2+1,\epsilon a_2)}
        \prllpp
        \to\cdots .
\]
Here the tail after \(\prllpp\) is the same as in the original descent
sequence. Then
 the descent sequence above has index \(\Gamma'\), and therefore \(\Gamma'\in\mathcal D(\prll)\).
 \end{proof}

The lemma proves the elementary mutation case.  Since the order on descent sequence indexes is generated by elementary mutations, iterating the lemma
proves Theorem~\ref{ds}(iii).

\subsection{Maximal descent sequences and Whittaker dimensions}

\begin{thm}\label{max}
Let \(\prll\in \Irr(\Sp_{2n}(\mathbbm{k}))\).  Assume that
\(
        \rho=\prod_{[a]\ne[\pm1]}\pi[a],
\)
where \(\pi[a]\) corresponds to the partition
\(
        \lambda[a]=[\lambda[a]_1,\lambda[a]_2,\ldots,\lambda[a]_{k[a]}].
\) 
\begin{enumerate}
    \item[(i)]  Let $\Gamma=\llb(2\ell_1,a_1),(2\ell_2,a_2),\cdots,(2\ell_k,a_k)\rrb$ be a maximal descent sequence index of $\prll$. For each \(i\), let
\(\ell[\ne\pm1]_i,\ell[1]_i,\ell[-1]_i\) denote the three contributions
to the \(i\)-th first-descent index appearing in
Corollary~\ref{fdi}. Recall from \eqref{eq:bd-iteration} that
\(
{}^{\bd}\Lambda_{\pm1}^{\,i}
:=({}^{\bd}\Lambda_{\pm1})^i.
\)
    We have
    \begin{itemize}
    \item $\ell_i=\ell[\ne \pm 1]_i+\ell[1]_i+\ell[-1]_i$;
        \item $\ell[\ne \pm1]_i=\sum_{[a]\ne[\pm1]}\widetilde{\#[a]}\lambda[a]^t_i$;
        \item $ \ell[(-1)^{i-1}]_i=\rank({}^{\bd}\Lambda_1^{i-1})-\rank({}^{\bd}\Lambda_1^{i})$;
        \item $\ell[(-1)^{i}]_{i}=\rank({}^{\bd}\Lambda_{-1}^{i-1})-\rank({}^{\bd}\Lambda_{-1}^{i})$;
        \item  $a_i=-\epsilon_{\rho_{i-1}}\epsilon_{\rho_{i}}\zeta({}^{\bd}\Lambda_1^{i-1})\zeta({}^{\bd}\Lambda_{-1}^{i-1})\varepsilon(-1)^{i-1}$;
        \item $k=\max\left\{\max_{[a]\ne[\pm1]}k[a], |{}^{\bd}\Lambda_1|,
|{}^{\bd}\Lambda_{-1}|\right\}$.
    \end{itemize}
    
    \item[(ii)] We have
    \[
    \ScO(\prll)^{\max}=\bigcup_{\textrm{maximal descent sequence indexes }\Gamma'\textrm{ of }\prll}\CO(\Gamma').
    \]
Moreover, \(\ScO(\prll)^{\max,\mathrm{st}}\) is contained in a unique
\(F\)-stable nilpotent orbit, whose partition
\(\lambda=[\lambda_1,\ldots,\lambda_k]\) is attached to any maximal
descent sequence index
\(
\Gamma'
=
\llb(2\ell_1,a_1),\ldots,(2\ell_k,a_k)\rrb,
\)
with \(\lambda_i\) corresponding to \(\ell_i\).

\item[(iii)] 
Define
\[
\mathtt{N}:=\Set{1\leqslant i<k | \begin{aligned}
& \ell[\ne\pm 1]_{i}=\ell[\ne\pm 1]_{i+1},\\
& \ell[ \epsilon]_{i}=\ell[-\epsilon]_{i+1}-1, \ \ell[ -\epsilon]_{i}=\ell[\epsilon]_{i+1} \textrm{ for some }\epsilon \in \{\pm 1\}
\end{aligned}}.
\]
Then \(\lambda_i\) is odd precisely when either \(i\) or \(i-1\) belongs to
\(\mathtt N\).

\item[(iv)] Define
\[
\mathtt{M}:=\Set{1\leqslant i<k | \begin{aligned}
& \ell[\ne\pm 1]_{i}>\ell[\ne\pm 1]_{i+1},\\
& \ell[ \epsilon]_{i}=\ell[-\epsilon]_{i+1}-1, \ \ell[ -\epsilon]_{i}=\ell[\epsilon]_{i+1} \textrm{ for some }\epsilon \in \{\pm 1\}
\end{aligned}}.
\]
Then for any
\(\CO\in\ScO(\prll)^{\rm max}\), we have
\[
\dim {\mathrm{Wh}}_{\CO}(\prll)=2^{\#\mathtt{M}}.
\]
In particular, the dimension above equals 1 for quadratic unipotent representations ($\ell[\ne\pm 1]_{i}\equiv 0$).
\end{enumerate}
\end{thm}
\begin{proof}
Choose a maximal descent sequence with index \(\Gamma\),
\[
\prll=\pi_0\to\pi_1\to\cdots\to\pi_k=\mathbbm{1},
\qquad
\pi_j=\pi_{\rho_j,\Lambda_{1,j},\Lambda_{-1,j}}.
\]
By Theorem~\ref{ds}\textup{(i)}, every arrow is a first descent.  At
each step, the first column is removed from every partition in the
\((\ne\pm1)\)-part, so the \(i\)-th step gives the stated formula for
\(\ell[\ne\pm1]_i\). For the \(\pm1\)-part, the first-descent recipe in
Section~\ref{sec6.1} constructs \(\Lambda_{\epsilon,j}\) from
\(\Lambda_{-\epsilon,j-1}\) for each
\(\epsilon\in\{\pm1\}\). Hence
\({}^{\bd}\Lambda_1^{i-1}\) and
\({}^{\bd}\Lambda_{-1}^{i-1}\), possibly with their two rows
interchanged, are respectively
\({}^{\bd}\Lambda_{(-1)^{i-1},i-1}\) and
\({}^{\bd}\Lambda_{(-1)^i,i-1}\).  Since interchanging rows does not
change rank, Corollary~\ref{fdi} gives the stated formulas for the
symbol contributions and
\[
\ell_i=\ell[\ne\pm1]_i+\ell[1]_i+\ell[-1]_i,
\]
which prove the $\ell$ part of (i).

For the rational sign, Theorem~\ref{ggp}\textup{(ii)} gives
\[
a_i=-\epsilon_{\rho_{i-1}}\epsilon_{\rho_i}
\zeta({}^{\bd}\Lambda_{1,i-1})
\zeta({}^{\bd}\Lambda_{-1,i-1}).
\]
For the \(j\)-th preceding descent, set
\(
(\ev)_j=\zeta({}^{\bd}\Lambda_{-1,j}),
\)
and let \(\epsilon_{-1,j}\) be determined by the representation
\(\pi_{\Lambda_{-1,j}}\), as in Theorem~\ref{ggp}.
By the first-descent recipe in Section~\ref{sec6.1},
the rows are interchanged in constructing
\(\Lambda_{1,j+1}\) when \(\epsilon_{-1,j}(\ev)_j=+\), and in
constructing \(\Lambda_{-1,j+1}\) when
\(\epsilon_{-1,j}(\ev)_j\varepsilon(-1)=+\).  Since interchanging rows
changes the sign of \(\zeta\), the product of the corresponding changes in the two
\(\zeta\)-values is
\[
\bigl(-\epsilon_{-1,j}(\ev)_j\bigr)
\bigl(-\epsilon_{-1,j}(\ev)_j\varepsilon(-1)\bigr)
=\varepsilon(-1).
\]
Iterating gives
\[
a_i=-\epsilon_{\rho_{i-1}}\epsilon_{\rho_i}
\zeta({}^{\bd}\Lambda_1^{i-1})
\zeta({}^{\bd}\Lambda_{-1}^{i-1})
\varepsilon(-1)^{i-1}.
\]
The formula for \(k\) follows because the descent terminates precisely
when all partition columns and all entries of the two dual symbols
have been removed. This proves (i).

By Theorem~\ref{ds}\textup{(ii)}, Corollary~\ref{od}, and
Lemma~\ref{lem:orbit-sequence-order}, the rational orbits attached to
maximal descent sequence indexes are exactly the rational
orbits in \(\ScO(\prll)^{\rm max}\).  This gives the asserted equality in
\textup{(ii)}.
By Theorem~\ref{ds}\textup{(i)}, all maximal descent sequence indexes
have the same numerical sequence.  Moreover, the formulas in
\textup{(i)} show that the three contributions
\(\ell[\ne\pm1]_i,\ell[1]_i,\ell[-1]_i\) are independent of the
choice of maximal descent sequence index.  Hence the corresponding
rational orbits have the same partition
\(\lambda=[\lambda_1,\ldots,\lambda_k]\), where \(\lambda_i\)
corresponds to \(\ell_i\), and lie in a unique \(F\)-stable nilpotent
orbit.  This proves \textup{(ii)}.

By Corollary~\ref{fdi-comp} and \textup{(i)}, 
\(i\in\mathtt N\) if and only if
\(\ell_i=\ell_{i+1}-1\).
Since \(\Gamma\) is good, each \(i\in\mathtt N\) gives the two equal
odd parts
\(
\lambda_i=\lambda_{i+1}=2\ell_i+1.
\)
If neither \(j-1\) nor \(j\) belongs to \(\mathtt N\), then
\(\lambda_j=2\ell_j\), which is even.  This proves (iii).

For (iv), we argue by induction along a maximal descent sequence. Fix \(\CO\in\ScO(\prll)^{\rm max}\) and a maximal descent sequence
\[
\prll=\pi_0\to\pi_1\to\cdots\to\pi_k=\mathbbm{1}
\]
with index \(\Gamma_{\CO}\) such that
\(\CO=\CO(\Gamma_{\CO})\).  
For \(r\geq1\), let \(\CO_{\geq r}\) be the rational nilpotent orbit
corresponding, via the descent-orbit dictionary in Section~\ref{sec4.1}, to
the tail index
\[
        \llb(2\ell_r,a_r),(2\ell_{r+1},a_{r+1}),\cdots\rrb .
\]
Thus \(\CO=\CO_{\geq1}\). By maximality of \(\Gamma_{\CO}\), the tail beginning at \(r\) is a
maximal descent sequence index of \(\pi_{r-1}\).

We first consider the case where
\(\CD^{\rm FJ}_{2\ell_1,a_1}(\prll)\) has only one irreducible
constituent, denoted by \(\pi'=\pi_1\). By Corollary~\ref{5.6},
\(1\notin\mathtt N\cup\mathtt M\).  Hence, by \textup{(iii)}, the first
part of the partition attached to \(\CO_{\geq1}\) is even, namely
\(2\ell_1\).  Since \(\Gamma_{\CO}\) is good and
\(1\notin\mathtt N\), we have \(\ell_1\geqslant\ell_2\), and therefore
\[
\CO_{\geq1}=\CO_{2\ell_1,a_1}*\CO_{\geq2}.
\]
By Lemma~\ref{erL3},
\[
        \dim \mathrm{Wh}_{\CO_{\geq1}}(\prll)
        =
        \dim \mathrm{Wh}_{\CO_{\geq2}}(\pi').
\]
Thus this step contributes no extra factor 2.

Now suppose that
\(\CD^{\rm FJ}_{2\ell_1,a_1}(\prll)\)
has two irreducible constituents, denoted by
\(\pi'=\pi_1\) and \(\pi''\).  By
Corollary~\ref{5.6}(i), the largest entries in the two rows of both
\({}^{\bd}\Lambda_1\) and \({}^{\bd}\Lambda_{-1}\) are equal.  Hence, by Corollary~\ref{fdi-comp} and
Theorem~\ref{max}\textup{(i)}, there exists \(\epsilon=\pm1\) such that
\[
        \ell[\epsilon]_1=\ell[-\epsilon]_2-1,
        \qquad
        \ell[-\epsilon]_1=\ell[\epsilon]_2 .
\]

Assume first that
\[
        \ell[\ne\pm1]_1=\ell[\ne\pm1]_2 .
\]
Then \(1\in\mathtt N\).  By (iii), the first two parts of the partition are
odd and equal:
\[
        \lambda_1=\lambda_2=2\ell_1+1 .
\]
By Corollary~\ref{5.6} (i), the first descents of \(\pi'\) and \(\pi''\) have a
common irreducible constituent, denoted by
\(\pi'''=\pi_2\).  Using
Theorem~\ref{thm5.1}, we get
\[
\begin{aligned}
\dim \mathrm{Wh}_{\CO_{\geq1}}(\prll)
&=
\dim \mathrm{Wh}_{\CO_{(2\ell_1+1)^2}*\CO_{\geq3}}(\prll)\\
&=
\frac{1}{2}
\dim \mathrm{Wh}_{\CO_{2\ell_1+2,\varepsilon(-1)a_1}*\CO_{\geq3}}
\left(
        \mathrm{Wh}_{\CO_{2\ell_1,a_1}}(\prll)
\right)\\
&=
\frac{1}{2}
\dim \mathrm{Wh}_{\CO_{2\ell_1+2,\varepsilon(-1)a_1}*\CO_{\geq3}}
\left(
        \pi'\oplus\pi''
\right)\\
&=
\dim \mathrm{Wh}_{\CO_{\geq3}}(\pi''').
\end{aligned}
\]
Here the two summands \(\pi'\) and \(\pi''\) give the same contribution to the
common constituent \(\pi'''\); the resulting factor \(2\) is exactly cancelled
by the factor \(\frac12\) in the odd composition law. Thus this step contributes no extra factor.

It remains to consider the case
\[
        \ell[\ne\pm1]_1>\ell[\ne\pm1]_2 .
\]
Then \(1\notin\mathtt N\), so by (iii) the first part of the partition
attached to \(\CO_{\geq1}\) is even.  Thus
\[
        \CO_{\geq1}=\CO_{2\ell_1,a_1}*\CO_{\geq2}.
\]
By the branch symmetry in Corollary~\ref{5.6}\textup{(i)},
\(\pi'\) and \(\pi''\) have the same remaining descent data.  Hence
\[
\dim \mathrm{Wh}_{\CO_{\geq2}}(\pi')
=
\dim \mathrm{Wh}_{\CO_{\geq2}}(\pi'').
\] 
Using Lemma~\ref{erL3}, we obtain
\[
\begin{aligned}
\dim \mathrm{Wh}_{\CO_{\geq1}}(\prll)
&=
\dim \mathrm{Wh}_{\CO_{\geq2}}
\left(
        \mathrm{Wh}_{\CO_{2\ell_1,a_1}}(\prll)
\right)\\
&=
\dim \mathrm{Wh}_{\CO_{\geq2}}(\pi'\oplus\pi'')\\
&=
2\dim \mathrm{Wh}_{\CO_{\geq2}}(\pi').
\end{aligned}
\]
In this case there is no factor \(\frac12\) from the odd composition law, so
the factor \(2\) coming from the two irreducible constituents remains.  This
case is precisely the condition \(1\in\mathtt M\).

Repeating the same argument along the maximal descent sequence, each index in
\(\mathtt M\) contributes one factor \(2\), while all other steps contribute
no extra factor. This proves (iv).

\end{proof}

\begin{exmp}
Let \(\mathbbm{1}\) be the trivial representation of
\(\Sp_2(\mathbbm{k})\).  Then
\(
        \ScO(\mathbbm{1})=\{0\}.
\)
Let
\(
        \sigma=R^{\Sp_2}_{\GL_1}\chi,
\)
where \(\chi\) is neither trivial nor quadratic.  Then \(\sigma\) is generic,
\(
        \ScO(\sigma)^{\max}=\{[(2,\pm)]\},
\)
and each of its Whittaker models has dimension one.
Now put
\(
        \pi
        :=
        R^{\Sp_4}_{\GL_1\times \Sp_2}
        (\chi\otimes\mathbbm{1}).
\) 
Then $\pi$ is irreducible and
\(
        \ScO(\pi)^{\max}=\{[(2^2,\pm)]\}.
\)
Moreover, for either sign \(\eta=\pm\) and \(\CO=[(2^2,\eta)]\), one has
\[
        \dim \mathrm{Wh}_{\CO}(\pi)=2.
\]
\end{exmp}

\subsection{Whittaker dimensions and canonical quotients}

The power-of-two formula for Whittaker dimensions suggests the
following relation with Lusztig's canonical quotient.

\begin{thm}\label{dimW}
Let \(G=\Sp_{2n}\), \(\pi\in\mathcal E(G,s)\), \(\CO\in\SOM\), and
\(\SCO=\SOMS\).  Let
\(
        \pi^*=\mathcal L^{\rm can}(\pi)\in\mathcal E(C_{G^*}(s),1)\)
        and
        \(\SCO^*=\mathscr O(\pi^*)^{\max,\mathrm{st}}.
\)
Then
\[
\dim \mathrm{Wh}_{\CO}(\pi)={\frac{|A(\SCO)|}{|\SOM||\overline{A}(\SCO^*)|}},
\]
where $\overline{A}(\SCO^*)$ is Lusztig's canonical quotient of $\SCO^*$ in  $C_{G^*}(s)$.
\end{thm}
To prove the theorem, we need the following lemma.

\begin{lem}\label{unm}
Let \(\pi\) be an irreducible quadratic unipotent representation.  With the
notation of Theorem~\ref{dimW}, for every \(\CO\in\SOM\) one has
\[
        \frac{|A(\SCO)|}
        {|\SOM|\,|\overline A(\SCO^*)|}
        =
        1 .
\]
\end{lem}

\begin{proof}
The statement is clear in rank \(0\) and for the trivial
representation.  We argue by induction on the rank in the remaining
cases.
Write
\[
        \pi=\pi_{\Lambda_1,\Lambda_{-1}}.
\]
Let \(a_1\ge a_2\) and \(b_1\ge b_2\) be the two largest entries in
\({}^{\bd}\Lambda_1\) and \({}^{\bd}\Lambda_{-1}\), respectively. If one of the symbols has fewer than two entries, we append entries equal to 
\(-\infty\)'s when defining \(a_2\) or \(b_2\). If the
two largest entries of a symbol are equal, we denote by \(a_1\)
(resp. \(b_1\)) the one removed by the first descent.  The remaining one is
denoted by \(a_2\) (resp. \(b_2\)). 

Let
\(
        \Gamma=\llb(2\ell_1,\alpha_1),(2\ell_2,\alpha_2),\ldots\rrb
\)
be a maximal descent sequence index corresponding to \(\CO\).
Let
\(\pi'=\pi_{\Lambda'_1,\Lambda'_{-1}}\)
be the irreducible constituent occurring after the first step of the
maximal descent sequence with index \(\Gamma\). Since \(\pi\) is nontrivial, Proposition~\ref{trivial}\textup{(i)}
implies that its first descent index is positive. Hence $\pi'\in {\rm Irr}(\Sp_{2n'}(\mathbbm{k}))$ with $n'<n$. Put
\[
        \SCOP=\ScO(\pi')^{\max,\mathrm{st}},
        \qquad
        \SCOPS=\mathscr O(\mathcal L^{\rm can}(\pi'))^{\max,\mathrm{st}}.
\]
We compare
\[
        \frac{|A(\SCO)|}
        {|\SOM|\,|\overline A(\SCO^*)|}
        \quad\text{with}\quad
        \frac{|A(\SCOP)|}
        {|\ScO(\pi')^{\max}|\,|\overline A(\SCOPS)|}.
\]

First assume that
\[
        a_1>a_2,\qquad b_1>b_2 .
\]
By Corollary~\ref{5.6}, the first descent of \(\pi\) has a unique irreducible
constituent, namely \(\pi'\).  Moreover, Corollary~\ref{fdi-comp} and
Theorem~\ref{max}\textup{(i)} imply that the conditions
\(a_1>a_2\) and \(b_1>b_2\) give
\(
\ell_1>\ell_2 .
\)
Hence \(1\notin\mathtt N\).  By Theorem~\ref{max}(iii), the first part of the
partition attached to \(\SCO\) is therefore even, namely \(2\ell_1\).  Since
\(\ell_1>\ell_2\), this first even part is strictly larger than the remaining
parts.  Thus the even composition law, Lemma~\ref{erL3}, gives
\[
        \CO\in\SOM
        \quad\Longleftrightarrow\quad
        \CO=\CO_{2\ell_1,\alpha_1}*\CO'
        \quad\text{for some }\CO'\in\ScO(\pi')^{\max}.
\]
In particular,
\[
        |\SOM|=|\ScO(\pi')^{\max}|.
\]
Since the first part is even and strictly larger than the remaining parts, the
description of component groups in Section~\ref{lcq} gives
\[
        |A(\SCO)|=2|A(\SCOP)|.
\]
We now compare the canonical quotients on the dual side.  We have
\[
        \mathcal L^{\rm can}(\pi')
        =
        \pi_{\Lambda'_1}\otimes \pi_{\Lambda'_{-1}},
\]
where \(\pi_{\Lambda'_1}\) is an irreducible unipotent representation of a
special odd orthogonal group, and \(\pi_{\Lambda'_{-1}}\) is an irreducible
unipotent representation of an even orthogonal group.  By Proposition~\ref{un}, the conditions
\(a_1>a_2\) and \(b_1>b_2\) imply that
\begin{itemize}
    \item the partition of
    \(\mathscr O(\pi_{\Lambda'_{-1}})^{\max}\) is obtained by removing the
    first part of \(\mathscr O(\pi_{\Lambda_1})^{\max}\);
    \item the partition of
    \(\mathscr O(\pi_{\Lambda'_1})^{\max}\) is obtained by removing the first
    part of \(\mathscr O(\pi_{\Lambda_{-1}})^{\max}\);
    \item the removed parts are odd.
\end{itemize}
Therefore, by the description of Lusztig's canonical quotient in
Section~\ref{lcq},
\[
|\overline A(\mathscr O(\pi_{\Lambda_1})^{\max,\mathrm{st}})|
=
2|\overline A(\mathscr O(\pi_{\Lambda'_{-1}})^{\max,\mathrm{st}})|
\]
and
\[
|\overline A(\mathscr O(\pi_{\Lambda_{-1}})^{\max,\mathrm{st}})|
=
|\overline A(\mathscr O(\pi_{\Lambda'_1})^{\max,\mathrm{st}})|.
\]
Equivalently,
\[
        |\overline A(\SCO^*)|
        =
        2|\overline A(\SCOPS)|.
\]
Combining the three equalities above, we get
\[
\frac{|A(\SCO)|}
{|\SOM|\,|\overline A(\SCO^*)|}
=
\frac{2|A(\SCOP)|}
{2|\ScO(\pi')^{\max}|\,|\overline A(\SCOPS)|}
=
\frac{|A(\SCOP)|}
{|\ScO(\pi')^{\max}|\,|\overline A(\SCOPS)|}.
\]
The last quotient is equal to \(1\) by the induction hypothesis applied to
\(\pi'\).  This proves the claim in the first case.

Next assume that
\[
        a_1>a_2,\qquad b_1=b_2 .
\]
The same argument as in the first case gives \(\ell_1\geqslant\ell_2\).  Hence
\(1\notin\mathtt N\), and the first part of the partition attached to
\(\SCO\) is even. 
By Corollary~\ref{5.6} and the same descent--orbit analysis as in the
proof of Theorem~\ref{max}\textup{(iv)}, together with the description
of component groups in Section~\ref{lcq}, we have
\[
|\SOM|=c|\ScO(\pi')^{\max}|,
\qquad
|A(\SCO)|=c|A(\SCOP)|,
\]
where
\[
c=
\begin{cases}
1,&\ell_1=\ell_2,\\
2,&\ell_1>\ell_2.
\end{cases}
\]

We now compare the canonical quotients on the dual side.  The contribution
from the \(a\)-part is the same as in the first case:
\[
|\overline A(\mathscr O(\pi_{\Lambda_1})^{\max,\mathrm{st}})|
=
2|\overline A(\mathscr O(\pi_{\Lambda'_{-1}})^{\max,\mathrm{st}})|.
\]
For the \(b\)-part, the equality \(b_1=b_2\) implies, by
Proposition~\ref{un}, that the first two parts of the partition of
\(\mathscr O(\pi_{\Lambda_{-1}})^{\max,\mathrm{st}}\) are equal and even.
Under the first descent, the first of these two equal even parts is removed,
while the second one is increased by one and hence becomes odd.  The partition
obtained in this way is the one corresponding to \(\mathscr O(\pi_{\Lambda'_{1}})^{\max,\mathrm{st}}\).
Therefore, by the description of Lusztig's canonical quotient in
Section~\ref{lcq}, the new odd part contributes one extra factor \(2\), and
we get
\[
|\overline A(\mathscr O(\pi_{\Lambda'_1})^{\max,\mathrm{st}})|
=
2|\overline A(\mathscr O(\pi_{\Lambda_{-1}})^{\max,\mathrm{st}})|.
\]
Together with the \(a\)-part equality above, this gives
\[
        |\overline A(\SCO^*)|
        =
        |\overline A(\SCOPS)|.
\]
Therefore
\[
\frac{|A(\SCO)|}
{|\SOM|\,|\overline A(\SCO^*)|}
=
\frac{c|A(\SCOP)|}
{c|\ScO(\pi')^{\max}|\,|\overline A(\SCOPS)|}
=
\frac{|A(\SCOP)|}
{|\ScO(\pi')^{\max}|\,|\overline A(\SCOPS)|}.
\]
The last quotient is equal to \(1\) by the induction hypothesis applied to
\(\pi'\).  This proves the claim in this case.

Next assume that
\[
        a_1=a_2,\qquad b_1>b_2 .
\]
This is analogous to the previous case, with the roles of the two symbols
interchanged; we only record the change in the \(a\)-part on the dual side.
The equality \(a_1=a_2\) gives two first-descent branches whose
irreducible representations \(\pi'\) and \(\pi''\) have the same
remaining descent data.
As before, Corollary~\ref{fdi-comp} and
Theorem~\ref{max}\textup{(i)} give
\(\ell_1\geqslant\ell_2\).  Hence
\(1\notin\mathtt N\).  Thus
\[
        |\SOM|=c|\ScO(\pi')^{\max}|,
        \qquad
        |A(\SCO)|=c|A(\SCOP)|,
\]
where $c$ is defined as in the previous case.
By Proposition~\ref{un}, the equality \(a_1=a_2\) means that the first two parts
of the partition of
\(\mathscr O(\pi_{\Lambda_1})^{\max,\mathrm{st}}\) are equal and odd.  Under
the first descent, the first of these two equal odd parts is removed, while
the second one remains unchanged.  The resulting partition is the one
corresponding to \(\mathscr O(\pi_{\Lambda'_{-1}})^{\max,\mathrm{st}}\).  Hence, by the description of
Lusztig's canonical quotient in Section~\ref{lcq},
\[
|\overline A(\mathscr O(\pi_{\Lambda_1})^{\max,\mathrm{st}})|
=
|\overline A(\mathscr O(\pi_{\Lambda'_{-1}})^{\max,\mathrm{st}})|.
\]
Together with the \(b\)-part equality from the first case, this gives
\[
        |\overline A(\SCO^*)|
        =
        |\overline A(\SCOPS)|.
\]
Combining this with the two equalities for \(|\SOM|\) and \(|A(\SCO)|\)
above, we reduce to the same quotient for \(\pi'\).  Hence the desired
equality follows from the induction hypothesis applied to \(\pi'\).  This
proves the claim in this case.

Finally assume that
\[
        a_1=a_2,\qquad b_1=b_2 .
\]
In this case both symbols are in the equal-row situation.  By
Corollary~\ref{5.6}, the first two descent steps fit into the diagram
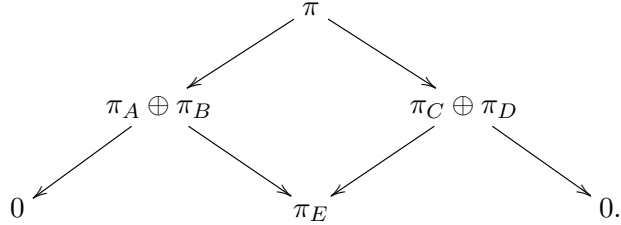
\[
\xymatrix{
& & \pi \ar[dl] \ar[dr] & &  \\
& \pi_A\oplus\pi_B \ar[dl] \ar[dr] & & \pi_C\oplus\pi_D \ar[dl] \ar[dr] & \\
0 & & \pi_E & & 0.
}
\]
By Corollary~\ref{fdi-comp} and
Theorem~\ref{max}\textup{(i)}, we have
\(1\in\mathtt N\).  Hence
Theorem~\ref{max}\textup{(iii)} shows that the first two parts of
the partition attached to \(\SCO\) are equal and odd. By the odd-pair
composition law and Theorem~\ref{thm5.1}, we have
\[
        \CO\in\SOM
        \quad\Longleftrightarrow\quad
        \CO=\CO_{(2\ell_1+1)^2}*\CO_E
        \quad\text{for some }\CO_E\in\ScO(\pi_E)^{\max}.
\]
Let
\(
        \CO^{\rm st}_E=\mathscr O(\pi_E)^{\max,\mathrm{st}}.
\)
We have
\[
        |\SOM|=|\ScO(\pi_E)^{\max}|,
        \qquad
        |A(\SCO)|=|A(\CO^{\rm st}_E)|.
\]

It remains to compare the canonical quotients on the dual side.  We follow
the two first descents in the diagram above.  For the first descent, the
\(a\)-part is governed by the comparison used in the third case, whereas the
\(b\)-part is governed by the comparison used in the second case.  For the
second descent, both parts are governed by the comparison used in the first
case.  Applying these factor comparisons and the description of Lusztig's
canonical quotient in Section~\ref{lcq}, we obtain
\[
        |\overline A(\SCO^*)|
        =
        |\overline A(\mathscr O(\mathcal L^{\rm can}(\pi_E))^{\max,\mathrm{st}})|.
\]
Combining this with
\[
        |\SOM|=|\ScO(\pi_E)^{\max}|,
        \qquad
        |A(\SCO)|=|A(\CO^{\rm st}_E)|,
\]
we reduce the desired equality to the corresponding quotient for \(\pi_E\).
The result follows from the induction hypothesis applied to \(\pi_E\).  This
completes the proof.
\end{proof}

\begin{proof}[Proof of Theorem~\ref{dimW}]
Write \(\pi=\prll\).  Choose a maximal descent sequence index
\[
        \Gamma=\llb(2\ell_1,a_1),(2\ell_2,a_2),\ldots\rrb
\]
corresponding to \(\CO\).  We argue by induction on the length of
\(\Gamma\).  If the length is one, then \(\pi\) is generic, and the formula
follows directly from the explicit description of the generic orbit and of
Lusztig's canonical quotient.

Assume now that the length of \(\Gamma\) is at least two.  Let
\(
        \Gamma'
\)
be obtained from \(\Gamma\) by deleting its first term.  Let \(\pi'\) be the
irreducible constituent of the first descent of \(\pi\) corresponding to
\((2\ell_1,a_1)\), and write
\[
        \pi'=\pi_{\rho',\Lambda'_1,\Lambda'_{-1}}.
\]
By the maximality of \(\Gamma\), its tail \(\Gamma'\) is a maximal
descent sequence index for \(\pi'\).  Let
\(
        \CO'\in\ScO(\pi')^{\max}
\)
be the corresponding rational nilpotent orbit.

Let
\(
        \pi^\star=\pi_{\Lambda_1,\Lambda_{-1}}
\)
be the irreducible representation obtained by keeping only the quadratic
unipotent part of \(\pi\).  By the first-descent recipe in Section~\ref{sec6.1} and
Corollary~\ref{5.6}, the first descent on the two quadratic symbols
\(\Lambda_1,\Lambda_{-1}\) is determined by their relevant largest entries;
the non-quadratic datum \(\rho\) is carried to \(\rho'\) independently and
does not change this operation.  Hence the quadratic unipotent part of
\(\pi'\) is
\[
        \pi^{\star\prime}=\pi_{\Lambda'_1,\Lambda'_{-1}}.
\]

Let \(\Gamma^\star\) and \(\Gamma^{\star\prime}\) be the maximal
descent sequence indexes obtained from \(\Gamma\) and \(\Gamma'\),
respectively, by retaining only the contributions from the
\(1\)- and \((-1)\)-parts in Theorem~\ref{max}\textup{(i)}.  Let
\(
\CO^\star=\CO(\Gamma^\star),
\) and \(
\CO^{\star\prime}=\CO(\Gamma^{\star\prime}).
\)  Write
\[
        (\CO')^{\mathrm{st}}
        =
        \mathscr O(\pi')^{\max,\mathrm{st}},
        \qquad
        (\CO^\star)^{\mathrm{st}}
        =
        \mathscr O(\pi^\star)^{\max,\mathrm{st}},
        \qquad
        (\CO^{\star\prime})^{\mathrm{st}}
        =
        \mathscr O(\pi^{\star\prime})^{\max,\mathrm{st}} .
\]
We also write \(\CO^{\mathrm{st}}\) for the stable orbit containing
\(\CO\).  On the dual side, denote the corresponding stable nilpotent orbits
in the centralizers of the relevant semisimple elements by
\[
        \CO^{*\mathrm{st}},\qquad
        (\CO')^{*\mathrm{st}},\qquad
        (\CO^\star)^{*\mathrm{st}},\qquad
        (\CO^{\star\prime})^{*\mathrm{st}}.
\]
The stable orbits \(\CO^{*\mathrm{st}}\) and
\((\CO')^{*\mathrm{st}}\) differ from
\((\CO^\star)^{*\mathrm{st}}\) and
\((\CO^{\star\prime})^{*\mathrm{st}}\) only by the nilpotent orbits coming
from general linear and unitary factors in the corresponding centralizers.
Since these factors have trivial Lusztig canonical quotient,
\begin{equation}\label{ao}
\frac{|\overline A(\CO^{*\mathrm{st}})|}
{|\overline A((\CO^\star)^{*\mathrm{st}})|}
=
\frac{|\overline A((\CO')^{*\mathrm{st}})|}
{|\overline A((\CO^{\star\prime})^{*\mathrm{st}})|}.
\end{equation}
Let
\[
        Q(\sigma,\mathcal P)
        :=
        \frac{|A(\mathcal P^{\mathrm{st}})|}
        {|\ScO(\sigma)^{\max}|\,
         |\overline A(\mathcal P^{*\mathrm{st}})|},
\]
where $\sigma$ is an irreducible representation of a finite symplectic group, $\mathcal P$ is a rational nilpotent orbit and \(\mathcal P^{*\mathrm{st}}\) denotes the corresponding dual stable
orbit listed above.  It is enough to prove
\[
        \dim \mathrm{Wh}_{\CO}(\pi)=Q(\pi,\CO).
\]

For quadratic unipotent representations, Theorem~\ref{max}(iv) gives
multiplicity one, and Lemma~\ref{unm} gives
\[
        Q(\pi^\star,\CO^\star)=1,
        \qquad
        Q(\pi^{\star\prime},\CO^{\star\prime})=1.
\]
Equivalently,
\[
\begin{aligned}
|A((\CO^\star)^{\mathrm{st}})|
&=
|\ScO(\pi^\star)^{\max}|\,
|\overline A((\CO^\star)^{*\mathrm{st}})|,\\
|A((\CO^{\star\prime})^{\mathrm{st}})|
&=
|\ScO(\pi^{\star\prime})^{\max}|\,
|\overline A((\CO^{\star\prime})^{*\mathrm{st}})|.
\end{aligned}
\]
By Theorem~\ref{max}(iv), the first descent gives
\[
\dim \mathrm{Wh}_{\CO}(\pi)
=
\begin{cases}
2\dim \mathrm{Wh}_{\CO'}(\pi') & \text{if }1\in\mathtt M,\\
\dim \mathrm{Wh}_{\CO'}(\pi') & \text{if }1\notin\mathtt M.
\end{cases}
\]

Assume first that \(1\in\mathtt M\).  By the definition of
\(\mathtt M\), we have
\[
        \ell[\ne\pm1]_1>\ell[\ne\pm1]_2,\qquad
        \ell[1]_1=\ell[-1]_2-1,\qquad
        \ell[-1]_1=\ell[1]_2.
\]
By Corollary~\ref{fdi-comp}, the largest entries in the two rows of
both \({}^{\bd}\Lambda_1\) and
\({}^{\bd}\Lambda_{-1}\) are equal. Put
\(
        k=\ell[1]_1+\ell[-1]_1 .
\)
By Theorem~\ref{max}(iii), the first two parts of
\((\CO^\star)^{\mathrm{st}}\) are \((2k+1,2k+1)\), while the corresponding
first part for \((\CO^{\star\prime})^{\mathrm{st}}\) is \(2k+2\), with all
remaining parts agreeing.  Hence, by Section~\ref{lcq},
\[
        2|A((\CO^\star)^{\mathrm{st}})|
        =
        |A((\CO^{\star\prime})^{\mathrm{st}})|.
\]
Also, \((\CO')^{\mathrm{st}}\) is obtained from
\(\CO^{\mathrm{st}}\) by removing its first part \(2\ell_1\).
If \(\ell_1=\ell_2\), then the first two parts of
\(\CO^{\mathrm{st}}\) are both equal to \(2\ell_1\); if
\(\ell_1>\ell_2\), then \(2\ell_1\) is strictly larger than all the
remaining parts.  Therefore
\[
|A(\CO^{\mathrm{st}})|
=
\begin{cases}
|A((\CO')^{\mathrm{st}})| & \text{if }\ell_1=\ell_2,\\
2|A((\CO')^{\mathrm{st}})| & \text{if }\ell_1>\ell_2.
\end{cases}
\]
It follows that
\begin{equation}\label{m1}
\frac{|A((\CO')^{\mathrm{st}})|}
{|A((\CO^{\star\prime})^{\mathrm{st}})|}
=
\begin{cases}
\dfrac{|A(\CO^{\mathrm{st}})|}
{2|A((\CO^\star)^{\mathrm{st}})|}
& \text{if }\ell_1=\ell_2,\\[1.2ex]
\dfrac{|A(\CO^{\mathrm{st}})|}
{4|A((\CO^\star)^{\mathrm{st}})|}
& \text{if }\ell_1>\ell_2 .
\end{cases}
\end{equation}
The proof of Lemma~\ref{unm} gives
\[
        |\ScO(\pi^\star)^{\max}|
        =
        |\ScO(\pi^{\star\prime})^{\max}|.
\]
The proof of
Theorem~\ref{max}\textup{(iv)} gives
\[
|\ScO(\pi)^{\max}|
=
\begin{cases}
|\ScO(\pi')^{\max}| & \text{if }\ell_1=\ell_2,\\
2|\ScO(\pi')^{\max}| & \text{if }\ell_1>\ell_2.
\end{cases}
\]
Combining these equalities with \eqref{m1}, in both subcases we get
\begin{equation}\label{m3}
\frac{
|A((\CO')^{\mathrm{st}})|\,|\ScO(\pi^{\star\prime})^{\max}|
}{
|A((\CO^{\star\prime})^{\mathrm{st}})|\,|\ScO(\pi')^{\max}|
}
=
\frac{
|A(\CO^{\mathrm{st}})|\,|\ScO(\pi^\star)^{\max}|
}{
2|A((\CO^\star)^{\mathrm{st}})|\,|\ScO(\pi)^{\max}|
}.
\end{equation}
Together with \eqref{ao} and the quadratic-unipotent equalities above,
\eqref{m3} gives
\[
        2Q(\pi',\CO')=Q(\pi,\CO).
\]
By the induction hypothesis,
\[
        \dim \mathrm{Wh}_{\CO'}(\pi')=Q(\pi',\CO').
\]
Hence
\[
        \dim \mathrm{Wh}_{\CO}(\pi)
        =
        2\dim \mathrm{Wh}_{\CO'}(\pi')
        =
        2Q(\pi',\CO')
        =
        Q(\pi,\CO).
\]
This proves the theorem when \(1\in\mathtt M\).

It remains to treat the case \(1\notin\mathtt M\).  In this case the first
descent contributes no multiplicity factor.  Repeating the same comparison of
component groups and of maximal rational orbits gives the analogue of
\eqref{m3} without the extra factor \(2\):
\[
\frac{
|A((\CO')^{\mathrm{st}})|\,|\ScO(\pi^{\star\prime})^{\max}|
}{
|A((\CO^{\star\prime})^{\mathrm{st}})|\,|\ScO(\pi')^{\max}|
}
=
\frac{
|A(\CO^{\mathrm{st}})|\,|\ScO(\pi^\star)^{\max}|
}{
|A((\CO^\star)^{\mathrm{st}})|\,|\ScO(\pi)^{\max}|
}.
\]
Together with \eqref{ao} and the quadratic-unipotent equalities above, this
gives
\[
        Q(\pi',\CO')=Q(\pi,\CO).
\]
By Theorem~\ref{max}(iv) and the induction hypothesis,
\[
        \dim \mathrm{Wh}_{\CO}(\pi)
        =
        \dim \mathrm{Wh}_{\CO'}(\pi')
        =
        Q(\pi',\CO')
        =
        Q(\pi,\CO).
\]
This proves the remaining case and completes the proof.
\end{proof}

\section{Rational wavefront sets for orthogonal groups}\label{sec8}

In this section, we use the theta correspondence to reduce the study of
rational wavefront sets for orthogonal groups to the symplectic case.
We first recall the relation between theta correspondence and generalized
Whittaker models, following \cite{GZ14,Z19}, and give the corresponding
statement over finite fields. We then apply this compatibility to prove the main results for
orthogonal groups.

\subsection{Theta correspondence and generalized Whittaker models}

Following the notation in Section \ref{subsec:theta}, assume that $G$ and $G'$ form a reductive dual pair, with Lie algebras $\fg$ and $\fg'$, respectively. Recall the
moment maps
\[
\xymatrix{
& \Hom(V, V') \ar[dl]_\phi \ar[dr]^{\phi'} & \\
\fg & & \fg'
}
\]
where $\phi(T) := T^*T$ and $\phi'(T) := T T^*$ for $T\in \Hom(V, V')$. 
Given an $\mathfrak{sl}_2$-triple $\gamma= \{X, H, Y \} \subset \fg$, let $\fg_\gamma:=\textrm{span}\{X, H, Y \}\subset  \fg$. Then we have the decompositions
\[
V = \bigoplus_{ k\in\bb{Z}} V_k = \bigoplus_{ j} V^{\gamma,t_j},
\]
where $V_k := \set{v \in V | Hv = kv}$, and $V^{\gamma,t_j}$ is a direct sum of irreducible $t_j$-dimensional $\fg_\gamma$-modules. For any $d$ and $k$, set \(V^{\gamma,d}_k:=V^{\gamma,d}\cap V_k\), and use the analogous notation for $V'$.
Similar notation applies to an $\mathfrak{sl}_2$-triple
$\gamma'=\{X',H',Y'\}\subset\fg'$. We write
\[
\CO:=\CO_X\subset\fg,
\qquad
\CO':=\CO_{X'}\subset\fg'.
\]

For the description of the theta lifting of nilpotent orbits,
we use the following notion of lifting of $\mathfrak{sl}_2$-triples.
\begin{defn}
Let $\gamma\subset  \fg$ and $\gamma'\subset  \fg'$
be  $\mathfrak{sl}_2$-triples, and $T \in \Hom(V, V ')$. We
say that $T$ lifts $\gamma$ to $\gamma'$
if
 \begin{itemize}
 \item $ T \in {\rm GenHom}(V, V ')$,
 \item $\phi(T) = X$ and $\phi'(T) = X'$,
 \item $T(V_k)\subset V_{k+1}'$ for all $k$, where $V'_k$ is defined similarly for $\gamma'$.
 \end{itemize}
Here
 \[
{\rm GenHom}(V, V') := \set{T \in \Hom(V, V ') | \Ker(T)\textrm{ is a nondegenerate subspace of }V }.
\]
\end{defn}
We denote the set of all such lifts by
\[
\CO_{\gamma,\gamma'}
:=
\set{T\in\Hom(V,V')\mid T\text{ lifts }\gamma\text{ to }\gamma'}.
\]
The argument of \cite[Lemma 3.4]{Z19} works over finite fields as well.
We obtain the following description. For any $F$-rational $\mathfrak{sl}_2$-triple $\gamma'=\{X',H',Y'\}\subset(\fg')^F$ with $X'$ in the image of $\phi'$, there is a unique $G^F$-conjugacy class of $F$-rational
$\mathfrak{sl}_2$-triples
$\gamma=\{X,H,Y\}\subset\fg^F$
such that $\CO_{\gamma,\gamma'}$ is nonempty. Moreover, $\CO_{\gamma,\gamma'}$ is a single
$M_X^F\times (M'_{X'})^F$-orbit, where $M_X$ and $M'_{X'}$ are the
stabilizers of the triples $\gamma$ and $\gamma'$, respectively.

\begin{defn} We  say that $\gamma$ (resp. $\CO$) is the generalized descent of $\gamma'$ (resp. $\CO'$) if there exists $T\in \Hom(V,V')$ lifting $\gamma$ to $\gamma'$. In this case  we write
\[
\CO = \nabla^{\rm{gen}}_{V',V}(\CO').
\]
If $\CO_{\gamma,\gamma'}$ contains an injective element, then we call $\gamma$ (resp. $\CO$) the descent of $\gamma'$ (resp. $\CO'$), and write
\[
\CO = \nabla_{V',V}(\CO').
\]
\end{defn}
This orbit-theoretic generalized descent should not be confused with the representation-theoretic descents and descent sequences introduced in
Section~\ref{sec4}.

The explicit description of $\nabla^{\mathrm{gen}}_{V',V}$ over local
fields is given in \cite[Section~3]{GZ14}. In the finite-field setting,
it takes the following form. 
\begin{prop}\label{orbit}
    Let $\CO'=[(\lambda_1^{k_1},a_1),\cdots,(\lambda_l ^{k_l},a_l),(2^{j},b_2),(1^s,b_1)]$ be an $F$-rational nilpotent orbit lying in the image of $\phi'$, where
$\lambda_1>\cdots>\lambda_l>2$. Then
    \[
    \CO = \nabla^{\rm{gen}}_{V',V}(\CO')=\left[\left((\lambda_1-1)^{k_1},a_1\right),\cdots,\left((\lambda_l-1) ^{k_l},a_l\right),(1^{j+t},b_1')\right].
    \]
 Here $t=\dim\Ker(T)$ for
$T\in\CO_{\gamma,\gamma'}$ realizing the generalized descent; this
dimension is independent of the choice of $T$.
The $(j+t)$-dimensional multiplicity space $W_1$ corresponding to
$(1^{j+t},b_1')$ contains the $j$-dimensional multiplicity space $W_2$
corresponding to $(2^j,b_2)$ as a nondegenerate subspace.
In particular, if $t=0$, then $b_1'=b_2$.
\end{prop}

\begin{prop}\label{8.4}
Put
\(
\beta=\dim V'-\dim V.
\)
Let
$\CO_1',\CO_2'\subset(\fg')^F$
be two $F$-rational nilpotent orbits lying in the image of $\phi'$,
and set
\(
\CO_i=\nabla^{\rm gen}_{V',V}(\CO_i')
\) with $i=1,2$. Then
\[
\CO_1'\leqslant\CO_2'
\quad\Longrightarrow\quad
\CO_1\leqslant\CO_2.
\]
Assume that the partition of $\CO_2'$ has exactly $\beta$ rows. Then
\[
\CO_1'\leqslant\CO_2'
\quad\Longleftrightarrow\quad
\CO_1\leqslant\CO_2.
\]
Moreover, assume that $\dim V'\geqslant2\dim V$ and that $V'$ is
skew-symmetric. Then every $F$-rational nilpotent orbit in $\fg^F$
is the generalized descent of a unique $F$-rational nilpotent orbit
in the image of $\phi'$ whose partition has exactly $\beta$ rows.
\end{prop}

\begin{proof}
The first two assertions follow directly from Proposition~\ref{orbit}
and a straightforward calculation.

For the final assertion, let $\CO\subset\fg^F$. Since
$\beta\geqslant\dim V$, its signed Young diagram has at most $\beta$
rows. Append zero rows until it has exactly $\beta$ rows and then add
one box to each row. Transporting the rational forms as in
Proposition~\ref{orbit} gives an $F$-rational nilpotent orbit
$\CO'\subset(\fg')^F$ with exactly $\beta$ rows. It lies in the image
of $\phi'$ and satisfies
\(
\nabla^{\rm gen}_{V',V}(\CO')=\CO.
\)
The construction also proves uniqueness.
\end{proof}

We now recall the stabilizer decomposition needed for the comparison of
generalized Whittaker models.
For a symmetric or skew-symmetric space \(W\), write \(\RG(W)\) for its isometry group.

\begin{defn}
Let $\gamma$ and $\gamma'$ be as above, and let
$T\in\CO_{\gamma,\gamma'}$. Put
\(
V_{\gamma,\gamma'}:=\Ker(T)\) and
\(
V'_{\gamma,\gamma'}:=(V')^{\gamma',1}_0,
\)
and let $L$ and $L'$ be their isometry groups, respectively.
Define
\[
\alpha=\alpha_T:M'_{X'}\longrightarrow M_X
\]
by letting $\alpha(m')$ act trivially on $\Ker(T)$ and as
$T^{-1}m'T$ on $(\Ker(T))^\perp$.
Let $t_1+1>\cdots>t_l+1>2$ be the distinct integers with 
$(V')^{\gamma',t_j+1}\neq 0$ for any $1\leq j\leq l$, listed in decreasing order. Via the isometries induced by $T^*$, identify
\[
(V')^{\gamma',t_j+1}_{t_j}
\simeq
V^{\gamma,t_j}_{t_j-1},
\qquad 1\leq j\leq l,
\]
and  via the isometry induced by \(T^*\), regard
\(
U_1:=(V')^{\gamma',2}_1
\)
as a nondegenerate subspace of $V$.
Then
\[
M'_{X'}=M_{X,X'}\times L',
\]
where
\[
M_{X,X'}
:=
\prod_{j=1}^l
\RG\bigl((V')^{\gamma',t_j+1}_{t_j}\bigr)
\times\RG(U_1).
\]
Let
\[
U:=\left(\bigoplus_{j=1}^l V^{\gamma,t_j}\right)^\perp.
\]
Then
\(
U=U_1\perp V_{\gamma,\gamma'},
\)
and
\[
M_X
=
\prod_{j=1}^l
\RG\bigl(V^{\gamma,t_j}_{t_j-1}\bigr)
\times\RG(U)
\supset M_{X,X'}\times L.
\]
\end{defn}

\begin{thm}[Finite-field analogue of Theorem 3.7 in \cite{Z19}]\label{gz}
Let $(G,G')$ be a reductive dual pair over $\mathbbm{k}$.
Let $\CO\subset\fg^F$ and
$\CO'\subset(\fg')^F$ be $F$-rational nilpotent orbits, and let
$\pi\in\Irr(G^F)$.
\begin{enumerate}
\item[(i)]  Assume that $\CO'$ lies in the image of $\phi'$ and that $\CO$ is the generalized descent of $\CO'$. Choose compatible
$F$-rational $\mathfrak{sl}_2$-triples
\[
\gamma=\{X,H,Y\},
\qquad
\gamma'=\{X',H',Y'\},
\]
representing $\CO$ and $\CO'$, respectively. Then the associated
stabilizers are of the form
\[
M_X \supset M_{X,X'} \times L, \quad M'_{X'} = M_{X,X'} \times L',
\]
for some reductive dual pair $(L, L')$ of the same type as $(G,G')$. Let $\tau'\in \Irr(L^{\prime F})$. Then
\[
{\mathrm{Wh}}_{\CO',\tau'} (\Theta(\pi))\cong {\mathrm{Wh}}_{\CO,\Theta(\tau')^\vee} (\pi^\vee)
\]
as $M^F_{X,X'}$-modules. Here $\Theta(\tau')$ is the theta lift of $\tau'$ with respect to the
dual pair $(L,L')$. We set
\[
\mathrm{Wh}_{\CO',\tau'}(\Theta(\pi))
:=
\Hom_{L^{\prime F}}
\bigl(\tau',\mathrm{Wh}_{\CO'}(\Theta(\pi))\bigr),
\]
and define
$\mathrm{Wh}_{\CO,\Theta(\tau')^\vee}(\pi^\vee)$ analogously.

\item[(ii)] If $\CO'$ does not lie in the image of $\phi'$, then
\[
{\mathrm{Wh}}_{\CO'} (\Theta(\pi)) = 0.
\]
\end{enumerate}
\end{thm}
\begin{sop} The argument is the finite-field counterpart of the proof of \cite[Theorem~3.7]{Z19}; see also \cite{GZ14}. A detailed implementation for finite unitary dual pairs is given in \cite[Section~4.4]{PW}. One realizes the Weil representation in the mixed model adapted to the $\mathfrak{sl}_2$-gradings associated with $\gamma$ and $\gamma'$ and successively takes the relevant equivariant quotients. In the present symplectic--orthogonal case, one replaces the unitary mixed model by the corresponding mixed model for the symplectic--orthogonal Weil representation; see \cite[Section~2.2]{MQZ24} and \cite{Ger77}. With this replacement, the successive equivariant-quotient calculations in the proof of \cite[Theorem~3.7]{Z19} carry over without further change.
\end{sop}

\subsection{Main results for orthogonal groups}

We use the theta correspondence to transfer the results of
Corollary~\ref{7.2} from symplectic groups to orthogonal groups.

\begin{thm}\label{omax}
Let $\pi\in\Irr(\RO_m(\mathbbm{k}))$. Then
\begin{enumerate}
\item[(i)]
\(
\ScO(\pi)^{\max,\mathrm{st}}=\ScO(\pi)^{\max}.
\)

\item[(ii)]
\(
\ScO(\pi)=\overline{\ScO(\pi)^{\max}}.
\)
\end{enumerate}
\end{thm}

\begin{proof}

We give the details when $m$ is even. The odd-dimensional case follows
from the same argument using the corresponding odd orthogonal Witt tower.

Assume that $m=2n$, and let $\pi^\vee$ denote the contragredient of $\pi$. Assume that
\[
\pi^\vee=\pi_{\rho,\Lambda_1,\Lambda_{-1}},
\qquad
\Lambda_1=
\begin{pmatrix}
a_1,a_2,\ldots,a_{m_1}\\
b_1,b_2,\ldots,b_{m_2}
\end{pmatrix}.
\]
We treat the case
$\RO_{2n}(\mathbbm{k})=\RO^+_{2n}(\mathbbm{k})$ and
$\mathrm{def}(\Lambda_1)\leqslant0$. The remaining even-dimensional
cases follow by interchanging the two rows of the relevant symbol and
choosing the corresponding Witt tower.

Let $s$ be a semisimple element such that
\(
\prll\in\mathcal E(\RO_{2n}(\mathbbm{k}),s).
\)
Choose a semisimple element
$s'\in\SO_{4n+1}(\mathbbm{k})$ such that the pair $(s',s)$
satisfies condition~\textup{(3)} of Theorem~\ref{thm:L-can}. Let
\(
\prllp\in\mathcal E(\Sp_{4n}(\mathbbm{k}),s')
\)
be determined by $\rho'=\rho$, $\Lambda_{-1}'=\Lambda_{-1}$ and
\begin{equation}\label{Lambda1p}
    \Lambda_1'
=
\begin{pmatrix}
k,b_1,b_2,\ldots,b_{m_2}\\
a_1,a_2,\ldots,a_{m_1}
\end{pmatrix}\qquad\rank(\Lambda_1')=\rank(\Lambda_1)+n.
\end{equation}
The rank condition gives
\(
k=n+\frac{m_1+m_2}{2}>n\geqslant b_1
\)
so $\Lambda_1'$ is well defined.

By Theorem~\ref{thm:L-can}, together with the explicit description of
the unipotent theta correspondence in \cite{AMR96} and
\cite[Section~1.5]{Pan19}, the first occurrence of $\prllp$ in the Witt
tower $\mathbf O^+_{\mathrm{even}}$ has dimension $2n$, and
\begin{equation}\label{thetaps0}
\Theta_{4n,2n}(\prllp)=\prll.
\end{equation}
Moreover,
\begin{equation}\label{thetaps}
    \Theta_{2n,4n}(\prll)=\bigoplus_{\prlls}\prlls,
\end{equation}
where
$\prlls=\pi_{\rho^\star,\Lambda_1^\star,\Lambda_{-1}^\star}$ runs over
the representations satisfying
\begin{equation}\label{thetaps1}
\rho^\star=\rho,
\qquad
\Lambda_{-1}^\star=\Lambda_{-1},
\qquad
\mathrm{def}(\Lambda_1^\star)
=
-\mathrm{def}(\Lambda_1)+1,
\end{equation}
and
\begin{equation}\label{thetaps2}
\bigl(\Upsilon(\Lambda_1^\star)^-\bigr)^t
\preccurlyeq
\bigl(\Upsilon(\Lambda_1)^+\bigr)^t,
\qquad
\bigl(\Upsilon(\Lambda_1)^-\bigr)^t
\preccurlyeq
\bigl(\Upsilon(\Lambda_1^\star)^+\bigr)^t.
\end{equation}

Define
\[
\ScO\bigl(\Theta_{2n,4n}(\prll)\bigr)
:=
\bigcup_{\prlls\subset\Theta_{2n,4n}(\prll)}
\ScO(\prlls).
\]
For every constituent $\prlls\neq\prllp$, the relations
\eqref{thetaps1} and \eqref{thetaps2}, after applying Alvis--Curtis
duality and appending zeros when necessary, imply the hypotheses of
Lemma~\ref{thetamax}\textup{(i)}. Hence
\begin{equation}\label{sip}
\ScO(\prlls)^{\max}\subset\ScO(\prllp).
\end{equation}

By Corollary~\ref{7.2}(ii), $\ScO(\prlls)$ is the rational closure of
$\ScO(\prlls)^{\max}$, whereas $\ScO(\prllp)$ is rationally closed.
It follows that
\[
\ScO(\prlls)\subset\ScO(\prllp).
\]
Since $\prllp$ itself occurs in $\Theta_{2n,4n}(\prll)$, we obtain
\begin{equation}\label{eq8.7}
\ScO\bigl(\Theta_{2n,4n}(\prll)\bigr)
=
\ScO(\prllp).
\end{equation}

By Theorem~\ref{gz}(ii), all nilpotent orbits occurring in
$\ScO(\Theta_{2n,4n}(\prll))$ lie in the image of $\phi'$, while
Proposition~\ref{8.4} shows that every nilpotent orbit on
the orthogonal side has a generalized lift. Applying
Theorem~\ref{gz}(i) and using
\eqref{eq8.7}, we obtain
\begin{equation}\label{eq8.8}
\ScO(\pi)=\ScO(\prll^\vee)
=
\nabla^{\rm gen}_{V',V}\bigl(\ScO(\prllp)\bigr).
\end{equation}

Together with Corollary~\ref{7.2}\textup{(ii)} and \eqref{eq8.8},
Proposition~\ref{8.4} and Lemma~\ref{theta-strict} show that
generalized descent gives a bijection
between $\ScO(\prllp)^{\max}$ and $\ScO(\pi)^{\max}$.
By Corollary~\ref{7.2}\textup{(i)} and Proposition~\ref{orbit}, the
corresponding maximal rational orbits lie in one stable orbit. This
proves \textup{(i)}. Finally, Corollary~\ref{7.2}\textup{(ii)}, \eqref{eq8.8}, and the
order equivalence in Proposition~\ref{8.4}, which applies because by Lemma~\ref{theta-strict},
every orbit in $\ScO(\prllp)^{\max}$ has a partition with exactly
$2n$ rows, imply \textup{(ii)}.

\end{proof}

\begin{lem}\label{thetamax}
Let $\prllp$ and $\prlls$ be irreducible representations of
$\Sp_{2n}(\mathbbm{k})$.

\begin{enumerate}
\item[(i)]
Assume that
\(
\rho^\star=\rho'\),
\(\mathrm{def}(\Lambda_1^\star)
=
\mathrm{def}(\Lambda_1')\), and
\(\Lambda_{-1}^\star=\Lambda_{-1}'\).
Assume further that there exists $\epsilon\in\{\pm\}$ such that, writing
\[
\begin{aligned}
\Upsilon({}^{\bd}\Lambda_1')^\epsilon
&=(a_1',\ldots,a_{m_1'}'),
&
\Upsilon({}^{\bd}\Lambda_1')^{-\epsilon}
&=(b_1',\ldots,b_{m_2'}'),\\
\Upsilon({}^{\bd}\Lambda_1^\star)^\epsilon
&=(a_1^\star,\ldots,a_{m_1^\star}^\star),
&
\Upsilon({}^{\bd}\Lambda_1^\star)^{-\epsilon}
&=(b_1^\star,\ldots,b_{m_2^\star}^\star),
\end{aligned}
\]
we have $m_2'\leqslant m_2^\star$ and
\[
\begin{aligned}
&a_i'-1\leqslant a_i^\star\leqslant a_i'
&& (i\geqslant1),\\
&b_i'-1\leqslant b_i^\star\leqslant b_i'
&& (1\leqslant i\leqslant m_2'),\\
&b_i^\star=1
&& (m_2'<i\leqslant m_2^\star),
\end{aligned}
\]
with the convention that $a_i'=0$ for $i>m_1'$ and
$a_i^\star=0$ for $i>m_1^\star$. Then
\[
\ScO(\prlls)^{\max}\subset\ScO(\prllp).
\]

\item[(ii)]
Assume that
\(
\rho^\star=\rho'\),
\(\mathrm{def}(\Lambda_{-1}^\star)
=
\mathrm{def}(\Lambda_{-1}')\), and
\(\Lambda_1^\star=\Lambda_1'\).
Assume further that there exists $\epsilon\in\{\pm\}$ such that, writing
\[
\begin{aligned}
\Upsilon({}^{\bd}\Lambda_{-1}')^\epsilon
&=(a_1',\ldots,a_{m_1'}'),
&
\Upsilon({}^{\bd}\Lambda_{-1}')^{-\epsilon}
&=(b_1',\ldots,b_{m_2'}'),\\
\Upsilon({}^{\bd}\Lambda_{-1}^\star)^\epsilon
&=(a_1^\star,\ldots,a_{m_1^\star}^\star),
&
\Upsilon({}^{\bd}\Lambda_{-1}^\star)^{-\epsilon}
&=(b_1^\star,\ldots,b_{m_2^\star}^\star),
\end{aligned}
\]
we have $m_2'\leqslant m_2^\star$ and
\[
\begin{aligned}
&a_i'-1\leqslant a_i^\star\leqslant a_i'
&& (i\geqslant1),\\
&b_i'-1\leqslant b_i^\star\leqslant b_i'
&& (1\leqslant i\leqslant m_2'),\\
&b_i^\star=1
&& (m_2'<i\leqslant m_2^\star),
\end{aligned}
\]
with the convention that $a_i'=0$ for $i>m_1'$ and
$a_i^\star=0$ for $i>m_1^\star$. Then
\[
\ScO(\prlls)^{\max}\subset\ScO(\prllp).
\]
\end{enumerate}
\end{lem}

\begin{proof}
We argue by simultaneous induction on $n$. Both assertions are
immediate for $n=0$. Assume that they hold for all smaller values of
$n$. We prove (i); the induction step for (ii) is obtained by
interchanging the roles of the $1$-part and the $(-1)$-part. We retain
the notation for the $\epsilon$- and $(-\epsilon)$-rows introduced in
the statement.

Let $\CO\in\ScO(\prlls)^{\max}$, and choose a maximal descent sequence
realizing $\CO$. By Theorem~\ref{ds}(i), this sequence is obtained by
successive first descents. Write its first descent as
\[
\prlls\xrightarrow{(\ell^\star,a^\star)}\prllsp.
\]
By maximality, the remaining arrows form a maximal descent sequence for
$\prllsp$. Let
\(
\CO'\in\ScO(\prllsp)^{\max}
\)
be the orbit realized by these arrows.

If $\ell^\star=0$, then Proposition~\ref{trivial}(i) implies that
$\prlls$ is the trivial representation. Hence $\CO$ is the zero orbit,
which belongs to $\ScO(\prllp)$ by Corollary~\ref{7.2}(ii). We may
therefore assume that $\ell^\star>0$.

By Theorem~\ref{ggp}(ii), the first descent of $\prlls$ is obtained
with
\[
(\ep,\ev)
=
\bigl(
-\zeta({}^{\bd}\Lambda_1^\star),
\zeta({}^{\bd}\Lambda_{-1}^\star)
\bigr).
\]
To match this descent, we apply Theorem~\ref{ggp}(i) to $\prllp$ with
the same pair $(\ep,\ev)$. Our goal is to construct an irreducible
representation $\prllpp$ satisfying
\[
\prllpp
\subset
\CD^{\rm FJ}_{2\ell^\star,a^\star}(\prllp).
\]
Since the $\rho$-parameters and $\Lambda_{-1}$-parameters of $\prllp$ and
$\prlls$ agree, the parameter relations in Theorem~\ref{ggp}(i) allow
us to take the $\rho$- and $\Lambda_1$-parameters of $\prllpp$ to be
those of $\prllsp$. Their contributions to the
descent-length formula are therefore unchanged. The common choice of
$(\ep,\ev)$ also gives the same rational sign $a^\star$. It remains
only to choose $\Lambda_{-1}''$ so that its contribution gives the
required length $\ell^\star$ or equivalently \(
\operatorname{rank}(\Lambda_{-1}'')
=
\operatorname{rank}(\Lambda_{-1}^{\star\prime})
\).

Since the same pair $(\ep,\ev)$ is used in the two descents, the
defect formula in Theorem~\ref{ggp}\textup{(i)} shows that, if such a
symbol $\Lambda_{-1}''$ exists, it must satisfy
\(
\mathrm{def}(\Lambda_{-1}'')
=
\mathrm{def}(\Lambda_{-1}^{\star\prime}).
\)
It then follows that
\[
\operatorname{rank}(\Lambda_{-1}'')
-
\operatorname{rank}(\Lambda_{-1}^{\star\prime})
=
\bigl|\Upsilon({}^{\bd}\Lambda_{-1}'')\bigr|
-
\bigl|\Upsilon({}^{\bd}\Lambda_{-1}^{\star\prime})\bigr|.
\]
Set
\[
r=
\begin{cases}
a_1'-a_1^\star,&-\ep=\epsilon,\\
b_1'-b_1^\star,&-\ep=-\epsilon.
\end{cases}
\]
By the hypotheses, $r\in\{0,1\}$. To ensure that
\(
\operatorname{rank}(\Lambda_{-1}'')
=
\operatorname{rank}(\Lambda_{-1}^{\star\prime}),
\)
choose, up to the row switch prescribed by $\ev$, a symbol
$\Lambda_{-1}''$ satisfying
\(
\mathrm{def}(\Lambda_{-1}'')
=
\mathrm{def}(\Lambda_{-1}^{\star\prime})
\)
and
\[
\Upsilon({}^{\bd}\Lambda_{-1}'')
=
\begin{cases}
\begin{bmatrix}
a_2',\ldots,a_{m_1'}',1^r\\
b_1',\ldots,b_{m_2'}'
\end{bmatrix},
&-\ep=\epsilon,\\[5mm]
\begin{bmatrix}
a_1',\ldots,a_{m_1'}'\\
b_2',\ldots,b_{m_2'}',1^r
\end{bmatrix},
&-\ep=-\epsilon.
\end{cases}
\]
A direct verification of the conditions in
Theorem~\ref{ggp}\textup{(i)} then shows that
\[
\prllpp
\subset
\CD^{\rm FJ}_{2\ell^\star,a^\star}(\prllp).
\]

By direct calculation, the pair $(\prllpp,\prllsp)$ satisfies the
hypotheses of~\textup{(ii)}, possibly with $\epsilon$ replaced by
$-\epsilon$ after the row switch. Since $\ell^\star>0$, the induction
hypothesis gives
\[
\CO'
\in
\ScO(\prllsp)^{\max}
\subset
\ScO(\prllpp).
\]
Since
\(
\prllpp
\subset
\CD^{\rm FJ}_{2\ell^\star,a^\star}(\prllp)
\)
and $\CO$ is obtained from $\CO'$ by adjoining the first descent index
$(\ell^\star,a^\star)$, the composition law for descents gives
\[
\CO\in\ScO(\prllp).
\]
As $\CO\in\ScO(\prlls)^{\max}$ was arbitrary, we conclude that
\[
\ScO(\prlls)^{\max}\subset\ScO(\prllp).
\]
This proves~\textup{(i)}, and the same argument proves~\textup{(ii)}.
\end{proof}

\begin{rmk}
Lemma~\ref{thetamax}~\textup{(ii)} is used both in the simultaneous induction
above and in the omitted odd case of Theorem~\ref{omax}.
\end{rmk}

\begin{lem}\label{theta-strict}
Retain the notation of the proof of Theorem~\ref{omax}, in particular
\eqref{Lambda1p}--\eqref{thetaps2}. For every irreducible constituent
\(\prlls\neq\prllp\) occurring in \eqref{thetaps}, the partition
attached to \(\ScO(\prlls)^{\max}\) is strictly smaller than that
attached to \(\ScO(\prllp)^{\max}\). Moreover, the former has more
than \(2n\) rows, whereas the latter has exactly \(2n\) rows.
\end{lem}

\begin{proof}
By the definition of \(\Lambda_1'\),
\[
\Upsilon(\Lambda_1')^+
=
\bigl(k-m_2,\Upsilon(\Lambda_1)^-\bigr),
\qquad
\Upsilon(\Lambda_1')^-=\Upsilon(\Lambda_1)^+.
\]
The second inequality in \eqref{thetaps2} means that each part of
\(\bigl(\Upsilon(\Lambda_1^\star)^+\bigr)^t\) exceeds the
corresponding part of
\(\bigl(\Upsilon(\Lambda_1)^-\bigr)^t\) by at most one. Therefore
\[
\bigl(\Upsilon(\Lambda_1^\star)^+\bigr)_1
\geq
\left|\Upsilon(\Lambda_1^\star)^+\right|
-
\left|\Upsilon(\Lambda_1)^-\right|.
\]
By \eqref{thetaps1} and the definition of \(\prllp\),
\(\prlls\) and \(\prllp\) have the same \(\rho\)-part and the same
\((-1)\)-symbol. Since both are
representations of \(\Sp_{4n}(\mathbbm{k})\), the parameter rank
formula gives
\(
\rank(\Lambda_1^\star)=\rank(\Lambda_1').
\)
Moreover, by \eqref{Lambda1p} and \eqref{thetaps1}, they also have the same defect.
Hence their associated bipartitions have the same total size.
Consequently,
\[
\begin{aligned}
\bigl(\Upsilon(\Lambda_1^\star)^+\bigr)_1
&\geq
\left|\Upsilon(\Lambda_1^\star)^+\right|
-
\left|\Upsilon(\Lambda_1)^-\right|\\
&=
k-m_2+
\left|\Upsilon(\Lambda_1)^+\right|
-
\left|\Upsilon(\Lambda_1^\star)^-\right|\\
&\geq k-m_2,
\end{aligned}
\]
where the last inequality follows from the first inequality in
\eqref{thetaps2}.

Suppose that
\(
\bigl(\Upsilon(\Lambda_1^\star)^+\bigr)_1=k-m_2.
\)
Then equality holds throughout the preceding inequalities. In
particular,
\(
\left|\Upsilon(\Lambda_1^\star)^-\right|
=
\left|\Upsilon(\Lambda_1)^+\right|,
\)
and hence the first relation in \eqref{thetaps2} gives
\[
\Upsilon(\Lambda_1^\star)^-
=
\Upsilon(\Lambda_1)^+.
\]
Moreover,
\[
\left|\Upsilon(\Lambda_1^\star)^+\right|
-
\left|\Upsilon(\Lambda_1)^-\right|
=
\bigl(\Upsilon(\Lambda_1^\star)^+\bigr)_1
=
k-m_2.
\]
By the second relation in \eqref{thetaps2}, the corresponding parts of
\(\bigl(\Upsilon(\Lambda_1^\star)^+\bigr)^t\) and
\(\bigl(\Upsilon(\Lambda_1)^-\bigr)^t\) differ by either zero or one.
Since the former has exactly \(k-m_2\) nonzero parts and the total
difference in size is \(k-m_2\), each of its first \(k-m_2\) parts
exceeds the corresponding part of the latter by one. Therefore
\[
\Upsilon(\Lambda_1^\star)^+
=
\bigl(k-m_2,\Upsilon(\Lambda_1)^-\bigr).
\]
Thus \(\Lambda_1^\star=\Lambda_1'\). Consequently,
\[
\bigl(\Upsilon(\Lambda_1^\star)^+\bigr)_1>k-m_2
\qquad
\text{if }\prlls\neq\prllp.
\]

By the rank condition in \eqref{Lambda1p}, we have
\(
k=n+\frac{m_1+m_2}{2}.
\)
If \(m_1>0\), then the rank formula for \(\Lambda_1\) gives
\[
a_1\leqslant n+\frac{m_1+m_2}{2}-1=k-1.
\]
If \(m_1=0\), then \(\Upsilon(\Lambda_1)^+\) is empty. Thus, in
either case,
\[
\bigl(\Upsilon(\Lambda_1)^+\bigr)_1<k-m_1+1,
\]
where the first part of the empty partition is understood to be zero.
The first relation in \eqref{thetaps2} therefore gives
\[
\begin{aligned}
\bigl(\Upsilon(\Lambda_1^\star)^-\bigr)_1
&\leq
\bigl(\Upsilon(\Lambda_1)^+\bigr)_1\\
&<
k-m_1+1\\
&=
k-m_2+{\rm def}(\Lambda_1^\star)\\
&\leq
\bigl(\Upsilon(\Lambda_1^\star)^+\bigr)_1
+
{\rm def}(\Lambda_1^\star).
\end{aligned}
\]
By Alvis--Curtis duality,
\[
\Upsilon({}^{\bd}\Lambda_1^\star)^+
=
\bigl(\Upsilon(\Lambda_1^\star)^-\bigr)^t,
\qquad
\Upsilon({}^{\bd}\Lambda_1^\star)^-
=
\bigl(\Upsilon(\Lambda_1^\star)^+\bigr)^t.
\]
The preceding strict inequality shows that the lower row of
\({}^{\bd}\Lambda_1^\star\) has
\(\bigl(\Upsilon(\Lambda_1^\star)^+\bigr)_1\) nonzero entries,
whereas the upper row has
\(
\bigl(\Upsilon(\Lambda_1^\star)^+\bigr)_1
+
\mathrm{def}(\Lambda_1^\star)
\)
entries, the last of which is zero. Hence
\[
|{}^{\bd}\Lambda_1^\star|
=
2\bigl(\Upsilon(\Lambda_1^\star)^+\bigr)_1
+
\mathrm{def}(\Lambda_1^\star)-1.
\]
For \(\prllp\), with the same calculation, we have
\[
|{}^{\bd}\Lambda_1'|
=
2(k-m_2)+\mathrm{def}(\Lambda_1')-1
=
2n.
\]
If \(\prlls\neq\prllp\), then
\(
\bigl(\Upsilon(\Lambda_1^\star)^+\bigr)_1>k-m_2.
\)
Since
\(
\mathrm{def}(\Lambda_1^\star)
=
\mathrm{def}(\Lambda_1'),
\)
it follows that
\[
|{}^{\bd}\Lambda_1^\star|
>
2(k-m_2)+\mathrm{def}(\Lambda_1')-1
=
2n.
\]
Since the \(\rho\)-part and the \(\Lambda_{-1}\)-part of both
\(\prllp\) and \(\prlls\) are inherited from those of \(\prll\), and
\(\prll\) is a representation of \(\RO^+_{2n}(\mathbbm{k})\),
Theorem~\ref{max}\textup{(i)} shows that these two parts can contribute
at most \(2n\) rows to the partitions attached to
\(\ScO(\prllp)^{\max}\) and \(\ScO(\prlls)^{\max}\). Therefore the
partition attached to \(\ScO(\prllp)^{\max}\) has exactly \(2n\)
rows, whereas that attached to \(\ScO(\prlls)^{\max}\) has more than
\(2n\) rows whenever \(\prlls\neq\prllp\). Together with
\eqref{sip} and Corollary~\ref{7.2}\textup{(ii)}, this shows that the
latter partition is strictly smaller than the former.
\end{proof}

We now define descent sequence indexes for orthogonal groups. Let $V$ be
a nondegenerate symmetric space. 
For $\ell>0$ and $a\in\{\pm\}$, let
$\CO_{2\ell+1,a}\subset\mathfrak{o}(V)^F$ be the corresponding
$F$-rational nilpotent orbit. For an $F$-rational $\mathfrak{sl}_2$-triple representing
$\CO_{2\ell+1,a}$, let $V_{\ell,a}$ be the nondegenerate multiplicity
space of the one-dimensional $\mathfrak{sl}_2$-summands in $V$. Thus
\[
\dim V_{\ell,a}=\dim V-2\ell-1.
\]
For any
\(
\pi\in\Irr(\RO(V)^F)
\), define
\[
\CD^{\rm B}_{2\ell+1,a}(\pi)
:=
\operatorname{Wh}_{\CO_{2\ell+1,a}}(\pi)
\]
and regard it as a representation of $\RO(V_{\ell,a})^F$.

For $\ell=0$ and $\epsilon\in\{\pm\}$, choose an orthogonal
decomposition
\[
V=V'\oplus V'',
\qquad
\dim V''=1,
\qquad
\disc(V'')=\epsilon,
\]
and define
\[
\CD^{\rm B}_{1,\epsilon}(\pi)
:=
\Res_{\RO(V')^F}^{\RO(V)^F}\pi.
\]

A descent sequence of $\pi\in\Irr(\RO(V_0)^F)$ is a sequence
\[
\pi=\pi_0
\xrightarrow{(\ell_1,a_1)}
\pi_1
\xrightarrow{(\ell_2,a_2)}
\cdots
\xrightarrow{(\ell_k,a_k)}
\pi_k=\mathbbm{1},
\]
where $\pi_i\in\Irr(\RO(V_i)^F)$ is an irreducible constituent of
\(
\CD^{\rm B}_{2\ell_i+1,a_i}(\pi_{i-1})
\) and $\mathbbm{1}$ is the trivial representation of the trivial orthogonal group.
Here, if $\ell_i>0$, then
\(
V_i=(V_{i-1})_{\ell_i,a_i},
\)
whereas, if $\ell_i=0$, then
\(
V_{i-1}=V_i\oplus L_i\) with
\(\dim L_i=1\) and 
\( \disc(L_i)=a_i
\).
 The corresponding descent sequence index is
\[
\Gamma
=
\llb
(2\ell_1+1,a_1),\ldots,(2\ell_k+1,a_k)
\rrb .
\]
The notions of goodness, the mutation order, and the orbit
$\CO(\Gamma)$ are defined as in Section~\ref{sec4}, with each
$(2\ell,a)$ replaced by $(2\ell+1,a)$. In particular, maximality refers
to this mutation order.

\begin{thm}\label{maxo1}
    Let $G=\RO^\epsilon_{2n}$, and $\prll\in \Irr(G^F)$. Assume that $\rho=\prod_{[a]\ne[\pm1]}\pi[a]$, where $\pi[a]$ is the irreducible representation corresponding to the partition $\lambda[a]=[\lambda[a]_1,\lambda[a]_2,\cdots,\lambda[a]_{k[a]}]$. 
   \begin{enumerate}
       \item[(i)]
    Let $\llb(2\ell_1+1,a_1),(2\ell_2+1,a_2),\cdots,(2\ell_k+1,a_k)\rrb$ be a maximal descent sequence index of $\prll$. For each \(i\), set
    \begin{itemize}
        \item $\ell[\neq\pm1]_i:=\sum_{[a]\ne[\pm1]}\widetilde{\#[a]}\lambda[a]^t_i$,
        \item $ \ell[1]_i:=\rank({}^{\bd}\Lambda_1^{i-1})-\rank({}^{\bd}\Lambda_1^{i})$,
        \item $\ell[-1]_{i}:=\rank({}^{\bd}\Lambda_{-1}^{i-1})-\rank({}^{\bd}\Lambda_{-1}^{i})$.
    \end{itemize}
    Then
    \begin{align*}
    &\ell_i
=
\begin{cases}
\ell[\ne\pm1]_i+\ell[1]_i+\ell[-1]_i-1,
& i\ \textrm{odd},\\
\ell[\ne\pm1]_i+\ell[1]_i+\ell[-1]_i,
& i\ \textrm{even},
\end{cases} \\
    & a_i=\epsilon_{\rho_{i-1}}\epsilon_{\rho_{i}}\zeta({}^{\bd}\Lambda_1^{i-1})\zeta({}^{\bd}\Lambda_{-1}^{i-1})\varepsilon(-1)^{i-1}, \\
    & k=\max\left\{\max_{[a]\ne[\pm1]}k[a],|{}^{\bd}\Lambda_1|,|{}^{\bd}\Lambda_{-1}|\right\}.
    \end{align*}

    \item[(ii)] We have
    \[
    \ScO(\prll)^{\max}=\bigcup_{\textrm{maximal descent sequence indexes }\Gamma\textrm{ of }\prll} \CO(\Gamma).
    \]
   Moreover, \(\ScO(\prll)^{\max,\mathrm{st}}\) is contained in a unique
\(F\)-stable nilpotent orbit, whose partition
\(\lambda=[\lambda_1,\ldots,\lambda_k]\) is attached to any maximal
descent sequence index
\(
\Gamma
=
\llb(2\ell_1+1,a_1),\ldots,(2\ell_k+1,a_k)\rrb,
\)
with \(\lambda_i\) corresponding to \(\ell_i\).

\item[(iii)] 
Define
\[
\mathtt{N}:=\Set{1\leqslant i<k | \begin{aligned}
& \ell[\ne\pm 1]_{i}=\ell[\ne\pm 1]_{i+1},\\
& \ell[ \delta]_{i}=\ell[\delta]_{i+1}-1, \ \ell[ -\delta]_{i}=\ell[-\delta]_{i+1} \textrm{ for some }\delta \in \{\pm 1\}
\end{aligned}}.
\]
Then \(\lambda_i\) is even precisely when either \(i\) or \(i-1\)
belongs to \(\mathtt N\).

\item[(iv)] Define
\[
\mathtt{M}:=\Set{1\leqslant i<k | \begin{aligned}
& \ell[\ne\pm 1]_{i}>\ell[\ne\pm 1]_{i+1},\\
& \ell[ \delta]_{i}=\ell[\delta]_{i+1}-1, \ \ell[ -\delta]_{i}=\ell[-\delta]_{i+1} \textrm{ for some }\delta \in \{\pm 1\}
\end{aligned}}.
\]
Then, for $\CO \in \ScO(\prll)^{\max}$, we have
\[
\dim {\mathrm{Wh}}_{\CO}(\prll)=2^{\# \mathtt{M}}.
\]
In particular, the dimension above equals 1 for quadratic unipotent representations ($\ell[\ne\pm 1]_{i}\equiv 0$).

\item[(v)] Let $\pi\in \cal{E}(G,s)$,   $\CO\in \SOM$ and $\SCO=\SOMS$. Let $\pi^*=\cal{L}^{\rm can}(\pi)\in \cal{E}(C_{G^*}(s),1)$ and $\SCO^*=\mathscr{O}(\pi^*)^{\rm{max},\rm{st}}$.
Then
\[
{\rm{dim}}{\rm Wh}_{\CO}(\pi)={\frac{|A(\SCO)|}{|\SOM||\overline{A}(\SCO^*)|}},
\]
where $\overline{A}(\SCO^*)$ is Lusztig's canonical quotient of
$\SCO^*$ in $C_{G^*}(s)$.
\end{enumerate}
\end{thm}
\begin{proof}
For the split even orthogonal case, let $\prllp$ be the representation of $\Sp_{4n}(\mathbbm{k})$
constructed in the proof of Theorem~\ref{omax}. As shown in the proof of Theorem~\ref{omax}, the maximal rational
orbits of $\prll$ are exactly the generalized descents of those of
$\prllp$. The explicit formula in
Proposition~\ref{orbit} then transfers the descent sequence indexes and
stable partitions, so \textup{(i)}--\textup{(iii)} follow from
Theorem~\ref{max}. 

For \textup{(iv)}, let
\(\CO'\in\ScO(\prllp)^{\max}\) correspond to
\(\CO\in\ScO(\prll)^{\max}\) under generalized descent. By
Lemma~\ref{theta-strict} and
Corollary~\ref{7.2}\textup{(ii)}, \(\prllp\) is the only irreducible constituent
of the theta lift in \eqref{thetaps} having a nonzero
\(\CO'\)-Whittaker model. Moreover, the partition of \(\CO'\) has
exactly \(2n\) rows. Then by
Proposition~\ref{orbit} and Theorem~\ref{gz}\textup{(i)},
\[
\dim\mathrm{Wh}_{\CO}(\prll)
=
\dim\mathrm{Wh}_{\CO'}(\prllp).
\]
Proposition~\ref{orbit} identifies the corresponding descent sequence
indexes, and hence the sets \(\mathtt M\). Therefore \textup{(iv)}
follows from Theorem~\ref{max}\textup{(iv)}.

Finally, Proposition~\ref{orbit} identifies the stable and rational
orbit data under generalized descent. Hence \textup{(v)} follows from
Theorem~\ref{dimW} and the description of the canonical quotient in
Section~\ref{lcq}.

The remaining even orthogonal cases are analogous, using the
corresponding Witt tower and interchanging the two rows of the
relevant symbol when necessary.
\end{proof}

\begin{thm}\label{maxo2}
    Let $G=\RO^{\epsilon'}_{2n+1}$, and $\prllz\in \Irr(G^F)$. Assume that $\rho=\prod_{[a]\ne[\pm1]}\pi[a]$, where $\pi[a]$ is the irreducible representation corresponding to the partition $\lambda[a]=[\lambda[a]_1,\lambda[a]_2,\cdots,\lambda[a]_{k[a]}]$.
 \begin{enumerate}
    \item[(i)]
    Let $\llb(2\ell_1+1,a_1),(2\ell_2+1,a_2),\cdots,(2\ell_k+1,a_k)\rrb$ be a maximal descent sequence index of $\prllz$. For each \(i\), set
    \begin{itemize}
        \item $\ell[\ne \pm1]_i:=\sum_{[a]\ne[\pm1]}\widetilde{\#[a]}\lambda[a]^t_i$,
        \item $ \ell[1]_i:=\rank({}^{\bd}\Lambda_1^{i-1})-\rank({}^{\bd}\Lambda_1^{i})$,
        \item $\ell[-1]_{i}:=\rank({}^{\bd}\Lambda_{-1}^{i-1})-\rank({}^{\bd}\Lambda_{-1}^{i})$.
    \end{itemize}
    Then
    \begin{align*}
    & \ell_i=
\begin{cases}
\ell[\ne\pm1]_i+\ell[1]_i+\ell[-1]_i,
& i\ \textrm{odd},\\
\ell[\ne\pm1]_i+\ell[1]_i+\ell[-1]_i-1,
& i\ \textrm{even},
\end{cases} \\
    & a_i=\epsilon'\epsilon_\rho\epsilon_{\rho_{i-1}}\epsilon_{\rho_{i}}\zeta({}^{\bd}\Lambda_1^{i-1})\zeta({}^{\bd}\Lambda_{-1}^{i-1})\varepsilon(-1)^{i-1}, \\
    & k=\max\left\{\max_{[a]\ne[\pm1]}k[a],|{}^{\bd}\Lambda_1|,|{}^{\bd}\Lambda_{-1}|\right\}.
    \end{align*}

    \item[(ii)] We have
    \[
    \ScO(\prllz)^{\max}=\bigcup_{\textrm{maximal descent sequence indexes }\Gamma\textrm{ of }\prllz} \CO(\Gamma).
    \]
Moreover, \(\ScO(\prllz)^{\max,\mathrm{st}}\) is contained in a unique
\(F\)-stable nilpotent orbit, whose partition
\(\lambda=[\lambda_1,\ldots,\lambda_k]\) is attached to any maximal
descent sequence index
\(
\Gamma
=
\llb(2\ell_1+1,a_1),\ldots,(2\ell_k+1,a_k)\rrb,
\)
with \(\lambda_i\) corresponding to \(\ell_i\).

\item[(iii)] 
Define
\[
\mathtt{N}:=\Set{1\leqslant i<k | \begin{aligned}
& \ell[\ne\pm 1]_{i}=\ell[\ne\pm 1]_{i+1},\\
& \ell[ \delta]_{i}=\ell[\delta]_{i+1}-1, \ \ell[ -\delta]_{i}=\ell[-\delta]_{i+1} \textrm{ for some }\delta \in \{\pm 1\}
\end{aligned}}.
\]
Then a part $\lambda_i$ is even precisely when either $i$ or $i-1$ belongs to $\mathtt{N}$. 

\item[(iv)] Define
\[
\mathtt{M}:=\Set{1\leqslant i<k | \begin{aligned}
& \ell[\ne\pm 1]_{i}>\ell[\ne\pm 1]_{i+1},\\
& \ell[ \delta]_{i}=\ell[\delta]_{i+1}-1, \ \ell[ -\delta]_{i}=\ell[-\delta]_{i+1} \textrm{ for some }\delta \in \{\pm 1\}
\end{aligned}}.
\]
Then, for $\CO \in \ScO(\prllz)^{\max}$, we have
\[
\dim {\mathrm{Wh}}_{\CO}(\prllz)=2^{\#\mathtt{M}}.
\]
In particular, the dimension above equals 1 for quadratic unipotent representations ($\ell[\ne\pm 1]_{i}\equiv 0$).

\item[(v)] Let $\pi\in \cal{E}(G,s)$,   $\CO\in \SOM$ and $\SCO=\SOMS$. Let $\pi^*=\cal{L}^{\rm can}(\pi)\in \cal{E}(C_{G^*}(s),1)$ and $\SCO^*=\mathscr{O}(\pi^*)^{\rm{max},\rm{st}}$.
Then
\[
{\rm{dim}}{\rm Wh}_{\CO}(\pi)={\frac{|A(\SCO)|}{2|\SOM||\overline{A}(\SCO^*)|}},
\]
where $\overline{A}(\SCO^*)$ is Lusztig's canonical quotient of
$\SCO^*$ in $C_{G^*}(s)$.
\end{enumerate}
\end{thm}

\begin{proof}
The proof is the same as that of Theorem~\ref{maxo1}. Let $s$ be a
semisimple element such that
\(
\prllz\in\mathcal E(\RO^{\eps'}_{2n+1}(\mathbbm{k}),s),
\)
and choose $s'\in\SO_{4n+3}(\mathbbm{k})$ such that the pair $(s',s)$ satisfies
condition~\textup{(2)} of Theorem~\ref{thm:L-can}. Write
\[
\Lambda_{1}
=
\begin{pmatrix}
c_1,c_2,\ldots,c_{m_1}\\
d_1,d_2,\ldots,d_{m_2}
\end{pmatrix}.
\]
Let
\(
\prllp\in\mathcal E(\Sp_{4n+2}(\mathbbm{k}),s')
\)
be the representation whose $\rho'$ and $\Lambda_1'$ are determined by
Theorem~\ref{thm:L-can}, and whose remaining symbol is
\[
\Lambda_{-1}'
=
\begin{pmatrix}
h,d_1,d_2,\ldots,d_{m_2}\\
c_1,c_2,\ldots,c_{m_1}
\end{pmatrix},
\qquad
\rank(\Lambda_{-1}')=\rank(\Lambda_{1})+n+1.
\]
The argument in the proof of Theorem~\ref{omax} gives the analogue of
\eqref{eq8.8} for $\prllz$ and $\prllp$. The claims then follow exactly as in the proof of
Theorem~\ref{maxo1}. The proof of
Lemma~\ref{theta-strict}, with the roles of the \(1\)- and
\((-1)\)-symbols interchanged, gives the corresponding strictness
assertion. The additional factor
$2$ in \textup{(v)} comes from the two $F$-rational quadratic spaces of
dimension $2n+1$, which give rise to the same odd orthogonal group.
\end{proof}

\section{Rational wavefront sets of unipotent representations and canonical quotients}\label{sec:un}

We now specialize the preceding results to unipotent representations.
The main purpose of this section is to describe their maximal rational
wavefront sets in terms of Lusztig's canonical quotient and
Alvis--Curtis duality.

\subsection{Preparatory results for the unipotent case}

We begin with two preparatory results. 
The following proposition specializes
Theorems~\ref{max}, \ref{maxo1}, and~\ref{maxo2} to unipotent
representations and records both the rational sign pattern and the relations between
the partition and the entries of ${}^{\mathbf d}\Lambda$ that
will be used in the sequel.

\begin{prop}\label{un}
Let $\pi_{\Lambda,\eta}$ be an irreducible unipotent representation of
$G^F$, where the additional parameter $\eta\in\{\pm\}$ occurs only for
odd orthogonal groups. Assume that the partition of $\SOMS$ is
\[
\lambda=[\lambda_1,\ldots,\lambda_k]
       =[\mu_1^{i_1},\mu_2^{i_2},\ldots,\mu_r^{i_r}].
\]
Let
\(
c_1\geq c_2\geq\cdots\geq c_k
\)
be the entries of ${}^{\mathbf d}\Lambda$, counted with multiplicity and
arranged in nonincreasing order.

\begin{enumerate}
\item[(i)]
An $F$-rational nilpotent orbit
\[
\CO=[(\mu_1^{i_1},\epsilon_1),\ldots,
     (\mu_r^{i_r},\epsilon_r)]
\]
lies in $\SOM$ if and only if there exist one-dimensional quadratic
spaces $V_1,\ldots,V_k$, with $v_j=\disc(V_j)$, satisfying the following
two conditions. First,
\begin{equation}\label{v}
(v_1,\ldots,v_k)
=
\Bigl(
\zeta({}^{\mathbf d}\Lambda^{0})x_1,\,
\varepsilon(-1)\zeta({}^{\mathbf d}\Lambda^{1})x_2,\,
\ldots,\,
\varepsilon(-1)^{k-1}
\zeta({}^{\mathbf d}\Lambda^{k-1})x_k
\Bigr),
\end{equation}
where
\[
x=
\begin{cases}
(y_1,-y_1,y_2,-y_2,\ldots),
   & G=\Sp_{2n}\text{ or }G=\mathrm O_{2n}^{\epsilon},\\
(\epsilon,y_1,-y_1,y_2,-y_2,\ldots),
   & G=\mathrm O_{2n+1}^{\epsilon},
\end{cases}
\]
with $y_i\in\{\pm\}$, truncated after its $k$-th entry. Second,
\[
\epsilon_m
=
\disc\left(
\bigoplus_{j=1+\sum_{a=1}^{m-1}i_a}^{\sum_{a=1}^{m}i_a}V_j
\right),
\qquad 1\leq m\leq r.
\]

\item[(ii)]
Suppose that $G=\Sp_{2n}$. Then:
\begin{enumerate}
\item[(a)]
$\lambda_{2i-1}=\lambda_{2i}$ is odd if and only if
\(
c_{2i-1}=c_{2i}.
\)

\item[(b)]
$\lambda_{2i-1}=\lambda_{2i}$ is even if and only if
\(
c_{2i-1}=c_{2i}+1.
\)

\item[(c)]
$\lambda_{2i}=\lambda_{2i+1}$ is even if and only if
\(
c_{2i}=c_{2i+1}.
\)
\end{enumerate}

\item[(iii)]
Suppose that $G=\mathrm O_{2n}^{\epsilon}$. Then:
\begin{enumerate}
\item[(a)]
$\lambda_{2i-1}=\lambda_{2i}$ is odd if and only if
\(
c_{2i-1}=c_{2i}+1.
\)

\item[(b)]
$\lambda_{2i-1}=\lambda_{2i}$ is even if and only if
\(
c_{2i-1}=c_{2i}.
\)

\item[(c)]
$\lambda_{2i}=\lambda_{2i+1}$ is odd if and only if
\(
c_{2i}=c_{2i+1}.
\)
\end{enumerate}

\item[(iv)]
Suppose that $G=\mathrm O_{2n+1}^{\epsilon}$. Then:
\begin{enumerate}
\item[(a)]
$\lambda_{2i-1}=\lambda_{2i}$ is odd if and only if
\(
c_{2i-1}=c_{2i}.
\)

\item[(b)]
$\lambda_{2i}=\lambda_{2i+1}$ is even if and only if
\(
c_{2i}=c_{2i+1}.
\)

\item[(c)]
$\lambda_{2i}=\lambda_{2i+1}$ is odd if and only if
\(
c_{2i}=c_{2i+1}+1.
\)
\end{enumerate}
\end{enumerate}
\end{prop}

\begin{sop}
For a unipotent representation, the non-quadratic part is absent,
$\Lambda_1=\Lambda$, and $\Lambda_{-1}$ is the trivial symbol
$\binom{-}{-}$ or $\binom{0}{-}$. We first determine the sign components of the maximal descent sequence
indexes. By Theorems~\ref{max}\textup{(i)},
\ref{maxo1}\textup{(i)}, and \ref{maxo2}\textup{(i)}, the contributions
coming from $\varepsilon(-1)$ and from the rows containing the successive
largest entries of ${}^{\mathbf d}\Lambda$ give the factors
\[
\varepsilon(-1)^{i-1}
\zeta({}^{\mathbf d}\Lambda^{i-1})
\]
in \eqref{v}. It remains to determine the contribution of the trivial
$(-1)$-part.

For symplectic and even orthogonal groups, the contribution of the
trivial $(-1)$-part is computed as in Example~\ref{dex2}.  Under the
exhausted-symbol convention it is represented during the iteration
by $\binom{0}{0}$.  The two zeros are removed successively from
opposite rows, and hence its signs have the form
\[
(y_1,-y_1,y_2,-y_2,\ldots).
\]

For $G=\mathrm O_{2n+1}^{\epsilon}$, the contribution of the $(-1)$-part
in Theorem~\ref{maxo2}\textup{(i)} also involves the factor $\epsilon$
determined by the form of the odd orthogonal group. Since
\(
\Lambda_{-1}=\binom{0}{-}\) and
\(
\zeta({}^{\mathbf d}\Lambda_{-1})=+,
\)
the product
\(
\epsilon\cdot\zeta({}^{\mathbf d}\Lambda_{-1})
\)
gives the initial contribution $\epsilon$. Thereafter the $(-1)$-part has the same paired form as above.
Since the $y_i$ are independent free signs, all further fixed
$\epsilon$-factors can be absorbed into their choice.

These computations determine the sign parts of the maximal descent
sequence indexes of $\pi_{\Lambda,\eta}$. Applying the descent--orbit correspondence of
Section~\ref{sec4.1} gives the rational signs in \textup{(i)}. The numerical parts of the descent indexes are given by the
descent-index formulas in the three theorems above. Comparing consecutive
descent indexes determines when two consecutive parts of $\lambda$ are
equal and their common parity, giving
\textup{(ii)}--\textup{(iv)}.
\end{sop}

Although generalized Whittaker models depend on the choice of additive
character, Proposition~\ref{un}\textup{(i)} shows that the
maximal rational wavefront set of a unipotent representation does not.

\begin{thm}\label{indep}
Let $\pi$ be an irreducible unipotent representation of $G^F$. Then
$\SOM$ is independent of the choice of the nontrivial additive character
$\psi:\mathbbm{k}\to\mathbb C^\times$.
\end{thm}

\begin{proof}
Write $\SOM_{\psi}$ for the maximal rational wavefront set defined using
$\psi$. Let
\(
\psi'(x)=\psi(\mathfrak a x)\) with \(
\mathfrak a\in\mathbbm{k}^{\times}\).
If $\mathfrak a$ is a square, the assertion is immediate. We may therefore
assume that $\mathfrak a$ is nonsquare. For
\(
\CO=[(\mu_1^{i_1},\epsilon_1),\ldots,
     (\mu_r^{i_r},\epsilon_r)],
\)
set
\[
\CO'=
[(\mu_1^{i_1},(-1)^{i_1}\epsilon_1),\ldots,
 (\mu_r^{i_r},(-1)^{i_r}\epsilon_r)].
\]

Suppose first that $G$ is symplectic or even orthogonal. Then $\CO'$ is
again an $F$-rational orbit of $G^F$, and
\[
\operatorname{Wh}_{\CO,\psi}(\pi)
\cong
\operatorname{Wh}_{\CO',\psi'}(\pi).
\]
Here $\operatorname{Wh}_{\CO,\psi}(\pi)$ denotes the generalized
Whittaker model of $\pi$ with respect to
$\CO$ and $\psi$.
Thus $\CO\mapsto\CO'$ carries $\SOM_{\psi}$ onto
$\SOM_{\psi'}$. On the other hand, negating all one-dimensional signs
preserves the paired pattern
\(
(y_1,-y_1,y_2,-y_2,\ldots).
\)
Hence Proposition~\ref{un}\textup{(i)} shows that the map $\CO\mapsto\CO'$ permutes
the elements of $\SOM_{\psi}$. It follows that
$\SOM_{\psi}=\SOM_{\psi'}$ in the symplectic and even orthogonal cases.

Now let
\(
G=\mathrm O_{2n+1}^{\epsilon}\) and \(
G'=\mathrm O_{2n+1}^{-\epsilon}\).
Then $\CO$ is an $F$-rational orbit of $G^F$, whereas $\CO'$ is an
$F$-rational orbit of $G^{\prime F}$. Note that $G^F$ and $G^{\prime F}$ are isomorphic as abstract groups.
We fix such an isomorphism and identify their representations
accordingly.

Choose an irreducible unipotent representation
$\pi_0=\pi_{\Lambda_0}$ of an even orthogonal group such that
${}^{\mathbf d}\Lambda_0$ is obtained from
${}^{\mathbf d}\Lambda$ by adjoining a sufficiently large entry, where
$\Lambda$ is the symbol corresponding to $\pi$. Then,
with respect to $\psi$, its first descent is
\(
\pi_0\xrightarrow{(\ell,a)}\pi,
\)
where the representation $\pi$ on the right is regarded as a
representation of $G^F$. The conjugation argument in
the proof of \cite[Theorem~5.13\textup{(iii)}]{Wang23} gives, with the
same additive character $\psi$, the descent
\(
\pi_0\xrightarrow{(\ell,-a)}\pi^c,
\)
where $\pi^c$ is regarded as a
representation of $G^{\prime F}$. Here
\[
\pi^c(x)=\pi(c^{-1}xc),\qquad x\in G^{\prime F},
\]
with $c$ a conformal orthogonal transformation having nonsquare
similitude factor. Since $\pi$ is unipotent, $\pi^c=\pi$ under the above
identification.

The even orthogonal case applied to $\pi_0$, together with the
composition law for descents and the two descents above, now shows that
the correspondence $\CO\mapsto\CO'$ carries $\SOM_{G,\psi}$ onto
$\SOM_{G',\psi'}$. Indeed, the change $a\mapsto -a$ changes the initial
sign from $\epsilon$ to $-\epsilon$, while the remaining
one-dimensional signs are negated. Thus the pattern
\(
(\epsilon,y_1,-y_1,y_2,-y_2,\ldots)
\)
is carried to the corresponding paired pattern for $G'$. By
Proposition~\ref{un}\textup{(i)}, this is precisely the relabelling of
the maximal rational orbits induced by scaling the defining quadratic
form. The assertion follows from the identification
$G^F\cong G^{\prime F}$.
\end{proof}

Henceforth, we suppress the additive character from the notation for
maximal rational wavefront sets.

\subsection{Canonical quotients and Alvis--Curtis duality}

Recall that the Kawanaka wavefront set of an irreducible
representation consists of a single $F$-stable orbit.
Although $\SOM$ may contain several $F$-rational orbits, we show that, for a unipotent representation, they all determine the same element of the canonical quotient of this stable orbit.

Recall that Section~\ref{lcq} defines a bijection $S$ from the relevant
rational orbits attached to $\SCO$ to $A(\SCO)$, and hence a surjection
$\overline S$ onto $\overline A(\SCO)$. For
$x\in\overline A(\SCO)$, let
\[
\mathscr O(x)=\overline S^{-1}(x).
\]
In the full orthogonal cases, the domain of $S$ is the disjoint union
of the rational orbits for the two quadratic forms. In even dimension,
for a rational orbit $\CO$ on either form, $\SCO$ denotes the
$F$-stable nilpotent orbit in the split form with the same underlying
partition as $\CO$.

\begin{thm}\label{un1}
Let $\pi$ be an irreducible unipotent representation of $G^F$ such that
$\SOMS=\SCO$. Then there exists $x\in\overline A(\SCO)$ such that
\[
\SOM=\mathscr O(x).
\]
If $\pi$ is special, regarded in the odd orthogonal case as a
representation of $\RO_{2n+1}^{+}(\mathbbm{k})$, then
\[
S(\SOM)=H(\SCO),
\]
where
\[
H(\SCO)
=
\ker\bigl(A(\SCO)\longrightarrow\overline A(\SCO)\bigr).
\]
In particular,
\[
A(\SCO)/S(\SOM)\cong\overline A(\SCO).
\]
\end{thm}

\begin{proof}
We prove the theorem for symplectic groups. The orthogonal cases are
analogous, using Proposition~\ref{un}\textup{(iii)} and
Proposition~\ref{un}\textup{(iv)}, respectively.

We repeatedly use Proposition~\ref{un} and Sommers'
description of the canonical quotient in Section~\ref{lcq}.

Let
\[
\lambda=[\lambda_1,\ldots,\lambda_k]
       =[\mu_1^{i_1},\ldots,\mu_h^{i_h}],
\qquad
\mu_1>\cdots>\mu_h,
\]
be the partition of $\SCO$. Since $\pi$ is unipotent, $\SCO$ is
special. We argue by induction on $k$, the cases $k\leq2$ being
immediate from Proposition~\ref{un}\textup{(i)}.

Whenever $\pi'$ is obtained from $\pi$ by two successive first
descents, Theorem~\ref{max}\textup{(i)} shows that the maximal descent
sequences of $\pi'$ have fewer terms than those of $\pi$. Apply Theorem~\ref{ggp} (ii) twice consecutively. we find that $\pi'$ is again unipotent. Hence the
induction hypothesis applies to $\pi'$. Let $\CO^{\prime\rm st}$ be the unique $F$-stable orbit containing
$\ScO(\pi')^{\max,\mathrm{st}}$, and choose
\[
x'\in\overline A(\CO^{\prime\rm st})
\qquad\text{such that}\qquad
\ScO(\pi')^{\max}=\mathscr O(x').
\]

For every such $\pi'$, choose a maximal descent sequence of $\pi$ passing
through $\pi'$ and write its first three terms as
\(
(2\ell_1,a_1),(2\ell_2,a_2),(2\ell_3,a_3).
\) 
 Then the partition of $\CO^{\prime\rm st}$ is
$[\lambda_3,\ldots,\lambda_k]$. Indeed, only its first part requires
verification. It could differ from $\lambda_3$ only if the second and
third terms of the corresponding descent sequence formed an odd pair.
Then
\[
\ell_2=\ell_3-1,\qquad \ell_1>\ell_2,
\]
and the orbit--sequence correspondence would give
\[
\lambda_1=2\ell_1>
\lambda_2=\lambda_3=2\ell_2+1.
\]
Thus exactly one even part would lie above the largest odd part,
contradicting the specialness of $\SCO$.

We shall repeatedly use the following fiber-counting observation.
\begin{center}
\fbox{\parbox{0.9\textwidth}{
\textit{Fiber-counting principle.}
Suppose that
\[
d:\SOM\longrightarrow\mathscr O(x')
\]
is an \(r\)-to-one surjection,
\(\SOM\subset\mathscr O(x)\), and
\[
|H(\SCO)|=r|H(\CO^{\prime\rm st})|.
\]
Since $S$ is a bijection, we have
\(
|\mathscr O(x)|=|H(\SCO)|\) and \(
|\mathscr O(x')|=|H(\CO^{\prime\rm st})|\).
Then
\[
|\SOM|
=r|\mathscr O(x')|
=r|H(\CO^{\prime\rm st})|
=|H(\SCO)|
=|\mathscr O(x)|,
\]
and hence \(\SOM=\mathscr O(x)\).
}}
\end{center}

Write $\pi=\pi_\Lambda$, and let
$c_1\geq c_2\geq\cdots$ be the entries of
${}^{\mathbf d}\Lambda$. The idea is to remove the first two terms of a maximal descent
sequence. The two possible choices of their rational signs either
give the same rational orbit or two distinct rational orbits. The
resulting fiber size is exactly accounted for by the corresponding
change in $H(\SCO)$.

\medskip
\noindent\textbf{Case 1: \(c_2>c_3\).}

For a maximal descent sequence index
\(
\Gamma=\llb(2\ell_1,a_1),(2\ell_2,a_2),\ldots,
(2\ell_k,a_k)\rrb
\)
of $\pi$, set
\[
d(\Gamma)=
\llb(2\ell_3,a_3),\ldots,(2\ell_k,a_k)\rrb.
\]
By Corollary~\ref{5.6}, the relevant first-descent diagram is
\begin{equation}\label{un1d}
\xymatrix{
& \pi
   \ar[dl]_{(2\ell_1,a_1)}
   \ar[dr]^{(2\ell_1,-a_1)}
&\\
\pi^{(1)}_{+}
   \ar[dr]_{(2\ell_2,a_2)}
&&
\pi^{(1)}_{-}
   \ar[dl]^{(2\ell_2,-a_2)}
\\
& \pi' &
}.
\end{equation}
Together with Theorem~\ref{max}\textup{(i)}, the diagram shows that
every maximal descent sequence index of $\pi'$ extends in exactly two
ways to a maximal descent sequence index of $\pi$, along the two
branches displayed above.
Hence $d$ is a two-to-one surjection on maximal
descent sequence indexes. 
In each of the two subcases below, this induces a surjection
\[
d:\SOM\longrightarrow\ScO(\pi')^{\max};
\]
its well-definedness and fiber size follow from the corresponding
orbit formula.

\smallskip
{\it Case 1 (a): \(\lambda_1=\lambda_2\).}
For
\(
\CO=[(\mu_1^{i_1},b_1),\ldots,(\mu_h^{i_h},b_h)]
\in\SOM,
\)
deleting the first two rows gives
\[
d(\CO)=
\begin{cases}
[(\mu_1^{i_1-2},a_1a_2\varepsilon(-1)b_1),
  (\mu_2^{i_2},b_2),\ldots],&i_1>2,\\
[(\mu_2^{i_2},b_2),\ldots],&i_1=2.
\end{cases}
\]
Indeed, deleting the first two terms removes two rows of length
$\mu_1$ and multiplies the remaining sign by
$a_1a_2\varepsilon(-1)$. Since this factor is unchanged under
\(
(a_1,a_2)\longmapsto(-a_1,-a_2),
\)
the displayed formula shows that $d$ is a bijection on rational
orbits.

If $\lambda_1=\lambda_2$ is odd, the deleted pair of rows contributes
no generator to $A(\SCO)$. Identifying the component-group generators
of $A(\SCO)$ and $A(\CO^{\prime\rm st})$ attached to the same even
parts identifies $H(\SCO)$ with $H(\CO^{\prime\rm st})$. Together
with the formula for $d(\CO)$, this gives an affine
isomorphism
\[
\iota:\overline A(\SCO)\longrightarrow
\overline A(\CO^{\prime\rm st})
\]
for which the diagram
\[
\xymatrix{
\SOM \ar[r]^{d} \ar[d]_{\overline S}
&
\ScO(\pi')^{\max} \ar[d]^{\overline S}
\\
\overline A(\SCO) \ar[r]_{\iota}
&
\overline A(\CO^{\prime\rm st})
}
\]
commutes. Therefore
\[
\SOM=\mathscr O\bigl(\iota^{-1}(x')\bigr).
\]

If $\lambda_1=\lambda_2$ is even, then
$\lambda_2>\lambda_3$, since equality would imply $c_2=c_3$.
Thus the first even block consists of exactly two rows. Let $x_1$ be
the generator of $A(\SCO)$ attached to this block. Its corner has
even height two, and hence
\[
\overline A(\SCO)\cong
\mathbb Z/2\mathbb Z\times\overline A(\CO^{\prime\rm st}),
\]
where the first factor is generated by the image of $x_1$. Under this
identification, \eqref{S} gives
\[
\overline S(\CO)=
\bigl(a_1a_2\varepsilon(-1),\overline S(d(\CO))\bigr).
\]
Consequently,
\[
\SOM\subset
\mathscr O\bigl((a_1a_2\varepsilon(-1),x')\bigr).
\]
The generators attached to all even blocks other than the deleted
first block identify $H(\SCO)$ with
$H(\CO^{\prime\rm st})$.
Since $d$ is a bijection,
the fiber-counting principle gives
\[
\SOM=
\mathscr O\bigl((a_1a_2\varepsilon(-1),x')\bigr).
\]

\smallskip
{\it Case 1 (b): \(\lambda_1>\lambda_2\).}
Here $\lambda_2>\lambda_3$, since equality would imply $c_2=c_3$.
Thus $\lambda_1$ and $\lambda_2$ are distinct even parts, each of
multiplicity one. The simultaneous change
\(
(a_1,a_2)\longmapsto(-a_1,-a_2)
\)
changes the rational signs of two distinct even blocks, so the two
branches in \eqref{un1d} determine distinct rational orbits.
Therefore
\[
d:\SOM\longrightarrow\ScO(\pi')^{\max}
\]
is two-to-one.

The corners corresponding to $\lambda_1$ and $\lambda_2$ have heights
one and two, respectively. Let $x_2$ be the generator of $A(\SCO)$
attached to $\lambda_2$. As in the even subcase of Case~1 (a),
\[
\overline A(\SCO)\cong
\mathbb Z/2\mathbb Z\times\overline A(\CO^{\prime\rm st}),
\]
where the first factor is generated by the image of $x_2$, and
\[
\overline S(\CO)=
\bigl(a_1a_2\varepsilon(-1),\overline S(d(\CO))\bigr).
\]
Hence
\[
\SOM\subset
\mathscr O\bigl((a_1a_2\varepsilon(-1),x')\bigr).
\]
The odd-height corner attached to $\lambda_1$ gives
\(
|H(\SCO)|=2|H(\CO^{\prime\rm st})|.
\)
The fiber-counting principle now yields
\[
\SOM=
\mathscr O\bigl((a_1a_2\varepsilon(-1),x')\bigr).
\]

\medskip
\noindent\textbf{Case 2: \(c_2=c_3\).}
By Proposition~\ref{un}\textup{(ii)}, this is precisely the case in
which $\lambda_2=\lambda_3$ are even; the same proposition shows that
$\lambda_1$ is even as well. By Corollary~\ref{5.6}, the relevant part
of the first-descent diagram has the form
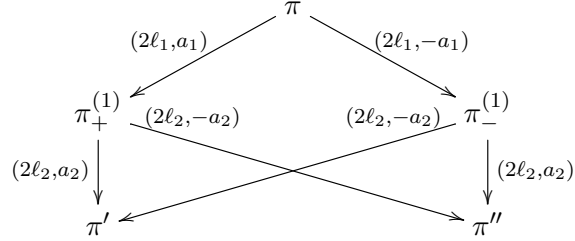
\[
\xymatrix{
&& \pi
  \ar[dll]_{(2\ell_1,a_1)}
  \ar[drr]^{(2\ell_1,-a_1)}
&&\\
\pi^{(1)}_{+}
  \ar[d]_{(2\ell_2,a_2)}
  \ar[drrrr]^(0.2){(2\ell_2,-a_2)}
&&&&
\pi^{(1)}_{-}
  \ar[dllll]_(0.2){(2\ell_2,-a_2)}
  \ar[d]^{(2\ell_2,a_2)}
\\
\pi' && &&\pi''
}
\]
where $\pi'$ and $\pi''$ are obtained according as the entries
$c_2$ and $c_3$, respectively, are removed at the second step.

Although not every maximal descent sequence passes through $\pi'$,
every orbit in $\SOM$ is represented by one that does. Indeed, interchanging the removals of $c_2$ and $c_3$ changes a path
through $\pi''$ into one through $\pi'$. Since the two equal rows
belong to the same even block, whose rational sign depends only on the
product of their signs, this interchange leaves the associated
rational orbit unchanged.
Deleting the first two
terms of such a sequence therefore gives a surjection
\[
d:\SOM\longrightarrow\ScO(\pi')^{\max}.
\]

\smallskip
{\it Case 2 (a):
\(\lambda_1=\lambda_2=\lambda_3\).}
Here $i_1>2$, and the formula for $d(\CO)$ in Case~1 (a) applies unchanged. Thus
$d$ is a bijection on rational orbits. Deleting two rows preserves the component-group generators and changes
all corner heights by two. Hence it identifies $H(\SCO)$ with
$H(\CO^{\prime\rm st})$ and, together with the formula for $d(\CO)$,
gives an isomorphism
\[
\iota:\overline A(\SCO)\longrightarrow
\overline A(\CO^{\prime\rm st})
\]
compatible with $d$. The same argument as in Case~1 (a) gives
\[
\SOM=\mathscr O\bigl(\iota^{-1}(x')\bigr).
\]

\smallskip
{\it Case 2 (b): \(\lambda_1>\lambda_2=\lambda_3\).}
For
\(
\CO=[(\mu_1,b_1),(\mu_2^{i_2},b_2),\ldots]\in\SOM,
\)
we have
\[
d(\CO)=
[(\mu_2^{i_2-1},a_2\varepsilon(-1)b_2),\ldots].
\]
As in Case~1 (b), simultaneous negation of the first two signs changes
the signs of two distinct even blocks, so $d$ is two-to-one on
rational orbits. Deleting the first two rows induces an affine
isomorphism
\[
\iota:\overline A(\SCO)\longrightarrow
\overline A(\CO^{\prime\rm st})
\]
compatible with $d$, while the odd-height corner attached to
$\lambda_1$ gives
\(
|H(\SCO)|=2|H(\CO^{\prime\rm st})|.
\)
The fiber-counting principle therefore gives
\[
\SOM=\mathscr O\bigl(\iota^{-1}(x')\bigr).
\]

Finally, suppose that $\pi$ is special. By the special symbol
condition and Proposition~\ref{un}\textup{(i)}, for every successive
pair in a maximal descent sequence,
\[
a_{2i-1}a_{2i}=\varepsilon(-1).
\]
Each Sommers coordinate is a product over one of these successive
pairs, with the normalization factor in \eqref{S}; hence all such
coordinatesof $\overline S(\CO)$  are $+$ for $\CO\in\SOM$. The
first assertion therefore gives
\[
\overline S(\SOM)=\{1\},
\qquad
S(\SOM)=H(\SCO),
\]
and consequently
\[
A(\SCO)/S(\SOM)\cong\overline A(\SCO).
\]
\end{proof}

We next compare the family realization of the canonical quotient
from Section~\ref{family} with the rational-orbit realization from
Section~\ref{lcq}.  Let $\mathcal F$ be a family, and let
$\CO^{\rm st}_{\mathcal F}$ be the $F$-stable nilpotent orbit
corresponding to its unipotent support. Recall from
Section~\ref{family} that the representations in $\mathcal F$ are
parametrized as
\[
\pi_{\Lambda,\eta}
=
\pi_{\CO^{\rm st}_{\mathcal F},x^+,x^-,\eta},
\qquad
x^\pm\in I^\pm_{\mathcal F},
\]
where the parameter $\eta\in\{\pm\}$ and the corresponding subscript
are present only in the odd orthogonal case. For symplectic and even
orthogonal groups, $I^+_{\mathcal F}$ is identified with
$\overline A(\CO^{\rm st}_{\mathcal F})$. In the odd orthogonal case,
$I^+_{\mathcal F}$ is identified with
$\overline A_{\SO}(\CO^{\rm st}_{\mathcal F})$. Here
$\overline A_{\SO}$ and $\overline A_{\RO}$ denote the canonical
quotients defined using the special and full odd orthogonal groups,
respectively, as in Section~\ref{lcq}.

Let $d=\dim_{\mathbb F_2} I_{\mathcal F}^+$ and write
\[
x^+=(\epsilon_1,\ldots,\epsilon_d),
\qquad
\epsilon_i\in\{\pm\},
\]
with respect to the ordered basis
\[
u_i=
\begin{cases}
e_{2i},
   &\text{if $G$ is symplectic or even orthogonal},\\
e_{2i-1},
   &\text{if $G$ is odd orthogonal},
\end{cases}
\]
where $e_i$ is defined in \eqref{ei}. We call
$\epsilon_1,\ldots,\epsilon_d$ the family coordinates of $x^+$.

In the notation of Section~\ref{lcq}, let $g_i=x_{r_i}$ be the ordered
Sommers generators of the relevant canonical quotient. We call the
coordinates with respect to these generators the Sommers coordinates.
Define $\Phi_{\mathcal F}$ by:
\begin{equation}\label{eq:symbol-Sommers}
\Phi_{\mathcal F}
(\epsilon_1,\ldots,\epsilon_d)
=
\begin{cases}
(\epsilon_1,\epsilon_1\epsilon_2,\ldots,
 \epsilon_{d-1}\epsilon_d),
 &\text{if $G=\Sp_{2n}$ or $\RO_{2n}^{\pm}$},\\
(\epsilon\epsilon_1,\epsilon_1\epsilon_2,\ldots,
 \epsilon_{d-1}\epsilon_d,\epsilon_d),
 &\text{if $G=\RO_{2n+1}^{\epsilon}$}.
\end{cases}
\end{equation}
When $d=0$, in the odd orthogonal case
$\Phi_{\mathcal F}$ is understood to take the unique element of
$I_{\mathcal F}^{+}$ to the one-coordinate vector $(\epsilon)$.
Here the coordinates on the right-hand side are the corresponding
Sommers coordinates. In particular, in the odd orthogonal case,
$\Phi_{\mathcal F}$ takes values in
$\overline A_{\RO}(\CO^{\rm st}_{\mathcal F})$ rather than
$\overline A_{\SO}(\CO^{\rm st}_{\mathcal F})$, and depends on the
quadratic form defining $G$. In the odd orthogonal case, only full orthogonal groups are considered
below; we therefore write $\overline A$ for $\overline A_{\RO}$.

\begin{thm}\label{un2}
Let
\(
\pi=\pi_{\SCO,x^+,x^-,\eta}
\)
be an irreducible unipotent representation of $G^F$ belonging to a
family $\mathcal F$, where $\SCO$ corresponds to the unipotent support
of $\pi$. The parameter $\eta$ is present only in the odd orthogonal
case. Let
\(
\pi'
=
\pi_{\CO^{\prime\mathrm{st}},x^{\prime+},x^{\prime-},\eta'}
\)
be the Alvis--Curtis dual of $\pi$, belonging to a family
$\mathcal F'$. Then
\[
\ScO(\pi')^{\max}
=
\mathscr O\bigl(\Phi_{\mathcal F}(x^+)\bigr),
\qquad
\ScO(\pi)^{\max}
=
\mathscr O\bigl(\Phi_{\mathcal F'}(x^{\prime+})\bigr).
\]
\end{thm}

\begin{proof}
We prove the first equality, giving the details for symplectic groups.
Assume first that $G=\Sp_{2n}$. Write
\(
\pi=\pi_\Lambda.
\)
Then
\(
\pi'=\pi_{{}^{\mathbf d}\Lambda}.
\)
By \cite{Lu92}, the stable orbit containing
$\ScO(\pi')^{\max}$ is $\SCO$. Hence Theorem~\ref{un1} gives
$x\in\overline A(\SCO)$ such that
\[
\ScO(\pi')^{\max}=\mathscr O(x).
\]
We prove
\[
x=\Phi_{\mathcal F}(x^+)
\]
by induction on the number of entries of $\Lambda$. The initial cases, in which $\Lambda$ contains at most one entry, follow directly from the definitions.

Following Section~\ref{sec:L-s}, let $\Lambda^2$ be obtained from
$\Lambda$ by removing its two largest entries. Write
\[
A=(a_1,\ldots,a_k),
\qquad
B=(b_1,\ldots,b_m).
\]
With $a_0\geq b_0$, there are four possibilities:
\begin{align*}
\textup{(I)}\quad
\Lambda&=
\begin{pmatrix}
a_0,A\\
b_0,B
\end{pmatrix},
&
\textup{(II)}\quad
\Lambda&=
\begin{pmatrix}
b_0,A\\
a_0,B
\end{pmatrix},
\\
\textup{(III)}\quad
\Lambda&=
\begin{pmatrix}
a_0,b_0,B\\
A
\end{pmatrix},
&
\textup{(IV)}\quad
\Lambda&=
\begin{pmatrix}
B\\
a_0,b_0,A
\end{pmatrix}.
\end{align*}
In \textup{(III)} and \textup{(IV)}, strictness of the rows forces
$a_0>b_0$. In \textup{(I)} and \textup{(II)},
$\Lambda^2=\binom{A}{B}$ has defect congruent to $1$ modulo $4$.
In \textup{(III)} and \textup{(IV)},
$\Lambda^2=\binom{B}{A}$, while
$(\Lambda^2)^t=\binom{A}{B}$ has defect congruent to $1$ modulo $4$.
Accordingly, put
\[
\pi_2=
\begin{cases}
\pi_{\Lambda^2},
   &\text{in \textup{(I)} and \textup{(II)}},\\
\pi_{(\Lambda^2)^t},
   &\text{in \textup{(III)} and \textup{(IV)}}.
\end{cases}
\]
Write
\(
\pi_2=\pi_{\CO^{\rm st}_2,x_2^+,x_2^-},
\)
let $\mathcal F_2$ be its family, and let $\pi_2'$ be its
Alvis--Curtis dual. By Theorem~\ref{un1} and the induction hypothesis,
there is $x_2\in\overline A(\CO^{\rm st}_2)$ such that
\[
\ScO(\pi_2')^{\max}
=
\mathscr O(x_2),
\qquad
x_2=\Phi_{\mathcal F_2}(x_2^+).
\]
Write
\[
x_2^+=(\epsilon_1,\ldots,\epsilon_r)
\qquad\text{and}\qquad
x_2=(\gamma_1,\ldots,\gamma_r)
\]
in family and Sommers coordinates, respectively. When $r=0$, both vectors are empty, and multiplication of the first
coordinate below is understood to have no effect. For $r\geq1$, we have
\begin{equation}\label{eq:un2-small-coordinates}
\gamma_1=\epsilon_1,
\qquad
\gamma_i=\epsilon_{i-1}\epsilon_i
\quad
(2\leq i\leq r).
\end{equation}

For the induction step, we compare the parameters attached to
$\Lambda$ with those attached to the smaller symbol $\Lambda^2$ or $(\Lambda^2)^t$ defining $\pi_2$,
for which the desired identity is known by the induction hypothesis. On the family side this is the change from
$x_2^+$ to $x^+$; on the rational wavefront side it is the change
from the canonical-quotient parameter $x_2$ of
$\ScO(\pi_2')^{\max}$ to the parameter $x$ of
$\ScO(\pi')^{\max}$. 

{\bf Rational wavefront side:}
Set
\[
c=
\begin{cases}
a_1,&\text{in \textup{(I)} and \textup{(III)}},\\
b_1,&\text{in \textup{(II)} and \textup{(IV)}},
\end{cases}
\qquad
\sigma=-\zeta^0(\Lambda)\zeta^1(\Lambda).
\]
Here $c$ is the only remaining entry that can equal $b_0$; if the
corresponding row is empty, set $c=-\infty$. Moreover,
$\sigma=+$ in \textup{(I)}, \textup{(II)}, and $\sigma=-$ in
\textup{(III)}, \textup{(IV)}.

Since
\(
{}^{\mathbf d}({}^{\mathbf d}\Lambda)=\Lambda,
\)
the three largest entries of $\Lambda$ are the entries denoted by
$c_1,c_2,c_3$ when Proposition~\ref{un} is applied to $\pi'$.
Proposition~\ref{un}\textup{(ii)} and the calculations in the proof
of Theorem~\ref{un1} give
\begin{equation}\label{eq:un2-transition}
x=
\begin{cases}
(\sigma,\gamma_1,\ldots,\gamma_r),
   &a_0>b_0>c,\\
(\gamma_1,\ldots,\gamma_r),
   &a_0=b_0,\\
(\sigma\gamma_1,\gamma_2,\ldots,\gamma_r),
   &a_0>b_0=c.
\end{cases}
\end{equation}
Indeed, as in the proof of Theorem~\ref{un1}, in the case of $a_0>b_0>c$, the ordered generators of
$\overline A(\CO^{\rm st}_2)$ correspond to the second and subsequent ordered
generators of $\overline A(\SCO)$, so $\overline A(\SCO)$ has one
additional leading generator, whose coordinate is $\sigma$. In the
case of $a_0=b_0$, the ordered generators of the two canonical quotients are
identified in the same order, since the deleted odd pair contributes
no new generator. In the case of $a_0>b_0=c$, they are again identified in the
same order, and \eqref{S} multiplies the first Sommers coordinate by
$\sigma$ when such a coordinate exists.

{\bf Family side:} 
By a retained single we mean a single of the smaller symbol $\Lambda^2$ or $(\Lambda^2)^t$ defining
$\pi_2$ that remains a single in $\Lambda$. If $a_0>b_0>c$ or
$a_0=b_0$, every single of the smaller symbol is retained. If
$a_0>b_0=c$, the single $c$ in $\Lambda^2$ or $(\Lambda^2)^t$ forms a double with $b_0$ in $\Lambda$
and is replaced in the ordered list of singles in $\Lambda$ by the new single
$a_0$; we call $a_0$ the replacement single.

We identify each retained single with the same entry in $\Lambda$ and,
when $a_0>b_0=c$, identify $c$ with the replacement single $a_0$.
Thus $x_2^+$ may be viewed as an ordered sign vector indexed by the
corresponding singles of $\Lambda$. Recall that the row-agreement sign
of a single is $+$ or $-$ according as it lies in the same or the
opposite row of the special symbol, and by the coordinate formula in
Section~\ref{family}, the $i$-th coordinate of
$x^+$ is the product of the signs of the first $2i$ ordered singles.

 We shall prove
\begin{equation}\label{eq:un2-symbol-transition}
x^+
=
\begin{cases}
(+,x_2^+),
   &\text{in \textup{(I)}, \textup{(II)} if }a_0>b_0>c,\\
x_2^+,
   &\text{in \textup{(I)}, \textup{(II)} if }
      a_0=b_0\text{ or }a_0>b_0=c,\\
(-,-x_2^+),
   &\text{in \textup{(III)}, \textup{(IV)} if }a_0>b_0>c,\\
-x_2^+,
   &\text{in \textup{(III)}, \textup{(IV)} if }a_0>b_0=c.
\end{cases}
\end{equation}
Here $-x_2^+$ denotes the coordinatewise negation of the sign vector
$x_2^+$.

In Cases~\textup{(I)} and \textup{(II)}, the retained singles keep
their row-agreement signs. In the strict case $a_0>b_0>c$, the two
new leading singles have signs $+,+$ in Case~\textup{(I)} and
$-,-$ in Case~\textup{(II)}; in either case their product is $+$.
Hence
\[
x^+=(+,x_2^+).
\]
If $a_0=b_0$, then $a_0,b_0$ form a double. If $a_0>b_0=c$, then
$b_0,c$ form a double and $a_0$ replaces $c$ with the same
row-agreement sign. Thus in both equality cases
\[
x^+=x_2^+.
\]

In Cases~\textup{(III)} and \textup{(IV)}, every retained
row-agreement sign is reversed. If $a_0>b_0>c$, the two new leading
singles have product $-$. Since each coordinate of $x_2^+$ involves
an even number of retained singles, their sign reversals do not change
the corresponding products. The leading contribution $-$ therefore
changes every subsequent coordinate, and
\[
x^+=(-,-x_2^+).
\]
If $a_0>b_0=c$, then $b_0,c$ form a double and $a_0$ replaces $c$
with the same row-agreement sign. Among the first $2i$ singles, the
other $2i-1$ signs are reversed, so every coordinate changes sign:
\[
x^+=-x_2^+.
\]
This proves \eqref{eq:un2-symbol-transition}.

Substituting \eqref{eq:un2-small-coordinates} into
\eqref{eq:un2-transition} and using
\eqref{eq:un2-symbol-transition} and
\eqref{eq:symbol-Sommers}, we obtain
\[
x=\Phi_{\mathcal F}(x^+).
\]
Consequently,
\[
\ScO(\pi')^{\max}
=
\mathscr O\bigl(\Phi_{\mathcal F}(x^+)\bigr).
\]

For even orthogonal groups, the same argument applies with
Proposition~\ref{un}\textup{(iii)} in place of
Proposition~\ref{un}\textup{(ii)}.

For odd orthogonal groups, the preceding argument applies with the
following changes. On the rational wavefront side, one uses
Proposition~\ref{un}\textup{(iv)} in place of
Proposition~\ref{un}\textup{(ii)}, while
Proposition~\ref{un}\textup{(i)} supplies the initial factor
$\epsilon$ determined by the quadratic form defining $G$. On the
family side, by the coordinate description in Section~\ref{family},
the $i$-th family coordinate is the product of the row-comparison
signs of the first $2i-1$ ordered singles. The same case-by-case comparison as above gives the first $d$
Sommers coordinates
\[
\epsilon\epsilon_1,\epsilon_1\epsilon_2,\ldots,
\epsilon_{d-1}\epsilon_d.
\]
For the full odd orthogonal group, it follows from \eqref{S} that the
product of all Sommers coordinates is the discriminant of the quadratic form defining $G$, and hence equals $\epsilon$. Therefore the last
coordinate is $\epsilon_d$. Thus
\[
x
=
(\epsilon\epsilon_1,\epsilon_1\epsilon_2,\ldots,
 \epsilon_{d-1}\epsilon_d,\epsilon_d)
=
\Phi_{\mathcal F}(x^+).
\]
The parameter $\eta$ does not enter this calculation.

Thus the first equality holds in all cases. Since Alvis--Curtis duality is involutive, applying the first equality
to $\pi'$ proves the second equality.

\end{proof}

\begin{rmk}
In the odd orthogonal case, although $\Phi_{\mathcal F}$ depends on
the quadratic form defining $G$, its image under the natural projection
\(
\overline A_{\RO}\longrightarrow\overline A_{\SO}
\)
is independent of this choice. Thus the two realizations of the same
representation give the same element of the special orthogonal
canonical quotient.
\end{rmk}

\subsection{Isolating rational nilpotent orbits}

Although the maximal rational wavefront sets of unipotent
representations admit a concise description in terms of Lusztig's
canonical quotients, the situation can be much more complicated for
general representations, even for quadratic unipotent ones. The
following result illustrates the flexibility that already occurs in
this class.

\begin{thm}
For every $F$-rational nilpotent orbit $\CO\subset\fg^F$, there exists
an irreducible quadratic unipotent representation
$\pi\in\Irr(G^F)$ such that
\[
\SOM=\{\CO\}.
\]
\end{thm}

\begin{proof}
We prove the symplectic case by induction
on $r$, the length of a good descent sequence index. The case $r=0$
is immediate. Recall from Section~\ref{sec4.1} that
$\mathfrak G(\CO)$ denotes the set of good descent sequence indexes
corresponding to $\CO$. Choose
\[
\Gamma=\llb(2\ell_1,a_1),\ldots,(2\ell_r,a_r)\rrb
\in\mathfrak G(\CO),
\]
and let $\CO'$ be the orbit corresponding to the good descent sequence
index obtained from $\Gamma$ by deleting its first term. By induction, there exists a
quadratic unipotent representation
$\pi'=\pi_{\Lambda_1',\Lambda_{-1}'}$ whose unique orbit in $\ScO(\pi')^{\rm max}$ is $\CO'$. Write
\[
{}^{\bd}\Lambda'_1=
\begin{pmatrix}
a^+_1,\ldots,a^+_{k^+}\\
b^+_1,\ldots,b^+_{m^+}
\end{pmatrix},
\qquad
{}^{\bd}\Lambda'_{-1}=
\begin{pmatrix}
a^-_1,\ldots,a^-_{k^-}\\
b^-_1,\ldots,b^-_{m^-}
\end{pmatrix}.
\]
Set
\[
{}^{\bd}\Lambda_1=
\begin{cases}
\begin{pmatrix}
a,a^-_1,\ldots,a^-_{k^-}\\
b^-_1,\ldots,b^-_{m^-}
\end{pmatrix},
&\mathrm{def}({}^{\bd}\Lambda'_{-1})\equiv0\pmod4,\\[6pt]
\begin{pmatrix}
a^-_1,\ldots,a^-_{k^-}\\
a,b^-_1,\ldots,b^-_{m^-}
\end{pmatrix},
&\mathrm{def}({}^{\bd}\Lambda'_{-1})\equiv2\pmod4,
\end{cases}
\]
and choose ${}^{\bd}\Lambda_{-1}$ from
\[
\begin{pmatrix}
b,a^+_1,\ldots,a^+_{k^+}\\
b^+_1,\ldots,b^+_{m^+}
\end{pmatrix},
\qquad
\begin{pmatrix}
a^+_1,\ldots,a^+_{k^+}\\
b,b^+_1,\ldots,b^+_{m^+}
\end{pmatrix}
\]
so that the first descent sign is $a_1$. Choose $a$ and $b$ so that
the resulting symbols have the required ranks and the first descent
has index $(2\ell_1,a_1)$ with constituent $\pi'$. The existence of
such choices follows from Theorem~\ref{ggp}\textup{(i)}.

The explicit first-descent formula, together with
Corollary~\ref{5.6} in the equal-head case and the induction
hypothesis for $\pi'$, shows that $\Gamma$ is a maximal descent
sequence index of $\pi=\pi_{\Lambda_1,\Lambda_{-1}}$ and that every
maximal descent sequence index of $\pi$ belongs to
$\mathfrak G(\CO)$. Theorem~\ref{max}\textup{(ii)} proves the
symplectic case.

Now let $G=\RO(V')$. By Proposition~\ref{orbit}, deleting one box
from every row of $\CO$, while preserving the rational forms, gives
\[
\CO_{\mathrm{sp}}
=\nabla^{\rm gen}_{V',V}(\CO)
\subset\mathfrak{sp}(V)^F.
\]
By the symplectic case, choose a quadratic unipotent
$\sigma\in\Irr(\Sp(V)^F)$ whose unique 
orbit in \(\ScO(\sigma)^{\rm max}\) is $\CO_{\mathrm{sp}}$. Since $\dim V<\dim V'$, the theta lifting from $\Sp(V)^F$ to $\RO(V')^F$  is nonzero:
$\Theta(\sigma)\neq0$. By Theorem~\ref{gz}\textup{(i)}, there is an
irreducible constituent $\pi$ of $\Theta(\sigma)$ such that
$\CO\in\ScO(\pi)$. By Theorem~\ref{thm:L-can}, $\pi$ is quadratic
unipotent.

For any $\widetilde{\CO}\in\ScO(\pi)$,
Theorem~\ref{gz}\textup{(ii)} places $\widetilde{\CO}$ in the image
of the moment map, while Theorem~\ref{gz}\textup{(i)} shows that its
generalized descent belongs to $\ScO(\sigma)$ and hence is at most
$\CO_{\mathrm{sp}}$. Since $\CO$ has exactly
$\dim V'-\dim V$ rows, Proposition~\ref{8.4} gives
$\widetilde{\CO}\leq\CO$. Thus $\CO$ is the unique
maximal rational orbit in $\ScO(\pi)$.
\end{proof}

\end{document}